\documentclass[11pt,reqno]{amsart}
\usepackage{amsmath}
\usepackage{amsfonts}
\usepackage{mathrsfs}
\usepackage{amssymb}
\usepackage{amsthm,mathabx}
\usepackage{url}
\usepackage{setspace}     \spacing{1}
\usepackage{enumitem}
\usepackage{bbm}
\usepackage{times}
\usepackage{soul}
\usepackage{pstricks}
\usepackage{graphicx}

\numberwithin{equation}{section}

\usepackage{caption}
\usepackage{makecell}
\usepackage{float}
\restylefloat{table}

\usepackage{alltt}    

\def \Aut{\mathrm{Aut}}
\def\calA{\mathcal A}
\def\diag{\mathrm{diag}}
	\def\av{\mathrm{top}}
\def\rank{\mathrm{rank}}
\def\vol{\mathrm{vol}}

\renewcommand{\emph}[1]{{\bf #1}}
\newcommand{\starsection}[1]{ \refstepcounter{section} \addcontentsline{toc}{section}{\hspace*{-.75em}$\bigast$\thesection.\hspace{1.em}#1}%
{\vspace{.7\linespacing}\normalfont\scshape
\hspace{.375em} $\bigast$\thesection.\hspace{.2em} #1}\par\vspace{.5\linespacing}}
\newcommand{\starsubsection}[1]{ \refstepcounter{subsection} \addcontentsline{toc}{subsection}{\hspace*{-.75em}$\bigast$\thesubsection.\hspace{1em}#1}%
{\vspace{1em} \par\noindent $\bigast$\thesubsection. \bfseries #1.}
} 

\usepackage[makeindex]{indextools}
\indexsetup{level=\part*,othercode=\small}

\usepackage[colorlinks=true,linkcolor=blue,citecolor=blue,linktocpage=true]{hyperref}

\usepackage[capitalize]{cleveref}

\usepackage[toc, title]{appendix}
\usepackage[abbrev,nobysame]{amsrefs}

\usepackage{caption}
\usepackage[labelformat=simple]{subcaption}

\def\Folner{F{\o}lner }

\theoremstyle{Theorem}
\newtheorem{theorem}{Theorem}[section]

\newtheorem{proposition}[theorem]{Proposition}

\newtheorem{claim}[theorem]{Claim} 
\newtheorem{lemma}[theorem]{Lemma}
\newtheorem{conjecture}[theorem]{Conjecture}
\newtheorem*{conjecture*}{Conjecture}
\newtheorem{corollary}[theorem]{Corollary}

\theoremstyle{definition}
\newtheorem{definition}[theorem]{Definition}
\newtheorem{question}[theorem]{Question}
\newtheorem{questions}[theorem]{Questions}
\newtheorem{example}[theorem]{Example}

\newtheorem{remark}[theorem]{Remark}

\theoremstyle{remark}
\theoremstyle{theorem}
  
\newtheorem{innercustomthm}{Theorem}
\newenvironment{customthm}[1]
  {\renewcommand\theinnercustomthm{#1}\innercustomthm}
  {\endinnercustomthm}

\usepackage{mdframed}
\newlist{enumlemma}{enumerate}{3} 
\setlist[enumlemma]{label*={ (\alph*)}, ref= {(\alph*)} }
\newlist{enumcount}{enumerate}{3} 
\setlist[enumcount]{label*={ (\arabic*)}, ref= {(\arabic*)} }

\usepackage{marginnote}

\newcommand{\diff}{\mathrm{Diff}}
\newcommand{\Diff}{\diff}

\newcommand{\supp}{\mathrm{supp}}
\newcommand{\td}{\tilde}

\newcommand{\id}{\mathrm{Id}}

\newcommand{\restrict}[2]{{#1}{\restriction_{{ #2}}}}
\newcommand{\Fol}{\mathcal{F}}

\newcommand{\diam}{\mathrm{diam}}

\newcommand{\sm}{\smallsetminus}
\newcommand{\C}{\mathbb {C}}
\newcommand{\R}{\mathbb {R}}
\newcommand{\Q}{\mathbb {Q}}
\newcommand{\Z}{\mathbb {Z}}
\newcommand{\N}{\mathbb {N}}
\newcommand{\T}{\mathbb {T}}
\newcommand{\vect}[1]{\vec{ #1}}
\newcommand{\inv}{^{-1}}

\newcommand{\Sl}{\mathrm{SL}}
\def\SL{\Sl}
\newcommand{\Gl}{\mathrm{GL}}
\def\GL{\Gl}

\newcommand{\SO}{\mathrm{SO}}
\def\So{\SO}
\newcommand{\SU}{\mathrm{SU}}
\newcommand{\su}{\mathfrak{su}}
\newcommand{\Sp}{\mathrm{Sp}}

\def\calE{\mathcal E}

\def\calG{\mathcal G}

\def \RP{\R P}

 \newcommand{\oldepsilon}{\mathchar"10F}
\newcommand{\eps}{\oldepsilon}

\newcommand{\fakeSS}[1]{\medskip \noindent{\bf #1.}}

\def\ae{a.e.\ }

\def\bs{\backslash}

\newcommand{\lieg}{\mathfrak g}
\newcommand{\lieh}{\mathfrak h}

\newcommand{\lien}{\mathfrak n}
\newcommand{\liea}{\mathfrak a}
\newcommand{\lie}{\mathrm{Lie}}

\def\calB{\mathcal B}

\renewcommand{\bar}{\overline}
\def\erg{\mathcal E}

\makeindex

\title{Entropy, Lyapunov exponents, and rigidity of group actions} 
\author[A.~Brown (et al.)]{Aaron Brown \\
\\
with 4 appendices by Dominique Malicet, Davi Obata, Bruno Santiago and Michele Triestino, S\'ebastien Alvarez and Mario Rold\'an\\
\\
Edited by Michele Triestino}

\long\def\symbolfootnote[#1]#2{\begingroup\def\thefootnote{\fnsymbol{footnote}}
\footnote[#1]{#2}\endgroup}

\usepackage{bbm}

\def\vect{\mathbf t}
\def\vecs{\mathbf s}
\def\vecv{\mathbf v}

\def\vecn{\mathbf n}
\def\vecm{\mathbf m}

\newcommand{\dfcn}[5]{\setlength{\arraycolsep}{1.5pt}\begin{array}{cccc}#1:&#2&\to&#3\\&#4&\mapsto&#5\end{array}}

\newcommand{\A}{\mathcal A}
\newcommand{\W}{W}
\newcommand{\B}{\mathcal B}
\newcommand{\Pp}{\mathcal P}
\newcommand{\Jac}{{\rm Jac}}
\newcommand{\essinf}{{\rm ess\,inf}}
\newcommand{\esssup}{{\rm ess\,sup}}
\newcommand{\HD}{{\rm HD}}
\newcommand{\vide}{\emptyset}
\def\lims{\mathop{\overline{\rm lim}}}
\def\limi{\mathop{\underline{\rm lim}}}
\def\dans{\mathop{\subset}}

\begin{document}


\maketitle

	\pagenumbering{roman}

	\begin{abstract}
	This  text is an expanded series of lecture notes based on a 5-hour course given at the workshop entitled  {\it Workshop for young researchers: Groups acting on manifolds} held in Teres\'opolis, Brazil in June 2016.  
	The  course  introduced a number of classical tools in smooth ergodic theory---particularly Lyapunov exponents and metric entropy---as tools to study  rigidity properties of group actions on manifolds.  
	
	We do not present a comprehensive treatment of group actions or general rigidity programs.  Rather, we focus   on two rigidity results in  higher-rank dynamics: the measure rigidity theorem for  affine Anosov abelian actions on tori due to A. Katok and R. Spatzier and recent the work of the author with D. Fisher, S. Hurtado, F. Rodriguez Hertz, and Z. Wang on  actions of lattices in higher-rank semisimple Lie groups on manifolds  We give complete proofs of these results and present sufficient background in smooth ergodic theory needed for the proofs. A unifying theme in this text is the use of metric entropy and its relation to the geometry of conditional measures along foliations  as a mechanism to verify invariance of measures.   
	
	\smallskip
	{\noindent\footnotesize \textbf{MSC\textup{2010}:} Primary 22F05, 22E40. Secondary 37D25, 37C85.}
	\end{abstract}
	
	
	\newpage

	\setcounter{tocdepth}{2}
	
	\phantomsection
\addcontentsline{toc}{part}{Front matter}	
	\phantomsection
\addcontentsline{toc}{section}{Contents}
\tableofcontents
	
	\newpage

\section*{Preface}
This collective work originates from the workshop for young researchers \textit{Groups acting on manifolds} held in Teres\'opolis, June 20-24 2016. The structure of the text respects the format of the event: Aaron Brown was lecturing about rigidity of smooth group actions, but other talks were given by young researchers and devoted to explain background notions. We thank all the participants, whose active presence contributed to a very pleasant scientific experience.  The event was sponsored by the French-Brazilian Network in Mathematics, CNPq, UFF and EDAI.

\medskip

Smooth ergodic theory takes a prominent r\^ole in the modern theory of differentiable dynamical systems.  Typically one is interested in studying iterations of one single map, or a flow; this can be done under many different aspects, one of the most successful being the description of statistical behaviors, or the ergodic properties. That is, given a probability measure which is invariant for the dynamics, one wants to describe the distribution of the orbits, the typical rates of contraction/expansion, etc. One of the difficulties, or the richness, is that a given system usually admits different invariant measures, and for each measure the statistical descriptions may differ in a significant way. The book \cite{BDV} was giving a good account of the state of the art about a decade ago (see also \cites{AA,Mane} as more classical references).

The study of dynamical systems can be enlarged to include also general \textbf{group actions} on manifolds. This was historically motivated by geometry and foliation theory, but since the appearance of hyperbolic dynamical systems, dynamicists found an own interest for it. A starting point for what will be discussed in this text, goes back to works of Hirsch, Pugh and Shub \cites{HPS,PS}, where the notion of \textbf{Anosov action} first appeared. 
Not all isomorphism classes of groups are adapted to usual dynamical tools, and one usually restricts the attention to abelian, or amenable groups, or at least to groups containing ``large'' abelian subgroups.
Compared to the previous discussion, a first relevant difference for group actions is that 
invariant probability measures (for the full group!) in general do not abound. In the first part of this work, Aaron Brown presents a pioneering result by Katok and Spatzier \cites{KS1,KS2} stating that, under suitable hypotheses, invariant ergodic probability measures for the action of a higher rank abelian group must be of \textit{algebraic nature}; see Theorem \ref{thm:KS}. Indeed, the main source of examples of Anosov actions is given by Lie groups, notably by discrete diagonalizable groups of matrices acting on homogeneous spaces. One usually refers to these kinds of results as \textbf{measure rigidity}. 

\smallskip

The term \textit{rigidity} here will also refer to a different, but closely related setting: the so-called \textbf{Zimmer program}. Before explaining it, let us make a preliminary digression. Instead of discrete groups, we first focus on Lie groups. The celebrated Montgomery--Zippin theorem \cite{MZ} tells that the topological structure is intimately connected to geometry; one can go even further:

\begin{conjecture*}[Hilbert--Smith]
	If a locally compact topological group $G$ acts faithfully on
	some connected $n$-manifold $M$, then $G$ is a Lie group.
\end{conjecture*}

This conjecture has  been validated only for dimension $n\le 3$ \cite{Pardon} (actually, the conjecture reduces to prove that a group of $p$-adic integers $\Z_p$ admits no faithful actions on manifolds). One may further ask, given a Lie group $G$, what is the lowest dimension $n(G)$ such that $G$ admits a faithful action on a manifold of dimension $n$? For example, the group $\mathrm{PSL}(n,\R)$ acts on the real projective space $\R P^{n-1}$, which is of dimension $n-1$, but it cannot act on a lower dimensional space. In general, for a simple Lie group $G$, the optimal lower bound depends on the maximal parabolic subgroup of $G$ \cite{Stuck}, but a satisfactory bound can be given in terms of the so-called \textit{(real) rank} of $G$ (this is by definition, the dimension of the largest abelian subalgebra $\mathfrak{A}$ of $\mathrm{Lie}(G)$ which satisfies that for every $a\in\mathfrak A$, the adjoint operator $ad(a):\mathrm{Lie}(G)\to\mathrm{Lie}(G)$ is diagonalizable over $\R$, see also Section \ref{sec:rank}). Since the fundamental work of Margulis \cite{Margulis}, it is natural to consider the same kind of question for \textit{lattices} (i.e.\ discrete, finite covolume subgroups) in simple Lie groups. The so-called \textit{Margulis's superrigidity} roughly states that every linear representation of a lattice of a simple Lie group of rank $\ge 2$ extends to the ambient group (more precisely, modulo finite subgroups and up to some bounded error, see Theorem~\ref{thm:MSR}), and therefore all linear representations of lattices are classified. Zimmer program is about the \textit{nonlinear} analogue of superrigidity.

\begin{conjecture*}[Zimmer]\label{c:zimmer}\index{conjecture!Zimmer}
	Let $G$ be a simple Lie group with $\mathrm{rank}(G)\ge 2$ and $\Gamma$ a lattice of $G$. Let $M$ be a closed $d$-dimensional manifold, and $\rho:\Gamma\to\mathrm{Diff}(M)$ a homomorphism. If $d< \mathrm{rank}(G)$ then  $\rho$ has finite image.
\end{conjecture*}

More generally, it is conjectured that a lattice $\Gamma\subset G$ cannot act (with infinite image) on a closed manifold whose dimension is lower than the least dimension $n(G)$ introduced above. For detailed discussions, we recommend the expository works by David Fisher \cites{F1,F2}.
Very recently, Aaron Brown, in collaboration with David Fisher and Sebastian Hurtado, solved Zimmer's conjecture \cites{BFH,BFH-SL} (some additional hypotheses are required, see Theorem \ref{slnr}). The third part of these notes discusses the main ingredients of their work. As for measure rigidity, the theory of nonuniformly hyperbolic dynamical systems takes the major part. The works by Ledrappier--Young \cites{LY1,LY2} on the relationship between \textit{Lyapunov exponents} and \textit{entropy} are of notable relevance here. These are very deep works, but we hope the reader will find a suitable introduction here. The main notions and results are recalled throughout the text, especially in the second part and in the appendices by Bruno Santiago and myself, Davi Obata, S\'ebastien Alvarez and Mario Rold\'an. Another important ingredient of similar flavor is the work of Ratner on unipotent flows, but we do not treat it in detail, as several very good introductions are available (in primis \cite{WitteRatner}). 

One delicate aspect in Zimmer's conjecture is about \textit{regularity}: in \cites{BFH,BFH-SL} the authors require the action to be by $C^2$ diffeomorphisms, that is $\rho(\Gamma)\subset \mathrm{Diff}^2(M)$. This is a very mild condition, at least compared to the previous approaches appearing in the literature, which had strong requirements such as invariant volume or geometric structures, or the action to be by real-analytic diffeomorphisms or on low dimensional manifolds \cites{CZ,Zimmer,FS,Ghys,WitteLine,BM,Polterovich,FH1,FH2}. We recommend the beautiful collection of contributions \cite{FF} for more detailed discussions on these works (and much more!).

As this text is more focused on smooth ergodic theory and applications to rigidity properties, we will shortly mention the other aspects of Zimmer's conjecture. These include the algebraic properties of Lie groups and their lattices, especially their \textit{rigidity} properties. An essential ingredient of \cites{BFH,BFH-SL} is the \textit{strong property (T)}, introduced by V.~Lafforgue \cite{Lafforgue} and studied by de Laat--de la Salle \cites{dldlS,dlS} which generalizes the more classical Kazhdan's property (T) \cite{BHV} and is also enjoyed by lattices in higher-rank simple Lie groups. Very roughly, strong property (T) is a machine to produce invariant vectors for representations as operators of Banach spaces, and is used by Brown--Fisher--Hurtado to reduce the nonlinear problem to a linear one, and then apply superrigidity.

Finally, in the fourth and last part, Aaron Brown discusses further rigidity results that rely on tools of smooth ergodic theory, and which are, in some sense, extensions of Katok--Spatzier \cite{KS1,KS2} discussed in the first part. First, he reviews theorems by Kalinin, Katok and F.~Rodriguez-Hertz \cite{prKK,prKRH,prKKRH} about rigidity problems for non-uniformly hyperbolic $\Z^2$ actions on general 3-manifolds. Secondly \textit{Cartan flows}, which constitute an other important class of Anosov actions, are treated. Among these, a classical example is the action of the group $A\subset \mathrm{SL}(3,\R)$ of diagonal matrices with positive entries, on the homogeneous space $X=\mathrm{SL}(3,\R)/\mathrm{SL}(3,\Z)$, which has dimension 8. An analogue of Katok--Spatzier result in this setting is motivated by a conjecture by Margulis (see \cite{prMargulis}) that $A$-invariant ergodic measures should be algebraic: this was solved by Einsiedler--Katok--Lindenstrauss \cite{prEKL} for measures of positive entropy.

Sections that are not required for the rest of the text or that may be skipped on first reading are marked  by an asterisk $\bigast$.

\begin{flushright}
	Michele Triestino (editor)
\end{flushright}

\bigskip

\noindent
{\scshape Michele Triestino}\\
Institut de Math\'ematiques de Bourgogne (IMB, UMR 5584)\\
9 av.~Alain Savary, 21000 Dijon, France\\
\texttt{michele.triestino@u-bourgogne.fr}
	
\newpage	

\pagenumbering{arabic}
\part*{Introduction}

\section{Groups acting on manifolds and rigidity programs}
In the classical theory of dynamical systems, one typically studies actions of 1-parameter groups:  Given a compact manifold $M$, a diffeomorphism $f\colon M\to M$ 
generates an action of the group $\Z$;  a smooth vector field $X$ on $M$ generates a flow $\phi^t\colon M\to M$ or an action of the group $\R$.   
However, one might consider groups more general than $\Z$ or $\R$ acting on a manifold $M$.  Natural families  of group  actions  arise in many geometric and algebraic settings and the  study of  group actions connects many areas of mathematics including    geometric group theory, representation theory,  Lie theory, geometry, and dynamical systems.  

This text focuses on  various {\it rigidity programs} for group actions. Roughly, such rigidity results  aim to classify all actions or   all invariant geometric structures (such as closed subsets, probability measures, etc.) under 
\begin{enumerate}
	\item  suitable algebraic hypotheses on the acting group, and/or
	\item suitable dynamical hypotheses on the action.
\end{enumerate}
This text primarily takes the first approach: under certain algebraic conditions on the acting group, we establish certain rigidity properties of the action.  Specifically,  we will consider actions  of various {\it higher-rank} discrete groups:  higher-rank, torsion-free abelian groups $\Z^k$ for  $k\ge 2$ or lattices $\Gamma$ in higher-rank simple Lie groups such as $\Gamma= \SL(n,\Z)$ for $n\ge 3$.  At times we impose certain dynamical hypotheses as well.  In particular, in Part \ref{part:I} we consider certain families of algebraic Anosov actions and  will discuss more general results on Anosov actions in this introduction.

Neither this introduction nor this text as a whole  gives  a comprehensive account of rigidity results for group actions on manifolds.  (For instance, we do not discuss the vast literature and many recent results concerning group action on the circle $S^1$.  See, however \cite{MR2809110}.) Our goal is rather to give detailed proofs (in simplified examples) of a small number of rigidity theorems coming from higher-rank dynamics and to present  the necessary background and constructions in smooth ergodic theory required for  these proofs.  
This introduction aims to give context to these results and give the reader some familiarity with broader rigidity programs in the literature.  

\subsection{Smooth group actions}\index{group!action}\index{action!group}
Let $M$ be a compact    manifold without boundary and  denote by $\Diff^r(M)$ the group of $C^r$ diffeomorphisms $f\colon M\to M$.   Recall that if $r\ge 1$ is not an integer then, writing $$r= k+ \beta  \quad \text{ for $k\in \N$ and $\beta\in (0,1)$}, $$ we say that  $f\colon M\to M$ is $C^r$ or is $C^{k+\beta}$ if it is $C^{k}$ and if the $k$th derivatives of $f$ are $\beta$-H\"older continuous.  

For $r\ge 1$,  the set $\Diff^r(M)$ has a group structure given by composition of maps. 
Given a (typically countably infinite, finitely generated)  discrete  group $\Gamma$, a  \emph{$ {\boldsymbol {C^r}}$ action} of   $\Gamma$ on $M$ is  a homomorphism 
$$ \alpha\colon \Gamma \to \Diff^r(M)$$
from the group $\Gamma$ into the group $\Diff^r(M)$; that is, for each $\gamma\in \Gamma$ the image $\alpha(\gamma)$ is a $C^r$ diffeomorphism $\alpha(\gamma)\colon M\to M$ and for $x\in M$ and $\gamma_1, \gamma_2\in \Gamma$ we have 
\begin{equation*}
	\alpha(\gamma_1\gamma_2)(x) = \alpha(\gamma_1)\big(\alpha(\gamma_2)(x)\big).\end{equation*}
If the discrete group $\Gamma$ is instead 
replaced by 
a Lie group $G$, we also require  that the map $ G\times M\to M$ given by $ (g,x)\mapsto \alpha(g)(x)$  be $C^r$. %

If $\vol$ is some fixed smooth volume form on $M$ (which we always normalize to be a probability measure)  we write $\diff^r_\vol(M)$ for the group of $C^r$-diffeomorphisms preserving 
$\vol$. A classical result of Moser ensures the $\vol$ is uniquely defined up to a   smooth change of coordinates \cite{Moser}.
  A \emph{volume-preserving} action is a homomorphism $\alpha\colon \Gamma \to  \diff^r_\vol(M)$ for some volume form $\vol$.  \index{action!volume-preserving}

As discussed above, actions of the group of integers $\Z$   are generated by iteration of a single diffeomorphism  $f\colon M\to M$  and its inverse.  For instance, given an integer $n>1$, the diffeomorphism $\alpha(n)\colon M\to M$ is defined as the $n$th iterate of $f$:  for $x\in M$
$$
\alpha(n)(x) = f^n(x):= \underbrace{f\circ f \circ \dots \circ f}_{n \text{ times}} (x).$$   
Given a manifold $M$, any pair of  diffeomorphisms $f,g\in \Diff(M)$  naturally induces an action $\alpha\colon F_2\to\Diff(M)$ of the free group on two generators $F_2 =\langle a,b\rangle$ which is uniquely defined by the conditions $\alpha(a)=f$ and $\alpha(b)=g$.  
If a   pair of diffeomorphisms $f  \colon M\to M$ and $g\colon M\to M$   commute, we naturally  obtain a   $\Z^2$-action $\alpha\colon \Z^2\to \diff(M)$ given  by $$
\alpha(n,m)  (x) = f^n \circ g^m(x).$$

\subsection{Rigidity of Anosov diffeomorphisms}
As a prototype for general rigidity results discussed below, we recall certain rigidity properties exhibited by Anosov diffeomorphisms $f\colon M\to M$.   
We first recall the definition of an Anosov diffeomorphism.  
\begin{definition}\label{def:anosov}\index{Anosov!diffeomorphism}
	A $C^1$ diffeomorphism $f\colon M\to M$ of a compact Riemannian manifold $M$ is \emph{Anosov} if there is a $Df$-invariant splitting of the tangent bundle $TM= E^s\oplus E^u$ and constants $0<\kappa<1$ and $C\ge 1$ such that for every $x\in M$ and every $n\in \N$
	\begin{align*} \|D_xf^n(v) \| \le C \kappa ^n \|v\|  &\quad \quad \text{ for all $v\in E^s(x)$}\\
		\|D_xf^{-n}(w) \| \le C \kappa ^n \|w\|& \quad \quad \text{ for all $w\in E^u(x)$}.
	\end{align*}
\end{definition}
As a primary example, consider  a matrix $A\in \Gl(n,\Z)$ with all eigenvalues of modulus different from $1$.  Then, with $\T^n: = \R^n/\Z^n$ the $n$-torus, the induced toral automorphism  $L_A\colon \T^n\to \T^n$ given by $$L_A(x+\Z^n) = Ax + \Z^n$$ is Anosov.  More generally, given $v\in \T^n$ we have $f\colon \T^n\to \T^n$ given by $$f(x) = L_A (x) + v$$ is an affine Anosov map.  
In  dimension 2, a standard example of an Anosov diffeomorphism is given by $L_A\colon \T^2\to \T^2$ where $A$ is  the matrix $A= \left(\begin{array}{cc}2 & 1 \\1 & 1\end{array}\right).$

As a prototype for {\it local rigidity} results, it is known (see \cites{MR0224110,MR0238357}, \cite[Corollary 18.2.2]{MR1326374}) that Anosov maps are \emph{structurally stable}:\index{rigidity!structural stability}\index{Anosov!structural stability} if $f$ is Anosov and $g$ is $C^1$ close to $f$ then $g$ is also Anosov and there is a homeomorphism $h\colon \T^n\to \T^n$ such that \begin{equation}\label{eq:Anosov1} h\circ g = f\circ h.\end{equation}  The map $h$ is always H\"older continuous but in general need not be $C^1$ even when $f$ and $g$ are $C^\infty$ or real-analytic.   
The map $h$ in \eqref{eq:Anosov1} is called a \emph{topological conjugacy}\index{conjugacy} between $f$ and $g$.

All known examples of Anosov diffeomorphisms occur on finite factors of tori and nilmanifolds.  From \cites{MR0271990,MR0358865} we have a complete classification---a prototype {\it global rigidity} result---of Anosov diffeomorphisms on  tori  (as well as nilmanifolds)  up to a continuous change of coordinates: If $f\colon \T^n\to \T^n$ is Anosov, then $f$ is homotopic to $L_A$ for some 
$A\in \Gl(n,\Z)$ with all eigenvalues of modulus different from $1$; moreover there is a homeomorphism $h\colon \T^n\to \T^n$ such that $$h\circ f = L_A\circ h.$$
Again, the topological conjugacy $h$ is  H\"older continuous but need not be $C^1$.  
Conjecturally, all Anosov diffeomorphisms are,  up to finite covers, topologically conjugate to affine maps on tori and nilmanifolds.

\subsection{Actions of higher-rank lattices and the Zimmer program}  \label{sec:ZZIM}\index{Lie group!lattice}
A principal family of discrete groups considered in this text are  lattices $\Gamma$ in (typically higher-rank, see Section \ref{sec:rank}) simple Lie groups $G$. That is, we consider discrete subgroups $\Gamma\subset G$ such that $G/\Gamma$ has finite volume.  Examples of such groups include $\Gamma = \Sl(n,\Z)$ where $G= \Sl(n,\R)$ (which is higher-rank if $n\ge 3$).  It is well known that   the matrices $$\text{$\left(\begin{array}{cc}1 & 2 \\0 & 1\end{array}\right)$ 
and $\left(\begin{array}{cc}1 & 0 \\2 & 1\end{array}\right)$ }$$  freely generate a finite-index subgroup of $\Sl(2,\Z)$ and thus the free group $F_2$  is a lattice subgroup of  $G= \Sl(2, \R)$ (which has rank $1$.)   See Section \ref{ss:latices} for background and additional  details.  

\subsubsection{Linear representations}
To motivate the  results and conjectures concerning smooth  actions of such $\Gamma$, first consider the setting of linear representations $\rho\colon \Gamma\to \Gl(d,\R)$. A  linear representation $\pi\colon \Z \to \Gl(d,\R)$ of the group of integers  is determined  by a  choice of a matrix $A\in \Gl(d,\R)$; similarly, a linear representation $\pi\colon F_2 \to \Gl(d,\R)$ of  the free group $F_2$ is determined  by a  choice of  a  pair of matrices $A,B\in \Gl(d,\R).$
These representations may be perturbed  to non-conjugate representations $\tilde \pi$.

In contrast, for groups such as  $\Gamma = \Sl(n,\Z)$ for $n\ge 3$ (and other lattices  $\Gamma$  in higher-rank simple Lie groups),  linear representations $\pi\colon \Gamma \to \Gl(d,\R)$   are very rigid as demonstrated by various classical   results including \cite{Selberg, Weil-I, Margulis-nonuniformtwo,Mostow-book,MR0385005,MR1090825}.  
  For instance,  for cocompact $\Gamma\subset \Sl(n,\R)$,  local rigidity results in  \cite{Selberg, Weil-I} established that any representation $\pi\colon \Gamma \to \Sl(n,\R)$ sufficiently close to the inclusion $\iota \colon \Gamma \to \Sl(n,\R)$ 
is   conjugate to $\iota$.  
A cohomological criteria for local rigidity of general linear representations $ \pi\colon \Gamma \to \Gl(d,\R)$  was given in  \cite{MR0169956}, further studied in \cite{MR0153028,MR0173730}, and is known to hold for lattices in higher-rank simple Lie groups.  
Margulis's superrigidity theorem (see Theorem \ref{thm:MSR} below and \cite{MR1090825}) establishes  that every linear  representation $\pi\colon \Gamma \to \Gl(d,\R)$ extends to a representation $\bar \pi\colon \Sl(n,\R) \to \Gl(d,\R)$ up to a ``compact error;''  this  effectively   classifies  all   representations $\Gamma \to \Gl(d,\R)$ up to conjugacy. 


\subsubsection{Smooth actions of lattices} 
As in the case of linear representations,  actions of    $\Z$  or $F_2$  on a manifold $M$ are determined by  a choice of diffeomorphism $f\in \diff^r(M)$ or pair of diffeomorphisms $f,g\in \diff^r(M)$. 
Such actions may be  perturbed to create new actions  that are inequivalent under continuous change of coordinates.  In particular, there is no possible classification of all actions of $\Z$ or $F_2$ on arbitrary  manifolds $M$. 
The free group on two generators  $ F_2$  and the group $\Sl(2,\Z)$ (which contains a finite-index subgroup isomorphic to $F_2$) are isomorphic to lattices in  the Lie group $\Sl(2,\R)$.  Both $F_2$ and $\Sl(2,\Z)$ admit many actions that are ``non-algebraic'' (i.e.\ not built from modifications of algebraic constructions) and the algebraic actions of such groups often display less rigidity then actions of higher-rank groups.  For instance, there exists a 1-parameter family of deformations   of   the standard $\Sl(2,\Z)$-action on the $2$-torus $\T^2$ such that  no continuous change of coordinates conjugates  the  deformed actions to  the original  affine action.  See Examples \ref{ex:free} and \ref{ex:hurder} for further discussion.  

However, as in the case of linear representations, the situation is expected to be very different for actions  by lattices in $\Sl(n,\R)$ for $n\ge 3$ and other higher-rank simple Lie groups.  In particular, the \emph{Zimmer program}\index{Zimmer!program} refers to a collection of conjectures and questions which roughly 
aim to establish analogues of rigidity results for  linear representations $\pi\colon \Gamma \to \Gl(d,\R)$ in the context of    smooth (often volume-preserving) actions  $$\alpha\colon \Gamma \to \diff^\infty(M)$$  or ``nonlinear representations.'' 
In particular, it is expected that  all nonlinear actions $\alpha\colon \Gamma\to \diff^r(M)$ are, in some sense, of ``an algebraic origin.''
We note that genuinely ``non-algebraic'' actions exist; see for instance the discussion in \cref{ex:exotic} and \cite[Sections 9, 10]{MR2807830}.  Thus, a complete classification of all   actions  of higher-rank lattices up to smooth conjugacy is impossible.  However, it seems plausible that certain families of actions (Anosov, volume-preserving, low-dimensional, actions on specific manifolds, actions preserving a geometric structure, etc.)\  are classifiable and that all such actions are constructed from  modifications of standard algebraic actions. See \cref{sec:ex} for   examples of standard algebraic actions.
We refer   to the surveys \cites{MR2369442,MR2807830,1711.07089,MR1648087} for further  discussion on  various notions of ``algebraic actions,''  the Zimmer program,  and precise statements of related conjectures and results.

For volume-preserving actions, the primary  evidence supporting  conjectures  in  the Zimmer program  is   Zimmer's   superrigidity\index{Zimmer!cocycle superrigidity} theorem for  cocycles,  Theorem \ref{thm:ZCSR} below. This extension of Margulis's superrigidity theorem (for homomorphisms) shows that the derivative cocycle of any  volume-preserving action $\alpha\colon \Gamma\to \diff^r_\vol(M)$  is---up to a  compact error and measurable coordinate change---given by a linear representation $\Gamma\to \Gl(d,\R)$.

\subsubsection{Actions in low-dimensions} 
Precise conjectures in the Zimmer program are easiest to formulate for actions in low dimensions.   See in particular  \cref{Q:qq}.  
For instance, if the dimension of $M$ is sufficiently small,
\emph{Zimmer's conjecture}\index{conjecture!Zimmer}  \index{Zimmer!conjecture} states that all actions should have finite image (see \cref{def:dumb}).  See Conjectures \ref{conj:slnr} and \ref{conjecture:zimmergne} for  precise  statements of this conjecture. Early results establishing this conjecture in the setting of actions the circle   appear in  \cites{MR1198459,MR1911660,MR1703323} and in the setting of volume-preserving (and more general measure-preserving) actions on surfaces  in   \cites{MR1946555, MR2026546, MR2219247}.  See also    \cite{MR1254981} and \cite{MR1666834} for results on  real-analytic actions and \cites{MR2103473,MR3014483,MR3849285} for results on holomorphic and birational actions.  There are also many results (usually in the $C^0$ setting) for actions of specific lattices on manifolds where there are topological obstructions  to the group acting; a partial list of such results includes \cites{MR2807834,MR1470739,MR2163900,MR2493377,MR2745276,MR3150210,Ye1,Ye2}.
Part \ref{part:III} of this text presents recent progress towards this conjecture made in \cite{1608.04995}.  

\subsubsection{Local rigidity}
Beyond the finiteness of actions in low dimensions, there are a number of local rigidity  conjectures that aim to classify perturbations of non-finite actions.   
We recall one common definition of local rigidity of a $C^\infty$ group action:
\begin{definition}\label{def:LR} \index{rigidity!local}
	An action $ \alpha\colon \Gamma\to \diff^\infty(M)$ of a finitely generated group $\Gamma$ is said to be \emph{locally rigid} if, for any action $\td \alpha\colon \Gamma\to \Diff^\infty(M)$ sufficiently $C^1$-close   to  $\alpha$, there exists  a $C^\infty$ diffeomorphism $h\colon M\to M$ 
	such that  \begin{equation}\label{eq:conj}h\circ  \td \alpha(\gamma)\circ h\inv  = \alpha(\gamma) \quad \quad \text{ for all $\gamma\in \Gamma$.}\end{equation}
\end{definition}
\noindent

In \cref{def:LR}, using that $\Gamma$ is finitely generated, we define the $C^1$ distance between $\alpha$ and $\td \alpha$ to be $$\max \{ d_{C^1}(\alpha (\gamma), \td \alpha(\gamma))\mid \gamma\in F\}$$
where $F\subset \Gamma$ is a finite, symmetric generating subset.  

Local rigidity results have been established for actions of higher-rank lattices in many settings.  For instance, local rigidity is known to hold  for   isometric actions by \cites{MR1779610,MR2198325}.   
In the non-isometric setting, local rigidity has been established for  affine   Anosov  actions. 
\begin{definition}\label{def:anosov2}\index{Anosov!action}\index{action!Anosov}
	We say an action $\alpha\colon \Gamma \to \diff (M)$ is \emph{Anosov} if $\alpha(\gamma) $ is an Anosov diffeomorphism for some $\gamma\in \Gamma$.  
\end{definition}
\noindent See  Example \ref{ex:standard} and \cref{rem:affineAnosov} for examples of affine Anosov actions of lattices  on tori.

For Anosov actions, note that while structural stability  \eqref{eq:Anosov1} holds for individual Anosov elements of an action, local rigidity requires that the map $h$ in \eqref{eq:conj} intertwines the action of the entire group $\Gamma$; moreover, unlike in the case of a single Anosov map where $h$ is typically only H\"older continuous, we ask that the map $h$ in \eqref{eq:conj} be smooth.

There are a number of results  establishing local rigidity of affine Anosov actions on tori and nilmanifolds including \cites{MR1154597,MR1631740, MR1164591, MR1367273,MR1740993,MR1332408}. The full result on  local rigidity of Anosov actions by higher-rank lattices was obtained in \cite[Theorem 15]{MR1632177}. 
See also related rigidity results including  \cite{MR1154597} for results on deformation rigidity and \cites{MR1154597, MR1338481, MR1058434, MR1415754} for  various infinitesimal   rigidity results.  
Additionally, see \cites{MR2521112,MR1826664} for local rigidity of closely related actions and  \cite{MR1421873}  and \cite[Theorem 17]{MR1632177} for results on the local rigidity of projective  actions by cocompact lattices.

\subsubsection{Global rigidity}
Beyond the study of  perturbations, there are a number of conjectures and results on the  \emph{global rigidity}\index{rigidity!global} of
smooth actions of higher-rank lattices.  
Most global rigidity results in the literature focus on various families of  Anosov actions.  (Though, see  \cref{conj:koko} for a global rigidity conjecture that is not about Anosov actions.)
Such conjectures and results aim to classify  all (typically volume-preserving) Anosov actions by showing they are smoothly conjugate to affine   actions on (infra-)tori and nilmanifolds.  
See for instance \cites{MR1154597,MR1380646, MR1367273,MR1740993, MR1401783,MR1643954,MR1866848, MR1826664} for various  global rigidity results for  Anosov actions.     

Recently, \cite{BRHW1} gave a new mechanism to study rigidity of Anosov actions on tori; in particular, it is shown in  \cite{BRHW1} that all Anosov actions (satisfying a certain lifting condition which holds, for instance, when the lattice is cocompact)  of  higher-rank lattices are smoothly conjugate to affine actions, even when the action is not assumed to preserve a measure.  This provides the most general global rigidity result for Anosov actions on tori and nilmanifolds.

\subsection{Actions of higher-rank abelian groups}\label{sec:HRabelian}
In Part \ref{part:I}, the discrete groups we consider are higher-rank abelian groups of the form $\Z^k$ for $k\ge 2$.  We  focus on certain affine   Anosov actions and aim to classify all invariant measures for   such actions.    

Recall that Anosov diffeomorphisms $f\colon \T^d\to \T^d$ on tori are classified up to continuous changes of coordinates.  Such maps $f$ leave invariant many closed  subsets and probability measures on $\T^d$ of intermediate dimension.  (See Proposition \ref{prop:zoo} and nearby discussion.)  
For Anosov actions   (satisfying certain non-degeneracy conditions) of higher-rank abelian groups $\Z^k$, a number of rigidity results show that properties of higher-rank actions are strikingly different from actions of a single Anosov diffeomorphism.    
We outline some of these results known to hold in  this setting:
\begin{enumerate}
	\item {\it Local rigidity}\index{rigidity!local} results---in which  perturbations of affine Anosov actions are  smoothly conjugate  to the original actions---have been established in \cites{MR1164591,MR1307298,MR1632177} with the most general results appearing in \cite{MR2342454}.  A partial list  of related local rigidity results in the setting of partially hyperbolic actions includes  \cites{MR2726100,MR2838045, MR2672298,MR2753946, MR3395259,1510.00848}.
	
	\item {\it Global rigidity}\index{rigidity!global}  results---in which all Anosov actions on tori and nilmanifolds are shown to be smoothly conjugate to affine actions---have been established in \cites{MR2240907,MR2372620,MR2322492,MR2318497,MR2776843,MR2983009,1801.04904} with the most complete result being \cite{MR3260859}.   Under strong dynamical hypotheses, a number of these results including \cites{MR2240907,MR2372620,1801.04904} establish global rigidity results without any assumption on the underlying manifold.   
	\item Results classifying all invariant sets (such as showing all closed invariant sets are finite or all of $M$) including Furstenberg's theorem (\cite{MR0213508},   Appendix \ref{App:furstenberg},  and Theorem \ref{thm:furst} below) and \cite{MR716835}. 
	\item {\it Measure rigidity}\index{rigidity!measure} results---in which all ergodic, invariant Borel probability measures with positive entropy are shown to be algebraic or smooth---have been established in a number of settings including Theorems \ref{thm:rudolph} and \ref{thm:KS} discussed below and in  \cites{MR1406432,MR2029471,MR1062766}.  See  also \cites{MR2261075,MR2811602} for versions of these results in non-linear and non-uniformly hyperbolic settings  (discussed in Section \ref{sec:NUHZd})  and \cites{MR1989231,MR2191228,MR2247967,MR2366231} for related results for diagonal actions on homogeneous spaces (discussed in Section \ref{sec:WCF}.)
\end{enumerate}

\subsection{Rigidity and classification of orbit closures and invariant measures}  \label{sec:ratner} 
A direction which is not pursued in this text concerns actions of groups $\Gamma$ with much less structure than those considered above.  As a prototype, one should   consider $\Gamma= F_2$, the free group on two generators. Instead of studying all  actions of such  groups,  one might consider  families of actions arising from geometric or algebraic constructions or actions satisfying certain dynamical properties.
The aim  is then to  classify certain dynamically defined objects, such as  orbit closures and invariant (or stationary)  measures, by showing that  such objects are  smooth or  homogeneous. 

Consider a discrete group  $\Gamma$  and an  action $\alpha\colon \Gamma\to \diff^r(M)$ on a compact manifold $M$.  Given $x\in M$ the \emph{orbit} of $x$ is $$O_x:=\{\alpha(\gamma)(x):\gamma\in \Gamma\}$$
and the \emph{orbit closure} of $x$ is $\overline{O_x}$, the closure of $O_x$ in $M$.  A probability measure $\mu$ on $M$ is \emph{$\Gamma$-invariant} if for all $\gamma\in \Gamma$  and Borel measurable $B\subset M$ we have $$\mu(B) = \mu\big(\alpha(\gamma\inv)(B)\big).$$   Given a probability measure $\nu$ on the acting group $\Gamma$, we say that a  probability measure $\mu$ on $M$ is \emph{$\nu$-stationary} if  for all Borel measurable $B\subset M$ we have  $$\mu(B) = \int _\Gamma \mu\big(\alpha(\gamma\inv)(B)\big) \ d \nu(\gamma).$$   That is, $\mu$ is $\nu$-stationary if it is ``invariant on average.''  While an action might not admit invariant measures (for instance if the group $\Gamma$ is non-amenable),   for any measure $\nu$ on $\Gamma$  there always exists at least one $\nu$-stationary measure (assuming $M$ is compact and the action is $C^0$.)

For a diffeomorphism  $f\colon M\to M$ exhibiting strong hyperbolicity properties, there always exist    orbit closures that  are  Cantor sets (of intermediate Hausdorff dimension) and  singular invariant probability measures supported on these Cantor sets.  This holds, for instance, if $f$ is Anosov or preserves an invariant measure  with no  Lyapunov exponent equal to zero; see Proposition \ref{prop:zoo} and nearby discussion as well as \cites{MR0442989,MR573822}.  
Similarly, singular orbit closures and invariant or stationary measures  may appear  for   actions of   free groups.

However,  there are  a number of extremely influential results establishing   homogeneity of orbit closures and invariant measures  in  certain  homogeneous or affine settings.
An  extremely important  setting in which such a  program was carried out is Raghunathan's conjecture  (see \cite[pg.\ 358]{MR629475}) on the homogeneity of orbit  closures for unipotent flows on homogeneous spaces.   Important special  cases of this conjecture were established in many papers including \cites{MR0393339,MR0578655,MR0407233,MR0447476,MR744294,MR629475,MR835804}.  Classification of orbit closures was   central to Margulis's proof of the  Oppenheim conjecture \cites{MR882782,MR993328} and   later results of  Dani and Margulis \cites{MR1016271,MR1032925}.  
The full conjecture on the  homogeneity  of all orbit closures and invariant measures for   unipotent flows was established by Ratner  in a series of papers   \cites{MR1054166,MR1075042,MR1135878,MR1262705}.  Similar results in more general homogeneous spaces and using   different techniques were    obtained in \cite{MR1253197}.  


More recently, there have been a number of breakthroughs in the setting of  homogeneous dynamics and Teichm\"uller dynamics where new techniques are developed to classify  orbit closures and  invariant  and  stationary measures for certain families of  group actions.  
In these settings, a number of common   rigidity properties of an action $\alpha\colon \Gamma\to \diff(M)$ are established:\index{rigidity} 
\begin{enumerate}
	\item {\it Stiffness of stationary measures:} all $\nu$-stationary measures are  $\Gamma$-invariant (for a finitely supported measure $\nu$ whose support generates $\Gamma$).
	\item {\it Rigidity of invariant measures:} all ergodic, 
	$\Gamma$-invariant measures  
	are a volume on a `nice' (e.g.\ homogeneous, affine, or smooth)  submanifold. 
	\item {\it Rigidity of orbit closures:} all orbit closures are  `nice' submanifolds.
\end{enumerate}

In  a  homogeneous setting, one may consider the natural action (see Example \ref{ex:standard}) of a subgroup $\Gamma$ of  $\Sl(n,\Z)$ on the torus $\T^n$.  In \cites{MR2136018,MR2182271,MR2060998} closed invariant sets were classified under various hypotheses on the acting group.  
Assuming certain algebraic properties of the group $\Gamma$,  in   \cite{MR2726604} and  \cite{MR2831114} all  stationary measures are shown to be either supported on a finite set or are the Lebesgue volume on $\T^n$ and hence  are   $\Gamma$-invariant; moreover, every orbit is  either finite or dense.
Similar  results appear in  \cite{MR2831114} for groups of translations on homogeneous spaces and under weaker hypotheses (which allow for orbit closures to be    finite unions of proper homogeneous submanifolds) in \cites{MR3037785,BQIII}.  
See also \cite{1708.04464} for an application of the method from \cite{MR3037785} to a certain non-volume-preserving homogeneous action and the  recent preprint \cite{ELlong} that extends many of the above results with fewer algebraic conditions.  

In Teichm\"uller dynamics,   an  affine but non-homogeneous action of $\Sl(2, \R)$  (the natural $\Sl(2,\R)$-action on a stratum $\mathcal H(\kappa)$ in the moduli space of  abelian differentials  on a  surface) is studied in the breakthrough work \cite{MR3814652}.  For the action of   the upper-triangular subgroup $P\subset \Sl(2,\R)$ and for certain measures  $\nu$ on   $\Sl(2,\R)$, the  $P$-invariant and $\nu$-stationary measures are shown in \cite{MR3814652} to be  $\Sl(2,\R)$-invariant and to coincide with  natural volume forms on affine submanifolds.   
This classification of $P$-invariant measures is used in  \cite{MR3418528} to show that    $P$- and $\Sl(2,\R)$-orbit closures are affine submanifolds.

In inhomogeneous settings,  there are  a number of families of actions for which   a classification of orbit closures and invariant measures is both expected and desired.   Such a  classification was attained  for nonlinear  group actions on surfaces (satisfying certain dynamical hypotheses) in \cite{1506.06826}.  Analogous results are expected to hold in higher-dimension.

\subsection{Common themes}
We end this introduction by outlining two common themes that recur  throughout this text.

\subsubsection{Entropy, exponents, and the geometry of conditional measures}
The first  major theme that runs throughout  this text is the relationship between metric entropy, Lyapunov exponents, and the geometry of measures along foliations and orbits.  The most basic relationship between these quantities is expressed in Lemmas \ref{entropyvatoms} and \ref{entropyvatoms2} which, for $C^2 $ (or $C^{1+\beta}$) diffeomorphisms, characterizes measures with zero metric entropy precisely  as those measures whose conditional measures along unstable Pesin manifolds are purely atomic.  

For measures with positive entropy but failing to attain equality in the Margulis--Ruelle inequality (Theorem \ref{entropyfacts}\ref{EFF1} below), the Ledrappier--Young entropy formula (Theorem \ref{thm:LYII} below) gives a very general relationship between the geometry of conditional measures on unstable manifolds (specifically, the transverse dimension relative to the stratification into fast unstable manifolds), Lyapunov exponents, and metric entropy.  

Our principal interest is in measures   which  attain  equality in the Margulis--Ruelle inequality. (At times we will also be interested in measures that attain the maximal value for entropy  conditioned along some expanding foliation or orbit of a group; see  Definition \ref{def:entsub} and \eqref{eq:popopopo}, page \pageref{eq:popopopo}.)  For such measures, Ledrappier  and Ledrappier--Young showed (see Theorem \ref{thm:led}, and \eqref{eq:LY}, page \pageref{eq:LY}) that conditional measures along unstable manifolds are absolutely continuous with respect to the Riemannian volume. Moreover, Ledrappier explicitly computes the density function of the conditional measures; in the case that the foliation and its dynamics are homogeneous, this yields invariance of the measure along the foliation.   See   \cref{prop:easyLed} and \cref{thm:led'}.

Deriving  invariance of a measure from entropy considerations underlies the ``invariance principle'' for linear cocycles in \cite{MR850070} and its extension to $C^1$-cocycles in \cite{MR2651382}.  It is one of the key ideas in the    classification theorem     of Margulis and Tomanov \cite{MR1253197} extending and giving some alternative arguments to Ratner's measure classification theorem.  See for example discussion in  \cite[Section 5.6]{MR2158954}.   Related entropy arguments are used in \cite{MR3814652}.  
The relationship between entropy and the geometry of conditional measures also plays a key role in \cites{MR2191228,MR2247967}.

In this text, we use  the relationship between entropy and geometry of conditional measures in our  proofs of  Theorem \ref{thm:KS} and Theorem \ref{thm:invmsr} (and its extension in Proposition \ref{prop:nonresinv}.)
In our  proof of  Theorem \ref{thm:KS}, we use \cref{prop:easyLed} 
(as well as the fact that all foliations considered are one-dimensional) to simplify certain arguments from   \cite{MR1406432}.  
In the proofs of Theorem \ref{thm:invmsr} and Proposition \ref{prop:nonresinv}, we use Theorem  \ref{thm:led'} to obtain invariance of certain measures under a group action by studying the entropy conditioned  along the orbits of the  group.  

\subsubsection{Linear functionals and higher-rank dynamics}
In the proofs of the rigidity results considered in this text, we   always reduce part of the proof to studying dynamics of  higher-rank groups of the form $\R^k$ for $k\ge 2$.  The proofs of Theorems \ref{thm:KS}, \ref{thm:invmeas},  \ref{slnr}, and \ref{thm:invmsr} all use similar tricks that rely on the fact that $\R^k$ is higher-rank when $k\ge 2$.  To each action of $\R^k$, we will associate certain dynamically defined linear functionals.  In the proof of Theorem \ref{thm:KS}, these are the Lyapunov exponents.  In the proofs of  Theorems \ref{thm:invmeas}, \ref{slnr}, \ref{thm:invmsr} these are the fiberwise Lyapunov exponents and the roots of the Lie algebra (where $\R^k \simeq A$ is the maximal split Cartan subgroup of diagonal matrices).  

The higher-rank tricks we employ are all variations on the following trivial fact: if $\lambda\colon \R^k\to \R$ is a non-zero linear functional and if $k\ge 2$, then there exists $s_0\in \R^k$ with $s_0\in \ker(\lambda)$ and $s_0\neq 0$.  In the proof of Theorem \ref{thm:KS}, the selection of such a $s_0$ ensures there exists nontrivial dynamics acting isometrically along a dynamical foliation (see Lemma \ref{lem:transinv}.)
In the proofs of Theorem \ref{thm:invmeas} and \cref{thm:invmsr}, the higher-rank assumption and the low-dimensionality of the fiber ensures we may find a nontrivial $s_0$ for which all fiberwise Lyapunov exponents vanish (see \eqref{eq:jazz}, page \pageref{eq:jazz}.)
In the proof of   \cref{slnr}, we use that if $\lambda, \beta \colon \R^k\to \R$ are non-proportional, non-zero linear functionals then we may select $s_0\in \R^k$ such that $s_0\in \ker\beta$ and $\lambda(s_0)>0$.  When $\lambda$ is a fiberwise Lyapunov exponent (for some $\R^k$-invariant measure) and $\beta$ is a root, this implies that $s_0$ is centralized by  a unipotent root subgroup and we can average (the measure) over this subgroup to obtain a new fiberwise Lyapunov exponent (for a new measure) $\lambda'\colon \R^k\to \R$ with $\lambda'(s_0)>0$.  See Claim \ref{average} and the proof of Proposition \ref{prop:goodmeas} in  \cref{ss:sl3}.   

\fakeSS{Acknowledgements}  The author would like to thank the organizers of the workshop for the invitation to present the  mini-course on which this text is based.   He is   especially grateful to  Michele Triestino  for his work in organizing the workshop and encouraging this publication.  
Finally, he would like to thank all who gave feedback on early drafts of this text including {Brian Chung,  David Fisher, Homin Lee, Emmanuel Militon, Michele Triestino, Dave Witte Morris, and the anonymous referee}.
He is especially grateful to the  anonymous referee who made a number of suggestions that substantially improved the text.  

This material is based upon work supported by the National Science Foundation under Grant No.~1752675.

\part{Rudolph and Katok--Spatzier measure rigidity theorems}\label{part:I}
\section{Furstenberg's conjecture; Theorems by Rudolph and Katok--Spatzier}

\subsection{Furstenberg conjecture}
Let $S^1 = \R/\Z$ be the additive circle.  Note that for $k\in \{2,3 ,4,   \dots\}$ the map $$M_k\colon x \mapsto kx \mod 1$$ is an expanding map of $S^1$.
The following properties of $M_k$ are well known.   For instance, using that $M_k$ is uniformly expanding for $k\ge 2$, one may pass to a symbolic extension and derive such properties using symbolic dynamics of the full $k$-shift.  
\begin{proposition}\label{prop:zoo}
	For  $k\in \{2,3,4, \dots\}$ there exist
	\begin{enumerate}
		\item uncountably many mutually disjoint, closed, invariant subsets $\Lambda\subset S^1$;
		\item uncountably many ergodic, $M_k$-invariant  Borel probability  measures $\mu$ with positive metric entropy $h_\mu(M_k)$.  
	\end{enumerate}
\end{proposition}
Analogous results 
hold for Anosov diffeomorphisms and Axiom  A systems  \cite{MR0442989}  and for any $C^{1+\beta}$-diffeomorphism of a surface with positive topological entropy \cite{MR573822}.  

Note that  each  map $M_k$ generates an action of the semigroup $\N_0$ on $S^1$.  In   \cite{MR0213508} Furstenberg considered the  action  of the semigroup $\N_0^2$ generated by $$ x \mapsto 2x \mod 1,\quad  x \mapsto 3x \mod 1.$$
\begin{theorem}[Furstenberg's theorem \cite{MR0213508}.  See Appendix \ref{App:furstenberg}]\label{thm:furst}
	The only closed subsets of $S^1$ that are invariant under both 
	$$ x \mapsto 2x \mod 1,\quad  x \mapsto 3x \mod 1$$
	are finite subsets (of rational numbers) or all of $S^1$.
\end{theorem}
In \cite{MR716835}, Berend extended Furstenberg's result to subsets of tori invariant under certain abelian groups of automorphisms.  (See also \cite{MR2974218} for further discussion on higher-rank abelian actions of toral automorphisms and \cites{MR2136018,MR2060998,MR2182271,MR2726604,MR2831114,MR3037785,BQIII} for results concerning actions by toral automorphisms of more general groups.)

Note that both generators $ x \mapsto 2x \mod 1$ and $  x \mapsto 3x \mod 1$ preserve the Lebesgue measure $m$ on $S^1$.   Thus, $m$ is invariant under the action of the semigroup $\N_0^2$ generated by  $ M_2$ and $M_3$.  
Also, for any rational point $p/q \mod 1\in \Q/\Z$, the orbit of $p/q$ under the action of $\N_0^2$  is finite and there exists  an $\N_0^2$-invariant measure supported on finitely many points of this orbit.  

From Theorem \ref{thm:furst} and the above observations, it is natural to conjecture the following.  
\begin{conjecture}[Furstenberg's conjecture] \label{conj:furst}
	\index{conjecture!Furstenberg}The only ergodic, Borel probability measure on $S^1$ that is invariant under both 
	$$ x \mapsto 2x \mod 1,\quad  x \mapsto 3x \mod 1$$
	is either   supported on a finite set (of rational numbers) or is the Lebesgue measure on $S^1$.
\end{conjecture}
\begin{remark} In Conjecture \ref{conj:furst}, the word \emph{ergodic}\index{measure(s)!ergodic} means ergodic for the semi-group action generated by $M_2$ and $M_3$. That is,  if $\mu$ is an $M_2$- and $M_3$-invariant measure, then $\mu$ is ergodic if any measurable set $D\subset \T^3$ satisfying   
	$$M_2\inv (D)= D \quad \quad M_3\inv (D) = D\label{eq:ll}
	$$
	has either $\mu(D) = 1$ or $\mu(D) =0$.  It is possible that $\mu$ is ergodic for the $\N_0^2$-action but not ergodic for either of the generators $M_2$ or $M_3$.
\end{remark}

\subsection{Rudolph's theorem}
Conjecture \ref{conj:furst} remains open.  Building on previous    results (specifically  \cite{MR941238} and \cite{MR1194793}),  Rudolph obtained what is still the optimal partial  resolution of    Conjecture \ref{conj:furst}.  

To state the result, we refer to the definition of metric entropy $h_\mu(f)$ for a $\mu$-preserving transformation $f$  defined in Section \ref{sec:ME} below.
If $f\colon X\to X$ is a continuous transformation of a compact metric space and if $\mu$ is an $f$-invariant measure supported on a finite set then the metric entropy $h_\mu(f)$ is zero.  The converse need not hold; indeed using symbolic dynamics one can build measures  $\mu$ on $S^1$ that are ergodic and invariant under $M_2$, satisfy $h_\mu(M_2) =0$, and have no atoms and hence have infinite support.  Explicit examples of such measures include measures supported on infinite minimal subshifts with zero topological entropy such as Sturmian subshifts and Morse-Thue (and more general substitution) subshifts; see \cite[\S13.7]{MR1369092}.

In \cite{MR1062766}, Rudolph resolved Conjecture \ref{conj:furst} except, possibly, for zero entropy measures with infinite support.  
\begin{theorem}[\cite{MR1062766}]\label{thm:rudolph}\index{theorem!Rudolph}
	The only ergodic Borel probability measure on $S^1$ that is invariant under both 
	$$M_2\colon  x \mapsto 2x \mod 1 \quad \text{and} \quad  M_3\colon  x \mapsto 3x \mod 1$$ and satisfies
	$$h_\mu (M_2)>0\quad \text{or} \quad  h_\mu (M_3)>0$$
	is the Lebesgue measure on $S^1$.
\end{theorem}

\subsection{Katok--Spatzier reformulation}
A minor technical nuisance when studying Furstenberg's conjecture is that the action is noninvertible.  That is, the maps 
$$  x \mapsto 2x \mod 1,\quad    x \mapsto 3x \mod 1$$
generate an action of the semigroup $\N_0^2$ rather than the action of a group.  
By passing to the \textit{natural extension solenoid} one can induce an action of the group $\Z^2$  that contains the $\N_0^2$-action as a topological factor.

One can view the natural extension solenoid as an analogue of 3-dimensional torus except that the solenoid has  non-Archimedean   directions.  A.\ Katok proposed studying a related action on a more familiar space: the action of two commuting (hyperbolic) automorphisms of $\T^3$.  
One then naturally obtains a version of Furstenberg's conjecture for  $\Z^k$-actions by automorphisms of tori and solenoids of arbitrary dimension.  
A generalization of Rudolph's theorem under a number of hypotheses was established in this setting by Katok and Spatzier \cites{MR1406432, MR1619571}.

We will focus  on the following concrete example which demonstrates many of the ideas in the paper \cite{MR1406432}.  
\begin{example}\label{ex:Key}\label{ex:key}
	Let \begin{equation}
		A=\left(
		\begin{array}{ccc}
			3 & 2 & 1 \\
			2 & 2 & 1 \\
			1 & 1 & 1 \\
		\end{array}
		\right) , \quad \quad B=\left(
		\begin{array}{ccc}
			2 & 1 & 1 \\
			1 & 2 & 0 \\
			1 & 0 & 1 \\
		\end{array}
		\right).
	\end{equation}
	One verifies the following properties of $A$ and $B$:
	\begin{claim}\label{claim:trivial} \
		\begin{enumerate}
			\item \label{trivial:1} $\det A = \det B =1$ so $A$ and $B$ preserve the orientation on $\R^3$ and the integer lattice $\Z^3$;
			\item \label{trivial:2}  $A$ has 3 distinct real eigenvalues $$\chi_A^1>1> \chi_A^2>\chi_A^3>0;$$
			\item \label{trivial:3} $B$ has
			3 distinct real eigenvalues $$\chi_B^1>\chi_B^3>1> \chi_B^2>0;$$
			\item \label{trivial:4} $A$ and $B$ commute: $AB= BA$;
			\item \label{trivial:5} $A^kB^\ell= \id $ only when $k = \ell = 0$.
		\end{enumerate}
	\end{claim}
	As $A$ and $B$ commute and are diagonalizable over $\R$, they are jointly diagonalizable.  The enumerations of the eigenvalues     of $A$ and $B$ are chosen so that $\chi_A^i$ and $\chi_B^i$  correspond to the same joint eigenvector; see \eqref{eq:jointdiag} below.
	
	Since both $A$ and $B$ preserve the integer lattice $\Z^3\subset \R^3$, they induce diffeomorphisms $$L_A\colon \T^3\to \T^3,\quad \quad L_B\colon \T^3\to \T^3 $$
	where $\T^3$ is the quotient group $\T^3 = \R^3/\Z^3$ and $L_A\colon \T^3\to \T^3$ is the automorphism
	$$L_A\colon (x+ \Z^3 )\mapsto Ax + \Z^3.$$
	Note that each of the diffeomorphisms $L_A\colon \T^3\to \T^3$ and $ L_B\colon \T^3\to \T^3 $ is an Anosov diffeomorphism.

	The maps $L_A$ and $L_B$ generate 
	a $\Z^2$-action $\alpha\colon\Z^2\to \Diff ( \T^3)$  on the 3-torus given by $$\alpha \left(n_1, n_2\right) ( x) =  L_A^{n_1} (L_B^{n_2}(x))= L_B^{n_2} (L_A^{n_1}(x)) = L_{A^{n_1}B^{n_2}}(x).$$
\end{example}
Note that given any $\Z$-action, any homomorphism $\psi\colon \Z^2\to \Z$ induces a ``fake'' $\Z^2$-action where the kernel of $\psi$ acts trivially.  Claim \ref{claim:trivial}(\ref{trivial:5}) ensures   the action $\alpha$ is not of this form; that is,  $\alpha$ is   a ``genuine'' $\Z^2$-action.

In \cites{MR1406432, MR1619571}, Katok and Spatzier proved a  generalization of Rudolph's theorem for $\Z^k$-actions on tori and solenoids generated by automorphisms under a number of  technical hypotheses.  
These hypotheses are satisfied by the action in Example \ref{ex:key}.  We note that some of  these hypotheses were later removed in \cite{MR2029471}.
The main result from \cite{MR1406432} applies to the action constructed in Example \ref{ex:Key} and yields the following natural analogue of Rudolph's Theorem, Theorem \ref{thm:rudolph}.
\begin{theorem}\label{thm:KS}\index{theorem!Katok--Spatzier}
	Let $L_A, L_B\colon \T^3\to \T^3$   be as in Example \ref{ex:Key}.  Then, the only  ergodic, Borel probability measure $\mu$ on $\T^3$ that is invariant under both $L_A$ and $L_B$ and satisfies $$h_\mu (L_A)>0\quad \text{or} \quad  h_\mu (L_B)>0$$ is the Lebesgue measure on $\T^3$.  
\end{theorem}

The rest of this part will be devoted to proving \cref{thm:KS}.  For a more concise yet complete proof of this result, see  \cite[Section 2.2]{MR1858547}.

\begin{remark}
	To generalize the action constructed from \cref{ex:key}, let $A\in \Gl(d,\Z)$ be a matrix whose characteristic polynomial is irreducible over $\Q$ and has $d$ distinct real roots.  It follows from Dirichlet's unit theorem (see \cite[Proposition 3.7]{MR1949111}) that the centralizer of $A$ in $\Gl(d,\Z)$ contains $\Z^{d-1}$ as a subgroup of finite index.  Let $\alpha\colon \Z^{d-1}\to \Diff(\T^d)$ be the induced action.   See   \cite{MR1949111}, where such actions of $\Z^{d-1}$ on $\T^d$ are called \emph{Cartan actions}, for further discussion.

	The proof we present of \cref{thm:KS} adapts  to show the following.
	\begin{theorem}[\cite{MR1406432}]\label{thmKSBig}
		For $d\ge 3$ and any Cartan action $\alpha\colon \Z^{d-1}\to \Diff(\T^d)$ as above, any ergodic, $\alpha$-invariant Borel probability measure  $\mu$ on $\T^d$ with $$h_\mu(\alpha(\vecn))>0$$ for some $\vecn\in \Z^{d-1}$ is the Lebesgue measure on $\T^d$.  
	\end{theorem}
\end{remark}

\section{Reductions in the proof of Theorem \ref{thm:KS}}
\subsection{Lyapunov exponent functionals}\label{sec:penisinthesalsa}
Let $A$ and $B$ be as in Example \ref{ex:Key}.  Since the eigenvalues  of $A$ and $B$ are  distinct real numbers, $A$ and $B$ are diagonalizable over $\R$. 
Moreover, since $A$ and  $B$ commute, they are {jointly diagonalizable}; that is there is a $Q\in \Gl(3,\R)$ such that \begin{equation}\label{eq:jointdiag}
	Q\inv A Q=   \left(\begin{array}{ccc}\chi^1_A  & 0 & 0 \\0 & \chi^2_A  & 0 \\0 & 0 &\chi^3_A  \end{array}\right), \quad \quad Q\inv B Q =   \left(\begin{array}{ccc}\chi^1_B  & 0 & 0 \\0 & \chi^2_B & 0 \\0 & 0 &\chi^3_B  \end{array}\right) .
\end{equation}

For $1\le j\le 3$, let $E^j$ denote the   $j$th joint eigenspace   of $A$ and $B$ (corresponding to $\chi_A^j$ and $\chi_B^j$.)  
As each $A$ and $B$ is irreducible over $\Q$, the eigenspaces $E^j$ are totally irrational: if $\vecv\in E^i\sm \{0\}$ has coordinates $\vecv= (v_1, v_2, v_3)$ then $v_1, v_2$, and $v_3$ are linearly independent over $\Q$.  

It is more convenient at times to work with the logarithm of the eigenvalues of $A$ and $B$.  For   $j\in \{1,2,3\}$  let 
$$\lambda^j_A = \log (\chi_A^j),  \quad \quad \lambda^j_B= \log (\chi_B^j).$$
Note that for any $(n_1, n_2)\in \Z^2$ we have 
\begin{align*}Q\inv A^{n_1}B^{n_2} Q &=  \left(\begin{array}{ccc}(\chi^1_A)^{n_1}(\chi^1_B)^{n_2}  & 0 & 0 \\0 & (\chi^2_A)^{n_1}(\chi^2_B)^{n_2}  & 0 \\0 & 0 &(\chi^3_A)^{n_1}(\chi^3_B)^{n_2}   \end{array}\right) \\& = 
	\left(\begin{array}{ccc}e^{n_1 \lambda^1_A + n_2  \lambda^1_B} & 0 & 0 \\0 & e^{n_1 \lambda^2_A + n_2 \lambda^2_B} & 0 \\0 & 0 &e^{n_1 \lambda^3_A + n_2 \lambda^3_B} \end{array}\right).\end{align*}

For  any $(n_1,n_2)\in \Z^2$, the subspace $E^j$ is an eigenspace for $A^{n_1} B^{n_2}$.  Let $\chi^j(n_1,n_2)$ be the eigenvalue of $A^{n_1} B^{n_2}$ corresponding to the eigenspace $E^j$.  We have
$$\log (\chi^j(n_1,n_2)) = n_1 \lambda^j_A + n_2  \lambda^j_B.$$
Thus, the map $\lambda^j\colon \Z^2\to \R$ given by 
$$\lambda^j(n_1,n_2) = \log (\chi^j(n_1,n_2)) = n_1 \lambda^j_A + n_2  \lambda^j_B$$
is additive.  In particular, each $\lambda^j$ extends to a linear functional $\lambda^j\colon \R^2\to \R$.  

\begin{definition}\index{Lyapunov exponent!functionals}
	The linear functionals 
	$\lambda^j\colon \R^2\to \R$ are called the \emph{Lyapunov exponent functionals} for the action $\alpha$.
\end{definition}

\subsection{Stable, unstable, and Lyapunov  foliations of $\T^3$} 
Note that $\R^3$ acts by translation on $\T^3$ as does any vector subspace $V\subset \R^3$.  
For $1\le j\le 3$ and any $x\in \T^3$ let $W^j(x)$ denote the orbit of $x$ under  translation by elements of the vector subspace $E^j$:  $$W^j(x) = \{x+ v: v\in E^j\}.$$
The sets $W^j(x)$ form a   {foliation} of $\T^3$ by lines.  We call $W^j(x)$ the $j$th \emph{Lyapunov manifold} through $x$ and call the corresponding foliation the \emph{$j$th Lyapunov foliation.}\index{Lyapunov manifold}
Note that if $x' \in W^j(x)$ with $x'= x+v$ for some $v\in E^j$ then for any $(n_1,n_2)\in \Z^2$ we have  $\alpha(n_1,n_2)(x') \in W^j(\alpha(n_1,n_2)(x))$ and 
$$ \alpha(n_1,n_2)(x') = \alpha(n_1,n_2)(x) + e^{\lambda^j(n_1,n_2)}v.$$
In particular, the action by $\alpha(n_1,n_2)$ dilates distances in $W^j$-leaves by exactly $e^{\lambda^j(n_1,n_2)}$.  


Given $(0,0)\neq (n_1,n_2)\in \Z^2$, let $$\text{$E^s_{(n_1,n_2)} = \bigoplus _{\lambda^j(n_1,n_2)<0} E^j$ and $E^u_{(n_1,n_2)} = \bigoplus _{\lambda^j(n_1,n_2)>0} E^j$}$$ be the \emph{stable} and \emph{unstable} subspaces for the matrix $A^{n_1} B^{n_2}$.  
For $x\in \T^3$ we  similarly define $W^s_{(n_1,n_2)}(x)$ and $W^u_{(n_1,n_2)}(x)$  to be the orbits of $x$ under $E^s_{(n_1,n_2)}$ and $E^u_{(n_1,n_2)}$, respectively.  For $(n_1, n_2)\neq (0,0)$, the map $$\alpha(n_1, n_2)= L_A^{n_1} L_B^{n_2} \colon \T^3\to \T^3$$ is Anosov and 
$W^s_{(n_1,n_2)}(x)$ and $W^u_{(n_1,n_2)}(x)$ are the stable and unstable manifolds through $x$ for the Anosov diffeomorphism $\alpha(n_1, n_2).$  

Observe
\begin{claim}\label{claim:plp} For any $(n_1, n_2)\in \Z^2$, any $x\in \T^3$, and any $(0,0)\neq  (m_1, m_2)\in \Z^2$
	\begin{enumcount}
		\item \label{chickenheadman1} $\alpha(n_1, n_2)(W^j(x)) =  W^j(\alpha(n_1, n_2)(x))$;
		\item \label{chickenheadman2} $\alpha(n_1, n_2)(W^u_{(m_1,m_2)}(x)) =  W^u_{(m_1,m_2)}(\alpha(n_1, n_2)(x))$;
		\item \label{chickenheadman3}  $\alpha(n_1, n_2)(W^s_{(m_1,m_2)}(x)) =  W^s_{(m_1,m_2)}(\alpha(n_1, n_2)(x))$;
		\item \label{chickenheadman4}   $E^u_{(m_1,m_2)}$ and $E^s_{(m_1,m_2)}$ are positive-dimensional and  have complementary dimension in $\R^3$; 
		\item \label{chickenheadman5}  the sets $W^u_{(m_1,m_2)}(x)$ and $W^s_{(m_1,m_2)}(x)$ are injectively immersed  planes or    lines that intersect transversally and have complementary dimension in $\T^3$.
	\end{enumcount} 
\end{claim}
Property  \ref{chickenheadman1}  is clear as $E^j$ is an eigenspace of $A^{n_1}B^{n_2}.$  Properties  \ref{chickenheadman2} and  \ref{chickenheadman3}  follow from the commutativity of $\alpha(n_1,n_2)$ and $\alpha(m_1,m_2)$.  Property \ref{chickenheadman4} follows as $\lambda^j(m_1,m_2)\neq 0$  for each $(0,0)\neq (m_1, m_2) \in \Z^2$  and  $$\lambda^1(m_1,m_2)+\lambda^2(m_1,m_2)+\lambda^3(m_1,m_2)=0$$ for every  $(m_1,m_2)\in \Z^2$.  
Property \ref{chickenheadman5} follows from  \ref{chickenheadman4} and that the spaces $E^i$ are totally irrational.  
\begin{remark}
	If $\alpha\colon \Z^{d-1}\to \Diff(\T^d)$ is as in \cref{thmKSBig} then there are $d$ Lyapunov exponent functionals $\lambda^i\colon \Z^{d-1} \to \R$, $1\le i\le d$.   Moreover, these are in general position. 
	Analogous properties to those in \cref{claim:plp} hold in this case.  For instance, we claim 
	that   for each $1\le i\le d$ there is some $\vecn\in \Z^{d-1}$ with $\lambda^i(\vecn)>0$ and $\lambda^j(\vecn)<0$ for all $j\neq i$; in particular, for such $\vecn$,  $E^u_\vecn$ is 1-dimensional  and $E^s_\vecn$ is $(d-1)$-dimensional.  We actually claim a stronger fact as in \cref{claim:plp}\ref{chickenheadman4}: for any non-trivial partition $$\{1,\dots, d\} = A      \sqcup B, \quad A\neq \emptyset,\quad B\neq \emptyset$$ there exists $\vecn\in \Z^{d-1}$ such that $$\lambda^i(\vecn)<0\quad \text{ for all $i\in A$}$$ and $$\lambda^i(\vecn)>0\quad \text{ for all $i\in B$.}$$
	
	This can be seen by observing there are $2^d-2$ such partitions.  Similarly, $d$ hyperplanes in $\R^{d-1}$ in general position divide $\R^{d-1}$ into $2^d-2$ connected components each of which corresponds to a different collection  of signs.  
\end{remark}

Recall  that  the eigenspaces  $E^j$ are totally irrational.  In particular, from the unique ergodicity of totally irrational flows on tori, we obtain  the following.
\begin{lemma}\label{lem:popo}
	A Borel probability measure $\mu$ on $\T^3$ is the Lebesgue (Haar) measure if and only if there exists  $1\le j\le 3$ such that the measure $\mu$ is invariant under the 1-parameter group of translations generated by $E^j$.  
\end{lemma}
Thus, to prove Theorem \ref{thm:KS}, it is enough to verify that any ergodic, $(L_A,L_B)$-invariant measure with positive entropy is invariant under translation by $E^j$ for some $1\le j\le 3$.  

\subsection{Conditional measures and leaf-wise measures} \label{sec:condi} (See Appendix \ref{App:rokhlin}, \cite{1208.4550}, \cref{ss:measpart}, and   \cite[Section 5]{MR2723325} for additional details and references.) 
Let $\mu$ be a Borel probability measure on $\T^3$.  In general, the partition of $(\T^3,\mu)$ into the $j$th Lyapunov  manifolds $W^j$ is not a measurable partition.  (See Lemma \ref{entropyvatoms} below for precise statement as well as  Appendix \ref{App:rokhlin}.)  
Let $\xi$ be a  {measurable partition} (see  \cref{ss:measpart} and Definition \ref{B:meas_partition} in Appendix \ref{App:rokhlin}) of $\T^3$ \emph{subordinate to $W^j$} (see Definition \ref{def:sub}); that is\index{partition!subordinate} \label{pg:opopopod}
\begin{enumerate}
	\item $\xi$ is a measurable partition of the measure space ($\T^3, \mu)$; 
	\item $\xi(x)\subset W^j(x)$ for $\mu$-\ae $x$;
	\item $\xi(x)$ contains an open neighborhood  of $x$  (in the immersed topology) in $W^j(x)$ for $\mu$-\ae $x$;  
	\item $\xi(x)$ is precompact in the immersed topology of $W^j(x)$ for $\mu$-\ae $x$;
\end{enumerate}

Let $\{\mu_x ^\xi\}$ denote a family of \emph{conditional measures} of $\mu$ relative to  the partition $\xi$.   \index{measure(s)!conditional}
That is (see \cref{def:condmeas} and  Definition \ref{B:disintegration} in Appendix \ref{B:secDisintegration})
\begin{enumerate}
	\item $\mu_x^\xi$ is a Borel probability measure on $\T^3$ such that  $\mu_x^\xi(\xi(x))=1$;
	\item if $y\in \xi(x)$ then $\mu_y^\xi= \mu_x^\xi$; 
	\item if $D\subset \T^3$ is a Borel set then $x\mapsto \mu_x^\xi(D)$ is measurable and 
	\item $\mu(D) = \int \mu_x^\xi(D) \ d \mu(x)$.
\end{enumerate}
Such a family $\{\mu_x^\xi\}$ of probability measures exists and is unique modulo $\mu$-null sets (see \cite{MR0047744}.)  



Rather than studying conditional probability measures $\{\mu^\xi_x\}$ discussed  above that depend on the choice of partition $\xi$, it is more convenient to study a family of \emph{leaf-wise measures}\index{measure(s)!leaf-wise}  denoted by $\{\mu^j_x\}$ along $W^j$.  Each measure $\mu^j_x$ in this family is  a locally finite, Borel measure (with respect to the {immersed} topology) on $W^j(x)$ but is typically an infinite measure.  We discuss the properties of these measures and then outline their construction.

Given $x\in \T^3$, let $$I^j_x:= \{x+ v : v\in E^j , |v| <1\}$$ denote the unit ball (i.e.\ interval) in $W^j(x)$ centered at $x$.
Given two locally finite, Borel   measures $\eta_1$ and $\eta_2$ on $W^j(x)$ we say $\eta_1$ and $\eta_2$ are \index{measure(s)!proportional}\emph{proportional}, written $\eta_1\,\propto\, \eta_2$, if there is $c>0$ with $$\eta_1 = c\, \eta_2.$$
\begin{proposition}[Leaf-wise measures]
	For almost every $x\in \T^3$ there is a locally finite, Borel (in the immersed topology) measure $\mu^j_x$ on $W^j(x)$ such that 
	\begin{enumcount}
		\item \label{rere1} each $\mu^j_x$ is  normalized so that $\mu^j_x (I^j_x) = 1$;
		\item \label{rere2} the family $\{\restrict{\mu_x^j }{I^j_x}\}$ of probability measures on $\T^3$ depends measurably on $x$;
		\item \label{rere3}  for $x'\in W^j(x)$ we have $\mu^j_x \,\propto\,   \mu^j_{x'};$
		\item \label{rere4} if $\mu$ is $\alpha$-invariant then for any $(n_1,n_2)\in \Z^2$  we have $$\alpha(n_1,n_2)_* (\mu^j_x) \,\propto\, \mu^j_{\alpha(n_1, n_2)(x)};$$
		\item \label{rere5} given any  measurable partition $\xi$ subordinate to $W^j$,  the conditional probability measure $\mu_x^\xi$ at $x$  is  given by $$\mu^\xi_x = \frac{1}{\mu^j_x(\xi(x))} \restrict {\mu^j_x}{\xi(x)}.$$
	\end{enumcount}
\end{proposition}

\begin{proof}[Outline of construction] To construct the family  $\{\mu^j_x\}$ of leaf-wise measures, consider a sequence  $\{\xi^k\}_{k\in \N}$  of  {measurable partitions} of $\T^3$ such that  
	\begin{enumerate}\item each $\xi^k$ is  {subordinate to $W^j$},
		\item  for almost every $x$, we have  $\xi^k(x)\subset \xi^{k+1}(x)$, and 
		\item  for almost every $x$, we have  $\bigcup _{k} \xi^k(x) = W^j(x).$
	\end{enumerate}
	From the uniqueness of conditional measures, for almost every $x$ we have that $
	\mu^{\xi^k}_x$ and $\mu^{\xi^\ell}_x$ coincide on $\xi^k(x) \cap \xi^\ell(x)$ up to normalization: if $\ell\ge k$ then 
	$$\mu^{\xi^k}_x= 
	\frac{1}{\mu^{\xi^\ell}_x(\xi^k(x))}\restrict{\mu^{\xi^\ell}_x}{\xi^k(x)} .$$
	For each $x$ and  every $k$ sufficiently large  so that  $$I_x^j \subset \xi^k(x),$$   
	let $$\td \mu^{\xi^k}_x = \frac 1 {\mu^{\xi^k}_x (I_x^j) } \mu^{\xi^k}_x.$$
	Then, given any compact (in the immersed  topology) subset $K\subset W^j(x)$, we have for any  $k$ and $\ell$ such that $K\subset \xi^k(x)$ and $K\subset \xi^\ell(x)$ that 
	\begin{equation}\label{eq:leafwise}\td \mu^{\xi^k}_x (K) = \td \mu^{\xi^\ell}_x (K).\end{equation}
	The measure $\mu^j_x$ on $W^j(x)$ is then defined to be the locally finite Borel measure defined by \eqref{eq:leafwise} for each compact $K\subset W^j(x)$. 
	Properties \ref{rere1},  \ref{rere2},  and \ref{rere5}  follow from construction and the properties of  the families of conditional measures  $\{\mu_x^{\xi_k}\}.$  Property  \ref{rere4} follows from the invariance of $\mu$.

	For Property \ref{rere3}, note that if $x'\in W^j(x)$ then both $\mu^j_x $ and $  \mu^j_{x'}$ are locally finite Borel  measures on the same space $W^j(x) = W^j(x')$.   Moreover, there is some $\ell$ such that $x'\in \xi^\ell(x)$.  Since the conditional measures $\mu^{\xi^\ell}_x = \mu^{\xi^\ell}_{x' }$ coincide, it follows that the leaf-wise measures  $\mu^j_x$ and $  \mu^j_{x'}$ are {proportional}; however, due to the choice of normalization we  typically have $\mu^j_x \neq  \mu^j_{x'}$.  \end{proof}

See also Section 6 of \cite{MR2648695}, especially Theorem 6.3, where  the  construction of the family of  leaf-wise measures is presented in a more general setting.  
We emphasize that the topology for  which $ \mu^j_x$ is Borelian is the {\it immersed} topology on the submanifold $W^j(x)$ rather than the topology inherited as a subset of $\T^3$;  as measures on $\T^3$,  the measures  $ \mu^j_x$ are   rather   pathological whenever they are non-atomic (see \cref{entropyvatoms} below.)

For $(0,0)\neq (n_1, n_2)\in \Z^2$ and for \ae $x\in \T^3$ we similarly construct locally finite,  leaf-wise measures $\mu^s_{(n_1,n_2),x}$ and $\mu^u_{(n_1,n_2),x}$ on the leaf of the corresponding stable or unstable foliation through $x$.  

Recall that an \emph{atom} of a locally finite  measure $\mu$ on a space $X$ is a point $x\in X$ with $\mu(\{x\}) >0$.  A probability measure $\mu$ on $X$ is \emph{an atom supported at $x$} if $\mu(X\sm \{x\}) = 0$ and $\mu(\{x\}) = 1$ in which case we write $\mu = \delta_x$.   
We have the following equivalences. (See Lemma \ref{entropyvatoms2} below for a proof of a more general statement.)
\begin{lemma}\label{entropyvatoms}
	Let $(n_1, n_2)\in \Z^2\sm\{(0,0)\}$ and let $\mu$ be an ergodic, $\alpha(n_1, n_2)$-invariant measure on $\T^3$.  The following are equivalent:
	\begin{enumerate}
		\item $h_\mu(\alpha(n_1, n_2)) = 0$;
		\item for $\mu$-\ae $x$, the measure $\mu^u_{(n_1,n_2),x}$ has at least one atom; 
		\item for $\mu$-\ae $x$ the measure $\mu^u_{(n_1,n_2),x}= \delta_x$ is a single  atom supported at $x$;
		\item the partition of $(\T^3, \mu)$ into full $W^u_{(n_1,n_2)}$-leaves is a measurable partition.
	\end{enumerate}
\end{lemma}

For every $x\in \T^3$,  the subspace $E^j\subset \R^3$ gives a coordinate system (inducing  the immersed topology) on the embedded line $W^j(x)$.  It is convenient to make these coordinates explicit:  for $x\in \T^3$, define an identification $\Phi_x$ between the vector space $E^j$  and the immersed manifold $W^j(x)\subset \T^3$ by 
$$\Phi_x \colon E^j \to W^j(x), \quad \quad \Phi_x(v) = x + v.$$
Let $\nu^j_x$ denote the locally finite Borel measure on $E^j$ given by pull-back under $\Phi_x$; that is (see also  {Figure \ref{fig:2}}, page \pageref{pgF2}),  let  
$$\nu^j_x = (\Phi_x\inv)_* (\mu_x^j).$$ 

\begin{remark}
	The map $\Phi_x$ 
	describes   $W^j(x)$ as  an immersed  copy of $E^j$ with $x$ as the origin.  Thus, for $x'\in W^j(x)$ with $x\neq x'$ we have  $\Phi_x\neq \Phi_{x'}$.  
	However, writing $x'= x+v$ for some $v\in E^j$ we have 
	$$\Phi_x(t) = \Phi_{x'}(t-v).$$

	What is the difference between $ \mu^j_x$ and $ \nu^j_x$?  
	For each $x\in \T^3$, the measure $ \mu^j_x$ is a locally finite measure on the immersed curve  $W^j(x)\subset \T^3$.  For $x'\notin W^j(x)$ it is difficult to compare the measures $ \mu^j_x$ and $ \mu^j_{x'}$.
	On the other hand, for each $x\in \T^3$, the measure  $ \nu^j_x$ is a locally finite measure on the vector space $E^j\simeq \R$; in particular, it is much more convenient to work with the collection $ \{ \nu^j_x\}$ as we can  easily compare  $ \nu^j_x$ and $ \nu^j_{x'}$ for $x\neq x'\in \T^3$.  

	\label{rem:kklloopp}
	For $x'\in W^j(x)$, recall that  $\mu^j_x \,\propto\, \mu^j_{x'}$ as $W^j(x) = W^j(x')$.  However, for $x'\in W^j(x)$ we  {do not} necessarily   have  that $\nu^j_x  \,\propto\, \nu^j_{x'}$. 
	The key step in the proof of Theorem \ref{thm:KS} is to establish that $\nu^j_x \,\propto\, \nu^j_{x'}$ for (typical)  $x'\in W^j(x)$.  
\end{remark}

The following lemma characterizes measures on $\T^3$ that are invariant under translations by $E^j$.   Together with Lemma \ref{lem:popo}, this reduces the proof of Theorem \ref{thm:KS} to studying the geometry of the family of measures $\{\nu_x^j\}$.  
\begin{lemma}\label{lem:invariance}\label{lem:popo2}
	A probability measure $\mu$  on $\T^3$ is invariant under translations by $E^j$   if and only if for $\mu$-\ae $x$ the measure $\nu^j_x$ is proportional  to  the Lebesgue (Haar) measure on $E^j\simeq \R$.  
\end{lemma}

\section{Interlude: Tools from smooth ergodic theory}
To complete the proof of Theorem \ref{thm:KS}, a number of additional tools from smooth ergodic theory are needed.  These tools and facts,  as well as many additional   facts that will be used in Part \ref{part:III}, are collected in Part \ref{part:II} below.  
For  the proof of Theorem \ref{thm:KS}, we encourage the reader to first consult
\begin{enumerate}
	\item  Section \ref{sec:abstractergodictheory}, especially Propositions \ref{prop:hopf} and \ref{prop:pinsker} (used in Section \ref{ss:pipart}), and 
	\item Section \ref{sec:led}, especially Theorems \ref{thm:led} and \ref{thm:led'} (used in Section \ref{sec:shear} and Section \ref{ss:prooff}.)
\end{enumerate} 
To understand these statements, the reader should refer to Section \ref{ss:lyapexp}, Section \ref{sec:unstmanifold} (especially \ref{unstmanifold11}), Section \ref{sec:ME}, Section \ref{unsentropy}, and Definition \ref{def:SRB}.

\section{Entropy, translation, and  geometry of leaf-wise measures} 
In this section, we present two key propositions, \cref{lem:abscon} and \cref{prop:easyLed} below, that will be used in the proof of Theorem \ref{thm:KS}. 
To motivate these results, in the setting of Theorem \ref{thm:KS}, let $\mu$ be an ergodic, $\alpha$-invariant probability measure with positive entropy (for some element of the action).  
By Lemmas \ref{lem:popo} and \ref{lem:popo2}, the proof of the theorem reduces to showing that for almost every $x$, the measure $\nu^j_x$ is proportional to the Lebesgue (Haar)  measure  $m$ on $E^j\simeq \R$ for some $j\in \{1,2,3\}.$  This is equivalent to showing that the measure $\mu^j_x$ is the Lebesgue (Haar) measure along the manifold $W^j(x) = x+ E^j$.  

We present here two key propositions that will give us such properties of the measures $\mu^j_x$ and $\nu^j_x$. First, under suitable geometric hypotheses on  a measure $\nu$ on $\R$, we show in \cref{lem:abscon} that $\nu$ is of the form $\nu=  \rho\, m $ where $m$ denotes the Lebesgue measure on $\R$ and $\rho\colon \R\to \R$ is a density function with $0< \rho(x) < \infty$ for $m$-\ae $x$.  In Section \ref{sec:proofKS}, we show for some $j\in \{1,2,3\}$ that these geometric hypotheses hold for the measure  $\nu^j_x$ for $\mu$-\ae $x$.   Second,  using an entropy computation due to Ledrappier (see Theorem \ref{thm:led} below),  we show in \cref{prop:easyLed} that the density function $\rho$ above is constant and, specifically, that $\mu^j_x$ is the Lebesgue measure on $W^j(x)$ for almost every $x$.

\subsection{Shearing  measures on  $\R$} \label{sec:shear}
Consider   a locally finite Borel measure $\nu$  on $\R$. For $t\in \R$, denote by  $T_t\colon \R\to \R$   the translation $$T_t(x) = x+ t$$ and let $(T_t)_*\nu$ denote the measure  defined by $$(T_t)_*\nu(B) = \nu (T_{-t}(B)) =  \nu (B-t).$$

Recall that  two locally finite Borel measures $\nu_1$ and $\nu_2$ on $\R$ are  \index{measure(s)!proportional} \emph{proportional}, written $\nu_1 \,\propto\, \nu_2$, if there is a constant $c>0$ with 
$$\nu_1 = c\,  \nu_2.$$
Two locally finite measures $\nu_1$ and $\nu_2$ on $\R$ are \emph{equivalent}\index{measure(s)!equivalent} if there is a measurable function $\rho$ with $0<\rho(x) <\infty$ for all $x$ such that  $$\nu_1= \rho\,  \nu_2$$
where $ \rho\, \nu_2$ indicates the measure defined as $$( \rho\, \nu_2)(E) = \int _E \rho (x) \ d \nu_2(x)  $$
for any Borel $E$.  

Given a locally finite Borel measure $\nu$ on $\R$, let $G(\nu)\subset \R$ denote the subgroup of translations satisfying 
$$G(\nu) = \{ t\in \R : (T_t)_*\nu \,\propto\, \nu\}.$$
\begin{example}
	Consider the   Lebesgue measure $m$ on $\R$.   Then $G(m) = \R$; in fact for every $t\in \R$ we have $(T_t)_* m = m$.
\end{example}
\begin{example}
	Consider $\nu$ to be the measure  on $\R$ given by the density $e^{x}$; that is $$d\nu(x) = e^x \, d m(x).$$
	For $t\in \R$ we have $$d((T_t)_*\nu )(x) = e^{x-t} d m(x-t) = e^{x-t} d m(x) = e^{-t}e^{x} d m(x)$$
	so $$(T_t)_*\nu = e^{-t} \nu\,\propto\, \nu.$$  Again we have $G(\nu) = \R$. 
\end{example}
Note that for a generic  density  function   $\rho\colon \R\to (0,\infty)$, we expect $G(\rho\, m)$ to be the trivial  subgroup $G(\rho\, m) = \{0\}$.

Although not needed in our proof of Theorem \ref{thm:KS}, one can show the following.
\begin{claim}\label{claim:bowl} 
	If $\nu$ is a locally finite Borel measure on $\R$ with $G(\nu) = \R$ then there exist $C>0$ and $\alpha\in \R$ such that $\nu = \rho \, m$ where  $$\rho (x) = C e^{\alpha t} .$$
\end{claim}
Indeed, we show in the proof of \cref{lem:abscon} below that,  if $G(\nu)= \R$, then $\nu $ is equivalent to $m$.  The density function $\rho$ is then a measurable function $\rho\colon \R\to (0,\infty)$ such that for each $t\in \R$, the function $$x\mapsto \frac {\rho(x)}{\rho(x-t)}$$ is a constant (in $x$) function $c_t$.  As $c_{s+t}= c_s c_t$ and as $t\mapsto c_t$ is measurable, the claim follows.

\begin{example}
	Consider   the measure $\nu$ on $\R$ given by  $$\nu = \sum _{n\in \Z} e^n \delta_{n}.$$
	For $t\in \R$ we have $$((T_t)_*\nu )(B) = \nu (\{ x-t : x\in B\}) =
	\sum _{n\in \Z} e^{n} \delta_{n+t} (B)=e^{-t}\sum _{n\in \Z} e^{n+t} \delta_{n+t} (B).$$
	Thus 
	\begin{enumerate}
		\item $(T_t)_*\nu $ is mutually singular with $\nu$ if $t\not \in \Z$, and 
		\item  if $t\in \Z$ then 
		$(T_t)_*\nu = e^{-t}\nu \,\propto\, \nu.$
	\end{enumerate}
	We thus have that $G(\nu) = \Z$ is a discrete subgroup of $\R$. 
\end{example}

We have the following elementary claim.  
\begin{claim}\label{claim:closedgroup} 
	$G(\nu)$ is a  {closed} subgroup of $\R$. 
\end{claim}


Recall the \emph{support}\index{measure(s)!support} of a measure $\nu$, written $\supp (\nu)$,  is the smallest closed subset of full measure.  Note that 
$G(\nu)$ restricts to a continuous action on $\supp (\nu)$.  In particular, as $G(\nu)$ is a closed subgroup, if $G(\nu)$ has a dense orbit in $\supp(\nu)$ then $G(\nu)$ acts transitively on $\supp(\nu)$.  

We state our  first key proposition of this section.
\begin{proposition}\label{lem:abscon}
	Suppose $G(\nu)$ acts transitively  on the support of $\nu$. Then either 
	\begin{enumcount}
		\item \label{gg1}the support of $\nu$ is a countable set and $G(\nu)$ is discrete, or 
		\item \label{gg2}$\nu$ is equivalent to the Lebesgue measure $m$ and $G= \R$.  
	\end{enumcount}
\end{proposition}
The assumption that $G(\nu)$ acts transitively  on the support of $\nu$ in  \cref{lem:abscon}    is a very strong hypothesis; for a typical measure on $\R$, we expect $G(\nu)= \{ 0\}$.

To prove  \cref{lem:abscon}, we recall the Lebesgue--Besicovitch differentiation and decomposition theorems:
\begin{proposition}[{c.f.~ \cite[Theorems 2.12, 2.17]{MR1333890}}]\label{prop:lebdecom}
	Let $\nu_1$ and $\nu_2$ be two locally finite Borel measures on $\R$.  Let 
	$$\rho(x):= \lim_{r\to 0} \dfrac{ \nu_1(B(x,r))}{\nu_2(B(x,r))}.$$
	Then 
	\begin{enumcount}
		\item the limit $\rho(x)$ exists $\nu_2$-\ae and defines a $\nu_2$-measurable 
		 function  $$\rho\colon \R \to [0,\infty);$$ 
		\item the set $$S= \{ x: \rho(x) = \infty\}$$ is  $\nu_1$-measurable  and $\nu_2$-null;
		\item $\nu_1$ decomposes as $$\nu_1 = \rho\, \nu_2 + \restrict{\nu_1} S.$$
	\end{enumcount}
\end{proposition}

\begin{proof}[Proof of  \cref{lem:abscon}]
	Since $G(\nu)\subset \R$ is a closed subgroup, there are only two options: either
	\begin{enumerate}
		\item  $G(\nu)$ is discrete in which case  \cref{lem:abscon}\ref{gg1} follows, or
		\item  $G(\nu) = \R$.  
	\end{enumerate}
	We show that if $G(\nu) = \R$ then $\nu$ is equivalent to the Lebesgue measure $m$. 
	
	We first consider the assertion that   $\nu\ll m$.   Let $\rho$ and $S$ be as in Proposition \ref{prop:lebdecom} with $\nu_1= \nu$ and $\nu_2= m$.  If $\nu$ is not absolutely continuous with respect to $m$ then the singular set $S$ has positive $\nu$-measure.  Fix $x\in S$.  Then $\rho(x) = \infty$.  
	
	Consider   any $y\in \R$.   Let $t = y-x$.  By the definition of $G(\nu)$ we have $$(T_t)_* \nu = c \nu$$ for some $c>0$ so 
	\begin{align*} \lim_{r\to 0} \dfrac{   \nu  (B(y,r))}{m(B(y,r))} &=
		\lim_{r\to 0} \dfrac{ \left(c\inv (T_t)_* \nu\right) (B(y,r))}{m(B(y,r))}\\& = 
		\lim_{r\to 0}  c\inv \dfrac{  \nu (B(x,r))}{m(B(x,r))} \\&=  \infty.\end{align*}
	It follows that $\rho(y) = \infty$ for {every} $y\in \R$.   This contradicts that $S$ is a $m$-null set.  It follows that  $\nu\ll m$. 
	
	The reverse absolute continuity  $\nu\gg m$ follows in the same manner and is left to the reader.  \end{proof}

\subsection{Invariance from  entropy considerations} 
We return to the setting and notation of \cref{thm:KS}.

For each $x\in \T^3$,  recall that $I_x^i$ denotes the unit ball (i.e.\ interval) in $W^i(x)$ centered at $x$.  Let $m^i_x$ denote the locally finite Lebesgue measure on the leaf $W^i(x)$ normalized so that $m^i_x(I_x) = 1$.  Note that for $x'\in W^i(x)$ we have $m^i_x = m^i_{x'}$ since $m^i_x$ is invariant under translations by $E^i$.  

Our second key  proposition of this section  shows that if the leaf-wise measures  $\mu^i_x$ are absolutely continuous with respect to $m^i_x$, then $\mu$ is automatically invariant under translations by $E^i$.  
\begin{proposition}\label{prop:easyLed} 
	For any $i\in \{1,2,3\}$, fix  $\vecn \in \Z^2 $ such that 
	\begin{enumerate}
		\item $\lambda^i(\vecn)>0$, and 
		\item $\lambda^j(\vecn)<0$ for both $j\neq i$.
	\end{enumerate}
	Then, the following are equivalent:
	\begin{enumlemma}
		\item\label{suck1}  $h_\mu(\alpha(\vecn))= \lambda^i$;
		\item \label{suck2} for $\mu$-\ae $x$, the measure $\mu^i_x$ is absolutely continuous with respect to $m^i_x$;
		\item \label{suck22} for $\mu$-\ae $x$, the measure $\mu^i_x$ is equivalent to  $m^i_x$;
		\item \label{suck3} for $\mu$-\ae $x$, we have equality of measures $\mu^i_x=m^i_x$;
		\item \label{suck4} $\nu_x^i$ is the Lebesgue (Haar) measure on $E^i$ for $\mu$-a.e.\ $x$.  
	\end{enumlemma}
\end{proposition}
\begin{remark}
	We only prove the  implication   \ref{suck1} $\implies $ \ref{suck3} of the proposition.  Given \ref{suck1} $\implies $ \ref{suck3}, the   only other non-trivial implication is \ref{suck2} $\implies $ \ref{suck1}.  This implication  follows, for instance, from \cite{MR693976} (see \cref{entropyfacts}\ref{EFF3} below) and can be shown using  calculations similar to those in the following proof.  
\end{remark}

Our proof essentially  follows  \cites{MR743818,MR819556} though we make certain simplifications using that the dynamics along $W^i$-manifolds is affine.
\begin{proof}[Proof that \ref{suck1} $\implies $ \ref{suck3}]
	We introduce some notation.  Fix $f= \alpha(\vecn)$. Then $f$ is a linear Anosov diffeomorphism of $\T^3$ such that for every $x\in \T^3$, the unstable manifold through $x$ is $W^i(x)$.  
	
	We may assume $\mu$ is ergodic for $f$.  Indeed, from the Margulis--Ruelle inequality (see Theorem \ref{entropyfacts}\ref{EFF1} below) we have that $$h_{\mu'}(f)\le \lambda^i(\vecn)$$ for any $f$-invariant probability measure $\mu'$.  As entropy is convex (see \eqref{eq:convex}, page \pageref{eq:convex}), it follows that $h_{\mu'}(f)= \lambda^i(\vecn)$ for almost every ergodic component  $\mu'$ of $\mu$ (see \cref{def:ergdecom} and Appendix \ref{B:secErgDecomposition}).

	Given a measurable partition $\xi$ of $\T^n$,  write $f\inv \xi$ for the partition $$f\inv \xi :=\{f\inv (C)\mid C\in \xi\}.$$  Then the atom of the  partition  $f\inv \xi$ containing $x$ is $$f\inv \xi(x) = f\inv ( \xi (f(x))).$$

	Recall  that the $W^i$-manifolds are the unstable manifolds for $f$.  Let $$J^u(x)= \left|\frac{\partial f }{\partial E^i}(x)\right|$$ be the unstable Jacobian of $f$: for any precompact, $m^i_x$-measurable subset  $C\subset W^i(x)$  we have $$m^i_{f(x)}(f(C)) = \int _{C} J^u(x) \ d m^i_{x}.$$
	As the dynamics of $f$ is affine along $W^i$-leaves, we have that $J^u(x) $ is constant in $x$.  Explicitly, we have $$J^u(x) = \chi^i(\vecn) = e^{\lambda^i(\vecn)}.$$
	
	For the remainder, fix $\xi$ to be a measurable partition of $(\T^3, \mu)$ such that
	\begin{enumerate} 
		\item $\xi$ is subordinate to the partition into $W^i$-manifolds (see \cref{sec:condi} and Definition \ref{def:sub} below), and 
		\item  $\xi$ is \emph{increasing} under  $f$:  for a.e.\   $x$ we have $f\inv\xi(x)\subset \xi(x)$.  \index{partition!increasing} 
	\end{enumerate}
	Using that there exists $\vecn\in \Z^2$ such that the  $W^i$ leaves are the unstable manifolds for the Anosov diffeomorphism $\alpha(\vecn)\colon \T^3\to \T^3$, a   partition $\xi$ with the above properties can be constructed, for instance,  by taking $\xi$ to be the unstable plaques  of a Markov partition (see \cites{Sinai1968,MR0442989,MR1326374}). See also the construction in \cite{MR693976} which holds for general $C^{1+\beta}$ diffeomorphisms.  
	
	Let $\{\mu^\xi_x\}$ be a family of conditional measures for this partition.  
	Also let $$m^\xi_x = \frac{1}{m^i_x(\xi(x))} \restrict {m^i_x}{\xi(x)}$$
	denote the normalized restriction of the Lebesgue measure $m^i_x$ to the atom $\xi(x)$ of this partition. 
	Note that we have ${m^i_x(\xi(x))} >0$  for $\mu$-\ae $x$ since each atom $\xi(x)$ contains a neighborhood of $x$ in $W^i(x)$;  in particular, the measure $m^\xi_x$ is well-defined for $\mu$-\ae $x$.

	We have  
	\begin{equation}
		\log \left(\int\frac { m^\xi_x (f\inv \xi(x) )}{ \mu^\xi_x (f\inv \xi(x) )}\ d \mu (x) \right) \le 0.
		\label{eq:rabitfood}\end{equation}
	Indeed,
	\begin{align*}
		\log & \left(\int\frac { m^\xi_x (f\inv \xi(x) )}{ \mu^\xi_x (f\inv \xi(x) )}\ d \mu (x) \right) \\
		&=  \log \left(\int \int _{\xi(x)} \frac{ m^\xi_x (f\inv \xi(y) )}{ \mu^\xi_x (f\inv \xi(y) )} \ d \mu^\xi_x (y) d \mu(x) \right) \\
		& \le \log 1 = 0\end{align*}
	where the inequality follows as   $$ \int _{\xi(x)} \frac{ m^\xi_x (f\inv \xi(y) )}{ \mu^\xi_x (f\inv \xi(y) )} \ d \mu^\xi_x (y) = \sum _{\substack{C\in {f\inv \xi }\\  \mu^\xi_x(C)>0 }}  m^\xi_x  (C) \le m^\xi_x (\xi(x))=1.$$

	We claim that \begin{align}\label{eq:rabitspit}\int \log &\left(\frac { m^\xi_x (f\inv \xi(x) )}{ \mu^\xi_x (f\inv \xi(x) )}\right) \ d \mu (x)  
		= 0.\end{align}
	Indeed, write \begin{align*} \int \log &\left(\frac { m^\xi_x (f\inv \xi(x) )}{ \mu^\xi_x (f\inv \xi(x) )}\right) \ d \mu (x) \\&=
		\int  \log\left( { m^\xi_x (f\inv \xi(x) )}\right)  \ d \mu (x)   -\int \log\left({ \mu^\xi_x (f\inv \xi(x) )}\right) \ d \mu (x).  
	\end{align*}
	From the properties of $\xi$, we  have   (see \cref{unsentropy} below)   $$ -\int   \log (\mu^\xi_x (f\inv \xi(x) ) ) \ d \mu(x)= h_\mu (f\inv \xi \mid \xi)=h_\mu(f).$$

	On the other hand, we claim that 
	\begin{equation}\label{eqhard}
		\int  \log \left( { m^\xi_x (f\inv \xi(x) )}\right) \ d \mu (x) =-\lambda (\vecn).
	\end{equation}
	To establish \eqref{eqhard}, let $$q(x) := m^i_x ( \xi(x) ).$$  As $f\inv\xi(x)\subset \xi(x)\subset f\xi(x)$ we have $$\frac {q(f(x))}{q(x)}
	= \frac { m^i_{f(x) }(\xi(f(x)) )} { m^i_x ( \xi(x) )}\le \frac { m^i_{f(x) }(f(\xi(x)) )} { m^i_x ( \xi(x) )}= 
	\frac { \int_{\xi(x)} J^u(x) \ dm^i_{x }} { m^i_x ( \xi(x) )}
	= \chi^i(\vecn)$$
	and $$\frac {q(f(x))}{q(x)}
	= \frac { m^i_{f(x) }(\xi(f(x)) )} { m^i_x ( \xi(x) )}\ge \frac { m^i_{f(x) }(\xi(f(x)) )} { m^i_x ( f\inv(\xi(f(x)) ))} =\frac 1{ \chi^i(\vecn)}.$$
	It follows that the function $$\log \frac {q \circ f }{q}$$ is $L^\infty(\mu)$ (in particular $L^1(\mu)$);    from \cite[Proposition 2.2]{MR693976} we have that $$\int \log \frac {q \circ f }{q} \ d \mu = 0.$$  
	We then have that 
	\begin{align*}
		\int    \log \left( { m^\xi_x (f\inv \xi(x) )}\right)  \, d \mu (x)
		&= \int  \log\left(\frac { m^i_x (f\inv \xi(x) )} { m^i_x ( \xi(x) )} \right) \, d \mu (x)\\
		&= \int  \log\left(\frac {\frac{1}{\chi^i(\vecn) } \, m^i_{f(x)} (\xi(f(x)) )} { m^i_x ( \xi(x) )} \right) \, d \mu (x)\\
		&= \int - \log \chi^i(\vecn) + \log \frac {q\circ f}{q}  \, d \mu\\
		&= -\lambda^i (\vecn)
	\end{align*}
	and \eqref{eqhard} follows.

	As we assumed $h_\mu(f) =  \lambda (\vecn)$, equation \eqref{eq:rabitspit}  follows.  
	From  the strict concavity of $\log$ we have 
	$$\int \log \left(\frac { m^\xi_x (f\inv \xi(x) )}{ \mu^\xi_x (f\inv \xi(x) )}  \right)\ d \mu (x) \le 
	\log \left(\int\frac { m^\xi_x (f\inv \xi(x) )}{ \mu^\xi_x (f\inv \xi(x) )}\ d \mu (x) \right)   $$ with equality if and only if the function $$x\mapsto \frac { m^\xi_x (f\inv \xi(x) )} { \mu^\xi_x (f\inv \xi(x) )}$$ is constant  off a $\mu$-null set. 
	From  \eqref{eq:rabitfood} and  \eqref{eq:rabitspit}, it thus follows   that $${\mu^\xi_x (f\inv \xi(x) )} ={ m^\xi_x (f\inv \xi(x) )}$$ for $\mu$-almost every $x$.  In particular, if $C\subset \xi(x)$ is a union of elements of $f\inv \xi$, then $\mu^\xi_x(C) = m^\xi_x(C).$
	
	We may repeat the above calculations with $f$ replaced by $f^n$ for $n\ge 1$ and obtain that $${\mu^\xi_x (f^{-n} \xi(x) )} ={ m^\xi_x (f^{-n} \xi(x) )}$$ for $\mu$-\ae $x$.  As the partitions $\{  f^{-n} \xi(x')\mid x'\in \xi(x)\}$ generate the point partition on each $\xi(x)$, it follows for $\mu$-\ae $x$  that $$\mu^\xi_x = m^\xi_x.$$
	
	Replacing $\xi$ with $f^n(\xi)$ for each $n\ge 1$, we obtain $\mu^{f^n(\xi)}_x = m^{f^n(\xi)}_x$
	and the equality   $\mu^i_x = m^i_x$  for $\mu$-\ae $x$ follows.   \end{proof}

\begin{remark}
	When $f\colon M\to M$ is an Anosov diffeomorphism or, more generally, a non-uniformly hyperbolic $C^{1+\beta}$ diffeomorphism we still have equivalence of   \ref{suck1}, \ref{suck2}, and \ref{suck22} in \cref{prop:easyLed} when the right-hand side of \ref{suck1} is replaced by the sum of all positive Lyapunov exponents counted with multiplicity and the measures are conditional  measures along unstable manifolds.   See \cref{thm:led} below and  Appendix \ref{App:LY}, especially Section \ref{D:Idea_Proof}, for details.   The proof is nearly identical to the above except for the analogue of   computation \eqref{eqhard}.  Multiplying the measures $m^\xi_x$ with an appropriate dynamically defined density (see \eqref{dyn_densities} in Appendix \ref{App:LY}), a  computation analogous to \eqref{eqhard} still holds.   See  \cite[Lemma 6.1.2]{MR819556}.

	The extra conclusion \ref{suck3}  in \cref{prop:easyLed}  follows in our   setting  from the fact that the $W^i$-manifolds are orbits of a group action (namely translation by $E^i$ on $\T^3$) and that  $f$ acts homogeneously between orbits.  The density  function guaranteed by \ref{suck22} is then constant and    equality in  \ref{suck3} holds by choice of normalization.   See \cref{thm:led'} below for a more general framework in which the extra invariance in  \ref{suck3} follows.  
\end{remark}

\begin{remark}
	In our proof of Theorem \ref{thm:KS} below, we  will apply  \cref{lem:abscon} above to conclude for some $j\in \{1,2,3\}$ that the leaf-wise measures $\mu^j_x$ along leaves of the $W^j$-foliation are absolutely continuous by showing the group $G(\nu_x^j)$ is not discrete and hence $G(\nu_x^j)= \R$ for \ae $x.$
	We will then apply \cref{prop:easyLed}    to conclude that the measures $\mu^j_x$ are Lebesgue along the $W^j$-foliation and conclude that $\mu$ is invariant under translation by $E^j$.   This approach heavily uses that the foliations $W^j$ are 1-dimensional.  

	Alternatively, one may follow \cite{MR1406432} (and many related  arguments including those in  \cites{MR2261075,MR2811602,1506.06826}) any apply dynamical arguments to  show that the density function in \cref{claim:bowl} is constant.  For instance, using that the dynamics expands unstable manifolds, one may show the curvature $\alpha^2$ of the density function $\rho$ in  \cref{claim:bowl} decreases and obtain a contradiction by Poincar\'e recurrence unless $\alpha=0$ and $\rho$ is constant.  
	This approach can be adapted when leaves of $W^j$-foliations are higher dimensional.  In this case, the group $G(\nu_x^j)$ is a closed subgroup of isometries (of some $\R^n$) that preserve the $\nu_x^j$ up to proportionality.  By classifying orbits of subgroups of the isometry group of $\R^n$ and using  the dynamics along $W^j$-leaves, one then argues that  $G(\nu_x^j)$ consists of translations that preserve $\nu_x^j$.   
\end{remark}

\section{Proof of Theorem \ref{thm:KS}}\label{sec:proofKS}
\subsection{Inducing from  a $\Z^2$-action to a $\R^2$-action}\label{ss:suspspace}
To prove  Theorem \ref{thm:KS},
it is convenient to induce from the $\Z^2$-action on the 3-manifold $\T^3$ to a  $\R^2$-action on a certain $5$-dimensional manifold which we denote by $N$.  
\index{suspension}
We first outline a general construction of $N$. (See \cref{ssec:susp} below for details of a related construction.) 
Let $\td N = \R^2\times \T^3$ and let $\Z^2$ act on $\td N$ on the right by $$(\vecs, x)\cdot \vecn = (\vecs+ \vecn, \alpha(-\vecn)(x)).$$
Let $\R^2$ act on $\td N$ on the left by $$\vect\cdot (\vecs, x) = (\vect +\vecs,x).$$
Let $N = \td N/\Z^2$ be the quotient manifold under the right $\Z^2$-action.  As the left $\R^2$-action and right $\Z^2$-action commute, we obtain a $\R^2$-action on $N$.  The manifold $N$ has a structure of a fiber-bundle with fibers diffeomorphic to $\T^3$.  The $\R^2$-action on $N$ permutes the 3-dimensional fibers and fibers over the natural $\R^2$-action on $\T^2= \R^2/\Z^2$ by translations.  
For each $j\in \{1,2,3\}$, there is a foliation $W^j$ of $N$ by injectively immersed curves; each curve $W^j(x)$ is contained in the $\T^3$-fiber through $x$.   
Moreover, there is a Riemannian metric on $N$ such   that, if $d^j_x(\cdot, \cdot)$ denotes the induced distance in $W^j(x)$, then for any $y,z\in W^j(x)$ and $\vecs \in \R^2$
$$   d^j_{\vecs \cdot x}(\vecs \cdot y, \vecs\cdot z) = e^{\lambda^j(\vecs)}d^j_x(y, z).$$
Given any Riemannian metric on $N$ and any $\R^2$-invariant probability measure $\td \mu$ on $N$ we may define fiberwise Lyapunov exponents for the action of $\R^2$ restricted to the fibers of $N$ (see \cref{sec:lyap} and \cref{ss:fibLyap}).  For any  $\R^2$-invariant measure $\td \mu$, these exponents coincide with the exponents $\lambda^1, \lambda^2$, and $\lambda^3$ above.

The above construction of $N$ works in full generality.   However,  in the context of \cref{thm:KS}, using that $A$ and $B$ are diagonalizable over $\R$ and have positive eigenvalues it is possible to give a more algebraic construction of the suspension manifold $N$.  
The algebraic construction has the advantage that the  dynamical properties outlined above follow immediately. We state the properties of the suspension space and induced $\R^2$-action and outline the construction in \cref{sec:susp}.  

Given  $\vecs=(s_1,s_2)\in \R^2$ define subspaces of $\R^3$ 
$$\text{$E^s_{(s_1,s_2)} = \bigoplus _{\lambda^j(s_1,s_2)<0} E^j$ \quad and \quad $E^u_{(s_1,s_2)} = \bigoplus _{\lambda^j(s_1,s_2)>0} E^j$.}$$
Note that if $(s_1,s_2)\in \ker \lambda^j\sm \{0\}$ for some $j\in \{1,2,3\}$  then both $E^s_{(s_1,s_2)}$ and $E^u_{(s_1,s_2)}$ are 1-dimensional.  

\begin{proposition}[See  \cref{sec:susp}]\label{prop:susp}
	There is a $5$-dimensional manifold $N$ and an $\R^2$-action $\td \alpha\colon \R^2 \to \diff(N)$ with the following properties:
	\begin{enumcount}
		\item $N$ is a fiber-bundle  over $\T^2$ with fibers diffeomorphic to $\T^3$; moreover the action $\td \alpha$ permutes the fibers.
		\item For every $1\le j\le 3$ the vector space $E^j$ acts by addition on $N$. For every $x\in N$,  the  orbit $W^j(x) = \{ x+ v: v\in E^i\}$ is contained in the $\T^3$ fiber containing $x$ and the leaves $W^j(x)$ form a  smooth foliation  of $N$.  
		The $W^j$-leaves are permuted by the action  $\td \alpha$.  
		\item For every $0\neq (s_1, s_2)\in \R^2$ the vector spaces $E^s_{(s_1, s_2)}$ and $E^u_{(s_1, s_2)}$ similarly act by addition on $N$.  The orbits $W^s(x) $ and $W^u(x)$ are contained in the fiber through $x$ and correspond to the stable and unstable manifolds, respectively,  for the partially hyperbolic diffeomorphism $\td \alpha(s_1, s_2)\colon N\to N$.  
		\item\label{55} For all $(t_1, t_2)\in \R^2$ and $x\in  N$, the map $$\td \alpha (t_1, t_2)\colon W^j(x)\to W^j(\td \alpha (t_1, t_2)(x))$$ dilates  distances in $W^j$ by exactly $e^{\lambda^j(t_1, t_2)}$.  That is, for $y\in  W^j(x) $ writing $y =x + v$ for some $v\in E^j$, we have $$
		\td \alpha (t_1, t_2)(y) \in W^j( \td \alpha (t_1, t_2)(x))$$ and 
		$$ \td \alpha (t_1, t_2)(y)=  \td \alpha (t_1, t_2)(x) + e^{\lambda^j(t_1, t_2)} v.$$
	\end{enumcount} 
\end{proposition}

Given $\vecs\in \R^2$ let $W^u_\vecs(x)=E^u_{\vecs}(x)$ be the unstable manifold through $x\in N$ for the 1-parameter flow $\alpha(t\vecs)$.
Similarly let $W^s_\vecs(x) = x + E^s_{\vecs}(x)$ be the stable manifold through $x\in N$ for the 1-parameter flow $\alpha(t\vecs)$. The stable and unstable manifolds $W^s_\vecs(x)$ and $W^u_\vecs(x)$ are contained in the $\T^3$-fiber of $N$ through $x$. 
Note that for $\vecs\neq (0,0)$, the leaves $W^u_\vecs(x)$  and $W^s_\vecs(x)$ are at least one-dimensional.  Moreover, 
if $\vecs $ is not in a kernel of any $\lambda^j$ (see Figure \ref{kernels}) then $W^u_\vecs(x)$  and $W^s_\vecs(x)$ are of complementary dimension in the $\T^3$-fiber through $x$.  If  $\vecs \neq (0,0)$ is contained  in the  kernel of  $\lambda^j$ then $W^u_\vecs(x)$  and $W^s_\vecs(x)$ are both 1-dimensional.  

To begin the proof of \cref{thm:KS}, fix an ergodic, $\alpha$-invariant probability measure $\mu$  on $\T^3$ with positive entropy $h_\mu(\alpha(n_1, n_2))>0$ for some  $(n_1,n_2)\in \Z^2$.  
Note that $\alpha(n_1, n_2)$ has either   2-dimensional unstable or 1-dimensional unstable manifolds.  Replacing $(n_1,n_2) $ with $(-n_1,-n_2)$ if necessary and recalling that $$h_\mu(\alpha(n_1, n_2)) = h_\mu(\alpha(n_1, n_2) \inv) = h_\mu(\alpha(-n_1, -n_2))$$ we can assume that $\alpha(n_1, n_2)$ has 1-dimensional unstable manifolds.  Fix $1\le i\le 3$ for the remainder of the proof such that $E^i = E^u_{(n_1, n_2)}$.  We will   show that $\mu$ is invariant under translations by $E^i$.  

We write  $\td \mu$ for the ergodic,  $\td \alpha$-invariant measure on $N$ corresponding to $\mu$.  
To construct the measure $\td \mu$, first let $m^2$ denote the Lebesgue measure on $\R^2$.   Then $m^2\times \mu$ is a locally finite Borel measure on $\td N = \R^2\times \T^3$  that is invariant under the actions of both $\R^2$ and $\Z^2$.  Let $\td \mu$ denote the image of $m^2\times \mu$ restricted to the fundamental domain $[0,1]^2\times \T^3.$
From Lemma \ref{entropyvatoms}  the leaf-wise measures $\mu_x^i$ of $\mu$ along $W^i$-leaves in $\T^3$ are nonatomic.  This holds if and only if  the leaf-wise measures $\td \mu_x^i$ of $\td \mu$ along $W^i$-manifolds in $N$ are nonatomic.   
Moreover, ergodicity of $\mu$ for the $\Z^2$-action on $\T^3$ implies that  $\td \mu$ is ergodic for the $\R^2$-action on $N$.  

For $x\in N$ we again parameterize $W^i(x)$ by $E^i$ via the map $$\Phi_x \colon E^i \to W^i(x),\quad \Phi_x (v) = x+ v$$    and let $\td \nu^i_x$ given by  $$\td \nu^i_x = (\Phi_x\inv)_* \td \mu_x^i$$
be the corresponding locally finite measure on $E^i\simeq \R$. Recall we fix a normalization of each $\td \nu^i_x$ so that each $\td \nu^i_x$ gives mass 1 to the unit ball (i.e.\ interval) in $E^i$.  From the choice of $E^i$, the measures $\td \nu^i_x$ are nonatomic for \ae $x\in N$. 

To prove   \cref{thm:KS}, we will show that $G(\td \nu^i_x) =\R$ for $\td \mu$-almost every $x\in N$.  This will imply that $G( \nu^i_x) =\R$ and thus $\nu^i_x$  is absolutely continuous with respect to Lebesgue for $ \mu$-almost every $x\in \T^3$ by \cref{lem:abscon}.  Applying \cref{prop:easyLed} then implies that $\mu$ is invariant under translations by $E^i$, showing that $\mu$ is the Lebesgue  measure on $\T^3$.

\subsection{Restriction to the kernel of $\lambda^i$}
We now heavily use that our acting group $\R^2$ is  {higher-rank}.  Recall that    $\lambda^1, \lambda^2, \lambda^3$ are linear functionals on $\R^2$; moreover, none of the  functionals $\lambda^1, \lambda^2, \lambda^3$ is the zero function and no pair is proportional.  It follows that each  functional $\lambda^1, \lambda^2, \lambda^3$ has a 1-dimensional kernel and that all kernels are distinct.  See Figure \ref{kernels}.  

\psset{unit=0.5cm}
\begin{figure}[h]
	\includegraphics{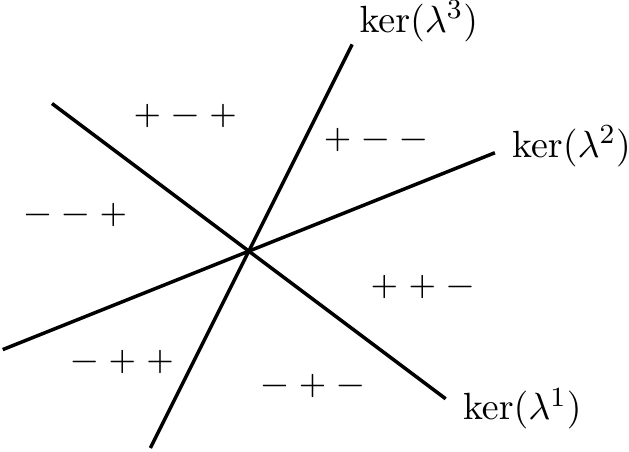}
%
	\caption{Kernels of the Lyapunov exponent  functionals in the acting group $\R^2$. Signs indicate the signs of the Lyapunov exponents in corresponding half-cones. For instance, $++-$ indicates the open half-cone of $\mathbf{s}\in \R^2$ such that $\lambda^1(\mathbf s)>0, \lambda^2(\mathbf s)>0$, and $\lambda^3(\mathbf s)<0$. (Recall that the sum of the $\lambda^i$ is zero.)}
	\label{kernels}
\end{figure} 
As the exponents $\lambda^i$ are not proportional to functionals defined over $\Q$, none of the kernels is defined over $\Q$.  In particular, there is no $0\neq \vecn\in \Z^2$ with $\vecn\in \ker \lambda^j$ for any $j\in \{1,2,3\}$.  This is the primary reason why we  induce to an $\R^2$-action on $N$ rather than studying the $\Z^2$-action on $\T^3$.  

Recall our fixed $i\in \{1,2,3\} $ above such that the leaf-wise measure  $\td \mu^i_x$ is non-atomic  for almost every $x\in N$.    As the kernel $ \ker \lambda^i$ is 1-dimensional, we may fix   $\vecs_0= (s_1, s_2) \in \ker \lambda^i\sm\{0\}$.  For the remainder of the proof we will (almost) exclusively study the  1-parameter  flow $\phi_t$ inside  $\ker \lambda^i$:  $$\phi_t\colon N\to N,\quad \quad \phi_t(x) = \td \alpha(t\vecs_0)(x) =\td \alpha(ts_1, ts_2)(x). $$

From the fact that $\td \mu$ is $\td \alpha$-invariant, the choice of $\vecs_0\in \ker \lambda^i$,  and the choice of normalization of the family of leaf-wise measures $\{\td \mu^i_x\}$, we immediately obtain the following.  
\begin{lemma}\label{lem:transinv} For every $t\in \R$ and almost every $x\in N$
	\begin{enumcount}
		\item $\phi_t\colon W^i(x)\to W^i(\phi_t(x))$ is an isometry;\label{aa}
		\item $(\phi_t)_* \td \mu^i_x = \td \mu^i_{\phi_t(x)}$;\label{bb} 
		\item $\td \nu^i_x = \td \nu^i_{\phi_t(x)}$.\label{Hugely important}\label{cc}
	\end{enumcount}
\end{lemma}
\begin{proof}
	Conclusion \ref{aa}  follows from Proposition \ref{prop:susp}\ref{55} and  the choice of $\vecs_0\in \ker \lambda^i$ so that $\td \alpha (\vecs_0)$ dilates distances in $W^i$ by $e^{\lambda^i(\vecs_0)} = 1.$
	
	For Conclusion \ref{bb}, from the invariance of $\td \mu$ we have  $$(\phi_t)_* \td \mu^i_x=(\alpha(\vecs_0))_* \td \mu^i_x \,\propto\, \td \mu^i_{\phi_t(x)}.$$ 
	On the other hand, since $\phi_t$ is an isometry along $W^i$-leaves, we have $$\phi_t(I^i_x) = I^i_{\phi_t(x)}$$ where $I^i_x$ is the unit ball (i.e.\ interval) in $W^i(x)$.  
	It follows that $$(\phi_t)_* \td \mu^i_x (I^i_{\phi_t(x)}) = \td \mu^i_x (I^i_{x}) = 1$$ and thus, from our choice of normalization, $$(\phi_t)_* \td \mu^i_x=  \td \mu^i_{\phi_t(x)}.$$
	
	Conclusion \ref{cc} then follows from \ref{aa}, \ref{bb}, and definition of $\Phi^i_x$.  
\end{proof}
Recalling Remark \ref{rem:kklloopp}, conclusion  \ref{Hugely important} of Lemma \ref{lem:transinv} is quite strong: the family of  measures  $\{\td \nu^i_x\} $ on $E^i$ is {constant} along orbits of the flow $\phi_t$.  

\subsection{Conclusion of the proof assuming ergodicity of $\phi_t$}\label{ss:prooff}
Note that while $\td \mu$ is assumed to be $\R^2$-ergodic, there is no reason that $\td \mu$ is ergodic for the 1-parameter flow $\phi_t$.  
We complete the proof of Theorem \ref{thm:KS} assuming the measure $\td \mu$ is  ergodic for the 1-parameter flow $\phi_t$.  Although this may not hold in general, we  will explain how to correct this in the next section.  

Recall the notation and conclusion of  \cref{lem:abscon}.  The next lemma verifies that the  measures $\td \nu_x^i$ satisfy the hypotheses of   \cref{lem:abscon}.
\begin{lemma}\label{lem:trans}
	Assume the 1-parameter flow $\phi_t$ acts ergodically on $(N,\td \mu)$.  Then for $\mu$-\ae $x\in N$,  the group $G(\td \nu_x^i)$ acts transitively on the support of $\td \nu_x^i$.  
\end{lemma}
\begin{proof}[Proof of Theorem \ref{thm:KS} assuming that $\phi_t$ is ergodic.] 
	Assume that Lemma \ref{lem:trans} holds, from    \cref{lem:abscon} we conclude that either  $\td \nu^i_x$ is supported on a countable set (and  thus the measure  $\td \nu^i_x$ has atoms) or the measure  $\td \nu^i_x$ is absolutely continuous with respect to Lebesgue measure on $E^i\simeq \R$.  From  our entropy assumptions (recall Lemma \ref{entropyvatoms}), the  measures  $\td \mu^i_x$ and thus $\td \nu^i_x$ have  no atoms for almost every $x\in N$  and thus from  \cref{lem:abscon} we conclude that 
	$\td \nu^i_x$ is absolutely continuous with respect to Lebesgue measure on $E^i$.  
	
	It follows that the leaf-wise measures $\td \mu^i_x$ are absolutely continuous with respect to Lebesgue measure on $W^i(x)$ for almost ever $x\in N$.  From the construction of $\td \mu$, it follows that  the leaf-wise measures  $\mu^i_x$ are absolutely continuous with respect to Lebesgue measure on $W^i(x)$ for almost every $x\in \T^3$.  (Alternatively, an analogue of \ref{suck2} $\implies$ \ref{suck1} in  \cref{prop:easyLed} implies that $h_{\td \mu}(\td \alpha(n_1, n_2)) =\lambda^i(n_1,n_2)$ which, from the structure of $\td \mu$, implies  $h_{ \mu}( \alpha(n_1, n_2)) =\lambda^i(n_1,n_2)$ and thus $\mu^i_x$ is absolutely continuous for $\mu$-almost every $x$.)
	From  \cref{prop:easyLed}\ref{suck3},    it follows for  \ae $x\in \T^3$ that $\mu^i_x$   coincides with the Lebesgue measure on $W^i(x)$ normalized on $I^i_x$.  From  \cref{lem:invariance}, it follows that $\mu$ is invariant under translations by $E^i$ and is hence the Lebesgue measure on $\T^3$ by \cref{lem:popo}.    
\end{proof}

We give the proof of Lemma \ref{lem:trans} (still   assuming   that $\phi_t$ is ergodic.)
\begin{proof}[Proof of  Lemma \ref{lem:trans} assuming ergodicity of $\phi_t$]
	Recall from Lemma \ref{lem:transinv}\ref{cc} that the parameterized collection
	of measures $x\mapsto \td \nu^i_x$ forms a $\phi_t$-invariant, measurable function.\footnote{There is a minor technical point we ignored here.  Namely, we are using that  the space of locally finite Borel measure on $\R$ is a reasonable topological space (with the topology dual to compactly supported continuous functions) and that   $x\mapsto \td \nu^i_x$  is a measurable function from $(N, \td \mu)$ to the space of locally finite Borel measure on $\R$.}  As we assume ergodicity of the flow $\phi_t$, it  follows that the assignment $x\mapsto \td \nu^i_x$ is  constant $\td \mu$-a.s.
	In particular, for $\mu$-almost every $x\in N$ and $ \td \mu^i_{x}$-almost every $x'\in W^i(x)$  we have 
	\begin{equation} \td \nu^i_{x'} =  \td \nu^i_{x}.\label{eq:lplp}\end{equation}

	Take such $x$ and $x'$.  Recall the parametrization $\Phi_x\colon E^i\to W^i(x)$.  Let $v\in E^i$ be such that $x' = x+ v$.  
	We observe (see Figures \ref{fig:2a} and \ref{fig:2b}) that
	$$\Phi_{x}\inv \circ \Phi_{x'} \colon E^i\to E^i$$ is the map $$\Phi_{x}\inv \circ \Phi_{x'} \colon t \mapsto t+v.$$
	Recall that $\td \nu^i_{x'}, \td \mu^i_{x'}, \td \mu^i_{x}$, and $\td \nu^i_{x}$ are canonically defined  {by our choice of normalization.}  Since  $x$ and $x'$ are in the same unstable manifold, we have $ \td \mu^i_{x'} \,\propto\,  \td \mu^i_{x}$ so  
	$$ (\Phi_{x'})_* \td \nu^i_{x'}= \td \mu^i_{x'} \,\propto\, \mu^i_{x} =  (\Phi_{x})_* \td \nu^i_{x}.$$
	and  $$(\Phi_{x}\inv \circ \Phi_{x'})_*  \td \nu^i_{x'}\,\propto\, \td \nu^i_{x}.$$
	It follows that 
	$$(T_v)_*  \td \nu^i_{x}= (T_v)_*\td \nu^i_{x'} \,\propto\, \td \nu^i_{x}. $$
	Thus $v\in G(\td \nu^i_{x})$.  Since $x'$ was a $\td \mu^i_{x}$-typical point of $W^i(x)$ it follows that $G(\td \nu^i_{x})$ has a dense orbit in the support of $\td \nu^i_{x}$ and  thus acts transitively on the support of $\td \nu^i_{x}$.  
\end{proof}

\begin{figure}[h]
	\begin{center}
		\psset{unit=0.7cm}
		\small
		\begin{subfigure}[t]{0.49\textwidth}
			\includegraphics{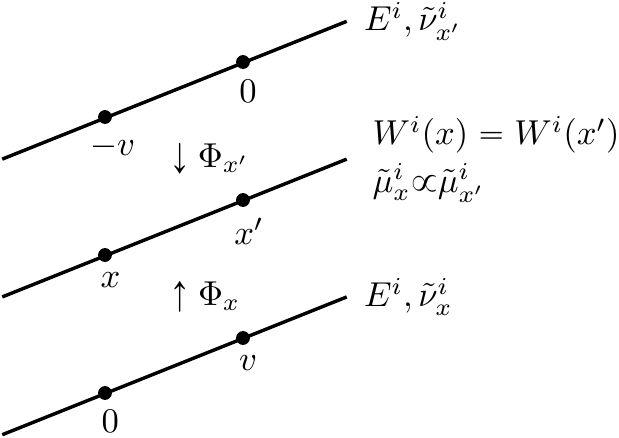}
%
%
%
%
			\subcaption{Parametrizations $\Phi_x$ and $\Phi_{x'}$ } \label{fig:2a}
		\end{subfigure}
		\begin{subfigure}[t]{0.49\textwidth}
			\small
			\includegraphics{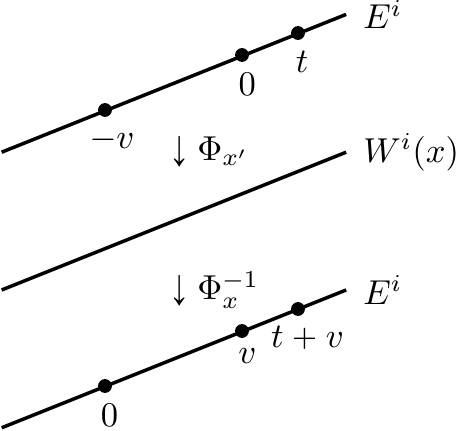}
%
%
%
			\subcaption{ $\Phi_{x}\inv \circ \Phi_{x'} \colon t \mapsto t+v$ } \label{fig:2b}

		\end{subfigure}
		\caption{Proof of Lemma \ref{lem:trans}}
		\label{fig:2}
		\label{pgF2}
	\end{center}
\end{figure}

\begin{remark}
	Above, we showed that   $\td \nu^i_{x'} = \td \nu^i_{x}$ for $x' = x+ v$ with $v\in E^i$ but  only obtained $(T_v)_*\td \nu^i_x \,\propto\, \td \nu^i_x $ rather than  $(T_v)_*\td \nu^i_x = \td \nu^i_x $.  
	The coefficient of proportionality is due to the  choice of normalization on  $\td \nu^i_{x}$  which is chosen so that  $\td \nu^i_{x}  \left(B^{E^i}_1(0)\right)=1$ where $B^{E^i}_1(0)\subset E^i$ is the unit   ball (interval) in $E^i$ centered at $0$.  We have $$(T_v)_* \td \nu^i_{x} \left(B^{E^i}_1(0)\right)
	=\td \nu^i_{x} \left(B^{E^i}_1(v)\right)$$ but 
	do not (yet) know that   $ \td \nu^i_{x} \left(B^{E^i}_1(v)\right)=1.$  However, we do know that 
	\begin{align*}
		\left((T_v)_*\td \nu^i_{x} \right) \left(B^{E^i}_1(v)\right)&=\left( (\Phi_{x}\inv \circ \Phi_{x'})_* \td \nu^i_{x'}\right)\left (B^{E^i}_1(v)\right)
		\\& = \td \nu^i_{x'}\left(  \Phi_{x'}\inv  (\Phi_{x}   ( B^{E^i}_1(v) ) )\right)  \\&=\td \nu^i_{x'} \left(B^{E^i}_1(0)\right) =1.\end{align*}
	In particular, we have that $(T_v)_*\td \nu^i_x = (T_v)_*\td \nu^i_{x'}= (\Phi_{x}\inv \circ \Phi_{x'})_* \td \nu^i_{x'} \,\propto\,  \td \nu^i_{x}$
	with explicit coefficient of proportionality:  
	$\displaystyle (T_v)_*\td \nu^i_x=   \dfrac{1}{\td \nu^i_{x}  (B^{E^i}_1(v) )}\td \nu^i_x.$

\end{remark}

\subsection{Overcoming lack of ergodicity: the $\pi$-partition trick}\label{ss:pipart}
The proof of Lemma \ref{lem:trans} seems to  fail if the measure $\td \mu$ is not $\phi_t$-ergodic.  Indeed, we used that the assignment $x\mapsto \td \nu^i_x$ was $\phi_t$-invariant to conclude that the assignment $x\mapsto \td \nu^i_x$ was constant  in order to  conclude that   $ \td \nu^i_x=  \td \nu^i_{x'}$ for $\td \mu^i_x$-typical  $x' \in W^i(x)$ in \eqref{eq:lplp}.  

We recall the following constructions and definitions.  See also  Theorem \ref{t.decomposicaoergodica} in  Appendix \ref{B:secErgDecomposition} and \cite[Section 3.5]{MR1086631}.    
\begin{definition} \label{def:ergdecom}\index{ergodic decomposition}
	Let $f\colon X\to X$ be a Borel map of a metric space $X$ preserving a Borel probability measure $\mu$.  Then,  there exists a   measurable partition $\erg$  of $(X,\mu)$ such that---writing $\{\mu^\erg_x\} $ for a family of conditional measures of $\mu$ relative to $\erg$ (see \cref{def:condmeas})---for $\mu$-\ae $x$ the measure $\mu^\calE_x$ is an ergodic, $f$-invariant Borel probability measure.  The partition $\erg$ is called the \emph{ergodic decomposition} or the \emph{partition into ergodic components}  of $\mu$ with respect to $f$.  The measures $\{\mu^\erg_x\}$ are called the \emph{ergodic components} of $\mu$. 
\end{definition}

To illustrate the most extreme defect when ergodicity fails, for a typical $x$,  it could be that  the conditional measure along $W^i(x)$ of the $\phi_t$-ergodic component $\td \mu^\erg_x$  of $\td \mu$ containing $x$ is an atom at the point $x$.  Then, the only point $x'\in W^i(x)$  for which  one could conclude that $ \td \nu^i_x=  \td \nu^i_{x'}$ would be $x'= x$.  

We now correct the proof of  Lemma \ref{lem:trans}.    This requires  tools  and notation discussed in Section \ref{sec:abstractergodictheory} below which we encourage the reader to read first.  

Examining the proof of Lemma \ref{lem:trans}, we did not fully use that the assignment $x\mapsto \td \nu^i_x$ was constant.  Rather, we used that the  assignment $x\mapsto \td \nu^i_x$ was constant   along the support of $\td \mu^i_x$ in $W^i(x)$.  From this we concluded that $ \td \nu^i_x=  \td \nu^i_{x'}$  for $\td \mu$-typical $x$ and  $\td \mu^i_x$-typical $x'$ in $W^i(x)$.

Recall that if  $f\colon N\to \R$ is a $\phi_t$-invariant, measurable function then  $f$ is constant on almost every $\phi_t$-ergodic component of $\td \mu$. 
Thus, the proof of  Lemma \ref{lem:trans}  above works if  we establish that almost every $\phi_t$-ergodic component of $\td \mu$ is ``saturated''  
by  full $W^i$-manifolds.  The precise statement appears in the following lemma, known as the ``$\pi$-partition trick.''   
We refer to \cref{sec:abstractergodictheory} and Appendix \ref{App:Pinsker} for details of the $\pi$-partition and measurable hulls.  
From entropy considerations in Lemma \ref{entropyvatoms}, we have that the partition of $N$ into full $W^i$-leaves is not measurable.  We let $\Xi^i$ denote the measurable hull of the partition of $(N,\td \mu)$ into full $W^i$-leaves (see Section \ref{ss:meashull}). Also see  \cref{sec:parord} for the definition of the partial order on the space of partitions.

\begin{lemma}[$\pi$-partition trick]\label{lem:Pipart}\index{partition!Pinsker!$\pi$-partition trick}
	$\Xi^i$  is finer than the partition of $(N, \td \mu)$ into $\phi_t$-ergodic components.
\end{lemma}
 Note that if the $W^i$-leaves were expanded (or contracted) by $\phi_t$, then the conclusion of \cref{lem:Pipart} would follow from Proposition \ref{prop:hopf} 
below.   However, we chose $\phi_t$ precisely  so that it neither expands nor contracts $W^i$-leaves.

It follows from Lemma \ref{lem:Pipart} that almost every $\phi_t$-ergodic component contains full $W^i$-leaves and hence the proof of  Lemma \ref{lem:trans} works   by replacing $\td \mu$ with a $\phi_t$-ergodic component of $\td \mu$.

We  complete the proof of Theorem \ref{thm:KS}  by giving the proof of Lemma \ref{lem:Pipart}.  
For $\vecs\in \R^2$ and  $x\in N$ recall that $W^s_\vecs(x)$ and   $W^u_\vecs(x)$ denote the stable and  unstable manifolds, respectively, for the 1-parameter flow $\alpha(t\vecs)$.
We let  $\Xi^u_\vecs$ denote the measurable hull (see \ref{ss:meashull} below) of the partition of $(N,\td \mu)$ into full $W^u_{\vecs}$-leaves.  Similarly 
$\Xi^s_\vecs$ denotes the measurable hull of the partition of $(N,\td \mu)$ into full $W^s_{\vecs}$-leaves.

Given $\vecs\in \R^2$, let $\erg_\vecs$ denote the measurable partition of $(N,\td \mu)$ into ergodic components of the 1-parameter flow $\alpha(t\vecs)$.  Similarly, let $\pi_\vecs$ denote the Pinsker partition ({see \ref{sss:pinsker} below}) for the 1-parameter flow $\alpha(t\vecs)$ on $(N,\td \mu)$.  
As stated in Proposition \ref{prop:hopf} below,    for any $\vecs\in \R^2$ we have  $$\erg_\vecs\prec \Xi^s_\vecs.$$ 
From Proposition \ref{prop:pinsker} below,  we have 
for any $\vecs\in \R^2$  that $$\Xi^u_\vecs = \pi_\vecs = \Xi^s_\vecs.$$
(See \ref{sec:parord} for definition of the partial order on space of partitions.)

With the above notation, the conclusion of  Lemma \ref{lem:Pipart} states,  for our fixed $\vecs_0$, that $$\erg_{\vecs_0} \prec \Xi^i.$$
Given the abstract ergodic theoretic facts   above, the proof of  Lemma \ref{lem:Pipart} is remarkably straightforward. 
\psset{unit=0.5cm}
\begin{figure}[h]
%
%
%
%
%
	\includegraphics{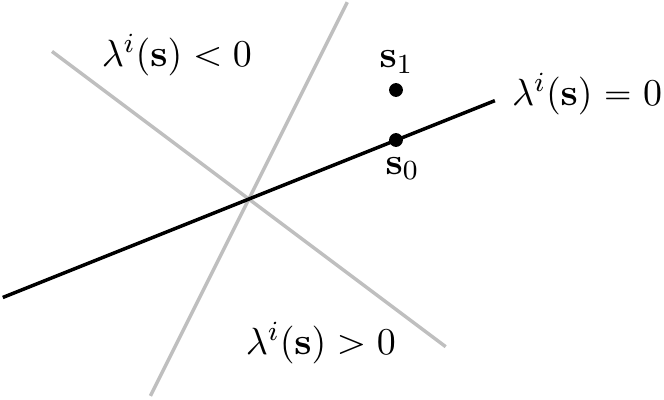}
	\caption{Choice of $s_1$} 
	\label{kernels2}
\end{figure}

\begin{proof}[Proof of Lemma \ref{lem:Pipart}]
	
	Recall we chose $\vecs_0\neq(0,0)$ so that  $\lambda^i(\vecs_0)=0$ and $\lambda^j(\vecs_0)\neq 0$ for each $j\neq i$.  Pick $\vecs_1\in \R^2$ close to $\vecs_0$ with the following properties (see Figure \ref{kernels2}):
	\begin{enumerate}
		\item $\lambda^j(\vecs_1)\neq 0$ for every $1\le j\le 3$; 
		\item $\lambda^i(\vecs_1)<0$;
		\item \label{334433} for each  $j\neq i$, the numbers $\lambda^j(\vecs_1) $ and  $\lambda^j(\vecs_0)$ have the same sign. 
	\end{enumerate}

	Since $\lambda^i(\vecs_1)<0$, it follows that $W^i(x)\subset W^s_{\vecs_1}(x)$ for every $x\in N$.  This immediately implies $$\Xi^s_{\vecs_1} \prec \Xi^i.$$  Also, note that $W^u_{\vecs_1}(x)= W^u_{\vecs_0}(x)$ for all $x\in N$ whence $ \Xi^u_{\vecs_0} = \Xi^u_{\vecs_1}$.  
	We then obtain the following string of refinements and equalities: 
	$$\erg_{\vecs_0} \prec \Xi^s_{\vecs_0} = \pi _ {\vecs_0}= \Xi^u_{\vecs_0} = \Xi^u_{\vecs_1} = \pi _ {\vecs_1}=\Xi^s_{\vecs_1} \prec \Xi^i.$$
	In symbols, this is exactly what we needed to prove.  
\end{proof}

\cref{lem:Pipart} completes the proof of \cref{thm:KS}.  We end this section with some technical remarks on the proof of  \cref{lem:Pipart}.

\begin{remark}
	
	In the  proof of  \cref{lem:Pipart}, to choose $\vecs_1$ satisfying condition \eqref{334433}, we heavily used the fact that there is no $\lambda^j$ with  $\lambda^j = -c \lambda^i$ for any $c>0$; that is, we are using that the action has no Lyapunov exponents that are negatively proportional to $\lambda^i$.  Indeed, if $\lambda^j = -c\lambda^i$ then, for any choice of  $\vecs_1$  such that $\lambda^i(\vecs_1)<0$, the sign of  $\lambda^j$ changes from zero at $\vecs_0$ to positive at $\vecs_1$.  On the other hand, if  $\lambda^j$ and $\lambda^i$ were positively proportional, we can adapt the proof by grouping all  exponents positively proportional to $\lambda^i$ together into a single \emph{coarse Lyapunov exponent} (see \cref{sss:cle}) and then study the leaf-wise measures along higher-dimensional \emph{coarse Lyapunov  manifolds} (see \cref{sss:clm}.)

	For an explicit example of a higher-rank action where the proof of  \cref{lem:Pipart} (and consequently  \cref{thm:KS}) fails, let $$ A= \left(\begin{array}{cc}2 & 1 \\1 & 1\end{array}\right) \quad \text{and} \quad B= \left(\begin{array}{cc}2 & 3 \\3 & 5\end{array}\right).$$  We have eigenvalues $$\chi^1_A>1>\chi^2_A, \quad \quad  \chi^1_B>1>\chi^2_B>0$$ and both  $L_A\colon \T^2\to \T^2$ and 
	$L_B\colon \T^2\to \T^2$ are Anosov.  Consider 
	$$L_A\times \id \colon \T^4\to \T^4, \quad \quad \id\times   L_B \colon \T^4\to \T^4.$$
	Clearly $ L_A\times \id $ and $\id\times   L_B$ commute and so generate a $\Z^2$-action by automorphisms of $\T^4$. Note that while the generators   $ L_A\times \id $ and  $ \id \times L_B$  are not Anosov, the $\Z^2$-action contains Anosov diffeomorphisms such as  $ L_A\times L_B $.
	
	Let $\mu_1$ be any ergodic, $L_A$-invariant measure on $\T^2$  and  let $\mu_2$ be any ergodic, $L_B$-invariant measure on $\T^2$.  Note that we can pick $\mu_1$ different from Lebesgue with $h_{\mu_1}(L_A)>0$. 
	Then $\mu_1\times \mu_2$ is   an ergodic,  $\Z^2$-invariant measure on $\T^4$ that is not Lebesgue and has $$h_{\mu_1\times \mu_2} (L_A\times \id)>0.$$
	
	The conclusion of Theorem  \ref{thm:KS} thus fails for this $\Z^2$-action on $\T^4$.  The  proof degenerates in a number of places. 
	\begin{enumerate}
		\item The $\Z^2$-action on $\T^4$ has four Lyapunov exponent functionals: $$(n_1, n_2)\mapsto n_1 \lambda^1_A  , \quad n_1 \lambda^2_A ,\quad n_2 \lambda^1_B ,\quad n_2 \lambda^2_B .$$
		We note that the functionals  $$(n_1, n_2)\mapsto n_1 \lambda^1_A, \quad \quad(n_1, n_2)\mapsto n_1 \lambda^2_A$$ are {negatively proportional}.  Similarly $$(n_1, n_2)\mapsto n_2 \lambda^1_B, \quad \quad(n_1, n_2)\mapsto n_2 \lambda^2_B$$  are {negatively proportional}.\footnote{In the literature,  negatively proportional exponents such as $(n_1, n_2)\mapsto n_1 \lambda^1_A$ and $(n_1, n_2)\mapsto n_1 \lambda^2_A$ are often referred to as a \emph{symplectic pair}.}
		The presence of negatively proportional Lyapunov exponents makes it impossible  to choose  $\vecs_1$ in the proof of  Lemma \ref{lem:Pipart} above with the desired properties.   In fact, the conclusion of \ref{lem:Pipart} is   false for this example.
		\item Consider the map $h\colon \T^4= \T^2\times \T^2 \to \T^2$ given by  $h(x,y) = x$.  Then $h$ semiconjugates the $\Z^2$-action generated by $ L_A\times \id $ and $\id\times   L_B$ with the $\Z$-action generated by $L_A$.  For any $L_A$-invariant measure $\mu$ on $\T^2$,  we can find a $\Z^2$-invariant measure on $\T^4$ projecting to $\mu$ under $h$.  In the language of \cite{MR1406432}, $L_A$ is a \emph{rank-1 factor}.
	\end{enumerate}
	To state the general version (\cite[Theorem 5.1]{MR1406432})  of Theorem \ref{thm:KS},  Katok and Spatzier   impose additional hypotheses on the  $\Z^2$-action to rule out the defects discussed above. For the action in Example \ref{ex:Key}, neither of these defects occurs.   
	
	The primary obstruction to the rigidity in this example is the presence of rank-1 factors.  For genuinely higher-rank actions with  negatively proportional   pairs,   other tools developed in  \cite{MR2029471} can be used to overcome the failure of Lemma \ref{lem:Pipart}.  
	
\end{remark}

\starsubsection{Algebraic construction of the suspension space $N$} \label{sec:susp}\index{suspension}
We outline the construction in Proposition \ref{prop:susp}.  Recall our fixed commuting matrices $A$ and $B$ are jointly diagonalizable: there is  $Q\in \Gl(3,\R)$ with 
$$Q\inv AQ= \left(\begin{array}{ccc}e^{\lambda^1_A} & 0 & 0 \\0 & e^{\lambda^2_A} & 0 \\0 & 0 &e^{\lambda^3_A}\end{array}\right), \quad \quad 
Q\inv BQ= \left(\begin{array}{ccc}e^{\lambda^1_B} & 0 & 0 \\0 & e^{\lambda^2_B} & 0 \\0 & 0 &e^{\lambda^3_B}\end{array}\right).$$ 
Also recall our Lyapunov exponent functionals $\lambda^j\colon \R^2\to \R$, given by $$\lambda^j(t_1,  t_2) = t_1 \lambda^j_A + t_2  \lambda^j_B.$$

Given $\vect= (t_1,t_2)\in \R^2$ let $M^\vect\in \Gl(3,\R)$ be the interpolation matrix
$$M^\vect = Q \left(\begin{array}{ccc}e^{\lambda^1(t_1, t_2)} & 0 & 0 \\0 & e^{\lambda^2(t_1, t_2)} & 0 \\0 & 0 &e^{\lambda^3(t_1, t_2)} \end{array}\right) Q\inv.$$
Note that for $\vect = (n_1,n_2)\in \Z^n$, $M^\vect = A^{n_1} B^{n_2}\in \Gl(3,\Z)$.  However, for $\vect\notin \Z^3$, we expect $M^\vect\notin \GL(3,\Z)$; in particular, $M^\vect$ does not define an action on the torus $\T^3$.  

For $\vect\in \R^2$ define the ``twisted'' lattice subgroup $\Lambda_\vect \subset \R^3$ by 
$$\Lambda_\vect  = M^\vect \Z^3.$$  Note that  if $\vecm\in \Z^2$ then $\Lambda_\vecm $ is the standard integer lattice $ \Z^3$.  
For  $\vect\in \R^2$ define a ``twisted torus'' $T_\vect$ by $$T_\vect = \R^3/ \Lambda_\vect .$$ Note that  if $\vecm\in \Z^2$ then $T_\vecm$ is the standard torus $\T^3$; also if $\vect'= \vect + \vecm$ where $\vecm\in \Z^2$ then $\Lambda_\vect = \Lambda_{\vect'}$ whence  $T_\vect = T_{\vect'}$.  

Consider $\R^2\times \R^3$.  Let $\Z^3$ act on $\R^2\times \R^3$ as follows: given $(\vect ,x)\in \R^2\times \R^3$ and $\vecn \in \Z^3$ define $$(\vect,x) \cdot \vecn = (\vect, x + M^\vect \vecn).$$
Let $\td N$ be the quotient of  $\R^2\times \R^3$ by this action.  Note that $\td N$ is a fiber-bundle over $\R^2$ whose fiber over $\vect$ is exactly the twisted torus  $T_\vect.$

Consider the following $\Z^2$-action on $\td N$: given  $\vecm\in \Z^2$ and $(\vect, x+ \Lambda_\vect )\in \td N$
$$(\vect, x+\Lambda_\vect ) \cdot \vecm = (\vect + \vecm, x+\Lambda_\vect).$$
Also consider the following $\R^2 $-action on $\td N$:  given $\vecs\in \R^2$ and $(\vect, x+\Lambda_\vect )\in \td N$ 
$$\vecs\cdot (\vect, x+\Lambda_\vect  ) = (\vecs+ \vect, M^\vecs x+\Lambda_{\vect  +\vecs}  ).$$

Let $N$ be the quotient manifold $\td N/\Z^2 $ where the quotient is by the $\Z^2 $-action  described above.  As the $\R^2$- and $\Z^2 $-actions on $\td N$ commute, the $\R^2$-action on $\td N$ descends to an $\R^2$-action on $N$ which we denote by $\td \alpha$. 

Note that $N$ is a $5$-dimensional manifold which fibers over the torus $\T^2= \R^2/\Z^2$  and has $3$-dimensional fibers  where the fiber over $\vect +\Z^2\in \T^2$ is $T_\vect.$
Also  note that  $N$ inherits a Riemannian metric from $\R^2\times \R^3$.  
Given $x\in  N$ the tangent space $T_xN$ decomposes as 
$$T_xN=\R^2 \oplus E^1 \oplus E^2 \oplus E^3 $$
where $E^i$ are the joint eigenspaces of $A$ and $B$ enumerated as before.  

Each $E^i$ acts on $N$ as follows: given $p = (\vect + \Z^2, x + \Lambda_\vect)\in N$ and $v\in E^i$
$$p+ v= (\vect + \Z^2, x + \Lambda_\vect) + v  = (\vect + \Z^2, x + v+ \Lambda_\vect).$$
Note that the  orbit $W^i(p) = \{ p+v : v\in E^i\}$ is contained in the fiber through $p$. 
Moreover,  for any $p =  (\vect + \Z^2, x + \Lambda_\vect)\in  N$,  any  $p' = p + v$ for $v\in E^i$,  and any $\vecs= (s_1, s_2) \in \R^2$ we have 
$$\td \alpha(\vecs)(p') = \td \alpha(\vecs)(p) + e^{s_1\lambda^i_A + s_2 \lambda^i_B} v= \td \alpha(\vecs)(p) + e^{\lambda^i(\vecs) } v.$$

\part{Primer: smooth ergodic theory for $\Z^d$-actions}\label{part:II}
We  present   background material and a number of tools from the theory of nonuniformly hyperbolic dynamics that will be used in Part \ref{part:III} and explain a number of facts and constructions that were used in Part \ref{part:I}.  We will be particularly interested in  the relation between entropy, conditional measures, and Lyapunov exponents for single diffeomorphisms and for actions of higher-rank abelian groups.

\section{Lyapunov exponents and Pesin manifolds} \label{sec:koko}

\subsection{Lyapunov exponents for  diffeomorphisms} \label{ss:lyapexp}
Let $f\colon M\to M$ be a $C^{1}$ diffeomorphism of a compact manifold $M$.  Let $\mu$ be an ergodic, $f$-invariant Borel probability measure. We recall Oseledec's Theorem \cite{MR0240280}; see also \cites{MR1073779,MR0571089}.  

\label{sss:osel}
\begin{theorem}[Oseledec \cite{MR0240280}]  
	\label{thm:oseled}\index{theorem!Oseledec}
	There are
	\begin{enumerate}
		\item a measurable set $\Lambda$ with $\mu(\Lambda)=1$;
		\item  numbers $\lambda^1> \lambda^2>\dots >\lambda^p$;
		\item    a $\mu$-measurable, $Df$-invariant  splitting $T_xM = \bigoplus_{i=1}^p E^i(x)$ defined for $x\in \Lambda$
	\end{enumerate}
	such that for every $x\in \Lambda$
	\begin{enumlemma}
		\item for every $v\in E^i(x)\sm \{0\}$ $$\displaystyle \lim_{n\to \pm \infty} \frac 1 n \log \|D_xf^n(v)\|= \lambda^i;$$
		\item if $Jf$ denotes the Jacobian determinant of $f$ then $$\displaystyle\lim_{n\to \pm \infty} \frac 1 n \log | {J}f^n|= \sum_{i=1}^p m_i \lambda^i$$
		where $m_i = \dim E^i(x)$;
		\item for every $i\neq j$
		we have $$\displaystyle \lim_{n\to \pm  \infty} \frac{1}{n} \log \Big(\sin \angle \Big(E^i(f^n(x)), E^j(f^n(x))\Big) \Big)= 0. $$
	\end{enumlemma}
\end{theorem}
The numbers $\lambda^i$ are called the \emph{Lyapunov exponents}\index{Lyapunov exponent} of $f$ with respect to $\mu$ and the subspaces $E^i(x)$ are called the \emph{Oseledec's subspaces}.  Above,  $m^i$ denotes the almost-surely constant value of $\dim E^i(x)$, called  the \emph{multiplicity} of $\lambda^i$.

Given any $f$-invariant measure  $\mu$ on $M$ (which may be nonergodic) the \emph{average top  Lyapunov exponent} \index{Lyapunov exponent!top} of $f$ with respect to $\mu$ is 
\begin{equation}\label{topex1}\lambda_\av(f,\mu)  =  \inf _{n\ge 1} \frac 1 n  \int \log \| {D_xf^n} \| \ d \mu (x).\end{equation}
Since $\mu$ is $f$-invariant, the sequence $n\mapsto \int \log \| {D_xf^n} \| \ d \mu (x)$ is subadditive and the infimum in \eqref{topex1} can be replaced by  a limit.

\index{theorem!subadditive ergodic}
By the subadditive ergodic theorem, the functions $$x\mapsto \frac 1 n  \log \| {D_xf^n} \|$$  converge \ae to an invariant, integrable function with  integral  $\lambda_\av(f,\mu)$; see \cites{MR0121828,MR0254907} and \cite[Chapter 3]{MR3289050}. 
If $\mu$ is ergodic, we  have in the notation of \cref{thm:oseled} that     $\lambda_\av(f,\mu) = \lambda^1$.  If $\mu$ is not ergodic, let $\{\mu^\erg_x\}$ denote the ergodic decomposition (see \cref{def:ergdecom}) of  $\mu$ and let $\lambda^1_x> \lambda_x^2>\dots >\lambda^{p(x)}_x$ denote the Lyapunov exponents of $f$ with respect to the ergodic invariant measure $\mu^\erg_x$.  Then  we have  $$\lambda_\av(f,\mu) = \int \lambda^1_x \ d \mu(x).$$

\subsection{Lyapunov exponents and (sub)exponential growth of derivatives}   
Let $M$ be a compact manifold and equip $TM$ with a background Riemannian metric and associated norm.  Let $f\colon M\to M$ be a $C^1$ diffeomorphism.  We say $f\colon M\to M$ has \emph{uniform subexponential growth of derivatives}\index{uniform subexponential growth of derivatives} if for all $\epsilon>0$ there is a $C_\epsilon>0$ such that $$\|Df^n\| := \sup_{x\in M } \|D_x f^n\| < C_\epsilon e^{\epsilon |n|}\quad \quad \text{for all   $n\in \Z.$}$$

Note that we allow  that $C_\epsilon\to \infty$ as $\epsilon \to 0$.

\begin{proposition}\label{prop:ZLEUSEGOD}
	A diffeomorphism $f\colon M\to M$ has  {uniform subexponential growth of derivatives} if and only if for any $f$-invariant Borel probability measure $\mu$, all Lyapunov exponents of $f$ with respect to $\mu$ are zero.
	
	That is, $f\colon M\to M$ has  {uniform subexponential growth of derivatives} if and only if $\lambda_\av(f,\mu)= \lambda_\av(f\inv,\mu)=0$ for every  $f$-invariant Borel probability measure  $\mu$.  
\end{proposition}
\begin{proof}
	We show that vanishing of all Lyapunov exponents for all $f$-invariant probability measures implies that $f$ has uniform subexponential growth of derivatives; the converse is clear. 
	
	Suppose that $f\colon M\to M$ fails to have uniform subexponential growth of derivatives.  Then there is an $\epsilon>0$ and   sequences of iterates $n_j\in \Z$ with   $|n_j|\to\infty$, base points $x_j\in M$, and unit vectors $v_j \in T_{x_j}M$  such that 
	\begin{equation}\label{eq:growth} \|D_{x_j} f^{n_j} v_j\| \ge  e^{\epsilon |n_j|}. \end{equation}
	Replacing $f$ with $f\inv$, we may assume without loss of generality that $n_j\to \infty.$
	
	Let $UM\subset TM$ denote the unit-sphere bundle.  We represent an element of $UM$ by a pair $(x,v)$ where $v\in T_xM$ with $\|v\|= 1$.  Note that $UM$ is compact.  Note also that $Df\colon TM\to TM$ induces a map $Uf\colon U\to U$ given by the renormalized  derivative: $$Uf(x,v) := \left(f(x), \frac{D_xf(v)}{\|D_xf(v)\|}\right).$$
	Define $\Phi\colon UM\to \R$ as follows: given $(x,v)\in UM$, let $$\Phi(x,v) := \log \|D_xf(v)\|.$$
	By the chain rule, we have 
	$$\log \|D_xf^n(v)\|  = \sum_{j=0}^{n-1} \Phi(Uf^j(x,v)).$$
	
	For each $j$, let $\nu^j$ denote the empirical measure along the orbit segment $$(x_j,v_j),Uf(x_j,v_j), \dots,  Uf^{n_j-1}(x_j,v_j)$$ in $UM$ given by $$\nu^j = \frac{1}{n_j} \sum_{k=0}^{n_j-1}  \delta_{Uf^k(x_j,v_j)}.$$
	From \eqref{eq:growth} we have   for every $j$ that  $$\int \Phi \ d \nu^j \ge \epsilon.$$  
	\begin{claim}\label{booboob}\
		Let $\nu$ be any weak-$*$ subsequential limit  of $\{\nu^j\}$.  Then 
		\begin{enumlemma}
			\item \label{peepee}$\nu$ is $Uf$-invariant;
			\item\label{poopoo} $\displaystyle \int \Phi \ d \nu \ge \epsilon$.
		\end{enumlemma}
	\end{claim}
	\begin{proof}Conclusion \ref{peepee} follows as in the proof of the Krylov-Bogolyubov theorem:  if $\phi\colon M\to \R$ is any (bounded) continuous function then  \begin{equation}\label{eq:KB} \lim _{j\to \infty} \left |\int \phi \ d \nu^j - \int \phi \circ f \ d \nu^j\right | \le \lim _{j\to \infty}\frac{2 \|\phi\|_{C^0}}{n_j} = 0\end{equation}
		showing that $\nu$ is $f$-invariant. 
		Conclusion \ref{poopoo} follows from continuity of $\Phi$ and   weak-$*$ convergence.  \end{proof}
	
	From \cref{booboob}\ref{poopoo}, we may replace $\nu$ with an ergodic component (see \cref{def:ergdecom})   $\nu'$ of $\nu$ such that $\int \Phi \ d \nu' \ge \epsilon$.  
	
	Take $\mu$ to be the   push-forward of $\nu'$ under the natural projection $UM\to M$.  Then $\mu$ is an $f$-invariant, ergodic measure on $M$.
	Let $\{\nu'_x\}$ denote a family of conditional measures of $\nu'$ for the partition of $UM$ into fibers over $M$.  
	By the pointwise ergodic theorem, for $\mu$-\ae $x\in M$  and $\nu'_x$-\ae $v\in UM(x)$ we have 
	$$\lim_{n\to  \infty} \frac{1}{n}\log \|D_xf^n(v)\| = \lim_{n\to  \infty} \frac{1}{n} \sum_{j=0}^{n-1} \Phi(Uf^j(x,v)) = \int \Phi \ d \nu' \ge \epsilon.$$
	
	On the other hand,
	\begin{align*}
		\lambda_\av(f,\mu)  &=  \lim _{n\to \infty} \frac 1 n  \int \log \| {D_xf^n} \| \ d \mu (x)\\
		&=  \lim _{n\to \infty} \frac 1 n  \int  
		\sup_{v\in UM(x)} \sum_{j=0}^{n-1} \Phi(Uf^j(x,v))
		\ d \mu (x)\\
		&\ge  \lim _{n\to \infty}  \int \int   \frac 1 n 
		\sum_{j=0}^{n-1} \Phi(Uf^j(x,v)) \ d \nu_x(v)
		\ d \mu (x)\\
		&= \int \Phi \ d \nu' \ge \epsilon
		.\end{align*}
	Above, the inequality follows from comparing the maximal  growth  with the  average growth (averaged by   $ \nu_x'$.)
	It follows that the largest Lyapunov exponent of  $f$  with respect to  $\mu$ is at least $ \epsilon>0$.  
\end{proof}

\subsection{Lyapunov exponents for nonuniformly hyperbolic $\Z^d$-actions}\label{sec:lyap}
How does the theory of Lyapunov exponents change for actions of more general abelian groups?  We state a version of Oseledec's theorem for actions of $\Z^d$ which can easily be extended to actions of $\R^\ell\times \Z^k$.  One should think of the following as a non-stationary version of the joint-diagonalizability or joint-Jordan-normal-form for commuting matrices as exploited in Section \ref{sec:penisinthesalsa}.

Let $M$ be a compact manifold, let $\alpha\colon \Z^d \to \Diff^{1} (M)$ be a  $\Z^d$-action, and let $\mu$ be an ergodic, $\alpha$-invariant measure.  
\begin{theorem}[Higher-rank Oseledec's theorem (see \cite{AWB-GLY-P1})]\label{thm:oscHR}\index{theorem!Oseledec!higher-rank}
	There are
	\begin{enumerate}
		\item a measurable set $\Lambda$ with $\mu(\Lambda)=1$;
		\item  linear functionals  $\lambda^1, \lambda^2,\dots,\lambda^p\colon \R^d\to \R$;
		\item    a $\mu$-measurable, $D\alpha$-invariant  splitting $T_xM = \bigoplus_{i=1}^p E^i(x)$ defined for $x\in \Lambda$
	\end{enumerate}
	such that for every $x\in \Lambda$
	\begin{enumlemma}
		\item \label{kk} for every $v\in E^i(x)\sm \{0\}$ $$\lim_{|n|\to   \infty} \dfrac { \log \|D_x\alpha(n)(v)\| -  \lambda^i(n)}{|n|} = 0;$$
		\item \label{klol} if $Jf$ denotes the Jacobian determinant of $f$ then $$\displaystyle\lim_{|n|\to   \infty} \frac { \log | {J}\alpha(n)| - \sum_{i=1}^p m^i \lambda^i(n)}{|n|};$$
		\item for every $i\neq j$
		$$ \lim_{n\to \infty} \frac{1}{|n|} \log \Big(\sin \angle \Big(E^i(\alpha(n)(x)), E^j(\alpha(n)(x))\Big) \Big)= 0. $$
	\end{enumlemma}
\end{theorem}

The linear functionals $\lambda^1, \lambda^2,\dots,\lambda^p\colon \R^d\to \R$ are called the \emph{Lyapunov exponent functionals}  or simply the \emph{Lyapunov exponents} of $\mu$.\index{Lyapunov exponent!functionals}\index{Lyapunov exponent}
In \ref{klol}, $m^i $ is the almost-surely constant value of $\dim E^i(x)$, called the \emph{multiplicity} of $\lambda^i$.  
Note that \ref{kk} implies convergence along rays: for any $n\in \Z^d$  and $v\in E^i(x)\sm \{0\}$
\begin{equation}\label{rays} \lim_{k\to   \infty} \dfrac { 1}k \log \|D_x\alpha(k n)(v)\|  =   \lambda^i(n).\end{equation}
The convergence in \ref{kk} is taken along \emph{any} sequence $n\to \infty$; this is  stronger than \eqref{rays} and is typically needed in applications.

\subsection{Unstable manifolds  and coarse Lyapunov manifolds} \label{sec:unstmanifold}

\newcommand{\cA}{{\mathcal A}}
\newcommand{\cF}{{\mathcal F}}
\newcommand{\cW}{{\mathcal W}}

\subsubsection{Unstable subspaces and unstable manifolds for a single diffeomorphism} \label{unstmanifold11}
Let $f\colon M\to M$ be a $C^{1}$ diffeomorphism of $M$ and let $\mu$ be an ergodic, $f$-invariant measure.  Let $\lambda^i$ be the Lyapunov exponents for $f$ with respect to $\mu$.  For $x\in \Lambda\subset M$, where $\Lambda$ is as in \cref{thm:oseled}, define $$E^u(x):= \bigoplus_{\lambda^i>0} E^i(x)$$
to be the \emph{unstable subspace} through $x$.  
We have that $$E^u(x):= \{ v\in T_xM : \limsup_{n\to \infty} \frac 1 n \log \| D_x f^{-n} (v)\| < 0\}.$$
We may similarly define \emph{stable} and \emph{neutral} (or center) subspaces through $x$, respectively, by $$E^s(x):= \bigoplus_{\lambda^i<0} E^i(x)$$ and $$E^c(x):= \bigoplus_{\lambda^i=0} E^i(x).$$

We now assume that $f\colon M\to M$ is $C^{1+\beta}$ for $\beta>0$.  
Through $\mu$-almost every point $x$ the set $$W^u(x) := \left\{y : \limsup_{n\to \infty} \frac 1 n \log (d(f^{-n}(x), f^{-n}(y)))<0 \right\}$$
is a connected $C^{1+\beta}$ injectively immersed manifold with $T_x W^u(x) = E^u(x)$ (see \cite{MR0458490}) called the \emph{(global) unstable Pesin  manifold of $f$ through $x$}\index{Pesin!unstable manifold}.  The collection of all $W^u(x)$ forms a partition of (a full measure subset of) $M$; in general, this partition does not have the structure of a nice foliation.  However, restricted to sets of large measure the partition into local unstable manifolds has the structure of a continuous lamination. 
That is, for almost every $x\in M$  and  any $\epsilon>0$ there is a neighborhood $U$ of $x$ such that, on a set $\Omega$ of  relative measure $(1-\epsilon)$ in   $U$,  the local leaves of $W^u$-manifolds form a partition of $\Omega$ by embedded $\dim(E^u)$-dimensional balls that vary continuously in the $C^{1+\beta}$-topology.  

Given the Lyapunov exponents $\lambda^1> \lambda^2>\dots >\lambda^p$ of $\mu$, fix $j\in \{1,\cdots, p\}$ such that $\lambda^j>0$.  Then, for almost every $x$, the set
$$W^j(x) := \left\{y : \limsup_{n\to \infty} \frac 1 n \log (d(f^{-n}(x), f^{-n}(y)))\le -\lambda^j \right\}$$
is again a connected, $C^{1+\beta}$ injectively immersed manifold with $$T_x W^j(x) = \bigoplus_{\lambda^i\ge \lambda^j} E^i(x)$$ called the \emph{(global) $j$th unstable manifold} through $x$.   
We remark that, in general, the intermediate unstable distributions, $E^i(x)$ for $\lambda^i>0$, do not integrate to  invariant family of immersed submanifolds.

\subsubsection{Coarse Lyapunov exponents and subspaces}\label{sss:cle}
Let $\alpha\colon \Z^d\to \diff^1 (M)$ be an action and let $\mu$ be an ergodic, $\alpha$-invariant probability measure.    We introduce objects that  play the role of unstable subspaces and unstable manifolds for the $\Z^d$-action  $\alpha$.

Given Lyapunov exponents $\lambda^1, \lambda^2,\dots,\lambda^p\colon \R^d\to \R$ we say $\lambda^i$ and $\lambda^j$ are \emph{positively proportional} if there is a $c>0$ with $$\lambda^i = c\lambda^j. $$
Note that this defines an equivalence relation on the linear functionals $$\lambda^1, \lambda^2,\dots,\lambda^p\colon \R^d\to \R.$$  
The positive proportionality  classes are called \emph{coarse Lyapunov exponents}.\index{coarse Lyapunov!exponent}\index{Lyapunov exponent!coarse} 
For a $\Z$-action generated by  a single diffeomorphism $f$, the coarse Lyapunov exponents are simply the collections of positive, zero, and negative Lyapunov exponents.  

Let $\chi= \{ \lambda^i\}$ be a coarse Lyapunov exponent.  While the size of $\chi(n)$ is not well defined, the sign of $\chi(n)$ is well defined.    Write $$E^\chi(x) = \oplus _{\lambda^i\in \chi}E^i(x)$$
called the corresponding \emph{coarse Lyapunov subspace.}
\index{coarse Lyapunov!subspace}

\subsubsection{Coarse Lyapunov manifolds for  $\Z^d$-actions}\label{sss:clm}
Analogous to the existence and properties of unstable Pesin manifolds for nonuniformly hyperbolic diffeomorphisms we have the following for actions of higher-rank abelian groups.  

Let $\alpha\colon \Z^d\to \diff^{1+\beta} (M)$ be an action and let $\mu$ be an ergodic, $\alpha$-invariant probability measure. Let $\Lambda$ be as in Theorem \ref{thm:oscHR}.
\begin{proposition}
	For almost every $x\in \Lambda$ and for every coarse Lyapunov exponent $\chi$ there is a connected, $C^{1+\beta}$, injectively immersed manifold $W^\chi(x)$ satisfying the following:
	\begin{enumerate}
		\item  $T_x W^\chi(x) = E^\chi(x)$;
		\item $\alpha(n) W^\chi(x) = W^\chi(\alpha(n)(x))$ for all $n\in \Z^d$;
		\item $W^\chi(x)$ is the set of all $y\in M$ satisfying $$\limsup_{k\to \infty} \frac 1 k \log d(\alpha(-kn) (y) , \alpha(-kn)(x)) <0 \text{ for all $n\in \Z^d$ with $\chi(n)>0$}.$$
	\end{enumerate}
\end{proposition}
\noindent The manifold $W^\chi(x) $ is called the \emph{coarse Lyapunov manifold} through $x$ associated with the coarse Lyapunov exponent $\chi$.\index{coarse Lyapunov!manifold}\index{Lyapunov manifold!coarse}

To construct  $W^\chi$-manifolds, given $n\in \Z^d$ with $\chi(n)>0$ let $W_{\alpha(n)}^u(x)$ denote the unstable manifold for the diffeomorphism $\alpha(n)\colon M\to M$ through $x$.  Then, for almost every $x\in M$ the manifold  $W^\chi(x)$ is the 
path component of the intersection $$\bigcap_{n\in \Z^d, \chi(n)>0}W_{\alpha(n)}^u(x)$$ containing $x$.

\section{Metric entropy}

\subsection{Metric entropy}\label{sec:ME}  General references for this subsection include \cite{MR02172581,MR1086631}.  Throughout, we take $(X,\mu)$ to be a \emph{standard probability space}.  That is, $(X,\mu)$ equipped with the $\sigma$-algebra of $\mu$-measurable sets is measurably isomorphic to an interval equipped with the Lebesgue measure and a countable  number of point masses;  see for instance \cite[Chapter 2]{MR1086631}. 

\subsubsection{Measurable partitions and conditional measures}\label{ss:measpart}
Recall that  a partition $\xi$ of $(X,\mu)$ is   \emph{measurable}\index{partition!measurable} if the quotient $(Y ,\hat \mu ) :=(X,\mu)/\xi$ is a {standard probability space}.  See also Definition \ref{B:meas_partition} in Appendix \ref{B:secMeasurablePart}.  
This is a technical but crucial condition.   For  more discussion and other characterizations of measurability see \cite{1208.4550}, \cite{MR0047744}, and   Appendices \ref{App:rokhlin} and \ref{App:Pinsker}.  

A key property of measurable partitions is the existence and uniqueness of a family of conditional measures  (or a \emph{disintegration})  of $\mu$ relative to this partition.   Given a partition $\xi$ of $X$, for $x\in X$ we write $\xi(x)$ for the element of $\xi$ containing $x$.
\begin{definition}\label{def:condmeas} Let $\xi$ be a measurable partition of $(X, \mu)$.  Then there is family of Borel probability measure $\{\mu^\xi_x\}_{x\in X}$, called a \emph{family of conditional measures}\index{measure(s)!conditional} of $\mu$ relative to $\xi$, with the following properties: For almost every $x$
	\begin{enumerate}
		\item $\mu_x^\xi$ is a Borel probability measure on $X$ with $\mu^\xi_x(\xi(x)) = 1$;
		\item if $y\in \xi(x)$ then $\mu_y^\xi= \mu_x^\xi$.
	\end{enumerate}
	Moreover, if $D\subset X$ is a Borel subset then 
	\begin{enumerate}[resume]
		\item $x\mapsto \mu_x^\xi(D)$ is measurable and 
		\item $\mu(D) = \int \mu_x^\xi(D) \ d \mu(x)$.
	\end{enumerate}
	Such a family is unique modulo $\mu$-null sets.  
\end{definition}
For construction and properties of $\{\mu^\xi_x\}$ see for instance \cite{MR0047744}.  See also Appendix \ref{B:secDisintegration} for further discussion.  

\subsubsection{Conditional information and conditional entropy}\label{ss:condentropy}
Given a measurable partition  $\xi$ of a standard probability space $(X,\mu)$, write $\{\mu_x^\xi\} $ for a family of conditional measures of $\mu$ with respect to the partition $\xi$.  
Given two measurable partitions $\eta, \xi$ of $(X,\mu)$ 
the \emph{conditional information} of $\eta$ relative to $\xi$  is   $$I_\mu(\eta \mid \xi) (x)= -  \log (\mu^\xi_x(\eta(x)))$$  and  
the \emph{conditional entropy}\index{entropy of a partition!conditional} of $\eta$ relative to $\xi$ is $$H_\mu(\eta \mid \xi) = \int I_\mu(\eta \mid \xi) (x) \ d \mu(x).$$ 
The \emph{join}\index{partition!join} $\eta\vee \xi$ of two partitions $\eta$ and $\xi$ is $$\eta\vee \xi = \{A\cap B\mid A \in \eta, B\in \xi\}.$$
The \emph{entropy} of $\eta$ is  $H_\mu(\eta) = H_\mu(\eta \mid \{\emptyset, X\})$.  Note that if $H_\mu(\eta)<\infty$ then $\eta$ is necessarily countable (mod zero) and $H_\mu(\eta) = -\sum_{C\in \eta} \log (\mu(C)) \mu (C).$
\index{entropy of a partition}

\subsubsection{Metric entropy of a transformation}
Let $f\colon (X,\mu)\to (X,\mu)$ be an invertible, measurable, measure-preserving transformation.  
Let $\eta$ be an arbitrary measurable partition of $(X,\mu)$.
We define 
$$\eta^+ := \bigvee_{i=0}^\infty f^i \eta, \quad \quad \quad \eta^f:= \bigvee_{i\in \Z}^\infty f^i \eta.$$
We define the 
\emph{entropy of $f$ given the partition $\eta$} to be $$h_\mu(f, \eta):= H_\mu( \eta  \mid f \eta^+) =  H_\mu( \eta^+ \mid f \eta^+) = H_\mu( f\inv \eta^+ \mid \eta^+).$$ 
We define the \emph{$\mu$-metric entropy of $f$}\index{entropy of a transformation} to be  $h_\mu(f)= \sup \{ h_\mu(f, \eta)\}$ 
where the supremum is taken over  {all measurable} partitions of $(X,\mu)$. 
If $$\mu = \alpha \mu_1 + \beta \mu_2$$ where $\alpha, \beta\in [0,1]$ satisfy  $\alpha+\beta =1$  and $\mu_1$ and $\mu_2$ are $f$-invariant Borel probability measures then \begin{equation}\label{eq:convex}
	h_\mu(f) = \alpha h_{\mu_1}(f) + \beta h_{\mu_2}(f).
\end{equation}

\subsection{Entropy under factor maps}
Let $(X,\mu)$ and $(Y,\nu)$ be standard probability spaces.  Let $f\colon X\to X$ and $g\colon Y\to Y$ be measure-preserving transformations.   Suppose there is a measurable map $\psi\colon X\to Y$ with $$\psi_*\mu = \nu$$ and $$\psi\circ f = g\circ \psi.$$
In this case, we say that $g\colon (Y,\nu) \to (Y,\nu)$ is a \emph{measurable factor} of $f\colon (X,\mu) \to (X,\mu)$.  

We note   that entropy only decreases under measurable factors: if $ g\colon (Y,\nu) \to (Y,\nu)$ is a {measurable factor} of $f\colon (X,\mu) \to (X,\mu)$ then
$$h_\nu(g) \le h_\mu(f).$$
The difference between the entropies $h_\nu(g) $ and $ h_\mu(f)$ is captured by the {Abramov--Rokhlin theorem}.  Let $\zeta$ be the measurable partition of $(X, \mu)$ into level sets of $\psi\colon X\to Y$.  Note that $\zeta$ is an $f$-invariant partition: $\zeta= \zeta^f$.  Define the conditional entropy $h_\mu(f\mid \zeta)$ of $f$ relative to $\zeta$ to be 
$$h_\mu(f\mid \zeta)= \sup _\xi h_\mu(f, \xi \vee \zeta)$$
where, as usual, the supremum is over all measurable partitions $\xi$ of $(X, \mu)$.  
We call $h_\mu(f\mid \zeta)$ the \emph{fiberwise entropy}\index{entropy of a transformation!fiberwise} of $f$.  The Abramov--Rokhlin theorem\index{theorem!Abramov--Rokhlin} (see  \cites{MR0140660,MR1179170,MR0476995}) states the following:\label{AR}
\begin{equation}\label{eq:AR}
	h_\mu(f) =  h_\nu(g) + h_\mu(f\mid \zeta). 
\end{equation}

\subsection{Unstable entropy of a diffeomorphism}\label{unsentropy}
Let $f\colon M\to M$ be a $C^{1+\beta}$ diffeomorphism and let $\mu$ be an ergodic, $f$-invariant measure.  

\subsubsection{Partitions subordinate to a foliation}
For the following discussion and in most applications considered in this text, we may take  $\Fol$ to be an $f$-invariant foliation of $M$ with $C^{1+\beta}$ leaves.  More generally, we may take $\Fol$ to be, in the terminology introduced in \cite{AWB-GLY-P1}, an $f$-invariant, \emph{tame measurable foliation}; that is, $\Fol$ a partition of a full measure set by $C^{1+\beta}$ manifolds with the property that locally, restricting to sets of large measure, $\Fol$ has the structure of a continuous family of $C^{1+\beta}$ discs.  The primary examples of such measurable foliations include the partition into global $j$th unstable Pesin manifolds and the partition into global  coarse Lyapunov manifolds in the setting of $\Z^d$-actions.  Note that the partition into  global leaves of a measurable foliation is not necessarily a measurable partition; rather locally the partition looks like a measurable family of $C^{1+\beta}$ discs.  

Write $\Fol(x)$ for the leaf of $\Fol$ through $x$.   We say $\Fol$ is \emph{expanding} (for $f$) if  $\Fol(x)\subset W^u(x)$, i.e.\ if $\Fol(x)$ is a subset of the global unstable manifold through $x$ for $f$ discussed in Section \ref{sec:unstmanifold}. 
As a key example,   one should consider $\Fol^u$, the partition of $M$ into  full global unstable  manifolds.  
\begin{definition}\label{def:sub} We say a measurable partition $\xi$ is \emph{subordinate to $\Fol$}\index{partition!subordinate} if 
	\begin{enumerate}
		\item $\xi(x)\subset\Fol(x)$ for $\mu$-\ae $x$;
		\item $\xi(x)$ contains an open (in the immersed topology) neighborhood of $x$ in $\Fol(x)$ for $\mu$-\ae $x$;  
		\item $\xi(x)$ is precompact in (the immersed topology of) $\Fol(x)$ for $\mu$-\ae $x$;
	\end{enumerate}
\end{definition}

\subsubsection{Partial ordering on the set of partitions}\label{sec:parord}
\index{partition!partial order}
We recall the partial order on partitions of $(M,\mu)$.      Let $\xi$ and $\eta$ be   partitions of the probability space $(M,\mu)$.  We write   $$\eta\prec\xi$$ 
and say that $\xi$ is  \emph{finer} than  $\eta$ (or that $\eta$ is  \emph{coarser} than $\xi$) if there is a subset $X\subset M$ with $\mu(X)=1$ such that 
for almost every $x$, $$\xi(x)\cap X \subset \eta(x) \cap X.$$
We say $\eta= \xi$ if $\eta\prec\xi$ and $\xi\prec\eta$.   

\subsubsection{Entropy conditioned on a foliation}  
We say that  a partition $\xi$ is \emph{increasing} if $f\xi \prec\xi$ where $f \xi $ denotes the partition $f\xi = \{ f(C) \mid C\in \xi\}$.  \index{partition!increasing}
\begin{definition}\label{def:entsub}
	Given an expanding, $f$-invariant foliation $\Fol$ we define the \emph{entropy of $f$ conditioned on $\Fol$}\index{entropy of a transformation!conditioned on a foliation}  to be $$h_\mu (f\mid \Fol)= h_\mu (f, \xi)$$ where $\xi$ is any increasing, measurable partition subordinate to $\Fol$.  
\end{definition}
There are two small claim in Definition \ref{def:entsub}: First we have that $h_\mu (f, \xi_1) =h_\mu (f, \xi_2)$ for any two increasing partitions $\xi_1 $ and $\xi_2$ subordinate to $\Fol$;  see for example \cite[Lemma 3.1.2]{MR819556}.  Second, such a partition $\xi$ always exists.  This was shown when $\Fol= \Fol^u$ is the partition into global  unstable Pesin manifolds for a $C^{1+\beta}$ diffeomorphism in \cite{MR693976} (see also discussion in \cite[(3.1)]{MR819556}) extending a construction due to Sinai for uniformly hyperbolic dynamics \cites{MR0197684,Sinai1968}; the proof in \cite{MR693976} can be adapted for general invariant expanding $\Fol$.  

When $\Fol= \Fol^u$ is the partition into full unstable manifolds, define the \emph{unstable metric entropy of $f$}\index{entropy of a transformation!unstable} to be $$h^u_\mu(f):= h_\mu(f\mid \Fol^u).$$ 
The principal result  (Corollary 5.3)  of \cite{MR819556} shows that for   $C^2$ diffeomorphisms 
we have equality of the metric entropy of $f$ and the  unstable metric {entropy} of $f$: \begin{equation}\label{eq:LY} h_\mu(f) = h_\mu^u(f).\end{equation} 
For $C^{1+\beta}$-diffeomorphism without zero Lyapunov exponents  equality \eqref{eq:LY} was shown by Ledrappier in  \cite{MR743818}; for the general case of $C^{1+\beta}$-diffeomorphisms, \eqref{eq:LY} holds from \cite{1608.05886}.

\subsection{Entropy, exponents, and geometry of conditional measures.}  (See Appendix \ref{App:LY} for further details).
In this section, we consider the relationships between metric entropy $h_\mu(f)$, Lyapunov exponents, and the geometry of conditional measures along unstable manifolds.

Let $f\colon M\to M$ be a $C^{1+\beta}$ diffeomorphism and let $\mu$ be an ergodic, $f$-invariant measure.
At one extreme we have the following generalization of  Lemma \ref{entropyvatoms} characterizing invariant measures with zero entropy. 
\begin{lemma}\label{entropyvatoms2}
	Let  $\mu$ be an ergodic, $f$-invariant measure on $M$ and let $\xi$ be a measurable partition of $(M,\mu)$ subordinate  to the partition into unstable manifolds.   
	The following are equivalent:
	\begin{enumcount}
		\item\label{truckyou1} $h_\mu(f) = 0$;
		\item\label{truckyou2} 
		for   $\mu$-\ae $x$, the conditional measure $\mu^\xi_{x}$ has at least one atom;
		\item \label{truckyou3} 
		for   $\mu$-\ae $x$, the conditional measure $\mu^\xi_{x}$ is a single atom supported at $x$;
		\item\label{truckyou4} the partition of $(M, \mu)$ into full $W^u$-manifolds  is a measurable partition.
	\end{enumcount}
\end{lemma} 
\begin{proof}[Proof sketch]
	The implications \ref{truckyou1} $\implies$ \ref{truckyou4}  and \ref{truckyou1} $\implies$ \ref{truckyou3} are  a consequence of \cite{MR819556}*{Theorem B} (see also  \cite{1608.05886} for $C^{1+\beta}$ setting.)  Indeed, if $h_\mu(f) = 0$, then the Pinsker partition (see Section \ref{sss:pinsker} below) is the point partition.  From \cite[Theorem B]{MR819556}  we have that the Pinsker partition is the measurable hull of (and in particular is coarser than) the partition into full unstable manifolds.  As the point partition is the finest partition, it follows that the partition into full unstable manifolds is measurably equivalent to the point partition and \ref{truckyou3}  and \ref{truckyou4}  follow.  
	
	The implications \ref{truckyou4} $\implies$ \ref{truckyou3}   and    \ref{truckyou2} $\implies$ \ref{truckyou3}   follow from the dynamics on unstable manifolds and ergodicity of the measure.   For instance,  to see \ref{truckyou4} $\implies$ \ref{truckyou3}, assume  the partition of $(M,\mu)$ into full $W^u$-manifolds  is measurable and let $\{\mu^u_x\}$ denote a family of conditional probability measures for this partition.  As $\mu$ is $f$-invariant and as the partition into full unstable leaves is $f$-invariant, we have $f_* \mu^u_x = \mu^u_{f(x)}$ for almost every $x$.  
	
	Given $x\in M$, let $W^u(x,R)$ denote the metric ball of radius $R$ centered at $x$ in the internal metric of $W^u(x)$.
	Given   $\delta>0$ and  $R>0$, define the set $G_{\delta, R}$ of $(\delta,R)$-good points to be  $$G_{\delta, R}:= \{x\in M\mid \mu^u_x (W^u(x,R))\ge 1-\delta\}.$$
	Fix $R>0$ such that $\mu(G_{\delta, R})>0$.  Take a subset $G'\subset G_{\delta, R}$ with $\mu(G')>0$  such that the function $$x'\mapsto \diam^u_{f^{-n}(x')} (f^{-n}(W^u(x', R)))$$ converges to $0$ uniformly on $G'$ as $n\to \infty$ where $\diam^u_x(B)$ denotes the diameter of $B\subset W^u(x)$ with respect to the internal metric on $W^u(x)$.  
	For almost every $x$, we have $f^n(x)\in G'$ for infinitely many $n\in \N$.  For such $x$ and any $\epsilon>0$, there is $n_0\in \N$ such that for all $n\ge n_0$ with $f^n(x)\in G'$ we have 
	$$f^{-n}(W^u(f^n(x),R))\subset W^u(x, \epsilon)$$whence 
	$$\mu^u_x (W^u(x,\epsilon)) \ge  \mu^u_{f^n(x)} (W^u(f^n(x),R))   \ge 1-\delta.$$
	Taking $ \epsilon\to 0$ we have   $\mu^u_{x}(\{x\})\ge 1-\delta$ and, as $\delta$ was arbitrary, \ref{truckyou3} follows.  
	
	Finally, the implication \ref{truckyou3} $\implies$ \ref{truckyou1} follows from Corollary 5.3  of \cite{MR819556} (see \eqref{eq:LY} below) and the computation of unstable entropy in  \cref{def:entsub}.  
\end{proof}

At the other extreme, we have the following definition.  
\begin{definition}\label{def:SRB}\index{measure(s)!SRB}
	We say $\mu$ is an \emph{SRB} measure (or satisfies the \emph{ SRB property}) if, for any measurable partition $\xi$ of $(M,\mu)$ subordinate  to the partition into unstable manifolds, for almost every $x$ the conditional measure $\mu^\xi_x$ is absolutely continuous with respect to Riemannian volume on $W^u(x)$.  
\end{definition}

We have the following summary of a number of important results.  
\begin{theorem}\label{entropyfacts}
	Let $f\colon M\to M$ be a $C^{1+\beta}$ diffeomorphism and let $\mu$ be an ergodic, $f$-invariant measure.  Then 
	\begin{enumcount}
		\item \label{EFF1}$h_\mu(f)\le   \sum_{\lambda^i>0}m^i \lambda^i$;
		\item \label{EFF2}if $\mu$ is absolutely continuous with respect to volume then  $$h_\mu(f) = \sum_{\lambda^i>0}m^i \lambda^i;$$
		\item \label{EFF3}if $\mu$ is SRB then $h_\mu(f) = \sum_{\lambda^i>0}m^i \lambda^i$.
	\end{enumcount}
\end{theorem}
Theorem \ref{entropyfacts}\ref{EFF1},  known as the \emph{Margulis--Ruelle inequality}\index{Margulis--Ruelle inequality}, is  proven in  \cite{MR516310}. 
Theorem \ref{entropyfacts}\ref{EFF2},  known as the \emph{Pesin entropy formula}\index{Pesin!entropy formula}, is shown in \cite{MR0466791}.  
Theorem \ref{entropyfacts}\ref{EFF3} was established by Ledrappier and Strelcyn\index{theorem!Ledrappier--Strelcyn} in \cite{MR693976}.  
In the next section, we will complete Theorem \ref{entropyfacts} with Ledrappier's Theorem, Theorem \ref{thm:led}, which provides a converse to Theorem \ref{entropyfacts}\ref{EFF3}.

For general measures invariant under a $C^2$-diffeomorphism  (for the case of $C^{1+\beta}$-diffeomorphisms, see \cite{1608.05886}), Ledrappier and Young  explain explicitly the defect from equality in Theorem \ref{entropyfacts}\ref{EFF1}. This captures the intermediate geometry of measures with positive entropy (and hence non-atomic unstable conditional measures) but entropy strictly smaller than the sum of positive Lyapunov exponents.

Let $\delta^i$ denote the (almost-surely constant value of the)  pointwise  dimension of $\mu$ along the $i$th unstable manifolds; see Section \ref{S:LY2} in Appendix \ref{App:LY} for definition.    With $\delta^0=0$, let $$\gamma^i = \delta^{i} - \delta^{i-1}.$$  The  coefficients $\gamma^i$ reflect the transverse geometry (in particular the transverse dimension) of the measure $\mu$ inside of the $i$th unstable manifold transverse to the collection of  $(i-1)$th unstable manifolds.  In particular, we  have $\gamma^i\le m^i$ (see \cite[Proposition 7.3.2]{MR819557}.)
\begin{theorem}[\cite{MR819557}]\label{thm:LYII}\index{theorem!Ledrappier--Young}
	Let $f\colon M\to M$ be a $C^{1+\beta}$ diffeomorphism and let $\mu$ be an ergodic, $f$-invariant measure.  Then
	$$h_\mu(f)=    \sum_{\lambda^i>0}\gamma^i \lambda^i.$$
\end{theorem}
(Note that the proof in \cite{MR819557} required $f$ to be $C^2$; following {\cite{1608.05886} and \cite{MR1709302}*{Appendix}}, the theorem holds when $f\in C^{1+\beta}.)$

\starsubsection{Coarse-Lyapunov entropy  and entropy product structure}\label{sec:entprod}
Consider now   $\alpha\colon \Z^d \to \Diff^{1+\beta} (M)$ a smooth $\Z^d$-action on a compact manifold $M$.  Let $\mu$ be an ergodic, $\alpha$-invariant measure.  Recall that a coarse Lyapunov exponent $\chi$ is a positive-proportionality class of Lyapunov exponents of $\alpha$.  For almost every $x\in M$ there is a coarse Lyapunov subspace $E^\chi(x)\subset T_x M$ and a coarse Lyapunov manifold $W^\chi(x) $ tangent to $ E^\chi(x) $ at $x$.  

Let $\Fol^\chi$ denote the partition of $M$ into full $W^\chi$-manifolds. Given $n\in \Z^d$ with $\chi(n)>0$, following the construction from \cite{MR693976}  we can find a measurable partition $\xi$ of $(M,\mu)$ that is subordinate to $\Fol^\chi$ and increasing for $\alpha(n)$.  We then define the $\chi$-entropy of $\alpha(n)$ to be $$h_\mu^\chi (\alpha(n)  )= h_\mu (\alpha(n)  \mid \chi) := h_\mu(\alpha(n) \mid \Fol^\chi) = h_\mu(\alpha(n) , \xi) .$$

The main result of \cite{AWB-GLY-P3} is the following ``product structure of entropy'' for $\Z^d$-actions.
\begin{theorem}[{\cite[Corollary 13.2]{AWB-GLY-P3}}]\label{thm:entprod}
	Let $\alpha\colon \Z^d \to \Diff^{1+\beta} (M)$ be a smooth $\Z^d$-action on a compact manifold $M$ and let $\mu$ be an ergodic, $\alpha$-invariant measure. Then for any $n\in \Z^d$
	$$h_\mu(\alpha(n)) = \sum_{\chi(n)>0} h_\mu(\alpha(n)\mid \chi).$$
\end{theorem}

Fix $n\in \Z^d$ and let $f= \alpha(n)$.
The formulas in  \cref{thm:LYII} and \cref{thm:entprod} then look quite similar.  However, the contribution of each Lyapunov exponent $\lambda^i$ to the total entropy in \cref{thm:LYII} is a ``transverse entropy'' (the coefficient $\gamma^i$ is a measure of ``transverse dimension'').  In  \cref{thm:entprod}, the entropy of each coarse Lyapunov exponent $\chi$ is a ``tangential entropy'' $ h_\mu(\alpha(n)\mid \chi)$   obtained by conditioning along $W^\chi$-manifolds.  Thus,  \cref{thm:LYII} does not immediately imply  \cref{thm:entprod}.  To show \cref{thm:entprod}, one first shows that the total ``transverse entropy'' in \cref{thm:LYII} contributed by all $\lambda^i\in \chi$ is equal to the total  conditional entropy   $ h_\mu(\alpha(n)\mid \chi)$.  This is done in \cite{AWB-GLY-P3}.  The idea is to first establish an analogue of  \cref{thm:LYII} for the conditional entropy $h_\mu(f\mid \chi)$; this is done in  \cite{AWB-GLY-P2} where a formula  of the form 
$$h_\mu(f \mid \chi)= h_\mu(\alpha(n) \mid \chi) = \sum _{\lambda^i\in \chi} \gamma^{\chi,i}_n \lambda^i(n)$$ is shown.   Then (following \cite{MR1213080}) one uses that  $n\mapsto h_\mu(\alpha(n)\mid \chi)$ is linear on any half-cone where no coarse Lyapunov exponent $\chi'$ changes sign to show that the transverse dimensions $\gamma^{\chi,i}_n$ of each $\lambda^i\in \chi$ are independent of $n$  and  coincide with the transverse dimensions $\gamma^{i}$ appearing in  \cref{thm:LYII} for $f= \alpha(n).  $

\subsection{{Abstract ergodic theoretic constructions in smooth dynamics}}(See Appendices \ref{App:rokhlin} and \ref{App:Pinsker} for further details.)\label{sec:abstractergodictheory}
Let $f\colon M\to M$ be a $C^{1+\beta}$ diffeomorphism and let $\mu$ be an $f$-invariant probability measure.  We do not assume $\mu$ to be ergodic.  
We   introduce here a number of measurable partitions of the measure space $(M,\mu)$ associated with the dynamics $f$:

\def\erg{\mathcal E}
\begin{enumerate}
	\item $\erg$, the  ergodic decomposition (see \cref{def:ergdecom} and 
	Theorem \ref{t.decomposicaoergodica} in  Appendix \ref{B:secErgDecomposition});
	\item \label{popolo2} $\pi$, the Pinsker partition;
	\item $\Xi^u$, the measurable hull of the partition into unstable manifolds;
	\item \label{popolo4} $\Xi^s$, the measurable hull of the partition into stable manifolds.
\end{enumerate}

 In this section, we will define  objects \eqref{popolo2}--\eqref{popolo4} above 
 and  explain the following two assertions:
\begin{enumerate}
	\item \label{prop:trivialing1} $\erg \prec \Xi^s$   
	\item \label{prop:trivialing2} $ \Xi^s= \pi = \Xi^u.$ 
\end{enumerate}
The first assertion is a standard fact in hyperbolic dynamics (which forms the first step in the Hopf argument for ergodicity) and is discussed in detail in Theorem \ref{Hopf} in Appendix \ref{App:Pinsker}.  The second is \cite[Theorem B]{MR819556}.  

\def\sigmafoot{\footnote{At first glance, the order $\eta\prec\xi$ convention might seem backwards.  It is more consistent  thinking via $\sigma$-algebras: if $\sigma(\eta)$ denotes the $\sigma$-algebra of $\eta$-saturated sets then $\eta\prec\xi$ if and only if $\sigma(\eta)\subset \sigma (\xi).$}
}

\subsubsection{Measurable  hull of a partition}\label{ss:meashull}
Given a (possibly nonmeasurable)  partition $\xi$  of $(M,\mu)$ we write $\Xi(\xi)$ for the \emph{measurable hull}\index{partition!measurable hull} of $\xi$; that is, $\Xi(\xi)$ is  the finest  measurable partition with $\Xi(\xi)\prec \xi$.  If $\xi$ is measurable, then we have $\Xi(\xi)= \xi$ but in general $\Xi(\xi)$ is strictly coarser than $ \xi$.
We  illustrate this concept with a few examples.  
\begin{example}\label{ex:partition_orbits}
	Suppose that $\mu$ is $f$-invariant and ergodic.  Let $\mathcal {O}$ be the partition into orbits of $f$.  Then $\mathcal {O}$ is  {not} measurable (see \cref{ex.circle} in \cref{App:rokhlin}).  The measurable hull of $\mathcal {O}$ is the trivial partition $\Xi(\mathcal {O}) = \{M, \emptyset\}$.  For example, given a totally irrational flow on the torus $\T^2$, the partition into flow lines is not measurable and the measurable hull is the trivial partition.

	More generally, if $\mu$ is not ergodic then the measurable hull of $\mathcal {O}$ is $\Xi(\mathcal {O})= \erg$, the {ergodic decomposition}  $(M,\mu)$.  (See  \cref{def:ergdecom} and \cref{Bex:partition_orbits} in  Appendix \ref{App:Pinsker}.)
\end{example}

For the following two examples, recall \cref{entropyvatoms} and \cref{entropyvatoms2}.
\begin{example}\label{ex:truck}
	Let $f$ be a $C^{1+\beta}$ volume-preserving Anosov diffeomorphism of a connected manifold $M$.  Let $\xi^u$ denote the partition of $M$ into  unstable manifolds.  Then $\xi^u$ is  {not} measurable (for the invariant volume).  In fact, it is known that the measurable hull of $\xi^u$ is again the trivial partition $\Xi(\xi^u) = \{M, \emptyset\}$.  
\end{example}

\begin{example}\label{ex:unstable_partition_measurable}
	Let $f\colon M\to M$ be a $C^{1+\beta}$ diffeomorphism and let $\mu$ be any ergodic, $f$-invariant probability measure.  Let $\xi^u$ denote the partition of $M$ into  (full) unstable manifolds.  Then $\xi^u$ is measurable if and only if  $h_\mu(f) = 0$.  In particular, if $h_\mu(f) > 0$ then the measurable hull of $\xi^u$  is strictly coarser than $\xi^u$.  
\end{example}

In general, given a $C^{1+\beta}$ diffeomorphism $f\colon M\to M$ and an ergodic, $f$-invariant probability measure  $\mu$ 
we let  $\Xi^u$ and $\Xi^s$ denote, respectively,  the measurable hulls of the partition of $(M,  \mu)$ into full unstable and stable manifolds. 

We state the first relationship between the above objects   in the following proposition whose proof follows immediately from the pointwise ergodic theorem.   (See Theorem \ref{Hopf}, Appendix \ref{App:Pinsker}.)
\begin{proposition}\label{prop:hopf}
	Let $f\colon M\to M$ be a $C^{1+\beta}$ diffeomorphism and let $\mu$ be any   $f$-invariant probability measure.  Then $\erg\prec \Xi^s$ and $\erg\prec \Xi^u$. 
\end{proposition}
\begin{proof}
	Let $\sigma(\erg)$ and $\sigma(\Xi^s)$ denote the $\sigma$-algebras of $\erg$-saturated and $\Xi^s$-saturated sets, respectively.  
	
	Consider any continuous function  $\phi\colon M\to \R$.  Then   $\phi^+\colon M\to \R$   defined by $$\phi^+(x):= \limsup_{n\to \infty} \frac 1 n \sum _{k=0}^{n-1} \phi(f^k(x))$$ is an $f$-invariant function that is constant along $W^s$-leaves.  
	In particular,  the function  $\phi^+$ is measurable with respect to $\sigma(\erg)$ and $\sigma(\Xi^s)$.  
	Moreover, using that $C^0(M)$ is separable and dense in $L^1(\mu)$ and applying the pointwise ergodic theorem, it follows that the $\sigma$-algebra $\sigma(\erg)$ is the minimal $\sigma$-algebra for which $\phi^+$ is measurable for all continuous  $\phi\colon M\to \R$.  It follows that  $ \sigma(\erg)\subset \sigma(\Xi^s)$    whence $\erg\prec \Xi^s$.  
\end{proof}
\def\pifoot{\footnote{Another characterization of $\pi$ is  the following: $\pi$ is the unique $f$-invariant partition such that, if $(g,Y,\nu)$ is a measurable factor of  $(f,X, \mu)$ with zero entropy, then $(g,Y,\nu)$ is a factor of the factor system $(f, X,\mu)/\pi$.  
	}} 
	\subsubsection{The Pinsker partition}\label{sss:pinsker}
	Let $f\colon (X,\mu)\to (X,\mu)$ be a  measure-preserving transformation of a standard probability space $(X,\mu)$.  The \emph{Pinsker partition}\index{partition!Pinsker}  $\pi$ of $f\colon (X,\mu)\to (X,\mu)$  is the finest   measurable partition $(X,\mu)$ with the following property: for any measurable partition $\xi\prec \pi$, we have  $$h_\mu(f, \xi) = 0.$$  Another characterization of $\pi$ is  the following: $\pi$ is the unique $f$-invariant partition such that, if $(g,Y,\nu)$ is a measurable factor of  $(f,X, \mu)$ with zero entropy, then $(g,Y,\nu)$ is  also a factor of the factor system $(f, X,\mu)/\pi$.

	Our second relationship,  stated as \cite[Theorem B]{MR819556},  characterizes the Pinsker partition in smooth dynamics.  
	\begin{proposition}[{\cite[Theorem B]{MR819556}}]\label{prop:pinsker}
		Let $f\colon M\to M$ be a $C^{1+\beta}$ diffeomorphism and let $\mu$ be any $f$-invariant Borel probability measure.  Then we have  equality of partitions $$\Xi^u =  \pi = \Xi^s.$$
		
	\end{proposition}
	
	\begin{remark} 
		We say that a measure-preserving transformation $f\colon (X,\mu)\to (X,\mu)$ has the  \emph{K-property} (or the \emph{Kolmogorov property}) if the Pinsker partition $\pi$ is the trivial partition $\pi = \{\emptyset, X\}$.  For such systems, every non-trivial factor has positive entropy.

		Let $f\colon M\to M$ be a $C^{1+\beta}$ volume-preserving Anosov diffeomorphism.  Anosov first showed that such maps are ergodic with respect to the invariant volume in \cite{MR0224110}.   In this setting, the analogue of \cref{prop:pinsker} is established in \cite{MR0197684}; that is any set $A\in \pi$ is equal modulo $ 0$ to a set fully saturated by stable manifolds and also equal modulo $ 0$ to a (possibly different set) that is fully saturated by unstable manifolds.  
		Using the absolute continuity of the stable and unstable foliations established by Anosov in his proof of ergodicity, one may show that any $A\in \pi$  is equal modulo $ 0$ to a set that is both  fully saturated by stable manifolds and unstable manifolds.    It follows that any $A\in \pi$ is null or conull.  In particular, this shows that volume-preserving Anosov diffeomorphisms have the K-property.  This explains the  conclusion in \cref{ex:truck} that $\Xi^u$ is the trivial partition.  See, for example, \cite{MR2630044} for a modern discussion of absolute continuity and the $K$-property in uniformly (partially) hyperbolic settings.
	\end{remark}

	\section{Entropy, invariance, and the SRB property} \label{sec:entropyetc}
	
	In dissipative (i.e.\ non-volume-preserving) dynamical systems, ergodic SRB measures $\mu$ without zero Lyapunov exponents provide examples of \emph{physical measures}\index{measure(s)!physical}: there is a set $B$ of positive Lebesgue measure such that for any continuous function $\phi$, the forward time average of $\phi$ along the orbit of points in $B$ converges to $\int \phi \ d \mu$.   
	In applications and specific examples, a recurring problem is to establish the existence of physical and  SRB measures.   We pose a related question that arises naturally in the settings considered in this text:
	\begin{question}\label{Q1}Given a diffeomorphism $f\colon M\to M$
		and an $f$-invariant measure $\mu$, how do you verify that $\mu$ is an SRB measure?
	\end{question}
	Seemingly unrelated, consider a group $G$ acting smoothly on a manifold $M$.  We pose the following:  
	\begin{question}\label{Q2}
		Given a  Borel probability measure $\mu$ on $M$ and a subgroup $H\subset G$, how do you verify that $\mu$ is $H$-invariant?  
	\end{question}
	One method to answer both of these questions is given in \cref{thm:led} and \cref{thm:led'} below.

	\subsection{Ledrappier's theorem}  (See Appendix \ref{App:LY} for further details.) \label{sec:led}
	We outline one approach that solves both Question \ref{Q1} and \ref{Q2} in a number of settings.  We discuss other approaches towards verifying the existence of  SRB measures below.
	
	We recall \cref{unsentropy} where the notion of unstable entropy was introduced.   The main  result  (Corollary 5.3)  of \cite{MR819556} shows  for a   $C^{2}$ (see  \cite{1608.05886} for the $C^{1+\beta}$ case) diffeomorphism  $f\colon M\to M$ preserving an ergodic probability measure $\mu$ that   the metric entropy of $f$ and the  unstable metric {entropy} of $f$ coincide: $$h_\mu(f) = h_\mu^u(f).$$ 
	Using this fact, Ledrappier gave a geometric characterization of all measures satisfying equality $h_\mu(f) = \sum_{\lambda^i>0} m^i\lambda^i$ in the Margulis--Ruelle inequality, giving a converse of Theorem \ref{entropyfacts}\ref{EFF3}.  
	
	\begin{theorem}[Ledrappier's Theorem \cite{MR743818}]\label{thm:led}\index{theorem!Ledrappier}
		Let $f$ be a $C^{1+\beta}$ diffeomorphism and let $\mu$ be an ergodic, $f$-invariant, Borel probability measure.  
		Then $\mu$ is SRB if and only if \begin{equation}h^u_\mu(f) = \sum_{\lambda^i>0}m^i \lambda^i. \label{eq:entequal} \end{equation}
	\end{theorem}

	In the  proof of Theorem \ref{thm:led}, Ledrappier actually proves something much stronger than the SRB property:  if $h^u_\mu(f) = \sum_{\lambda^i>0}m^i \lambda^i$ then the leaf-wise measures $\mu^u_x$ of $\mu$ along unstable manifolds are {\it equivalent} to the Riemannian volume with a H\"older continuous density.  That is, if $m^u_x$ the Riemannian volume along $W^u(x)$ then for \ae $x$  there is a H\"older continuous, nowhere vanishing  function $\rho\colon W^u(x) \to (0, \infty)$ with \begin{equation}\label{eq:SRB}\mu^u_x = \rho \ m^u_x.\end{equation} In particular, the leaf-wise measure $\mu^u_x$ has full support in $W^u(x)$.  
	Moreover, Ledrappier explicitly computes the density function $\rho$; see \eqref{dyn_densities} in Appendix \ref{App:LY} and \cite[Corollary 6.1.4]{MR819556}.

	We make use of the explicit formula for the density $\rho$ in the following setup.  
	Consider a Lie group $G$ and a  smooth, locally free action of $G$ on a manifold $M$.  We denote the action by $g\cdot x$ for $g\in G$ and $x\in M$.  Consider a Lie subgroup $H\subset G$ and $s\in G$ that normalizes $H$.   Let $f\colon M\to M$ be the diffeomorphism given by $s$; that is $f(x) = s\cdot x$.  Let $\mu$ be an ergodic, $f$-invariant Borel probability measure and suppose that the orbit $H\cdot x$ is contained in the unstable manifold $W^u(x)$ for $\mu$-almost every $x$.  
	
	Since $s$ normalizes $H$, the partition of $M$ into $H$-orbits is preserved by $f$; in particular,  the partition into $H$-orbits is a subfoliation of the  partition 
	into unstable manifolds.  Given a Borel probability measure $\mu$ on $M$ and  a measurable partition $\xi$ subordinate to the partition into $H$-orbits we can define conditional measures $\mu^\xi_x$ of $\mu$.    
	Given $x\in M$ (using that the action is locally free) we can push forward the left-Haar measure on $H$ onto the orbit $H\cdot x$ via the parametrization $H\cdot x = \{ h\cdot x: h\in H\}.$  
	
	\begin{lemma}
		$\mu$ is $H$-invariant if and only if for any measurable partition $\xi$ subordinate to the partition into $H$-orbits and $\mu$-\ae $x$ the conditional measure $\mu^\xi_x$ coincides---up to normalization---with the restriction of the left-Haar measure on $H\cdot x$ to $\xi(x)$. 
	\end{lemma}

	Similar to the definition of metric  {entropy} of $f$ {conditioned on unstable manifolds}, we can define  the metric \emph{entropy} of $f$ \emph{conditioned on $H$-orbits}\index{entropy of a transformation!conditioned on an orbit}, written  $h_\mu(f\mid H)$, by  $$h_\mu(f\mid H)  :=  h_\mu(f,\xi)$$ where $\xi$ is  any increasing, measurable partition  $\xi$  {subordinate to $H$-orbits.} 
	Let $\lambda^i$, $E^i(x)$, and $m^i$ be as in \ref{sss:osel} for the dynamics of $f$ and the measure $\mu$.  We define the \emph{multiplicity of $\lambda^i$ relative to $H$} to be (the almost surely constant value of)  $$m^{i,H} = \dim (E^i(x) \cap T_x (H\cdot x)).$$
	Generalizing Theorem \ref{entropyfacts}\ref{EFF1} we have (see for instance \cite{AWB-GLY-P2})
	\begin{equation}\label{eq:popopopo} h_\mu(f\mid H) \le  \sum_{\lambda^i>0} \lambda^i m^{i,H}.\end{equation}
	From the proof of Theorem \ref{thm:led}, (in particular, the explicit formula for the density function $\rho$ in \eqref{eq:SRB}; see \eqref{dyn_densities} in Appendix \ref{App:LY} and proof of \cref{prop:easyLed}) we have the following.
	
	\begin{theorem}\label{thm:led'}
		With the above setup, the following are equivalent:
		\begin{enumcount} 
			\item\label{11111} $h_\mu(f\mid H) = \sum_{\lambda^i>0} \lambda^i m^{i,H}$;
			\item for any measurable partition $\xi$ subordinate to the partition into $H$-orbits and almost every $x$, $\mu^\xi_x$ is absolutely continuous with respect to the Riemannian volume on the $H$-orbit $H\cdot x$;
			\item\label{3333333333} $\mu$ is $H$-invariant.  
		\end{enumcount}
	\end{theorem}
	The proof is only a slightly more complicated version of the proof of \cref{prop:easyLed}.  Note that as \cref{thm:led'} only concerns the entropy and dynamics inside $H$-orbits, the result holds for $C^1$ or even $C^0$ actions since the dynamics permuting $H$-orbits is affine and hence $C^\infty$.  See for instance \cite{MR2191228} where related entropy results are shown for $C^0$ actions of Lie groups.  
	
	A possible critique of Theorem \ref{thm:led} is that in examples it seems nearly impossible to verify equality in \eqref{eq:entequal} without first knowing that the measure is SRB.  However, in a number of settings of  group actions on manifolds, it turns out one can, in fact, verify equality in \eqref{eq:entequal} (or typically, equality in Theorem \ref{thm:led'}\ref{11111}) and thus derive the SRB property  or gain additional invariance of the measure only from entropy considerations.  This is one key idea in this text, the papers \cites{AWBFRHZW-latticemeasure,1608.04995}, and also appears as a main tool in \cites{MR3814652, MR1253197}.

	\begin{remark}\label{rem:invariance_principle}\index{invariance principle!Ledrappier}\index{invariance principle!Avila--Viana}
		The statement and proof of  Theorem \ref{thm:led}, especially the reformulation in Theorem \ref{thm:led'}, is very similar to the  {\it invariance principle} for fiberwise disintegrations of measures invariant under skew products.  
		The earliest version of this invariance principle is due to Ledrappier \cite{MR850070} for projectivized linear cocycles.   Avila--Viana  extended this to  cocycles taking values in the group of $C^1$ diffeomorphisms in \cite{MR2651382}. See Proposition \ref{prop:nonresinv} for a related invariance principle in the setting of actions of lattices on manifolds.  
	\end{remark}

	\subsection{Approaches to  Questions \ref{Q1} and \ref{Q2}} 
	Although not the main focus of this text, we summarize a number of alternative approaches towards approaching Questions \ref{Q1} and \ref{Q2} that arise in various dynamical settings.

	\fakeSS{SRB property from  dynamical hypotheses} In the setting of uniformly hyperbolic dynamics, SRB measures are known to  exist for Anosov diffeomorphisms, Anosov flows,  and Axiom A attractors.    See \cites{MR0442989,MR0380889}.  In the setting of  partially hyperbolic diffeomorphisms, under suitable conditions on the central dynamics SRB measures are known to exist; related results hold for dynamics with a dominated splitting.    See for example \cites{MR1743717,MR1749677,MR1757000,MR3712997}.   

	\fakeSS{SRB measures via detailed analysis} For specific families of examples exhibiting nonuniform hyperbolicity, tools of parameter exclusion,  normal forms, and detailed analysis can be used to show the existence of an SRB measure.  See for example,  \cites{MR1701385,MR1087346,MR799250,MR630331,MR1218323,MR1835392,MR3098967,MR2005855}.  General hypotheses that can be verified in a number of examples are given in \cites{MR1637655,MR1824198} which guarantee    the existence of SRB measures.  See the survey article \cite{MR1933431} for more background.  

	\fakeSS{Verifying equality in the entropy formula} As discussed above, the culmination of the results of \cites{MR693976, MR743818, MR819556} characterizes  SRB measures exactly as those for which 
	the equality $h_\mu(f) = \sum_{\lambda^i>0} m^i \lambda^i$ holds.  
	Similarly, equality in Theorem \ref{thm:led'}\ref{11111}   holds if and only if  the   measure $\mu$  is  invariant under the action of the subgroup $H$.  
	This approach---verifying equality in the entropy formula to obtain invariance of a measure---has been exploited in particular in \cites{MR1253197,MR3814652,AWBFRHZW-latticemeasure,1608.04995}.
	
	
	
	\fakeSS{Shearing and translation invariance in a homogeneous structure}
	A common tool to obtain invariance or absolute continuity properties of leaf-wise measures is to manufacture a shear of leaf-wise measures along leaves of a foliation.  That is, given an invariant measure $\mu$ and an affine foliation $\Fol$ with family of normalized leave-wise measures $\{\mu^\Fol_x\}$, for a $\mu$-typical $x$ one may be able to use the dynamics to construct approximations of translations along the  support of $\mu^\Fol_x$ in the  leaf $\Fol(x)$ that preserve the measure $\mu^\Fol_x$ up to normalization. Taking a limit, one has that $\mu^\Fol_x$ is preserved up to normalization under some translations which gives strong information (see \cref{lem:abscon}) on the geometry  of $\mu^\Fol_x$.  
	Additional dynamical arguments can then often establish translation invariance of the leaf-wise measures $\mu^\Fol_x$.   
	Manufacturing translation invariance of leaf-wise measures along their support in an affine foliation $\Fol$ is a main tool used to establish  Ratner's measure classification results in  \cites{MR1075042,MR1262705}  and \cite{MR1253197}.   
	This was also one of the main steps (see \cref{lem:trans}) in the proof of Theorem \ref{thm:KS}.   
	
	In a setting similar to that of Theorem \ref{thm:KS}, for  higher-rank  diagonal actions on semisimple homogeneous spaces (see \cref{sec:diag}), the {\it high entropy method} \cites{MR1989231,MR2191228} and {\it low entropy} method \cites{MR2195133, MR2366231} provide mechanisms to  obtain translation invariance of leaf-wise measures, culminating in the landmark paper \cite{MR2247967}. 
	Another mechanism to obtain translation invariance of leaf-wise measures appears in \cites{MR2831114, MR3037785} and is used to establish   measure rigidity results for stationary measure for  affine random walks.  In \cite{MR3814652},  a  mechanism inspired by \cite{MR2831114} is used to obtain invariance for certain affine actions of $\Sl(2,\R)$.

	This approach, and specifically the method  presented in   \cref{part:I} from \cite{MR1406432}, has been adapted to establish measure rigidity in a number of nonlinear settings including \cite{MR2261075} and \cite{MR2811602}.  In non-linear settings, unstable manifolds $W^u(x)$ are $C^2$ injectively immersed copies of $\R^k$ for some $k$.  Although there might be no natural notion of translation, relative to certain coordinate systems $H_x \colon \R^k\to W^u(x)$  obtained from normal forms of the dynamics along unstable manifolds, leaf-wise measures $\mu^u_x$ are absolutely continuous if and only if their images  $(H_x\inv)_*\mu^u_x$  in these coordinates  are {translation invariant} in $\R^k$.  In a number of non-linear settings including  \cites{MR2261075, MR2811602, 1506.06826} absolute continuity properties of a measure $\mu$ along unstable foliations is shown by  establishing translation invariance of the leaf-wise measures $(H_x\inv)_*\mu^u_x$  in these coordinates.

	\part{Smooth lattice actions and new results in the Zimmer program}\label{part:III}
	
	The main goal of this part will be to understand properties and to classify smooth actions of certain countable groups $\Gamma$ on compact manifolds.  The main results of this section are \cref{thm:invmeas} and \cref{slnr}.  We give their proofs after introducing some terminology and motivation.

	\section{Smooth lattice actions} 
	We give some background on lattices in semisimple Lie groups and a number of examples of smooth actions of  lattices on manifolds. 
	References with additional details for this and the next section include \cites{MR3307755, MR2655311,MR1920389,MR1648087,MR1090825,MR2807830}.
	\subsection{Lattices in semisimple Lie groups}\label{ss:latices}\index{Lie group}
	Recall that a Lie algebra $\lieg$ is \emph{simple} if it is non-abelian and has no non-trivial ideal.  A Lie algebra $\lieg$ is \emph{semisimple} \index{Lie group!(semi-)simple} 
	  if it is the direct sum $\lieg= \oplus_{i=1}^\ell \lieg_i$ of simple Lie algebras $\lieg_i$; this is equivalent to the fact that $[\lieg,\lieg]=\lieg.$
	We say a  Lie group $G$ is \emph{simple}  (resp.\  \emph{semisimple}) if its Lie algebra $\lieg$ is   {simple}  (resp.\   {semisimple}).  
	The main example for this text is the simple Lie group $G = \Sl(n,\R)$.  
	
	Let $G$ be a connected semisimple Lie group with finite center.  Semisimple Lie groups are unimodular and hence admit a bi-invariant measure, called the \emph{Haar measure}, which is unique up to normalization.  
	A \emph{lattice}\index{Lie group!lattice} in $G$ is a discrete subgroup $\Gamma\subset G$ with finite co-volume.  That is, if $D$ is a measurable fundamental domain for the right-action of $\Gamma$ on $G$ then $D$ has finite volume.  
	If the quotient $G/\Gamma$ is compact, we say that $\Gamma$ is a \emph{cocompact} lattice.  If $G/\Gamma$ has finite volume but is not compact we say that $\Gamma$ is \emph{nonuniform}.  \index{Lie group!lattice!cocompact}  \index{Lie group!lattice!nonuniform}
	The quotient manifold $G/\Gamma$ by the right action of $\Gamma$  admits a left-action by $G$ and the Haar measure on $G$ descends to a finite,  $G$-invariant measure on $G/\Gamma$ which we normalize to be a probability measure.  
	
	\begin{example}
		The standard  example of a lattice in $G=\Sl(n,\R)$ is $\Gamma=\Sl(n,\Z)$.  Note that $\Sl(n,\Z)$ is not cocompact in $\Sl(n,\R)$. However,  $\Sl(n,\R)$ and more general simple and semisimple Lie groups  possess both nonuniform and cocompact lattices.   (See for example \cite[Sections 6.7, 6.8]{MR3307755} for examples and constructions.)
	\end{example}

	\begin{example}
		In the case  $G=\Sl(2,\R)$, the fundamental group of any finite area hyperbolic surface is a lattice in $G$.  In particular, the fundamental group of a compact hyperbolic surface is a cocompact lattice in $G$.  
		This can be seen by identifying the fundamental group of $S$ with the deck group of the hyperbolic plane $\mathbb{H}=\So(2,\R)\bs \Sl(2,\R)$.  For instance, the free group $\Gamma=F_2$  on two generators is a lattice in $G$ as can be seen by giving the punctured torus  $S= \T^2\sm\{\mathrm{pt}\}$ a hyperbolic metric.  
	\end{example}
	
	See \cite{MR3307755} for further details on  constructions and properties of lattices in Lie groups.  
	
	\subsection{Rank of $G$}\label{sec:rank}\index{Lie group!rank}\index{Lie group!Iwasawa decomposition}
	Every semisimple matrix group admits an Iwasawa decomposition $G= KAN$ where $K$ is compact, $A$ is a simply connected free abelian group of $\R$-diagonalizable elements, and $N$ is unipotent.    
	For general semisimple Lie groups with finite center, we  have a similarly defined Iwasawa decomposition $G= KAN$ where the images of $A$ and $N$ under the adjoint representation are, respectively,  $\R$-diagonalizable and unipotent.  
	See for instance \cite{MR1920389} for details.
	The  dimension of $A$ is the  \emph{rank} of $G$.  
	We call such a group $A$ a maximal split Cartan subgroup.  
	
	In the case of $G= \Sl(n,\R)$, the standard choice of $K$, $A$, and $N$ are
	$$K = \So(n,\R),\quad \quad A = \left\{ \diag(e^{t_1}, e^{t_2},\dots, e^{t_n}) : t_1+\dots +t_n = 0\right\},$$ 
	and  $N$ the group of upper-triangular matrices with all diagonal entries equal to $1$.  
	Note that, as elements in $\Sl(n,\R)$ have  determinant $1$, we have $$\diag(e^{t_1}, e^{t_2},\dots, e^{t_n})\in \Sl(n,\R)$$ if and only if $t_1 + \dots + t_n = 0$.  Thus $A\simeq \R^{n-1}$ and the rank of $\Sl(n,\R)$ is $n-1$.  
	
	We say that a simple Lie group $G$ is \emph{higher-rank}\index{Lie group!higher-rank} if its rank is at least $2$.  We will say that a lattice $\Gamma$ in a higher-rank simple Lie group $G$ is a \emph{higher-rank lattice.}  
	In particular, $G= \Sl(n,\R)$ and its lattices are higher-rank when   $n\ge 3$.  
	
	In   \cref{ex:isom} below,  we present  an example of a cocompact lattice $\Gamma$ in the group $G= \So(n,n)$ when $n\ge 4$.  The group $\So(n,n)$ has rank $n$ and thus $\Gamma$ is a higher-rank, cocompact lattice.  
	
	For further examples, see Table \ref{tab:stupid} for calculations of rank for various matrix groups and \cite[VI.4]{MR1920389} for examples of Iwasawa decompositions for various matrix groups.

	\subsection{Standard actions of  lattices in Lie groups}\label{sec:ex}
	We present a number of standard examples of ``algebraic''  actions of lattices in Lie groups.  We also discuss in \cref{ex:exotic} some modifications of algebraic actions and constructions of more exotic actions.  
	\begin{example}[Finite actions] 
		Let $\Gamma'$ be a finite-index normal subgroup of $\Gamma$.  Then $F= \Gamma/\Gamma'$ is finite.    Suppose the finite group $F$ acts on a   manifold $M$.  Since $F$ is a quotient of $\Gamma$ we naturally obtain  a $\Gamma$-action on $M$.  
		
		Note that an action of a finite group  {preserves a volume} simply by averaging any volume form by the action.  
	\end{example}
	\begin{definition} \label{def:dumb} \index{action!finite}
		An action $\alpha\colon \Gamma\to \diff(M)$ is \emph{finite} or \emph{almost trivial} if it factors through the action of a finite group.  That is, $\alpha$ is finite  if there is a finite-index normal subgroup $\Gamma'\subset \Gamma$ such that $\restrict \alpha {\Gamma'}$ is the identity.  
	\end{definition}
	We remark that by a theorem of Margulis \cite{MR515630}, if $\Gamma$ is a lattice in a higher-rank, simple Lie group with finite center then  all normal subgroups of $\Gamma$ are either finite or of finite-index.

	\begin{example}[Affine actions] \label{ex:standard}\index{action!affine}
		Let $\Gamma= \Sl(n,\Z)$ (or any finite-index subgroup of $\Sl(n,\Z)$).  Let $M= \T^n= \R^n/\Z^n$ be the $n$-dimensional torus.   We have a natural action $\alpha \colon \Gamma\to \diff(\T^n)$ given by 
		$$\alpha(\gamma)(x +\Z^n) = \gamma\cdot  x + \Z^n$$ for any matrix  $\gamma \in \Sl(n,\Z)$.  
		
		To generalize this example to other lattices, let $\Gamma\subset \Sl(n,\R)$ be any lattice and let $\rho\colon \Gamma\to \Sl(d,\Z)$ be any representation.   
		Then we have a natural action $\alpha \colon \Gamma\to \diff(\T^d)$ given  by $$\alpha(\gamma) (x+ \Z^d) = \rho(\gamma) \cdot x + \Z^d.$$
		Note that these examples  {preserve a volume} form, namely, the Lebesgue measure on $\T^d$.   Also note that these actions are non-isometric.  
	\end{example}
	
	\begin{remark}\label{rem:affineAnosov}
		Both constructions in \cref{ex:standard} give actions $\alpha\colon \Gamma\to  \diff(\T^d)$ that have global fixed points.  That is, the coset of $0$ in $\T^d$ is  a fixed point of $\alpha(\gamma)$ for every $\gamma\in \Gamma$.

		The construction can be modified further to obtain genuinely affine actions without global fixed points.  Given a lattice $\Gamma\subset \Sl(n,\R)$   and a representation $\rho\colon \Gamma\to \Sl(d,\Z)$,  there may exist  non-trivial elements  $c\in H^1_\rho(\Gamma, \T^d)$; that is, $c\colon \Gamma\to \T^d$ is a function with \begin{equation} c(\gamma_1\gamma_2) = \rho(\gamma_1)c(\gamma_2) + c(\gamma_1)\label{eq:coho1}\end{equation} and such that there does not exist any $\eta\in \T^d$ with \begin{equation}c(\gamma) =\rho(\gamma) \eta -\eta  \label{eq:coho2}\end{equation} for all $\gamma\in \Gamma$.  (Equation \eqref{eq:coho1} says that $c$ is a cocycle with coefficients in the $\Gamma$-module $\T^d$; \eqref{eq:coho2} says $c$ is not a coboundary.)
		We  may then define $\td \alpha\colon \Gamma\to \Diff(\T^d)$ by 
		$$\td \alpha(\gamma) (x+ \Z^d) = \rho(\gamma) \cdot x +  c(\gamma) + \Z^d.$$
		Equation \eqref{eq:coho1} ensures that $\td \alpha$ is an action and \eqref{eq:coho2} ensures that $\td \alpha$ is not conjugate to the action $\alpha$.

		In the above construction, any cocycle  $c\colon \Gamma \to \T^d$ is necessarily cohomologous to a torsion-valued (that is,  $\Q^d/\Z^d$-valued) cocycle.  This follows from Margulis's result (see  \cite[Theorem 3 (iii)]{MR1090825}) on the vanishing of $H^1_\rho(\Gamma, \R^d).$   In particular, $\td \alpha$ and $\alpha$ are conjugate when restricted to a  finite-index subgroup of $\Gamma$.  See \cite{MR1236179} for more details.  
	\end{remark}
	
	\begin{example}[Projective actions]\label{ex:sphere}\index{action!projective}
		Let $\Gamma\subset  \Sl(n,\R)$ be any lattice.  Then $\Gamma$ has a natural linear action on $\R^n$.  
		The linear action of $\Gamma$ on $\R^n$ induces an action of $\Gamma$ on the sphere $S^{n-1}$ thought of as the {set of unit vectors} in $\R^n$:    we have $\alpha\colon \Gamma\to \diff(S^{n-1})$    given by $$\alpha(\gamma)(x) = \frac{\gamma\cdot x}{\|\gamma\cdot x\|}.$$
		
		Alternatively  we could act on the {space of lines} in $\R^n$ and obtain  an action of $\Gamma$ on the $(n-1)$-dimensional real projective space $\R P^{n-1}$. 
		This action  {does not preserve} a volume; in fact there is no invariant probability measure for this action.  Additionally, these actions are not isometric for any Riemannian metric.  
	\end{example}
	\begin{remark}[Actions on boundaries]\label{rem:parabolic}
		Example \ref{ex:sphere} generalizes to  actions   of lattices $\Gamma $ in $G$ acting on   boundaries of $G$. 
		Given a semisimple  Lie group $G$ with Iwasawa decomposition $G= KAN$, let $M = K\cap C_G(A)$ be the centralizer of $A$ in $K$.  A closed subgroup $Q\subset G$ is \emph{parabolic}\index{Lie group!parabolic subgroup} if it is conjugate to a group  containing  $MAN$.  When $G= \SL(n,\R)$ we have that $M$ is a finite group and any parabolic subgroup $Q$ is conjugate to a group containing all upper triangular matrices.  See \cite[Section VII.7]{MR1920389} for further discussion on the structure of parabolic subgroups.  
		
		Given a semisimple Lie group $G$, a (finite-index subgroup of a) proper parabolic subgroup $Q\subset G$, and a lattice $\Gamma\subset G$, the coset space $M= G/Q$ is compact and $\Gamma$ acts on $M$ naturally as $$\alpha(\gamma) (xQ) = \gamma xQ.$$
		These actions   never preserve a volume form or any Borel probability measure and are not isometric.    
		
		In \cref{ex:sphere}, the action on the projective space $\R P^{n-1}$ can be seen as the action on $\Sl(n,\R)/Q$ where $Q$ is the parabolic subgroup
		$$Q=\left \{ \left(\begin{array}{cccc}* & * & \cdots & * \\0 & * & \cdots & *  \\ \vdots & \vdots  & \ddots & \vdots \\0 & * & \cdots & *\end{array}\right)\right\}.$$
	\end{remark}

	\begin{example}[Isometric actions]\index{action!isometric} Another important family of algebraic actions are isometric actions obtained from embeddings of  cocompact lattices in Lie groups into compact groups.  
		\label{ex:isom}
		\subsubsection*{Isometric actions of cocompact lattices in split orthogonal groups of type $D_n$ ($n\ge 4$)}For $n\ge 4$, consider the quadratic form in ${2n}$ variables $$Q(x_1, \dots, x_n, y_1, \dots, y_n) = x_1^2 + \dots x_n ^2 -\sqrt 2 (y_1 ^2 + \dots + y_n ^2).$$ Let $$B=\diag  \left (1,\dots, 1, -\sqrt 2, \dots , - \sqrt 2 \right) \in \Gl(2n,\R)$$ be the matrix such that   $Q(x) = x^TBx$ for all $x\in \R^{2n}$ and let $$G= \SO(Q)= \{g\in \Sl(2n,\R) \mid g^TBg = B\}$$ be the special orthogonal group associated with $Q$.  We have that $$\SO(Q)\simeq \SO(n,n)$$ is a Lie group of rank $n$ with restricted root system of type $D_n$ when $n\ge 4$.\footnote{For $n= 1$, $\So(1,1)$ is a one-parameter group and for $n= 2$, $\So(2,2)$ is not simple (it is double covered by $\Sl(2,\R) \times \Sl(2,\R)$). For $n= 3$, $\So(3,3)$ is  double covered by $\Sl(4,\R)$.}
		
		\def\K{\mathbb{K}}
		Let $\K= \Q[\sqrt 2] $ and let $\Z[\sqrt 2]$ be the ring of integers in $\K$.  
		Let $$\Gamma = \{ g\in \Sl(2n,\Z[\sqrt 2])  \mid g^TBg = B\}.$$
		Then $\Gamma$ is a cocompact lattice in $G$.  (See for example \cite{MR3307755}, Proposition 5.5.8 and Corollary 5.5.10.)

		Let $\tau\colon \K\to \K$ be the nontrivial Galois automorphism, $\tau(\sqrt{2})=-\sqrt{2}$, and let $\tau$ act coordinate-wise on matrices with entries in $\K$. 
		Given $\gamma\in \Gamma$ we have $\tau(\gamma) = \id$ if and only if $\gamma=\id$.  Moreover, as $\tau^2= \id$ we have $$\tau(\gamma)\in \SO(\tau(Q)):= \{g\in \Sl(2n,\R) \mid g^T \tau (B) g = \tau (B)\} \simeq \SO(2n).$$
		In particular, the map $\gamma\to \tau(\gamma)$ gives a representation $\Gamma\to \So(2n)$ with infinite image into the compact group $\So(2n)$.
		
		As $\So(2n)$ is the isometry group of the sphere $S^{2n-1}= \So(2n)/\So(2n-1)$ we  obtain an action of $\Gamma$ by isometries on a manifold of dimension $2n-1$.

		\subsubsection*{Isometric actions of cocompact lattices in $\Sl(n,\R)$}
		A more complicated construction can be used to build cocompact lattices $\Gamma\subset \Sl(n,\R)$ that possess  infinite-image representations $\pi\colon \Gamma \to \SU(n)$ (see discussion in  \cite[Sections 6.7, 6.8]{MR3307755} as well as \cite[Warning 16.4.3]{MR3307755}.)  In this case, one obtains isometric actions of certain cocompact lattices $\Gamma$ in $\Sl(n,\R)$ on the $(2n-2)$-dimensional homogeneous space $$\SU(n)/\mathrm{S}(\mathrm{U}(1)\times \mathrm{U}(n-1)).$$
	\end{example}

	\begin{example}[Modifications of standard examples and exotic actions]\index{action!exotic} 
		\label{ex:exotic}
		Beyond the  ``algebraic actions'' discussed  in Examples  \ref{ex:standard}--\ref{ex:isom},  
		it is possible to modify certain algebraic constructions to construct genuinely new actions; these actions might not be conjugate to algebraic actions and may exhibit much weaker  rigidity properties.   One such construction starts with the standard action of (finite-index subgroups of) $\Sl(n,\Z)$ on $\T^n$ and creates a non-volume-preserving  action by blowing-up fixed points or finite orbits of the action.  In \cite[Section 4]{MR1380646},  Katok and Lewis showed this example can be modified to obtain volume-preserving, real-analytic actions of $\Sl(n,\Z)$  that are not $C^0$ conjugate to an affine action; moreover, these   actions   are not locally $C^1$-rigid.  
		In \cites{MR2716615,MR2154667,MR2342012}, constructions of non-locally $C^1$-rigid, ergodic, volume-preserving actions of any lattice in a simple Lie group are constructed by more general blow-up constructions.
		
		Another example due to Stuck \cite{MR1406436}  demonstrates that it is impossible to fully  classify all lattice actions. Let $P\subset \Sl(n,\R)$ be the group of upper triangular matrices.  There is a non-trivial homomorphism $\rho\colon P\to \R$.  Now consider any flow (i.e.\ $\R$-action) on a manifold $M$ and view the flow as a $P$-action via the image of $\rho$.  Then $G$ acts on the induced space $N= (G\times M)/P$ and the restriction induces a non-volume-preserving,  non-finite action of $\Gamma$.  This example shows---particularly in the non-volume-preserving-case---that care is needed in order to formulate any precise conjectures that assert that every  action should be ``of an algebraic origin.''  Note, however, that we obtain a natural map $N\to G/P$ that intertwines $\Gamma$-actions; in particular, this action has an ``algebraic action'' as a factor.  
		
		We refer to \cite[Sections 9 and 10]{MR2807830} for more detailed discussion and references to modifications of algebraic actions and  exotic actions.  
	\end{example}

	\subsection{Actions of lattices in rank-1 groups}
	Actions by lattices in higher-rank Lie groups are expected to be rather constrained.  Although 
	\cref{ex:exotic} shows there exists exotic, genuinely ``non-algebraic'' actions of such groups, these actions are built from modifying algebraic constructions or factor over algebraic actions.  
	For lattices in rank-one Lie groups such as $\SL(2,\R)$, the situation is very different.  
	There exist natural actions that have no algebraic origin and the algebraic actions  of such groups seem to exhibit far less rigidity (for example  \cref{ex:hurder} which is not locally rigid) than those above. 
	\begin{example}[Actions of free groups]\label{ex:free}
		Let $G= \Sl(2,\R)$.  The free group $\Gamma = F_2$ is isomorphic to a lattice in $G$.  (For instance, the fundamental group of the punctured torus is isomorphic to $F_2$; more explicitly, $\Sl(2,\Z)$ contains a copy of $F_2$ as an index 12 subgroup.)   Let $M$ be any manifold and let $f,g\in \Diff(M)$.  Then $f$ and $g$ generate an action of $\Gamma$ on $M$ which in general is not of an algebraic origin and does not exhibit any local rigidity. 
In particular, there is no expectation that any rigidity phenomena should hold for  actions of all lattices in $\Sl(2,\R)$.  
	\end{example}

	For the next example, recall  Definitions \ref{def:anosov} and \ref{def:anosov2} of  Anosov actions.
	\begin{example}[Non-standard Anosov actions of $\Sl(2,\Z)$]\label{ex:hurder}\index{action!exotic}
		Consider the standard action $\alpha_0$ of $\Sl(2,\Z)$ on the 2 torus $\T^2$  as constructed in Example \ref{ex:standard}.  In \cite[Example 7.21]{MR1154597}, Hurder presents an example of a 1-parameter family of deformations $\alpha_t\colon \Sl(2,\Z)\to \Diff(\T^2)$ of $\alpha_0$ with the following properties:
		\begin{enumerate}
			\item Each $\alpha_t$ is a real-analytic, volume-preserving action;
			\item For $t>0$, $\alpha_t$ is not topologically conjugate to $\alpha_0$, (even when restricted to a finite-index subgroup of $\Sl(2,\Z)$.)
		\end{enumerate}
		Moreover, since $\alpha_0$ is an Anosov action and since the Anosov property is an open property we have that 
		\begin{enumerate}[resume]
			\item  each $\alpha_t$ is an Anosov action. 
		\end{enumerate}
		This shows that even affine Anosov actions of $\Sl(2,\Z)$ fail to exhibit local rigidity properties and that there exist genuinely exotic Anosov actions of $\Sl(2,\Z)$. This is in stark contrast to the affine Anosov actions of higher-rank lattices which are known to be locally rigid by  \cite[Theorem 15]{MR1632177}. 
	\end{example}
	In contrast, it is expected that all Anosov actions of higher-rank lattices are smoothly conjugate to affine actions as in Example \ref{ex:standard} or \cref{rem:affineAnosov} (or analogous constructions in infra-nilmanifolds).   See Question \ref{Q:qq}\ref{555} below. Recent progress towards this conjecture appears  in \cite{BRHW1}.  
	
	\begin{remark}
		There are a number of rank-1 Lie groups whose lattices are known to exhibit some  rigidity properties relative to linear representations.  For instance, Corlette established   superrigidity and arithmeticity of lattices in certain  rank-1 simple Lie groups  in \cite{MR1147961}.  In particular, Corlette establishes  superrigidity  for lattices in $\Sp(n,1)$ and $F_4^{-20}$, the isometry groups of quaternionic hyperbolic space and the Cayley plane.  It seems plausible that  lattices in certain rank-1 Lie groups would  exhibit some rigidity properties for actions by diffeomorphisms;  currently,  there do not seem to be any results in this direction.  
	\end{remark}
	
	\section{Actions in low dimension and Zimmer's conjecture} 
	\subsection{Motivating questions} 
	For actions by lattices in rank-1 groups, we have seen that it is  easy to construct exotic actions of free groups and Example \ref{ex:hurder} shows there are exotic Anosov actions of $\Sl(2,\Z)$ on tori.  
	
	However, for actions of lattices in higher-rank, simple Lie groups, the situation is expected to be far more rigid.  
	In particular, the  examples from the previous section  lead to a number of more precise questions and conjectures.  For concreteness,  fix $n\ge 3$ and let $G= \Sl(n,\R)$.  Let  $\Gamma\subset G$ be a lattice. 
	Recall the  action of $\Gamma$ on $S^{n-1}$  and the volume-preserving Anosov action of $\Gamma = \Sl(n,\Z)$ on $\T^n$.

	\begin{questions}\label{Q:qq}
		Consider the following questions:
		\begin{enumcount}
			\item \label{111} Is there a non-finite action of $\Gamma$ on  a manifold of dimension  at most $ n-2$?
			\item \label{222}  If the answer to \ref{111} is unknown, does every  action of $\Gamma$ on  a manifold of dimension at most  $n-2$ preserve a volume form?
			\item \label{333} Is there a non-finite, volume-preserving  action of $\Gamma$ on  a manifold of dimension at most  $n-1$?
			\item \label{444spit} Is every   non-finite  action of $\Gamma$ on  an $n$-torus of the type considered in Example \ref{ex:standard}?  What about volume-preserving actions? That is, if $\alpha\colon \Gamma\to \Diff(\T^n)$  is a non-finite action 
			is $\alpha$ smoothly conjugate to an affine   action as in Example \ref{ex:standard} (or as in \cref{rem:affineAnosov})?
			\item \label{666} Are the only   non-finite  actions  of $\Gamma$ on  a connected $(n-1)$-manifold  those considered in Example \ref{ex:sphere}?    That is, if $\alpha\colon \Gamma\to \Diff^\infty(M)$  is a non-finite action is $M$ either $S^{n-1}$ or $\R P^{n-1}$ and is $\alpha$ smoothly conjugate to the projective action? 
		\end{enumcount}

		Motivated by various conjectures on the classification of Anosov diffeomorphisms and Question \ref{Q:qq}\ref{444spit}, we also pose the following.
		\begin{enumcount}[resume]
			
			\item\label{555}
			Is every (volume-preserving) Anosov action of $\Gamma$ of the type considered in Example \ref{ex:standard}?  That is, if $\alpha\colon \Gamma\to \Diff(M)$ is an Anosov action is $M$ a (infra-)nilmanifold and is $\alpha$ smoothly conjugate to an affine action as in Example \ref{ex:standard} 
			(or as in \cref{rem:affineAnosov})?
		\end{enumcount}
		
	\end{questions}
	
	\cref{Q:qq}\ref{111} and \ref{333} are referred to as \emph{Zimmer's conjecture}, discussed in the next section.  
	Question \ref{Q:qq}\ref{222} is irrelevant given a negative answer to Question \ref{Q:qq}\ref{111} but motivated the result stated in \cref{thm:invmeas} below and was natural to conjecture before an   answer to  Question \ref{Q:qq}\ref{111} was known.  It may be that answering Question \ref{Q:qq}\ref{222}  is possible in dimension ranges where Conjecture \ref{conjecture:zimmergne}\ref{gnea}  below is expected to hold but is not yet known.

	\subsection{Zimmer's conjecture for actions by lattices in $\Sl(n,\R)$} \label{ss:ZimmerSLNR} \index{Zimmer!conjecture}\index{conjecture!Zimmer}
	Recall  \cref{ex:standard} and \cref{ex:sphere}.  For lattices in $\Sl(n,\R)$, \emph{Zimmer's conjecture} asserts that these are the minimal dimensions in which non-finite actions  can occur.  We have the following precise formulation.
	\begin{conjecture}
		\label{conj:slnr}
		For $n\ge 3$, let $\Gamma\subset \Sl(n,\mathbb R)$ be a  lattice. Let $M$ be a compact manifold. \begin{enumcount}
			\item \label{conj:1} If $\dim(M)< n-1$ 
			then any homomorphism $\Gamma\rightarrow \Diff(M)$ has finite image.  
			\item \label{conj:2} In addition, if $\vol$ is a volume form on $M$
			and if $\dim(M)=n-1$ then any homomorphism $\Gamma \rightarrow \Diff_\vol(M)$ has finite image.
		\end{enumcount}
	\end{conjecture}
	
	We are intentionally vague about the regularity in \cref{conj:slnr} (and \cref{conjecture:zimmergne} below). 
Zimmer originally stated \cref{conj:slnr}\ref{conj:2} for the case of  $C^{\infty}$ volume-preserving actions; see \cites{MR682830,MR934329,MR900826}.   \cref{conj:slnr}\ref{conj:1} for $C^\infty$ actions first appears in \cite{MR1666834}.  Most evidence for the  conjecture  requires the action to be at least  $C^1$. It is possible  the conjecture holds for actions by homeomorphisms; see for instance  \cites{MR1198459,MR2807834, MR3150210} for a partial list of results in this directions.  
	The results we discuss below require the action to be at least $C^{1+\beta}$ as we use tools nonuniformly hyperbolic dynamics though some of our results still hold for actions by $C^1$ diffeomorphisms (see \cref{slnrC1} below.)

	\def\C{\mathbb C}
	\def\rep{\mathrm{rep}}
	\def\cpt{\mathrm{cmt}}
	\starsubsection{Zimmer's conjecture for actions by lattices in other Lie groups} \label{sec:ZC2}
	To formulate Zimmer's  conjecture for lattices in general Lie groups, to each simple, non-compact Lie group $G$ we associate 3 positive integers $d_0(G), d_\rep(G) , d_\cpt(G) $  defined roughly as follows:
	\begin{enumerate}
		\item $d_0(G) $ is the minimal dimension of $G/H$ as $H$ varies over proper closed subgroups $H\subset  G$.  (We remark that $H$ is necessarily  a parabolic subgroup in this case.)
		\item $d_\rep(G) $ is the minimal dimension of a non-trivial linear representation of (the Lie algebra) of $G$.  
		\item $d_\cpt(G)$ is the minimal  dimension of a non-trivial homogeneous space of a compact real form of $G$.
	\end{enumerate}
	See Table \ref{tab:stupid} where we compute the above numbers for a number of  matrix groups, (split) real forms of exceptional Lie algebras, and complex matrix groups.  We also include another number $r(G)$ which is defined in \cites{AWBFRHZW-latticemeasure,1608.04995} and arises from certain dynamical arguments\footnote{\label{foot:r}A precise definition that is equivalent to that in \cites{AWBFRHZW-latticemeasure,1608.04995} is that  $r(G)$ is $d_0(G')$ where $G'$ is the largest $\R$-split simple subgroup in $G$.}; this number  gives the bounds appearing in the most general result, Theorem \ref{slnr:popop} below, towards solving Conjecture \ref{conjecture:zimmergne}.   For complete tables of values of $d_\rep(G)$, $d_\cpt(G)$, and $d_0(G)$, we refer to  \cite{Cantat}. 
 
   	\begin{table}[h]\label{page:table}
		\begin{center}
			\begin{tabular}{ | c | c | c | c |c |c |c|}
				\hline
				$G$ & \makecell{\small restricted \\ \small root  system} & rank& $d_\rep(G)$& $d_\cpt(G)$& $d_0(G)$ &$r(G)$\\ \hline 
				$\Sl(n,\R)$ &$A_{n-1}$ & $n-1$ & $n$& $2n-2 $& $n-1$&$n-1$ \\ \hline
				$\So(n,n+1)$ &$B_n$& $n$ & $2n+1$& $ 2n $& $2n-1$ &$2n-1$\\ \hline
				
				$\Sp(2n,\R)$&$C_n$ & $n$ & $2n$& $ 4n-4 $& $2n-1$ &$2n-1$\\ \hline
				$\So(n,n)$ &$D_n$& $n$ & $2n$& $ 2n-1 $& $2n-2$ &$2n-2$\\ \hline
				$E_{I}$ &$E_6$& $6$ & $27$ & $26$ & $16$ & $16$  \\ \hline
				$E_{V}$ &$E_7$& $7$ & $56$ & $54$ & $27$ & $27$  \\ \hline
				$E_{VIII}$ &$E_8$& $8$ & $248$ & $112$ & $57$ & $57$  \\ \hline
				$F_{1}$ &$F_4$& $4$ & $26$ & $16$ & $15$ & $15$  \\ \hline
				$G$ &$G_2$& $2$ & $7$ & $6$ & $5$ & $5$  \\ \hline
				$\Sl(n,\C)$ &$A_{n-1}$ & $n-1$ & $2n$& $2n-2 $& $2n-2$&$n-1$ \\ \hline
				$\SO(2n,\C)$ &$D_{n}$ & $n$ & $4n$& $2n-1 $& $4n-4$&$2n-2$ \\ \hline
				$\SO(2n+1,\C)$ &$B_{n}$ & $n$ & $4n+2$& $2n $& $4n-2$&$2n-1$ \\ \hline
				$\Sp(2n,\C)$ &$C_{n}$ & $n$ & $4n$& $4n-4 $& $4n-2$&$2n-1$ \\ \hline
				\makecell{$\So(p,q)$ \\$ p<q$} &$B_p$& $p$ & $p+q$& $ p+q-1 $& $p+q-2$ &$2p-1$\\ \hline
				
			\end{tabular}
		\end{center}
		\caption{Numerology appearing in Zimmer's conjecture for various   groups.  See also \cite{Cantat} for more complete tables.  See Theorem \ref{slnr:popop}  where the number $r(G)$ appears and \cites{AWBFRHZW-latticemeasure,1608.04995}  or \cref{foot:r} for definition.}\label{tab:stupid}
		\bigskip 
	\end{table}
	
	Given the examples in Section \ref{sec:ex} and the integers $d_\rep(G), d_\cpt(G)$, and $d_0(G)$ defined above, it is natural to conjecture the following.
	
	\begin{conjecture}[Zimmer's Conjecture]\label{conjecture:zimmergne}\index{Zimmer!conjecture} \index{conjecture!Zimmer}
		Let $G$ be a  connected, simple Lie group with finite center.  Let $\Gamma\subset G$ be
		a lattice.  Let $M$ be a compact manifold and $\vol$ a volume form on $M$.  Then
		\begin{enumcount}
			\item \label{gnea}if $\dim(M) < \min \{d_\rep(G), d_\cpt(G), d_0(G)\}$ then any  homomorphism $\alpha\colon \Gamma \rightarrow \Diff(M)$ has finite image;
			\item \label{gneb} if $\dim(M) < \min \{d_\rep(G), d_\cpt(G)\}$ then any homomorphism $\alpha\colon \Gamma \rightarrow \Diff_\vol(M)$ has finite image;
			\item if $\dim(M) < \min\{d_0(G),d_\rep(G)\}$ then for any homomorphism $\alpha\colon \Gamma \rightarrow \Diff(M)$, the image $\alpha(\Gamma)$ preserves a Riemannian metric;
			\item if $\dim(M) < d_\rep(G)$ then for any homomorphism $\alpha\colon \Gamma \rightarrow \Diff_\vol(M)$, the image $\alpha(\Gamma)$ preserves a Riemannian metric.
		\end{enumcount}
	\end{conjecture}

	\subsection{Recent results in the Zimmer program}\label{fullspit}
	The following  two recent results address \cref{Q:qq}\ref{111}--\ref{333} above.  
	In the remainder of this part, we  outline their proofs (at times, specializing to the case of $C^\infty$ actions and the case of  $G= \Sl(3,\R)$.)   We also refer the reader to the excellent  article by Serge Cantat \cite{Cantat} that presents (in French) a complete proof of Theorem \ref{slnr}.

	Before an answer to \cref{Q:qq}\ref{111} and \ref{333} were known, the author together with Federico Rodriguez Hertz and Zhiren Wang studied Question \ref{Q:qq}\ref{222}  and were able to show that all such actions preserve some probability measure.   
	
	\begin{theorem}[{\cite[Theorem 1.6]{AWBFRHZW-latticemeasure}}]\label{thm:invmeas}
		For $n\ge 3$, let $\Gamma\subset \Sl(n,\R)$ be a lattice.  
		Let $M$ be a manifold with $\dim(M)< n-1$.  Then, for any $C^{1+\beta}$ action $\alpha\colon \Gamma\to \diff^{1+\beta}(M)$, there exists an $\alpha$-invariant Borel probability measure.  
	\end{theorem}
	For actions on the circle, an analogue of \cref{thm:invmeas} is shown in \cite[Theorem 3.1]{MR1703323} for actions by homeomorphisms.  
	
	In the critical dimension, $\dim(M)=n-1$,   the projective action on $\R P^{n-1}$  discussed in \cref{ex:sphere} gives an example of an action  that does not preserve any Borel probability measure. 
	If $\alpha$ is an action of $\Gamma$   on a space $X$, we say that a Borel probability measure $\mu$ is \emph{nonsingular} for $\alpha$ if the measure class of $\mu$ is preserved by the action.  In particular, any smooth volume on  $\R P^{n-1}$ is nonsingular for the projective action.     
	In \cite[Theorem 1.7]{AWBFRHZW-latticemeasure},  it is shown that all non-measure-preserving actions on manifolds of the critical dimension $(n-1)$ have  the projective action on $\R P^{n-1}$ equipped with a smooth volume     as a measurable factor.  Precisely, for any action $\alpha\colon \Gamma\to \diff^{1+\beta}(M)$  where $\dim(M)=n-1$ it is shown that either
	\begin{enumerate}
		\item there exists an $\alpha$-invariant Borel probability measure $\mu$ on $M$; or 
		\item there exists a Borel probability measure $\mu$ on $M$ that is nonsingular for the action $\alpha$; moreover the action $\alpha$ on $(M, \mu)$ is measurably isomorphic to a finite extension of the projective action in \cref{ex:sphere} and the image of $\mu$ factors to a smooth volume form on $\R P^{n-1}$.  
	\end{enumerate}
	
	This  gives strong evidence for a positive answer to Question \ref{Q:qq}\ref{666} which we pose as a formal conjecture.  
	\begin{conjecture} \label{conj:koko} For $n\ge 3$, let $\Gamma\subset \Sl(n,\R)$ be a lattice, let $M$ be a closed $(n-1)$-dimensional manifold, and let $\alpha\colon \Gamma\to \Diff^\infty(M)$ be an action with infinite image.  Then, either $M= S^{n-1}$ or $M= \R P^{n-1}$ and the action $\alpha$ is $C^\infty$ conjugate  to the projective action on either $S^{n-1}$ or $\RP^{n-1}$ in  \cref{ex:sphere}.\end{conjecture}

	Returning to actions on manifolds below the critical dimensions in Zimmer's conjecture,  the author with David Fisher and Sebastian Hurtado recently  answered   \cref{Q:qq}\ref{111} and \ref{333} for actions by cocompact lattices in $\Sl(n,\R)$ in \cite{1608.04995}.  
	
	\begin{theorem}[{\cite[Theorem 1.1]{1608.04995}}]\label{slnr}
		For $n\ge 3$, let $\Gamma\subset \Sl(n,\mathbb R)$ be a cocompact lattice. Let $M$ be a compact manifold. \begin{enumcount}
			\item If $\dim(M)< n-1$ 
			then any homomorphism $\Gamma\rightarrow \Diff^{2}(M)$ has finite image.  
			\item In addition, if $\vol$ is a volume form on $M$
			and if $\dim(M)=n-1$ then any homomorphism $\Gamma \rightarrow \Diff^{2}_\vol(M)$ has finite image.
		\end{enumcount}
	\end{theorem}
	
	The proof of \cref{slnr} uses  ideas and results from \cite{AWBFRHZW-latticemeasure}, particularly the proof of Theorem \ref{thm:invmeas}, as  ingredients.  Thus, while \cref{thm:invmeas} follows trivially from \cref{slnr},  we include the proof of  \cref{thm:invmeas} below as   key ideas (namely, \cref{thm:invmsr}, \cref{thm:potatosalad}, and \cref{prop:nonresinv}) will be needed in the proof of \cref{slnr}.
	
	\fakeSS{Remarks on Theorem \ref{slnr}} 
	We give a number of remarks on extensions of Theorem \ref{slnr}.  See also the discussion in  Section \ref{sec:othergps}.  
	\begin{enumerate}
		\item Recently, the author, together with David Fisher and Sebastian Hurtado,  announced in \cite{1710.02735} that the conclusion of Theorem \ref{slnr} holds for actions of $\Sl(n,\Z)$ for $n\ge 3$.  The result for general lattices in $\Sl(n,\Z)$ as well as  analogous results for lattices in other higher-rank simple Lie groups,  has  been announced  \cite{BFHWM}.   
		This establishes \cref{conj:slnr} for actions by $C^2$ (and even $C^{1+\beta}$ diffeomorphisms).  See   Theorem \ref{slnr:popop}.  
		
		The results for actions of $\Sl(n,\Z)$ and of general nonuniform lattices use many of the ideas  presented in this text but also  require a number of new techniques (including the structure of arithmetic groups, reduction theory, and ideas from \cite{MR1828742}) which will not be discussed in this text.  
		\item We state Theorem \ref{slnr} for actions by $C^2$ diffeomorphisms though the proof can be adapted for actions by $C^{1+\beta}$ actions.  Our proof below will assume the action is by $C^\infty$ diffeomorphisms to simplify certain  Sobolev space arguments.  
		\item The result for actions by lattices in  general Lie groups is stated in \cref{slnr:popop} below.   In particular,  
parts \ref{gnea} and \ref{gneb} of Conjecture \ref{conjecture:zimmergne} hold for $C^{1+\beta}$ actions by lattices in   simple Lie groups  that are  non-exceptional, split real forms.   For $C^{1+\beta}$ actions by lattices in simple Lie groups  that are exceptional, split real forms, Conjecture \ref{conjecture:zimmergne}\ref{gnea} holds. 
		\item  D. Damjanovich and Z. Zhang observed that the proof of Theorem \ref{slnr} can be adapted to the setting of actions by $C^1$-diffeomorphisms.  Together with the author, they have announced the following theorem.  
		\begin{theorem}[\cite{BDZ}]
		\label{slnrC1}
			Let $\Gamma\subset G$ be a  lattice in a higher-rank simple Lie group $G$ with finite center. Let $M$ be a compact manifold. \begin{enumerate}
				\item If $\dim(M)< \rank(G)$ 
				then any homomorphism $\Gamma\rightarrow \Diff^{1}(M)$ has finite image.  
				\item In addition, if $\vol$ is a volume form on $M$
				and if $\dim(M)=\rank(G)$ then any homomorphism $\Gamma \rightarrow \Diff^{1}_\vol(M)$ has finite image.
			\end{enumerate}
		\end{theorem}
		For actions by lattices in other higher-rank groups, there is a gap between what is known for $C^1$ versus $C^{1+\beta}$-actions.  Indeed, our number $r(G)$  in \cref{slnr:popop} always satisfies  $r(G)\ge \rank(G)$ and is a strict inequality unless $G$ has restricted root system of type $A_n$.  
	\end{enumerate}

	\section{Superrigidity and heuristics   for Conjecture \ref{conj:slnr}}
	The original conjecture (for actions by lattices in $\Sl(n,\R)$)  posed by Zimmer was Conjecture \ref{conj:slnr}\ref{conj:2} (see for example \cite[Conjecture II]{MR934329}).
	Conjecture  \ref{conj:slnr}\ref{conj:1} was formulated later and first appears in print in \cite[Conjecture I]{MR1666834}.   The reason Zimmer posed his conjecture as Conjecture \ref{conj:slnr}\ref{conj:2} is that the strongest evidence for the    conjecture---Zimmer's {\it cocycle superrigidity theorem}---requires  the action  to preserve some Borel probability  measure.  Zimmer's cocycle superrigidity theorem also provides strong evidence for local and global rigidity conjectures related to Questions \ref{Q:qq}\ref{444spit}--\ref{555} and is typically used in proofs of results towards solving such conjectures.  
	
	In this section, we state a version of Zimmer's cocycle superrigidity theorem and some consequences.  We also state a version of Margulis's superrigidity theorem (for linear representations).   We also give some heuristics for Zimmer's conjecture that follow from the superrigidity theorems.  
	General references for this section include \cites{MR1090825,MR776417,MR3307755}.
	\subsection{Cocycles over group actions}
	Consider a standard probability space $(X, \mu)$.  Let $G$ be a locally compact topological group and let $\alpha\colon G\times X\to X$ be a measurable action of $G$ by   $\mu$-preserving transformations.  In particular,  $\alpha(g)$ is a $\mu$-preserving, measurable transformation of $X$ for each $g\in G$.  Below,  we will always assume the measure $\mu$ is ergodic for this action.

	\index{linear cocycle!measurable}
	A $d$-dimensional \emph{measurable linear cocycle} over $\alpha$ is a measurable map $$\calA\colon G\times X\to \Gl(d,\R)$$ satisfying for a.e.\ $x\in X$ the cocycle condition: for all $g_1, g_2\in G$, \begin{equation}\label{cocycl2} \calA(g_1g_2, x)= \calA\left(g_1, \alpha(g_2)(x)\right ) \calA (g_2, x).\end{equation}
	If $e$ is the identity element of $G$, then   \eqref{cocycl2} implies  that $$\calA(e,x) = \calA(e,x) \calA(e,x)$$  whence  $\calA(e, x) = \id$ for a.e.\ $x$.
	
	We say two cocycles $\calA,\calB\colon G\times X\to \Gl(d,\R)$ are (measurably) \emph{cohomologous} if there is a measurable map $\Phi\colon X\to \Gl(d,\R)$ such that for a.e.\ $x$ and every $g\in G$,\index{linear cocycle!cohomologous}
	\begin{equation}\label{eq:everynowandthenIgetalittlebitnervousthatthebestofallmyyearshavegoneby}\calB(g,x) = \Phi(\alpha(g)(x)) \inv \calA(g,x) \Phi(x).\end{equation}
	We say a cocycle $\calA\colon G\times X\to \Gl(d,\R)$ is \emph{constant} if  $\calA(g,x)$   is independent of $x$, that is, if $\calA\colon G\times X\to \Gl(d,\R)$ coincides  with a representation $\pi\colon G\to \Gl(d,\R)$ on a set of full measure.  	\index{linear cocycle!constant}

	As a primary example, let $\alpha\colon G\to \Diff^1_\mu(M)$ be an action of $G$ by $C^1$ diffeomorphisms of a compact manifold $M$ preserving some Borel probability measure $\mu$.    Although the tangent bundle $TM$ may not be a trivial bundle, we may choose a   Borel measurable trivialization  $\Psi\colon TM\to M\times \R^d$ of the vector-bundle $TM$ where $d = \dim(M)$.  We have that $\Psi$ factors over the identity map on $M$ and, writing $\Psi_x \colon T_x \to \R^d$ for the identification of the fiber over $x$ with $\R^d$, we moreover assume that $\|\Psi_x\|$ and $\|\Psi_x\inv\|$ are uniformly bounded in $x$.

	Fix such a trivialization $\Psi$  and define $\calA$ to be the derivative cocycle relative to this trivialization: $$\calA(g,x) = D_x \alpha(g)$$ where, we view 
	$  D_x \alpha(g)$ as an element of $\Gl(d,\R)$ transferring the fiber $\{x\} \times \R^d $ to $\{\alpha(g)(x)\}\times  \R^d$ via the measurable trivialization $\Psi$.  To be precise, if $\Psi\colon TM\to M\times \R^d$ is the measurable vector-bundle trivialization then  $$\calA(g,x) := \Psi(\alpha(g)(x))  D_x \alpha(g) \Psi(x) \inv .$$
	In this case, the cocycle relation \eqref{cocycl2} is simply the chain rule.   Note that if we choose another Borel measurable trivialization  $\Psi'\colon TM\to M\times \R^d$ then we obtain a cohomologous  cocycle $\calA'$.   Indeed, we have $$\calA'(g,x) =  \Psi'(\alpha(g)(x))  \Psi(\alpha(g)(x))\inv   \calA(g,x) \Psi(x) \Psi'(x) \inv $$
	so we may take $\Phi(x) =  \Psi(x) \Psi'(x) \inv $ in \eqref{eq:everynowandthenIgetalittlebitnervousthatthebestofallmyyearshavegoneby}.
	
	We have the following elementary fact which we frequently use in the case of volume-preserving actions.  
	\begin{claim}\label{cor:stupidcor}
		Let $\alpha\colon G\to \diff^1_\vol(M)$ be an action by volume-preserving diffeomorphisms.  Then, for any $\alpha$-invariant measure $\mu$, the derivative cocycle $\calA$ is cohomologous to a $\Sl^{\pm}(d,\R)$-valued cocycle.  
	\end{claim}
	Above, $\Sl^{\pm}(d,\R)$ is the subgroup of $\Gl(d,\R)$ defined by $\det(A) = \pm 1$.  
	
	\subsection{Cocycle superrigidity}\index{Zimmer!cocycle superrigidity}\index{linear cocycle!superrigidity}\index{theorem!Zimmer cocycle superrigidity}
	We  formulate the statement of Zimmer's  cocycle superrigidity theorem when $G$ is either $\Sl(n,\R)$ or a lattice subgroup of $\Sl(n,\R)$ for $n\ge 3$. Note that the version formulated by Zimmer (see \cite{MR776417}) had a slightly weaker conclusion.  We state the stronger version formulated and proved in \cite{MR2039990}.  
	\begin{theorem}[Cocycle superrigidity \cites{MR776417,MR2039990}]\label{thm:ZCSR}
		For $n\ge 3$, let $G $ be either $G=  \SL(n,\R)$ or let $G$ be a lattice in $\SL(n,\R)$.  Let $\alpha\colon G\to \Aut(X, \mu)$ be an ergodic, measurable action of $G$ by  $\mu$-preserving transformations of a standard probability space $(X, \mu)$.  Let $\calA \colon G\times X\to \Gl(d,\R)$ be a bounded,\footnote{Here, \emph{bounded} means that for every compact $K\subset G$, the map $K\times X\to \Gl(d,\R)$ given by $(g,x)\mapsto \calA(g,x)$ is bounded.   More  generally, we may replace the boundedness hypothesis with the hypothesis  that the function $x\mapsto\sup_{g\in K} \log \|\calA(g,x)\|$  is $L^1(\mu)$.  See  \cite{MR2039990}.} measurable 
		linear cocycle over  $\alpha$.  
		
		Then there exist 
		\begin{enumerate}
			\item a linear representation $\rho\colon\Sl(n,\R)\to \Sl(d,\R)$;
			\item a compact subgroup $K\subset \Gl(d,\R)$ that commutes with the image of $\rho$;  
			\item a $K$-valued cocycle $\mathcal C\colon G\times X\to K$;
			\item and a measurable function $\Phi\colon X\to \Gl(d,\R)$
		\end{enumerate}
		such that for \ae $x\in X$ and every $g\in G$
		\begin{equation}\label{ez:ZCS}\calA(g,x)= \Phi(\alpha(g)( x) ) \inv \rho(g)  \mathcal C(g,x) \Phi(x).\end{equation}
	\end{theorem}
	
	In particular, Theorem \ref{thm:ZCSR} states that any bounded measurable linear cocycle $\calA \colon G\times X\to \Gl(d,\R)$ over the action $\alpha$ 
	is cohomologous to the product of constant cocycle $\rho\colon G\to \Sl(d,\R)$ and a compact-valued cocycle $\mathcal C\colon G\times X\to K\subset \Gl(d,\R)$.

	\subsection{Superrigidity for linear representations}\index{Margulis superrigidity}\index{theorem!Margulis superrigidity}
	Zimmer's  cocycle superrigidity theorem is an extension of Margulis's superrigidity theorem for linear representations.  We formulate a version of this theorem for linear representations of lattices in $\Sl(n,\R)$.  
	\begin{theorem}[Margulis superrigidity \cite{MR1090825}]\label{thm:MSR}
		For $n\ge 3$, let $\Gamma $ be a lattice in $ \SL(n,\R)$.  Given a representation $\rho\colon \Gamma \to \Gl(d,\R)$ there are 
		\begin{enumerate}
			\item a linear representation $\hat \rho\colon\Sl(n,\R)\to \Sl(d,\R)$;
			\item a compact subgroup $K\subset \Gl(d,\R)$ that commutes with the image of $\hat \rho$
		\end{enumerate}
		such that $$\hat \rho(\gamma) \rho(\gamma)\inv\in K$$ for all $\gamma\in \Gamma$.  
		
		That is, $ \rho  = \hat \rho\cdot  c  $ is the product of  the restriction of a representation $$\hat \rho\colon\Sl(n,\R)\to \Sl(d,\R)$$  to $\Gamma$ and a compact-valued representation $c\colon \Gamma \to K$.   Moreover the image of $\hat \rho$ and $c$ commute.
	\end{theorem}
	
	In the case that $\Gamma$ is nonuniform, one can show that all compact-valued representations $c\colon \Gamma \to K$ have finite image.    See for instance the discussion in \cite[Section 16.4]{MR3307755}, especially  \cite[Exercise 16.4.1]{MR3307755}. 
	
	For certain cocompact $\Gamma\subset \Sl(n,\R)$, there exists compact-valued representations $c\colon \Gamma \to \SU(n)$ with infinite image. (See   discussion in Example \ref{ex:isom}.)\index{Margulis superrigidity!compact codomain}
	The next theorem,  characterizing all homomorphisms from lattices in $\Sl(n,\R)$ into compact Lie groups,   shows that representations into $ \SU(n)$  are   more-or-less the only such examples.  The proof uses  the $p$-adic version of  Margulis's superrigidity theorem and some algebra.  See \cite[ Theorem VII.6.5]{MR1090825} and  \cite[Corollary 16.4.2]{MR3307755}.  
	
	
	\begin{theorem}\label{thm:MSRcompactcodomain}
		For $n\ge 3$, let $\Gamma\subset \Sl(n,\R)$ be a lattice.  Let $K$ be a compact Lie group and $\pi\colon \Gamma\to K$ a homomorphism.
		\begin{enumcount}
			\item If $\Gamma$ is nonuniform then  $\pi(\Gamma)$ is finite.
			\item \label{MSR2} If $\Gamma$ is cocompact and  $\pi(\Gamma)$ is infinite then there is a closed subgroup $K'\subset K$ with 
			$$\pi(\Gamma)\subset K'\subset K$$
			and the Lie algebra of $K'$ is of the form $\lie(K') = \su(n)\times \dots \times \su(n)$.
		\end{enumcount}
	\end{theorem}
	The  appearance of $\su(n)$ in  \ref{MSR2} of Theorem  \ref{thm:MSRcompactcodomain} is due to the fact that $\su(n)$ is the compact real form of $\mathfrak{sl}(n,\R)$, the Lie algebra of $\Sl(n,\R)$.  
	For a cocompact lattice $\Gamma$ in $\So(n,n)$ as in \cref{ex:isom}, the analogue of \cref{thm:MSRcompactcodomain} states that $$\lie(K') = \mathfrak{so}(2n)\times \dots\times \mathfrak{so}(2n).$$

	\subsection{Heuristic evidence for Conjecture \ref{conj:slnr}}
	We present a number of heuristics that motivate  the conclusions of  Conjectures \ref{conj:slnr} and \ref{conjecture:zimmergne}.  
	\subsubsection{Analogy with linear representations}
	Note that if $d< n$, there is no non-trivial representation $\hat \rho\colon \Sl(n,\R)\to\Sl(d,\R)$; moreover, by a dimension count, there is no embedding of $\su(n)$ in $\mathfrak{sl}(d,\R)$.   We thus immediately obtain as corollaries of Theorems \ref{thm:MSR} and  \ref{thm:MSRcompactcodomain} the following. 
	\begin{corollary}
		For $n\ge 3$, let $\Gamma $ be a lattice in $G=  \SL(n,\R)$.  Then, for $d< n$,  the image of any representation $\rho\colon \Gamma \to \Gl(d,\R)$ is finite.  
	\end{corollary} 
	
	Conjecture \ref{conj:slnr} can be seen  as a ``nonlinear'' analogue  of this corollary.   That is, we aim to prove the same result when the linear group $\Gl(d,\R)$   is  replaced by    certain diffeomorphism groups $\diff(M)$.  
	
	\subsubsection{Invariant measurable metrics}
	For $n\ge 3$, let $\Gamma $ be a lattice in $G=  \SL(n,\R)$ and consider a measure-preserving action  $\alpha \colon \Gamma \to \Diff^1_\mu(M)$ where $M$ is a compact manifold of dimension at most $d\le n-1$ and $\mu$ is an arbitrary Borel probability measure on $M$ preserved by $\alpha$.  The derivative cocycle of the action  $\alpha$ is then $\Gl(d, \R)$-valued.  Since there are no representations $\rho\colon \Sl(n,\R)\to \Sl(d,\R)$  for $d<n$, Theorem \ref{thm:ZCSR} implies that the derivative cocycle is cohomologous to a compact-valued cocycle.  In particular, we have the following:
	
	\begin{corollary}\label{cor:cor} For $\Gamma, M, \mu$ and $\alpha \colon \Gamma \to \Diff^1_\mu(M)$  as above
		\begin{enumcount}
			\item \label{corcor1} $\alpha$ preserves a `$\mu$-measurable Riemannian metric,' namely there is a $\mu$-measurable, $\alpha$-invariant,  positive-definite symmetric two-form on $TM$;
			\item \label{corcor2} for any $\epsilon>0$  and $\gamma\in \Gamma$, the set of $x\in M$ such that 
			$$ \liminf _{n\to \infty} \frac 1 n \log \|D_x \alpha(\gamma^n)\| \ge \epsilon $$
			has zero $\mu$-measure.  
		\end{enumcount}
	\end{corollary}
	For \ref{corcor1}, suppose the derivative cocycle is cohomologous to a $K$-valued cocycle for some compact group $K\subset \Gl(d,\R)$.  One may then pull-back any $K$-invariant inner product on $\R^d$ to $T_xM$ via the map $\Phi(x)$ in Theorem \ref{thm:ZCSR} to an $\alpha (\Gamma)$-invariant inner product. 
	Conclusion \ref{corcor2} follows from Poincar\'e recurrence to sets on which the function $\Phi\colon M\to \Gl(d,\R)$ in Theorem \ref{thm:ZCSR}  has bounded norm and conorm. Note   from \ref{corcor2} that all Lyapunov exponents for individual elements of the action must vanish.  
	
	From Corollary \ref{cor:cor},  given $n\ge 3$ and   a lattice $\Gamma $ in $G=  \SL(n,\R)$,  we have that every action  $\alpha \colon \Gamma \to \Diff^1_\vol(M)$ preserves a Lebesgue-measurable Riemannian metric $g$ whenever $M$ is a compact manifold of dimension at most $n-1$.  Suppose one could show that $g$ was continuous or $C^\ell$.  As we discuss in Step 3 of Section \ref{sec:outline3steps} below, this combined with Theorem   \ref{thm:MSRcompactcodomain} implies the image $\alpha(\Gamma)$ is finite.    Thus, Conjecture \ref{conj:slnr}\ref{conj:2} follows if one can promote the measurable invariant metric $g$ guaranteed by Corollary \ref{cor:cor} of Theorem \ref{thm:ZCSR} to a continuous  Riemannian metric.
	
	The discussion in the previous paragraphs suggests the following variant of  Conjecture \ref{conj:slnr}\ref{conj:2} might hold: \begin{quote} {\it For $n\ge 3$, if $\Gamma\subset \Sl(n,\R)$ is a lattice and if $\mu$ is  any fully supported Borel probability measure on a compact manifold $M$ of dimension at most $(n-1)$ then any    homomorphism $$\Gamma \rightarrow \Diff_\mu(M)$$ has finite image.}\end{quote}
	Our method of proof of Conjecture \ref{conj:slnr}\ref{conj:2} does not establish  this conjecture.  However, the conjecture would follow (even allowing for $\mu$ to have partial support) if the  global rigidity result in \cref{conj:koko} holds.  
	%
	
	\subsubsection{Actions with discrete spectrum}
	Upgrading the measurable invariant Riemannian metric in \cref{cor:cor} to a continuous Riemannian metric in the above heuristic seems   quite difficult  and is not the approach we take in the proof of \cref{slnr}.   In \cite{MR743815}, Zimmer was able to upgrade the measurable metric to a continuous metric  for volume-preserving actions that are very close to isometries.  This result now follows from the local rigidity of isometric actions in \cites{MR1779610,MR2198325}.  
	
	Zimmer later established a  much stronger result  in \cite{MR1147291} which  provides very strong evidence for the volume-preserving cases in  \cref{conjecture:zimmergne}.  Using   the   invariant, measurable metric discussed above and that higher-rank lattices have Property (T), Zimmer showed  that any volume-preserving action appearing  in \cref{conjecture:zimmergne} has discrete spectrum.   In particular, this result  implies that (the ergodic components of) all volume-preserving actions appearing  in  \cref{conjecture:zimmergne}  are measurably isomorphic to isometric actions.  

	\section{Structure theory of $\Sl(n,\R)$ and Cartan flows on $\Sl(n,\R)/\Gamma$}\label{sec:slstructure}
	Let $G = \Sl(n,\R)$ and let  $\Gamma\subset G$ be a lattice.  Recall we write   $G= KAN$ for  the Iwasawa decomposition 
	where  
	$$K = \So(n,\R),\quad \quad A = \{ \diag(e^{t_1}, e^{t_2},\dots, e^{t_n}) : t_1+\dots + t_n = 0\},$$ 
	and $N $ is the group of  upper triangular matrices with 1s on the diagonal.  
	
	We will be interested in certain subgroups of $G$ and how they capture dynamical information of the action of the Cartan subgroup $A$ on the homogeneous space $G/\Gamma$.  
	
	\subsection{Roots and root subgroups}\label{ss:roots}\index{Lie group!roots}
	We consider the following  linear functionals    $$\beta^{i,j}\colon A\to \R$$ given as follows: 
	for $i\neq j$, 
	$$\beta^{i,j} \left(\diag(e^{t_1}, e^{t_2}, \dots, e^{t_n}) \right)= t_i - t_j.$$
	The linear functionals $\beta^{i,j}$ are the \emph{roots} of $G$.
	
	Associated to each root $\beta^{i,j}$ is a 1-parameter unipotent subgroup $U^{i,j}\subset G$.  \index{Lie group!root subgroup}
	For instance, in $G= \Sl(3, \R)$ we have the following 1-parameter flows
	$$u^{1,2}(t)=   \left(\begin{array}{ccc}1 & t & 0 \\0 & 1 & 0 \\0 & 0 & 1\end{array}\right),\hspace{.7em} 
	u^{1,3}(t)=   \left(\begin{array}{ccc}1 & 0 & t \\0 & 1 & 0 \\0 & 0 & 1\end{array}\right),\hspace{.7em} 
	u^{2,3}(t)=   \left(\begin{array}{ccc}1 & 0 & 0 \\0 & 1 & t \\0 & 0 & 1\end{array}\right),
	$$
	$$
	u^{2,1}(t)=  \left(\begin{array}{ccc}1 & 0 & 0 \\t & 1 & 0 \\0 & 0 & 1\end{array}\right),\hspace{.7em} 
	u^{3,1}(t)=   \left(\begin{array}{ccc}1 & 0 & 0 \\0 & 1 & 0 \\t & 0 & 1\end{array}\right),\hspace{.7em} 
	u^{3,2}(t)=   \left(\begin{array}{ccc}1 & 0 & 0 \\0 & 1 &0\\0 & t & 1\end{array}\right).$$

	We let $U^{i,j}$ denote  the associated 1-parameter  unipotent subgroups of $G$: \begin{equation}\label{eq:rootgroup}U^{i,j}:= \{ u^{i,j}(t): t\in \R\}.\end{equation}
	The groups $U^{i,j}$ have the property that conjugation by $s\in A$  dilates their parametrization by $e^{\beta^{i,j}(s)}$: \begin{equation}\label{kokololo} s u^{i,j}(t) s\inv = u^{i,j}(e^{\beta^{i,j}(s)} t).\end{equation}   In particular, if $g'= u^{i,j}(t) \cdot g$ is in the $U^{i,j}$-orbit of $g$ and $s\in A$ then \begin{equation*}\label{eq:kokololo}s\cdot g' = u^{i,j}(e^{\beta^{i,j}(s)} t) \cdot s\cdot g.\end{equation*}
	
	
	\subsection{Cartan flows}\label{sec:diag}\index{Cartan flow}
	For concreteness, consider $G= \Sl(3, \R)$ and let $\Gamma $ be a lattice in $\SL(3,\R) $ such as $ \Sl(3,\Z)$.
	Let $X$ denote the coset space $X= G/\Gamma$.  This is an $8$ dimensional  manifold (which is noncompact when  $\Gamma $ is a nonuniform lattice such as $\Sl(3,\Z)$.) 
	$G$ acts on $X$ on the left:  given $g\in G$ and $x = g'\Gamma \in X$ we have $$g\cdot x= gg'\Gamma\in X.$$

	The Cartan subgroup  $A\subset G$ is the subgroup of  diagonal matrices with positive entries $$A := \left \{ \left(\begin{array}{ccc}e^{t_1} & 0 & 0 \\0 & e^{t_2} & 0 \\0 & 0 & e^{t_3}\end{array}\right): t_1+ t_2+ t_3= 0 \right\}.$$
	The group $A$ is isomorphic to $\R^2$, for instance, via the embedding 
	\begin{equation*}\label{eq:LLOOP}(s,t)\mapsto \diag(e^s , e^t, e^{-s-t}).\end{equation*}
	We consider  the action 
	$\alpha \colon A\times X\to X$ of $A$ on $X$  given by $$\alpha (s)(x) = s  x.$$

	For $x\in X$ let $W^{i,j}(x)$ be the orbit of $x$ under the 1-parameter group $U^{i,j}$:
	$$W^{i,j}(x) = \{ u^{i,j}(t) x: t\in \R\}\}.$$
	For $s\in A$, we claim that the $s$-action on $X$ dilates the natural parametrization of each $W^{i,j}(x)$  by exactly $\beta^{i,j}(s)$.  Indeed, if $x\in X$ and if $x' = u^{i,j}(v)\cdot x \in W^{i,j}(x)$ then for  $s \in  A$ we have 
	\begin{align*} \alpha(s)(x') 
		&=   s u ^{i,j}(v)   x \\
		&=   s u ^{i,j}(v)  s\inv s x \\
		&=  u^{i,j}(v') \alpha(s)(x)
	\end{align*}
	where, using \eqref{kokololo}, we have that have  $$v' = e^{\beta^{i,j}(s)} v.$$   In particular,  we interpret the functionals $\beta^{i,j}$ as the (non-zero) Lyapunov exponents for the $A$-action on $X$ (with respect to any $A$-invariant measure).  Note that the zero functional is a Lyapunov exponent of multiplicity two corresponding to the $A$-orbits.   The tangent spaces to each $W^{i,j}(x)$ as well as the tangent space to the orbit $A\cdot x$ gives the $A$-invariant splitting guaranteed by Theorem \ref{thm:oscHR}.
	Note that no two roots $\beta^{i,j}$ are positively proportional and hence are their own coarse Lyapunov exponents for the action (see Section \ref{sss:cle}).

	\section{Suspension space and fiberwise exponents}
	\label{sec:cons} 
	We now begin the proofs of \cref{thm:invmeas} and \cref{slnr} with a technical but crucial construction.  Here, we induce from an action $\alpha$ of a lattice $\Gamma$ on a manifold $M$ to an action of $G=\Sl(n,\R)$ on an auxiliary manifold denoted by $M^\alpha$.  The properties of the $G$-action on $M^\alpha$ mimic the properties of the $\Gamma$-action on $M$.  However, for a number of reasons it is much more convenient to study the $G$-action on $M^\alpha$.  The construction is parallel to the construction described in    \cref{ss:suspspace}.  

	\subsection{Suspension space and   induced $G$-action} \label{ssec:susp}\index{suspension} Fix  $G= \Sl(n,\R)$ and let $\Gamma\subset G$ be a lattice.  
	Let $M$ be a compact manifold and let $\alpha\colon \Gamma\to \diff(M)$ be an action.

	On the product $G\times M$ consider the right $\Gamma$-action
	$$ (g,x)\cdot \gamma= (g\gamma , \alpha(\gamma\inv)(x))$$ and the left $G$-action $$a\cdot (g,x) = (ag, x).$$
	Define  the quotient manifold $M^\alpha:= (G\times M)/\Gamma $.  As the  $G$-action on $G\times M$ commutes with the $\Gamma$-action, we have an induced left $G$-action  on $M^\alpha$.  For $g\in G$ and $x\in M^\alpha$ we denote this action by $g\cdot x$ and denote the derivative of the diffeomorphism $x\mapsto g\cdot x$ at $x\in M^\alpha$  $D_xg\colon T_xM^\alpha\to T_{g\cdot x}M^\alpha$. 
	
	We write $$\pi\colon M^\alpha\to \Sl(n,\R)/\Gamma$$ for the natural projection map.  Note that $M^\alpha$ has the structure of a fiber-bundle over $\Sl(n,\R)/\Gamma$ induced by the map $\pi$ with fibers diffeomorphic to $M$.   The  $G$-action permutes the $M$-fibers of $M^\alpha$.
	We let $F= \ker(D\pi)$ be the \emph{fiberwise tangent bundle:} for $x\in M^\alpha$, $F(x)\subset T_xM^\alpha$ is the $\dim(M)$-dimensional subspace tangent to the   fiber through $x$.  
	
	Equip $M^\alpha$ with a continuous Riemannian metric. For convenience, we  moreover  assume the  restriction of the metric to $G$-orbits coincides under push-forward by  the  projection $\pi\colon M^\alpha \to \Sl(n,\R)/\Gamma$ with  the metric on $\Sl(n,\R)/\Gamma$ induced by a  right-invariant (and left $K$-invariant) metric on $G$. (We note that if $\Gamma$ is cocompact, $M^\alpha$ is compact and all metrics are equivalent. In the case that $\Gamma$ is not cocompact, some additional care is needed to ensure the metric is well behaved in the fibers. We will not discuss the technicalities of this case here.)

	We outline the  construction of such a metric.  Fix  a $C^\infty$ Riemannian metric $\langle \cdot , \cdot \rangle$ on $TM$. Passing to a finite index subgroup, we may assume that $K\bs G/\Gamma$ is a manifold. When $\Gamma$ is cocompact in $G$, this manifold is compact.  (More generally, when $\Gamma$ is nonuniform,  $K\bs G/\Gamma$ has a compactification as real-analytic manifold with corners; see \cite{MR2126641}.) 
  Let  $\{\hat \psi_i, i=1,\dots, m\}$ be a finite,  $C^\infty$  partition of unity  of the locally symmetric space $K\bs G/\Gamma$ subordinate to finitely many coordinate charts. Lift each $\hat \psi_i$ to a $K$-invariant function defined on $G/\Gamma$.  For each $i$, we select at compactly supported  $\psi_i\colon G\to [0,1]$ such that $\psi_i(g) = \hat \psi_i(g\Gamma)$, the map $g\mapsto g\Gamma$ is injective on the support of $\psi_i$, and  the support of each $ \psi_i$ intersects a fixed  compact fundamental domain containing the identity.  Write $\psi_{i,\gamma}\colon  G\to [0,1]$ for the function $$\psi_{i,\gamma}(g) = \psi_i(g\gamma\inv).$$
	The supports satisfy $\supp(\psi_{i,\gamma})\cap \supp(\psi_{i,\gamma'})= \emptyset$ whenever $\gamma\neq \gamma'$
	and the collection $\{\psi_{i,\gamma}\mid i\in \{1,\dots, m\}, \gamma\in \Gamma\}$ is a partition of  unity on $G$.
	Given $v,w\in \{g\} \times T_xM$ set $$\langle v,w\rangle_{g,x} := \sum_{i=1}^m\sum_{\gamma\in \Gamma} \phi_{i,\gamma}(g)  \langle D_x\alpha(\gamma)(v),D_x\alpha(\gamma)(w)\rangle_{x}.$$
	Equip $T_{(g,x)}(G\times M)= T_gG\times T_xM$ with the product of the left $K$-invariant, right $\Gamma$-invariant metric on $G$  and $\langle v,w\rangle_{g,x}.$  Note that this metric is $\beta$-H\"older continuous if $\alpha$ is an action by $C^{1+\beta}$ diffeomorphisms.  
	We then verify that    $\Gamma$ acts by isometries and thus the metric  descends to a metric on $M^\alpha$.  Indeed, writing $\|\cdot \|_{g,x}$ for the  norm associated to  $\langle \cdot,\cdot \rangle_{g,x}$, for $v \in \{g\hat \gamma \} \times T_xM$ we have
	\begin{align*}
		\|v\|_{g\hat \gamma,x}^2
		&=\sum_{i=1}^m\sum_{\gamma\in \Gamma} \phi_{i,\gamma}(g\hat\gamma) \| D_x\alpha(\gamma)(v)\|_0 ^2\\
		&=\sum_{i=1}^m\sum_{\gamma\in \Gamma} \phi_{i,\gamma\hat\gamma\inv}(g) \| D_x\alpha(\gamma)(v)\|_0^2\\
		&=\sum_{i=1}^m\sum_{\gamma\in \Gamma} \phi_{i,\gamma\hat\gamma\inv}(g) \| D_x\alpha(\gamma\hat\gamma\inv\hat\gamma)(v)\|_0^2\\
		&=\sum_{i=1}^m\sum_{\gamma\in \Gamma} \phi_{i,\gamma\hat\gamma\inv}(g) \| D_{\alpha(\hat \gamma)(x)}\alpha(\gamma\hat\gamma\inv,\alpha(\hat\gamma)(x))D_x\alpha(\hat\gamma,x)(v)\|_0^2\\
		&=\|D_x\alpha(\hat \gamma) v\|_{g, \alpha(\hat \gamma)(x)}^2.
	\end{align*}

	\subsection{Fiberwise Lyapunov exponents}\label{ss:fibLyap} \index{Lyapunov exponent!fiberwise}
	Recall  that $A\subset G$ is the subgroup  $$A= \{ \diag(e^{t_1}, e^{t_2}, \dots, e^{t_n})\}\simeq \R^{n-1}.$$
	The $G$-action on $M^\alpha$  restricts to  an  $A$-action on $M^\alpha$.  Let $\mu$ be any ergodic, $A$-invariant Borel probability measure on $M^\alpha$.  The $G$-action (and hence the $A$-action) permutes the  fibers of $M^\alpha$ and hence the derivatives of the $G$- and $A$-actions  preserve the   fiberwise tangent subbundle $F\subset TM^\alpha$. 
	
	We equip $A\simeq \R^{n-1}$   with a norm $|\cdot|$.  
	We may restrict Theorem \ref{thm:oscHR} to the  $A$-invariant subbundle $F \subset T M^\alpha$ and obtain Lyapunov exponent functionals for the fiberwise derivative cocycle.    We thus obtain 
	\begin{enumerate}
		\item an $A$-invariant set $\Lambda\subset M^\alpha$ with $\mu(\Lambda)=1$;
		\item  linear functionals  $\lambda^F_{1,\mu}, \lambda^F_{2,\mu},\dots,\lambda^F_{p,\mu}\colon A\to \R$; and
		\item    a $\mu$-measurable, $A$-invariant  splitting $F(x) = \bigoplus_{i=1}^p E_i^F(x)$ defined for $x\in \Lambda$
	\end{enumerate}
	such that for every $x\in \Lambda$ and $v\in E^F_i(x)\sm \{0\}$ 
	$$\lim_{|a|\to   \infty} \dfrac { \log \|D_x a (v)\| -  \lambda^F_{i,\mu}(a)}{|a|} = 0.$$
	In particular, for any $a\in A$ and $v\in F(x)\sm \{0\}$ we have 
	$$\lim_{k\to   \infty} \dfrac { 1}k \log \|D_xa^k(v)\|  =   \lambda^F_{i,\mu}(a).$$


	
	
	A \emph{coarse fiberwise Lyapunov exponent}  $\chi^F_{\mu}$ is a positive proportionality class of fiberwise Lyapunov exponents.  \index{coarse Lyapunov!exponent!fiberwise}

	\section{Invariance principle and Proof of Theorem \ref{thm:invmeas}}
	\label{sec:IP}
	\subsection {Proof of Theorem \ref{thm:invmeas}}\label{ss:IP1}
	Given the constructions in Section \ref{sec:cons} and Ledrappier's theorem as formulated in  Theorem \ref{thm:led'} (see also  \cref{prop:easyLed}), we are now in a position to prove Theorem  \ref{thm:invmeas}.  In fact, we prove the following invariance principle: 
	
	\begin{theorem}\label{thm:invmsr}\index{invariance principle!low dimensional fibers}
		Let $\Gamma\subset \Sl(n,\R)$ be a lattice.   Let  $\alpha \colon \Gamma\to \diff^{1+\beta}(M)$ be an action and let $M^\alpha$ denote the suspension space with induced $G$-action.
		Let $\mu$ be an ergodic, $A$-invariant Borel probability measure on $M^\alpha$ whose projection to $\Sl(n,\R)/\Gamma$ is the Haar measure.  
		
		Then, if $\dim(M)\le n-2$ the measure $\mu$ is $G$-invariant.  Moreover, if $\alpha$ preserves a volume form $\vol$ and if   $\dim(M)\le n-1$ then the measure $\mu$ is $G$-invariant.
	\end{theorem}
	Note that Theorem \ref{thm:invmsr} does not require that $\Gamma$  be cocompact.\footnote{However, in the case that $\Gamma$ is nonuniform, the space $M^\alpha$ is not compact and some care is needed to define Lyapunov exponents; in particular, we must specify a Riemannian metric on $M^\alpha$.  A  Riemannian metric on $M^\alpha$ adapted to this setting is constructed  in \cite{AWBFRHZW-latticemeasure}.}
	Theorem \ref{thm:invmeas} follows immediately from Theorem \ref{thm:invmsr}:
	since $A$ is abelian (in particular amenable) and the space of probability measures on $M^\alpha$  projecting to the Haar measure on $\Sl(n,\R)/\Gamma$ is nonempty,  $A$-invariant, and weak-$*$ compact, the Krylov-Bogolyubov theorem implies there is an  $A$-invariant Borel probability measure $\mu$ on $M^\alpha$  projecting to the Haar measure on $\Sl(n,\R)/\Gamma$.  
	Since the Haar measure on  $\Sl(n,\R)/\Gamma$ is $A$-ergodic, we may moreover assume that $\mu$ is $A$-ergodic.  
	Theorem \ref{thm:invmsr} implies $\mu$ is $G$-invariant and  Theorem \ref{thm:invmeas} then follows from the following elementary claim.
	\begin{claim}\label{claim:oolong}
		The $\Gamma$-action $\alpha$ on $M$ preserves a Borel probability measure if and only if the induced $G$-action on $M^\alpha$ preserves a Borel probability measure (which necessarily projects to the Haar measure on $G/\Gamma$).  
	\end{claim} 
	Indeed, if $\mu$ is a $G$-invariant measure on $M^\alpha$ then conditioning on the fiber of $M^\alpha$ over $e\Gamma\in G/\Gamma$ gives an $\alpha$-invariant measure on $M$ viewed as the fiber of $M^\alpha$ over $e\Gamma$.  On the other hand, if $\hat \mu$ is an $\alpha$-invariant measure on $M$ then, writing $m_G$ for the Haar measure on $G$, we have $m_G\times \hat \mu$ is a (right) $\Gamma$-invariant and (left) $G$-invariant measure on $G\times M$ and hence descends to a (finite) $G$-invariant measure on $M^\alpha$.

	\begin{remark}
		For more general semisimple Lie groups $G$ we have the following theorem which follows from the proof of Theorem \ref{thm:invmsr}.  In this setting, we take $A$ to be a maximal split Cartan subgroup; that is, $A$ is a maximal, connected, abelian subgroup of  $\R$-diagonalizable elements.   
		
		\begin{customthm}{\ref{thm:invmsr}$'$}\label{thm:potatosalad}\label{thm:invmsr'}
			Let $G$ be a simple Lie group and let  $\Gamma\subset G$ be any lattice.  Let  $\alpha \colon \Gamma\to \diff^{1+\beta}(M)$ be an action and let $M^\alpha$ denote the suspension space with induced $G$-action.  
			Let $\mu$ be an ergodic, $A$-invariant Borel probability measure on $M^\alpha$ whose projection to $G/\Gamma$ is the Haar measure.  
			
			Then, if $\dim(M)<\rank(G)$ then the measure $\mu$ is $G$-invariant.  Moreover, if $\alpha$ preserves a volume form $\vol$ and if   $\dim(M)\le \rank(G)$ then the measure $\mu$ is $G$-invariant.
		\end{customthm}
	\end{remark}
	
	\begin{remark}
		In fact, Theorem \ref{thm:invmsr} and \ref{thm:potatosalad} hold for actions by $C^1$-diffeo\-morphisms.  This can be shown by the {\it invariance principle}\index{invariance principle!Avila--Viana} of Avila and Viana \cite{MR2651382} (see  the discussion in \cref{rem:invariance_principle}).  We present below a proof that uses (mildly) the $C^{1+\beta}$ hypotheses as this motivates the proof of Proposition \ref{prop:nonresinv} (which allows us to establish an analogue of Theorem  \ref{thm:potatosalad} for manifolds of higher critical dimension) in the next section which requires the higher regularity of the action.  
	\end{remark}
	
	We proceed with the proof of  \cref{thm:invmsr} which is adapted  from  \cite{Cantat}.  This argument is somewhat simpler than the argument in \cites{1608.04995,AWBFRHZW-latticemeasure} (though is special for the case $\Sl(n,\R)$).  The main simplification was observed by  S. Hurtado.  This argument is exploited in \cite{BDZ} to obtain results for actions by $C^1$ diffeomorphisms.  
	\begin{proof}[Proof of Theorem \ref{thm:invmsr}]
		Let $\mu$ be an ergodic, $A$-invariant Borel probability measure on $M^\alpha$ whose projection to $\Sl(n,\R)/\Gamma$ is the Haar measure.  
		
		Recall that $A\simeq \R^{n-1}$.  In the non-volume-preserving case, since $\dim(M)\le n-2$ there are at most $n-2$ fiberwise Lyapunov exponents.  In particular, the intersection of the kernels of the fiberwise Lyapunov exponents is a subspace of $A$ whose dimension is at least $1$.  
		In the volume-preserving case, there are at most $(n-1)$ fiberwise Lyapunov exponents; however, these satisfy the linear relation  they necessarily sum to zero since the cocycle is cohomologous to an $\Sl^{\pm}(n-1, \R)$-valued cocycle (recall \cref{cor:stupidcor}) whence for every $g\in G$, $$0 = \int \log |\det (\restrict{Dg} F)| \ d \mu= \sum \lambda_{i,\mu}^F.$$  Thus, if $\dim(M)\le n-1$ and if $\alpha$ is a volume-preserving action, then the intersection  of the kernels of all fiberwise Lyapunov exponents again has dimension at least 1.  In particular, in either case we may find a 
		nonzero $s_0\in A$ such that \begin{equation} \label{eq:jazz} \lambda^F_{i,\mu}(s_0) = 0 \text{ for every fiberwise Lyapunov exponent $\lambda^F_{i,\mu}$.}\end{equation} 
		
		%
		%
		%

		Recall that  entropy can only decrease under a factor.  Thus 
		$$h_\mu(s_0 ) \ge   h_{\mathrm{Haar}}(s_0) $$
		where $h_{\mathrm{Haar}}(s_0)$ denotes the entropy of translating by $s_0$ on $\Sl(n,\R)/\Gamma$ with respect to the Haar measure.  
		
		Recall we  interpret the roots $\beta$ of $\Sl(n,\R)$ as the (non-zero) Lyapunov exponents for the $A$-action on $\Sl(n,\R)/\Gamma$ with respect to any $A$-invariant measure and hence also as Lyapunov exponents for the $A$-action on the fiber bundle $M^\alpha$ transverse to the fibers and tangential to the local $G$-orbits.     See discussion in \cref{sec:diag}. Let $N_+ \subset G$ be the subgroup generated by all root subgroups $U^\beta$ with  $\beta (s_0)>0$. 
		Similarly, let $N_- \subset G$ be the subgroup generated by all root subgroups $U^\beta$ with $\beta (s_0)<0$.  The orbits of $N_+$ and $N_-$  in $\Sl(n,\R)/\Gamma$ correspond, respectively, to the unstable and stable manifolds for the action of translation by $s_0$ on $G/\Gamma$.  Since $s_0$ is in the kernel of  all fiberwise Lyapunov exponents,  each tangent space $F(x)$ to the fibers of $M^\alpha$ is contained in the neutral Lyapunov subspace $E^c_{s_0}(x)$ for the action of $s_0$ on $(M^\alpha,\mu)$ for almost every $x$. Thus, the orbits of $N_+$ and $N_-$  in $M^\alpha$ also correspond, respectively, to the unstable and stable manifolds for the action of    $s_0$ on $M^\alpha$.
		
		We have that 
		$$h_{\mathrm{Haar}}(s_0) = \sum_{\beta(s_0)>0}\beta(s_0) = h_{\mathrm{Haar}}(s_0\inv) = \sum_{\beta(s_0)<0} (- \beta(s_0)).$$
		In particular,  from the choice of $s_0$, the Margulis--Ruelle inequality (Theorem \ref{entropyfacts}\ref{EFF1}), and the Ledrappier--Young Theorem \eqref{eq:LY} (page \pageref{eq:LY}) 
		$$\sum_{\beta(s_0)>0}\beta(s_0)= h_{\mathrm{Haar}}(s_0)\le h_\mu(s_0 ) = h_\mu(s_0\mid N_+) \le \sum_{\beta(s_0)>0}\beta(s_0).$$
		It follows that $$ h_\mu(s_0\mid N_+) = \sum_{\beta(s_0)>0}\beta(s_0).$$
		By Theorem \ref{thm:led'}, it follows that $\mu$ is $N_+$-invariant.  Similarly we have that $\mu$ is $N_-$-invariant.  
		
		In particular, $\mu$ is invariant by the subgroups $N_-$, $N_+$, and $A$ of $G$.   To end the proof, we  claim the following standard fact: the subgroups  $N_-$ and  $N_+$  generate all of $\Sl(n,\R)$.  It follows from the claim that the measure $\mu$ is $G$-invariant.  
		
		To prove the claim, it is best to work with Lie algebras.  Let $\lien_+,$  $\lien_-$, and $\liea$ be the Lie algebras of $N_-$ and $N_+$, and $A$, respectively.  Let $\lieh$ be the Lie algebra generated by $\lien_+$ and $\lien_-$.  For any $X\in \liea$ we have $$[X, \lieh] = \lieh$$ since $\liea$ normalizes each root space $\lieg^\beta$.  
		For roots $\beta, \hat \beta$ with $\beta(s_0) \neq0$ and  $\hat \beta(s_0) \neq0$ we have $$[\lieg^{\hat \beta}, \lieg^\beta] \subset \lieh$$
		by definition.  
		For roots $\beta, \hat \beta$ with $\beta(s_0) >0$ and $\hat \beta(s_0)= 0$ we have 
		$$[\lieg^{\hat \beta}, \lieg^\beta] = \lieg^{\beta+ \hat \beta} \subset \lieh$$
		since either  $\lieg^{\beta+ \hat \beta}=0$ (if $\beta+ \hat \beta$ is not a root) or  $(\hat \beta+ \beta)(s_0) = \beta(s_0)>0$ (if $\beta+ \hat \beta$ is  a root).  
		Similarly,  for roots $\beta, \hat \beta$ with $\beta(s_0) <0$  and $\hat \beta(s_0) =0$ we have 
		$$[\lieg^{\hat \beta}, \lieg^\beta]  \subset \lieh.$$
		It follows that $\lieh$ is an ideal of the Lie algebra $\lieg = \mathfrak{sl}(n,\R)$ of $\Sl(n,\R)$.  But $\mathfrak{sl}(n,\R)$ is simple (i.e.\ has no nontrivial ideals).  Since $\lieh\neq \{0\}$, it follows that $\lieh= \mathfrak{sl}(n,\R)$ and the claim follows.  
	\end{proof}

	\starsubsection{Advanced invariance principle: nonresonance implies invariance}  
	\label{sec:advancedIP}\index{invariance principle!nonresonant roots}
	Theorem \ref{thm:invmsr} gives the optimal dimension count in Theorem \ref{thm:invmeas} for actions by lattices $\Gamma$ in $\Sl(n,\R)$.  However, for lattices in other simple Lie groups, the critical dimension in Theorem \ref{thm:potatosalad} falls below the critical dimension expected for the analogous versions of   Theorem \ref{thm:invmeas} and Theorem \ref{slnr}.
	For instance, the group $G= \Sp(2n,\R)$, the group of $(2n)\times (2n)$ symplectic matrices over $\R$, has rank $n$.  Theorem \ref{thm:potatosalad}  implies that for any lattice $\Gamma\subset G$ and any compact manifold $M$ with $\dim(M)\le n-1$, any action $\alpha\colon \Gamma\to \Diff^{1+\beta}(M)$ preserves a Borel probability measure.  However, the main result of \cite{AWBFRHZW-latticemeasure} shows for  a lattice $\Gamma $ in $\Sp(2n,\R)$ that any action $\alpha\colon \Gamma\to \Diff^2(M)$ preserves a Borel probability measure when $\dim(M)\le 2n-2$.  To obtain the optimal critical dimensions, it is necessary to use a more advanced invariance principle  developed in \cite{AWBFRHZW-latticemeasure} and based on 
	key ideas from \cite{AWB-GLY-P3}.

	Recall that    we interpret  roots  $\beta^{i,j}\colon A\to \R$ as the nonzero Lyapunov exponents for the action of $A\simeq \R^{n-1}$ on $\Sl(n,\R)/\Gamma$   (for any $A$-invariant measure on $G/\Gamma$.)  Each root $\beta^{i,j}$ has a corresponding root subgroup $U^{i,j}\subset \Sl(n,\R)$.  Given an ergodic,  $A$-invariant measure $\mu$ on $M^\alpha$ we also have fiberwise Lyapunov exponents $\lambda^F_{1,\mu}, \lambda^F_{2,\mu},\dots,\lambda^F_{p,\mu}\colon A\to \R$  for the restriction of the derivative of the  $A$-action on $(M^\alpha, \mu)$ to the fiberwise tangent bundle $F\subset TM^\alpha$ in $M^\alpha$.  
	Then, the roots $\beta^{i,j}$ and fiberwise Lyapunov exponents $\lambda^F_{i,\mu}$ are linear functions on the common vector space $A\simeq \R^{n-1}$.
	We   say that a root  $\beta^{i,j}$ is \emph{resonant}\index{Lie group!roots!resonant} with a fiberwise Lyapunov exponent $\lambda^F_{i,\mu}$ of $\mu$ if they are positively proportional; that is   $\beta^{i,j}$ is  {resonant}  with $\lambda^F_{i,\mu}$  if there is a $c>0$  with $$\beta^{i,j}= c\lambda^F_{i,\mu}.$$
	Otherwise we say that $\beta^{i,j}$ is \emph{not resonant with} $\lambda^F_{i,\mu}$.
	We say that a root $\beta^{i,j}$ of $G$ is \emph{nonresonant}  \index{Lie group!roots!nonresonant}if it is not resonant with any fiberwise Lyapunov exponent $\lambda^F_{i,\mu}$ for the ergodic, $A$-invariant measure $\mu$.  
	
	The following is the key proposition from \cite{AWBFRHZW-latticemeasure}. 
	\begin{proposition}[{\cite[Proposition 5.1]{AWBFRHZW-latticemeasure}}]\label{prop:nonresinv}
		Suppose  $\mu$ is an  ergodic,  $A$-invariant measure on $M^\alpha$  projecting to the Haar measure on $\Sl(n,\R)/\Gamma$ under the projection $\pi\colon M^\alpha\to \Sl(n,\R)/\Gamma$.  
		
		Then, for every nonresonant root $\beta^{i,j}$, the measure $\mu$ is $U^{i,j}$-invariant.  
	\end{proposition}
	\begin{remark}
		Since each  root $\beta^{i,j}$ is a nonzero functional on $A$, if a fiberwise exponent $\lambda^F_{i,\mu}$ is zero, then  {every} root $\beta^{i,j}$ is not resonant with  $\lambda^F_{i,\mu}$.  Since no roots of $\Sl(n,\R)$ are positively proportional, if there are $p$ fiberwise Lyapunov exponents $\{\lambda^F_{i,\mu}, 1\le i\le p\}$ or, more generally, $p'\le p$ coarse fiberwise Lyapunov exponents $\{\chi^F_{i,\mu}, 1\le i\le p'\}$   then Proposition \ref{prop:nonresinv} implies that $\mu$ is  invariant under all-but-$p'$ root subgroups $U^{i,j}$.  Moreover,    if every fiberwise Lyapunov exponent $\lambda^F_{i,\mu}$ is in general position with respect to every root  $\beta^{i,j}$ then from \cref{prop:nonresinv}, $\mu$ is automatically $G$-invariant.  
		
	\end{remark}
	
	\starsubsection{Coarse-Lyapunov Abramov--Rokhlin Theorem and Proof of  Proposition \ref{prop:nonresinv}}
	The proof of Proposition \ref{prop:nonresinv} follows from a version of the Abramov--Rokhlin theorem  (see equation \eqref{eq:AR}, page \pageref{AR}) for entropies subordinated to coarse-Lyapunov foliations.  We outline these ideas  and the proof of  Proposition \ref{prop:nonresinv} in this section.

	Each root $\beta^{i,j}$ of $\Sl(n,\R)$  is  a Lyapunov exponent for the $A$-action on $(M^\alpha, \mu)$ (corresponding to vectors tangent to $U^{i,j}$ orbits in $M^\alpha$.)  Let $\chi^{i,j}$ denote the {coarse Lyapunov exponent} for the $A$-action on $(M^\alpha, \mu)$ containing $\beta^{i,j}$; that is, $\chi^{i,j}$  is the equivalence class of all Lyapunov exponents for the $A$-action on $(M^\alpha, \mu)$ that are positively proportional to $\beta^{i,j}$.
	Let  $\{\lambda^F_{i,\mu}, 1\le i\le p\}$ denote the collection of  fiberwise Lyapunov exponents.  
	We have that   $$ \text{$\chi^{i,j} = \{ \beta^{i,j}\}$ if $\beta^{i,j}$  is not resonant with any  $\lambda^F_{i,\mu}$}.$$ 
	Otherwise, $\chi^{i,j} $ contains   $ \beta^{i,j}$ and all fiberwise Lyapunov exponents $\lambda^F_{i,\mu} \colon A\to \R$ that are positively proportional to $\beta^{i,j}$.  
	
	For $\mu$-\ae $x\in M^\alpha$ there is a {coarse Lyapunov manifold} $W^{\chi^{i,j}}(x)$ through $x$ (see \cref{sss:clm}).  If $\chi^{i,j} = \{ \beta^{i,j}\}$ then for $x\in M^\alpha$, $W^{\chi^{i,j}}(x)$ is simply the $U^{i,j}$-orbit of $x$.  
	Otherwise,   $W^{\chi^{i,j}}(x)$ is a higher-dimensional manifold which intersects the fibers of $M^\alpha$ nontrivially.  
	The partition of $(M^\alpha, \mu)$ into $W^{\chi^{i,j}}$-manifolds forms an $A$-invariant  partition $\Fol^{\chi^{i,j}}$ with $C^{1+\beta}$-leaves.  
	
	If $\beta^{i,j}$ is resonant with some fiberwise Lyapunov exponent, let  $\chi^{i,j, F}$ denote the corresponding  {coarse fiberwise Lyapunov exponent}; that is, $\chi^{i,j, F}$ is the equivalence class of  fiberwise Lyapunov exponents that are positively proportional to $\beta^{i,j}$.  If $\beta^{i,j}$ is not resonant with any fiberwise Lyapunov exponent, let $\chi^{i,j, F}$  denote the zero functional.  If  $\chi^{i,j, F}$  is nonzero,  for $\mu$-\ae $x\in M^\alpha$ there is  a  \emph{coarse fiberwise Lyapunov manifold}\index{coarse Lyapunov!manifold!fiberwise} $W^{\chi^{i,j,F}}(x)$ through $x$.
	(To construct fiberwise coarse Lyapunov manifolds $W^{\chi^{i,j,F}}(x)$, recall that the fibers of $M^\alpha$ are permuted by the dynamics of $A$; all constructions in \cref{sec:koko} may be carried out fiberwise in the setting of a skew-product of diffeomorphisms over a measurable base  if the $C^{1+\beta}$ norms of the fibers are uniformly bounded.) 
	If  $\chi^{i,j, F}$  is zero, we simply define $W^{\chi^{i,j,F}}(x) = \{x\}$.  
	We have that $W^{\chi^{i,j,F}}(x)$ is contained in the fiber through $x$ and that $W^{\chi^{i,j}}(x)$  is the $U^{i,j}$-orbit of  $W^{\chi^{i,j,F}}(x)$.

	For each $\chi^{i,j}$ and  $a\in A$ with $\beta^{i,j}(a)>0$  we    define a   {conditional entropy} of $a$ conditioned on  $\chi^{i,j}$-manifolds, denoted by $h_\mu(a\mid \chi^{i,j} )$   as in Section \ref{sec:entprod}.
	Similarly,   we  can define a   {conditional entropy} of $a$ conditioned on the fiberwise coarse Lyapunov manifolds associated to  $\chi^{i,j,F}$, denoted by $h_\mu(a\mid \chi^{i,j,F})$.  
	In this setting, we have the following ``coarse-Lyapunov Abramov--Rokhlin formula.''  
	\begin{theorem}\label{thm:ARsuspension}\index{theorem!Abramov--Rokhlin}
		Let $\mu$ be an ergodic, $A$-invariant measure on $M^\alpha$ that projects to the Haar measure on $\Sl(n,\R)/\Gamma$.  For any $a\in A$ with $\beta^{i,j}(a)>0$,   
		\begin{equation}\label{eq:CLAR}h_\mu(a\mid \chi^{i,j})=  h_{\mathrm{Haar}}(a\mid \beta^{i,j}) +  h_{ {\mu}}(a\mid \chi^{i,j,F}).\end{equation}
	\end{theorem}
	Above, $$ h_{\mathrm{Haar}}(a\mid \beta^{i,j})$$ denotes the  conditional entropy of  translation by $a$ in $\Sl(n,\R)/\Gamma$ conditioned along $U^{i,j}$-orbits in $\Sl(n,\R)/\Gamma$.  

	\begin{proof}[Proof of Theorem \ref{thm:ARsuspension}] 
		We first show the upper bound 
		\begin{equation}\label{eq:moorepedophilesplease}h_\mu(a\mid \chi^{i,j})\le  h_{\mathrm{Haar}}(a\mid \beta^{i,j}) +  h_{ {\mu}}(a\mid \chi^{i,j,F}).\end{equation}
		This is a standard estimate  in abstract ergodic theory whose proof  we include  for completeness.

		Fix $a\in A$ with $\beta^{i,j}(a)>0$.  Let $\hat \eta$ be an increasing measurable partition of $G/\Gamma$ subordinate to the partition into $U^{i,j}$-orbits.  Let $\pi\colon M^\alpha\to G/\Gamma$ be the natural projection and let $\eta= \pi\inv \hat \eta$.  Let $\xi \succ \eta$ be an increasing measurable partition of $(M^\alpha,\mu)$ subordinate to the partition into $W^{\chi^{i,j}}$-manifolds.  Let 
		$\zeta$ be the partition of $(M^\alpha,\mu)$ into the level sets of $\pi\colon M^\alpha\to G/\Gamma$; that is, $\zeta$ is the partition of $M^\alpha$ into fibers of the fibration $\pi\colon M^\alpha \to G/\Gamma$.  Let $\xi^F:= \xi \vee\zeta$ be the join of $\xi$ and $\zeta$.  
		The partitions $\hat\eta$, $\xi$, and $\xi^F$ satisfy
		\begin{enumerate}
			\item  $ h_\mu(a, \eta) = h_{\mathrm{Haar}}(a, \hat\eta) = h_{\mathrm{Haar}}(a\mid \beta^{i,j})$,     
			\item  $h_\mu(a, \xi)= h_\mu(a\mid \chi^{i,j}),$ and 
			\item    $h_\mu(a, \xi^F)= h_\mu(a\mid \chi^{i,j,F}).$
		\end{enumerate}

		We have the following    computation (see for example \cite[Lemma 6.1]{MR2729332}):
		\begin{align*}h_\mu(a \mid  \chi^{i,j}) &:= h_\mu(a, \xi )\\
			&=h_\mu(a, \eta \vee \xi) \\& \le  h_\mu(a,   \eta) + h_\mu\left (a,  \xi\vee \bigvee_{n\in \Z}a^n(\eta)\right) \\
			&= h_{\mathrm{Haar}}(a ,\hat\eta) + h_\mu(a,  \xi \vee \zeta ) \\
			& = h_{\mathrm{Haar}}(a\mid \beta^{i,j}) +  h_\mu(a\mid \chi^{i,j,F})\end{align*}
		and \eqref{eq:moorepedophilesplease} follows.  
		
		On the other hand, summing over all roots $\beta$ with $\beta(a)>0$ we have from the classical Abramov--Rokhlin theorem \eqref{eq:AR}, the product structure of entropy in Theorem \ref{thm:entprod}, and an analogous version of Theorem \ref{thm:entprod} for the fiberwise entropy $h_\mu(a\mid \zeta)$ appearing in \eqref{eq:AR} that
		\begin{align*}h_\mu(a ) &=
			\sum_{\chi(a)>0}
			h_\mu(a \mid \chi)  \\
			&=  \sum_{\beta^{i,j}(a)>0} 
			h_\mu(a \mid \chi^{i,j})
			+ \sum_{\substack{\text{$\chi^F$ nonres.} \\ \chi^F(a)>0} } h_\mu(a\mid  \chi^F)
			\\
			&\le \sum_{\beta^{i,j}(a)>0}\left( h_{\mathrm{Haar}}(a\mid  \beta^{i,j}) +h_\mu(a\mid \chi^{i,j,F})\right) +   \sum_{\substack{\text{$\chi^F$ nonres.} \\ \chi^F(a)>0} } h_\mu(a\mid  \chi^F)\\
			&= \sum_{\beta^{i,j}(a)>0} h_{\mathrm{Haar}}(a\mid  \beta^{i,j})   + \sum_{ \chi^F(a)>0 } h_\mu(  a\mid \chi^F)\\
			& =  h_{\mathrm{Haar}}(a) + h_\mu(a\mid \zeta)\\
			&= h_\mu(a).
		\end{align*}
		In the second and third lines, the second sum is over all fiberwise coarse Lyapunov exponents that are not resonant with any root $\beta$ of $G$. 
		Since entropies are non-negative quantities, it follows that 
		$$h_\mu(a \mid \chi^{i,j}) =  h_{\mathrm{Haar}}(a\mid  \beta^{i,j}) + h_\mu(a\mid \chi^{i,j,F})  $$
		for all $ \beta^{i,j}$ with  $ \beta^{i,j}(a)>0$.  
	\end{proof}
	\begin{remark}
		A more general version of Theorem \ref{thm:ARsuspension} appears in \cite[Theorem 13.6]{AWB-GLY-P3} where the factor map $\pi$ is allowed to be measurable and the measure $\pi_*(\mu)$ on the factor system is an arbitrary ergodic, $A$-invariant measure.
	\end{remark}
	
	The proof of Proposition \ref{prop:nonresinv} is a straightforward consequence of Theorem \ref{thm:ARsuspension}.
	\begin{proof}[Proof of Proposition \ref{prop:nonresinv}]
		Given a root $\beta^{i,j}$ and $a\in A$ such that $\beta^{i,j}(a)>0$ we have  defined the conditional entropy $h_\mu(a\mid \beta^{i,j})$ for the translation by $a$ conditioned on $U^{i,j}$-orbits in $M^\alpha$.  
		From an appropriate version of the Margulis--Ruelle inequality (see Theorem \ref{entropyfacts}\ref{EFF1} and \eqref{eq:popopopo}), for $a\in A$ with $\beta^{i,j}(a)>0$ we have that
		\begin{equation}h_\mu(a\mid \beta^{i,j})\le \beta^{i,j}(a).\label{eqll}\end{equation}
		On the other hand, if $\beta^{i,j}$ is  nonresonant then $\chi^{i,j,F}$ is the zero functional whence the coarse Lyapunov manifold  $W^{\chi^{i,j}}(x)$ associated to $\chi^{i,j}$ is simply the $U^{i,j}$-orbit of $x$ for every $x\in M^\alpha$ and the term $h_{\mathrm{\mu}}(a\mid \chi^{i,j,F})$ in \eqref{eq:CLAR} of Theorem \ref{thm:ARsuspension} vanishes.  
		Hence, by Theorem \ref{thm:ARsuspension},
		\begin{equation}h_\mu(a\mid \beta^{i,j}) = 
			h_\mu(a\mid \chi^{i,j}) = h_{\mathrm{Haar}}(a\mid \beta^{i,j}) + 0 = \beta^{i,j}(a).\label{pp}\end{equation}
		From \eqref{eqll} and \eqref{pp}, we have that the conditional entropy $h_\mu(a\mid \beta^{i,j})$ attains its maximal possible value.  
		In particular, from the invariance principle in Theorem \ref{thm:led'}\ref{3333333333}, it follows that  $\mu$ is $U^{i,j}$-invariant.  
	\end{proof}

	\starsubsection{Proof of Theorem \ref{thm:invmeas} using the advanced invariance principle}\label{ss:proveinv}
	We outline another proof of Theorem \ref{thm:invmeas} based on  Proposition \ref{prop:nonresinv}.  
	This more closely mimics the arguments in \cite{1608.04995}.

	\begin{proof}[Proof  of Theorem \ref{thm:invmeas} using \cref{prop:nonresinv}]  
		From Claim \ref{claim:oolong},  it is sufficient to construct a $G$-invariant probability measure on $M^\alpha$.  
		Note that $A\simeq \R^{n-1}$ is abelian (and in particular amenable, see \cref{rem:folner}) and that the space of probability measures on $M^\alpha$  projecting to the Haar measure on $\Sl(n,\R)/\Gamma$ is nonempty,  $A$-invariant, and {weak-$*$} compact.  The Krylov-Bogolyubov theorem thus  gives an  $A$-invariant probability measure $\mu$ on $M^\alpha$  projecting to the Haar measure on $\Sl(n,\R)/\Gamma$.  Moreover, since the Haar measure on $\Sl(n,\R)/\Gamma$ is $A$-ergodic, we may assume $\mu$ is $A$-ergodic.

		
		Let $\dim (M) = d\le n-2$.   The  fiberwise tangent bundle $F$  of $M^\alpha$ is $d$-dimensional and therefore there are  at most $d$ fiberwise Lyapunov exponents $$\lambda^F_{1,\mu},\cdots,\lambda^F_{k, \mu},\quad  k\le d.$$  As no pair of roots of $\Sl(n,\R)$ is positively proportional, there are at most $d$ roots that are {resonant} with the  fiberwise Lyapunov exponent $\lambda^F_{j,\mu}$.  All other roots $\beta^{i,j}$ are nonresonant.   By Proposition \ref{prop:nonresinv}, if $\beta^{i,j}$ is not resonant with any  $\lambda^F_{j, \mu}$, then $\mu$ is $U^{i,j}$-invariant.   
		
		Let $H\subset \Sl(n,\R)$ be the subgroup that preserves $\mu$. We claim $H=G$ completing the proof.  
		As $d\le n-2$, $\mu$ is invariant under $A$ and all-but-at-most-$(n-2)$ root subgroups $U^{i,j}$.    Then $H$ has codimension at most $(n-2)$.  
		From \cite[Lemma 2.5]{1608.04995}, we have that  $H$ is parabolic; that is, $H$ is conjugate to a group of block-upper-triangular matrices (see \cref{rem:parabolic}). However, the proper closed parabolic subgroups  
		of $\Sl(n,\R)$ of maximal codimension are   conjugate to the     codimension  $(n-1)$ subgroup \begin{equation}\label{libel}\left \{ \left(\begin{array}{cccc}* & * & \cdots & * \\0 & * & \cdots & *  \\ \vdots & \vdots  & \ddots & \vdots \\0 & * & \cdots & *\end{array}\right)\right\}.\end{equation}    (See Section VII.7, especially Proposition 7.76 of \cite{MR1920389} for discussion on the structure of parabolic subgroups.)  
		As  $H$ has codimension at most $n-2$, it thus follows that $H=G$ as there are no proper parabolic subgroups of $G$  with codimension less than $(n-1)$. \end{proof}

	%


	\begin{remark}  
		The above proof has the advantage that it generalizes to give invariance of measures in the optimal critical dimension for actions by lattices in other Lie groups including $\Sp(2n,\R), $ $\SO(n,n), $ or $\So(n,n+1)$ on manifolds of the optimal dimension.  
		As discussed in \cref{sec:advancedIP} for a lattice $\Gamma$ in a group such as  $G=\Sp(2n, \R)$, the proof in \cref{ss:IP1} yields that any $C^{1+\beta}$ action of $\Gamma$ on a manifold of dimension at most $ \mathrm{rank}(G)-1$, any $A$-invariant measure on $M^\alpha$ that projects to Haar on $G/\Gamma$ is $G$-invariant.  However, the above proof establishes this result  for  manifolds $M$ where the critical  dimension is  $r(G)$, the number in the last column of \cref{tab:stupid}  (page \pageref{page:table}) defined in \cites{AWBFRHZW-latticemeasure,1608.04995} (see also   \cref{foot:r}.)  For $\R$-split groups $G$ we have  $r(G) = d_0(G)$.  
		In particular,  the above proof can be adapted to show  the following: 
		\begin{theorem} \label{thm:lplpl}Let $G$ be  a  higher-rank simple Lie group $G$ with finite center, let  $\Gamma$  be a lattice in $G$,  let $M$ be a closed manifold, and let $\alpha\colon \Gamma\to \diff^{1+\beta}(M)$ be an action.   Then
			\begin{enumerate}
				\item if  $\dim(M)\le r(G)-1$, every $A$-invariant probability measure on $M^\alpha$ that projects to the Haar measure  on $G/\Gamma$ is $G$-invariant;
				\item if  $\dim(M)\le r(G)$ and $\alpha$ is volume-preserving, every $A$-invariant probability measure on $M^\alpha$ that projects to the Haar measure  on $G/\Gamma$ is $G$-invariant. 
			\end{enumerate}
			In particular, if  $\dim(M)\le r(G)-1$, every action $\alpha\colon \Gamma\to \diff^{1+\beta}(M)$ preserves a Borel probability measure.
		\end{theorem}
		
		
	\end{remark}
	
	\section{Proof outline of Theorem \ref{slnr}} 
	\label{sec:outline3steps}
	We outline the proof of Theorem \ref{slnr}  for the case of $C^\infty$ actions of cocompact lattices in $\Sl(n,\R)$.  
	That is, for $n\ge 3$, we consider a  cocompact lattice    $\Gamma$ in $\Sl(n,\R)$ and show that every homomorphism $\alpha\colon \Gamma\to \diff^\infty(M)$ has finite image when 
	\begin{enumerate}
		\item $M$ is a compact manifold of dimension at most $(n-2)$, or 
		\item  $M$ is a compact manifold of dimension at most $(n-1)$ and $\alpha$ preserves a volume form $\vol$.
	\end{enumerate}

	The broad outline of the proof consists of 3 steps.

	\subsection{Step 1: Subexponential growth} In the case that $\Gamma\subset \Sl(n,\R)$ is cocompact, using its action on $\Sl(n,\R)$ and that $\Sl(n,\R)$ is a proper length space one may show that $\Gamma$ is  finitely generated (see for example  \cite[Theorem 8.2]{MR2850125}).  More generally, it is a  classical fact that all lattices $\Gamma$ in semisimple Lie groups are finitely generated.  
	
	Fix a finite  symmetric generating set $S$ for $\Gamma$.  Given $\gamma\in \Gamma$, let $|\gamma|= |\gamma|_S$ denote the word-length of $\gamma$ relative to this generating set; that is, $$|\gamma|= \min\{ k: \gamma = s_k \cdots s_1, s_i\in S\}.$$
	Note that if we replace the finite   generating set $S$ with another finite generating set $S'$, there is a uniform  constant $C$ such that  the word-lengths are uniformly distorted: $$|\gamma|_{S'}\le C | \gamma|_{S}.$$ 
	Thus all definitions below will be independent of the choice of $S$.

	Equip $TM$ with a Riemannian metric and corresponding norm.  
	\begin{definition}\label{def:USEGOD}
		We say that an action $\alpha\colon \Gamma\to \Diff^1(M)$ has \emph{uniform subexponential growth of derivatives}\index{uniform subexponential growth of derivatives} if  for every $\epsilon>0$ there is a $C= C_\epsilon$ such that for every $\gamma\in \Gamma$,
		$$\sup_{x\in M} \|D_x \alpha(\gamma)\|\le Ce^{\epsilon |\gamma|}.$$
	\end{definition}
	
	Note that if $\alpha\colon \Gamma\to \Diff^1(M)$ has {uniform subexponential growth of derivatives} it follows for every $\epsilon>0$ that there is a $C= C_\epsilon$ such that for every $\gamma\in \Gamma$, 
	$$\sup_{x\in M} \|D_x \alpha(\gamma)\|\ge Ce^{-\epsilon |\gamma|}.$$
	
	The following is the main result  of  \cite{1608.04995} in the case of cocompact lattices in $\Sl(n,\R)$.  
	\begin{theorem}[{\cite[Theorem 2.8]{1608.04995}}]\label{thm:USEGOD}
		For $n\ge 3$, let $\Gamma\subset \Sl(n,\R)$ be a cocompact lattice.  Let $\alpha\colon \Gamma\to \Diff^2(M)$ be an action.    Suppose that either
		\begin{enumcount}
			\item \label{USEGOD1} $\dim(M)\le n-2$, or
			\item \label{USEGOD2} $\dim(M)= n-1$ and $\alpha$ preserves a smooth volume.
		\end{enumcount}
		Then $\alpha$ has uniform subexponential growth of derivatives.
	\end{theorem}
	\begin{remark}
		The proof of Theorem \ref{thm:USEGOD} is the only place in the proof of Theorem \ref{slnr} where cocompactness of $\Gamma$ is used.  It is not required for Steps 2 or 3 below.    For $\Gamma = \Sl(m,\Z)$, the analogue of Theorem \ref{thm:USEGOD} is established in \cite{1710.02735} and has been announced for general lattices in \cite{BFHWM}.
	\end{remark}
	
	\subsection{Step 2: Strong property (T) and averaging Riemannian metrics} 
	Assume $\alpha\colon \Gamma\to \Diff^\infty(M)$ is an action by $C^\infty$ diffeomorphisms.\footnote{For   $C^2$ actions, one replaces the Hilbert Sobolev  spaces $W^{2,k}(S^2(T^*M)))$ below with appropriate Banach Sobolev spaces $W^{p,1}(S^2(T^*M)))$ and verifies such spaces are of the type $\mathcal {E}_{10}$ considered in \cite{MR3407190}.}  
	The action  $\alpha$ of $\Gamma$ on $M$ induces an action $\alpha_\#$ of $\Gamma$ on tensor powers of the cotangent bundle of $M$ by pull-back:  
	Given $\omega\in (T^*M)^{\otimes k}$ write $$\alpha_\#(\gamma)\omega= \alpha(\gamma\inv)^*\omega;$$ that is, 
	if $v_1,\dots v_k\in T_xM$ then $$\alpha_\#(\gamma)\omega(x)(v_1, \dots, v_k) = \omega(x)(D_x \alpha(\gamma\inv) v_1, \dots, D_x \alpha(\gamma\inv)  v_k).$$
	In particular, we obtain  an action of $\Gamma$ on the set of   Riemannian metrics which naturally sits as a half-cone inside $S^2(T^*M)$,  the vector space of all symmetric 2-forms on $M$. 
	\def\H{\mathcal H}
	Note that $\alpha_\#$ preserves $C^\ell(S^2(T^*(M)))$, the subspace of all $C^\ell$ sections of $S^2(T^*M)$ for any $\ell \in \N$.

	Fix a volume form $\vol$ on $M$.  The norm on $TM$ induced by the  background Riemannian metric induces a norm on each fiber of $S^2(T^*M)$.  We then obtain a natural notion of  measurable and integrable sections of $S^2(T^*M)$ with respect to $\vol$.  Let $\H^k = W^{2,k}(S^2(T^*M))$  be the Sobolev space of symmetric 2-forms whose weak derivatives  of order $\ell$ are bounded with respect to the $L^2(\vol)$-norm for $0\le \ell\le k$.   Then $\H^k$ is a Hilbert space.  Let $\| \cdot\| _{\H^k}$ denote the corresponding Sobolev norm on $\H^k$ as well as the induced operator norm on the space $B(\H^k) $ of bounded operators on $\H^k$.   Working in local coordinates, the Sobolev embedding theorem implies   that $$\H^k\subset C^\ell(S^2(T^*(M)))$$ as long as $$\ell< k- \dim(M)/2.$$  
	In particular, for $k$ sufficiently large, an element $\omega$ of $\H^k$   is a $C^\ell$ section of $S^2(T^*M)$  which will be a $C^\ell$
	Riemannian metric on $M$ if it is positive definite.  
	
	The   action $\alpha_\#$ is a representation of $\Gamma$ by bounded operators on $\H^k$.  
	From Theorem \ref{thm:USEGOD}, we obtain strong control on the norm growth of the induced representation $\alpha_\#$.  In particular, we obtain that the representation $\alpha_\#\colon \Gamma\to B(\H^k)$ has  \emph{subexponential norm growth:}
	\begin{lemma} \label{lem:normgrowth} Let $\alpha\colon \Gamma\to \diff^{\infty}(M)$  have uniform subexponential growth of derivatives.  Then, 
		for all $\epsilon'>0$ there is $C>0$ such that $$\|\alpha_\#(\gamma)\|_{\H^k} \le Ce^{\epsilon' |\gamma|}$$ for all $\gamma\in \Gamma$.
	\end{lemma} 
	The proof of Lemma \ref{lem:normgrowth} follows   from the chain rule, Leibniz rule,  and    computations  that bound  the growth of higher-order derivatives by  polynomial functions in the growth of the first derivative.   See \cite[Lemma 6.4]{MR2198325} and discussion in \cite[Section 6.3]{1608.04995}.

\index{strong property (T)} 
	We use the main result from \cites{MR3407190,MR2423763}: cocompact lattices $\Gamma$ in higher-rank simple Lie groups (such as $\Sl(n,\R)$ for $n\ge 3$) satisfy Lafforgue's \emph{strong Banach property (T)} first introduced in \cite{MR2423763}.  The result for  $\Sl(n,\R)$ and its cocompact lattices (as well as other higher-rank simple Lie groups containing a subgroup isogenous to $\Sl(3,\R)$) is established by Lafforgue in Corollary  4.1 and Proposition 4.3 of \cite{MR2423763}; for cocompact lattices in certain other higher-rank Lie groups (containing a subgroup isogenous to $\Sp(4,\R)$), the results of \cite{MR3407190} are needed.  
	See also \cite{delaSallenonuniform} for the case of nonuniform lattices.   Strong Banach property (T) considers representations $\pi$  of $\Gamma$ by bounded operators on certain Banach spaces $E$ (of type $\mathcal E_{10}$).  If such representations have sufficiently slow exponential norm growth, then there exists a  sequence of  averaging operators $p_n$ converging to a projection $p_\infty$ such that for any vector $v\in E$, the limit  $p_\infty( v)$ is $\pi$-invariant.  
	In the case that $E$ is a Hilbert space (which we may assume when $\alpha$ is an action by $C^\infty$ diffeomorphisms)   
	we have the following formulation.  Note that  Lemma \ref{lem:normgrowth} (which follows from   Theorem \ref{thm:USEGOD}) ensures our representation $\alpha_\#$ satisfies the hypotheses of the theorem.  
	\begin{theorem}[\cites{MR3407190, delaSallenonuniform,MR2423763}]\label{lafforgue1} Let $\mathcal{H}$ be a Hilbert space and for $n\ge 3$, let $\Gamma$ be a  lattice in $\Sl(n,\R)$.  
		
		There exists $\epsilon>0$ such that for any representation $\pi\colon \Gamma \to B(\mathcal{H})$, if there exists $C_\epsilon>0$ such that
		\begin{align*}
			\|\pi(\gamma)\| \leq C_{\epsilon}e^{\epsilon |\gamma|}
		\end{align*} for all $\gamma\in \Gamma$
		then there exists a sequence of averaging operators $p_n = \sum w_i \pi(\gamma_i)$ in $B(\mathcal{H})$---where $w_i\ge 0$, $\sum w_i= 1$, and $w_i =0$ for every $\gamma_i \in \Gamma$ of word-length larger than  $n$---such that for any vector $v\in \mathcal{H}$, the sequence $v_n = p_n(v) \in \mathcal{H}$ converges to an invariant vector $v^{*} = p_\infty(v)$. 
		
		Moreover the convergence is exponentially fast: there exist  $0< \lambda <1$   and $C=C_\lambda$ such  that  $\|v_n-v_*\| \leq C \lambda^n \|v\|$.  
	\end{theorem}
	Theorem \ref{lafforgue1} as stated in \cites{MR3407190,MR2423763} requires that $\Gamma$ be cocompact.  The extension to nonuniform lattices is announced in \cite{delaSallenonuniform}.  The exponential convergence in Theorem \ref{lafforgue1} is often not explicitly stated  in the definition of strong property (T) or  in statements of theorems establishing that the property holds for lattices in higher-rank simple Lie groups; however, the exponential convergence   follows from the proofs.
	
	We complete Step 2 with the following computation.
	\begin{proposition}
		For $n\ge 3$, let $\Gamma \subset \Sl(n,\R)$ be a lattice and let $\alpha\colon \Gamma\to \Diff^\infty(M)$ be an action with uniform subexponential growth of derivatives.  Then for any $\ell$, there is a $C^\ell$ Riemannian metric $g$ on $M$ such that $$\alpha(\Gamma)\subset \mathrm{Isom}_g(M).$$
	\end{proposition}
	
	\begin{proof}
		Consider an arbitrary  $C^\infty$ Riemannian metric $g$.  For any $k$, we have $g\in \H^k$.  We apply \cref{lafforgue1} and its notation to the representation $\alpha_\#\colon \Gamma\to B(\H^k)$ with the $g$ the initial vector $v$. As averages of finitely many Riemannian metrics are still Riemannian metrics we have that $g_n := p_n(g)$ is positive definite for every $n$.  In particular, the limit $g_\infty = p_\infty(g)$ is in the closed cone of positive (possibly indefinite) symmetric 2-tensors in $\H^k$. Having taken $k$ sufficiently large we have that $g_\infty$ is $C^\ell$; in particular, $g_\infty$ is continuous, everywhere defined, and positive everywhere.    We  need only confirm that  $g_\infty$ is non-degenerate, i.e.\ is positive {definite} on $T_xM$ for every $x\in M$.

		Given any $x\in M$ and   unit vector $\xi\in T_xM$, for any $\epsilon>0$ we have from Definition \ref{def:USEGOD} that there is a  $C_\epsilon>0$ such that 
		\begin{align*}
			p_n(g)(\xi,\xi)&=  \left(\sum w_i \alpha_\#(\gamma_i) g\right)(\xi,\xi)\\
			&=  \sum w_i  g (D_x\alpha(\gamma_i\inv) \xi,D_x\alpha(\gamma_i\inv)  \xi)\\
			&\ge \frac{1}{C_\epsilon^2 } e^{-2 \epsilon n}
		\end{align*}
		where we use that $w_i > 0$ only when  $\gamma_i$  has  word-length at most $n$.  
		
		On the other hand, from the exponential convergence in Theorem \ref{lafforgue1} we have 
		\begin{align*}
			|p_n(g)(\xi,\xi)- p_\infty(g)(\xi,\xi)| &\le C_\lambda \lambda^n.
		\end{align*}
		Thus $$p_\infty(g)(\xi,\xi) \ge \frac{1}{C_\epsilon^2 } e^{-2 \epsilon n}- C_\lambda \lambda^{n}$$
		for all $n\ge 0$.  Taking $\epsilon>0$ sufficiently small we can ensure that $$ {C_\epsilon^2 } e^{2 \epsilon n}< \frac{1}{C_\lambda }\lambda^{-n}$$
		for all sufficiently large $n$.  Then, for all sufficiently large $n$ we have  $$\frac{1}{C_\epsilon^2 } e^{-2 \epsilon n}>  C_\lambda \lambda^n$$ and thus $p_\infty(g)(\xi,\xi)>0.$
	\end{proof}

	\subsection{Step 3: Margulis superrigidity with compact codomain}\label{ssec:MSRcompact} 
	From Steps 1 and 2 we have that any action $\alpha\colon \Gamma\to \diff^\infty(M)$ as in  Theorem \ref{slnr} preserves a $C^\ell$ Riemannian metric $g$. 
	In the general case of $C^2$-actions (or even $C^{1+\beta}$-actions), we have that any action  $\alpha\colon \Gamma\to \diff^2(M)$ preserves a continuous Riemannian metric $g$.  See   \cite[Theorem 2.7]{1608.04995}.
	We thus have $$\alpha\colon \Gamma\to \mathrm{Isom}_g(M).$$

	The group $\mathrm{Isom}_g(M) $ is compact.  When $g$ is at least $C^1$, it follows from the classical result of Myers and Steenrod \cite{MR1503467} that $\mathrm{Isom}_g(M) $ is a Lie group; for continuous $g$ we use the solution to the Hilbert--Smith conjecture for actions by bi-Lipschitz maps \cite{MR1464908} to conclude that $\mathrm{Isom}_g(M) $ is a Lie group.  
	Let $\dim(M) = m$.   Then	
	\begin{equation}\label{eq4545}\dim(\mathrm{Isom}_g(M))\le \dfrac{m(m+1)}{2}.\end{equation}
	Indeed, the orbit of any point $p\in M$ under $\mathrm{Isom}_g(M)$ has dimension at most $m$ and the dimension of the stabilizer of a point is at most $\frac{m(m-1)}{2}$, the dimension of $\SO(m)$; thus $$\dim(\mathrm{Isom}_g(M)) \le m+ \frac{m(m-1)}{2}.$$

	With  $K=\mathrm{Isom}_g(M)$ we thus obtain  a  compact-valued representation $\alpha\colon \Gamma\to K$. By equation 
	\eqref{eq4545}, if $m< \frac{1}{2} \sqrt{8 n^2-7}-\frac{1}{2}$
	then    $\dim  (\mathfrak{su}(n) )= n^2-1 > \dim (K)$; by conclusion \ref{MSR2} of   Theorem  \ref{thm:MSRcompactcodomain}, $\alpha(\Gamma)$ is  thus contained in a  0-dimensional subgroup of $K$.  
	This holds in particular  if $m\le n-1$. 
	We thus conclude that the image $$\alpha(\Gamma)\subset K = \mathrm{Isom}_g(M)$$ is finite.

	Summarizing the  arguments from Steps 2 and 3, we obtain the following.  
	\begin{theorem}\label{thm:MSRcompact}
		For $n\ge 3$,  let $\Gamma\subset \Sl(n,\R)$ be a  lattice.  Let $\alpha\colon \Gamma\to \Diff^2(M)$ be an action with uniform subexponential growth of derivatives.  
		
		Then, if  $$\dim(M)< \frac{1}{2} \sqrt{8 n^2-7}-\frac{1}{2},$$
		the image $\alpha(\Gamma)$ is finite.  
	\end{theorem}
	
	\section{Proof outline of Theorem \ref{thm:USEGOD}}
	To establish Theorem \ref{slnr}, from the discussion in Section \ref{sec:outline3steps} it is enough to establish
	Theorem \ref{thm:USEGOD}: the action $\alpha$ has  uniform subexponential growth of derivatives.  
	We outline the  proof of Theorem \ref{thm:USEGOD}.  

	\subsection{Setup for proof}
	For $n\ge 3$, let $\Gamma\subset \Sl(n,\R)$ be a cocompact  lattice.  Let $M$ be a compact manifold and let  $\alpha\colon \Gamma\to \Diff^2(M)$ be an action.  Assume either that  $\dim(M)\le n-2$ or that   $\dim(M)\le n-1$ and that $\alpha$ preserves a volume form.  
	We recall the following constructions from the proof of Theorem \ref{thm:invmeas}:
	\begin{enumerate}
		\item The manifold $M^\alpha= (\Sl(n,\R)\times M)/\Gamma$ is the suspension space introduced in Section \ref{ssec:susp}.  $M^\alpha$ is a fiber bundle over $\Sl(n,\R)/\Gamma$ with fibers diffeomorphic to $M$.  Moreover, $M^\alpha$ and $\Sl(n,\R)/\Gamma$ have  natural (left) $\Sl(n,\R)$-actions and the projection $\pi\colon M^\alpha \to \Sl(n,\R)/\Gamma$ intertwines these $G$-actions.  
		\item  $A\subset \Sl(n,\R)$ denotes the subgroup of diagonal matrices with positive entries.  We have $A\simeq \R^{n-1}$ which is a higher-rank, free abelian group if $n\ge 3$.  
		\item Given an ergodic, $A$-invariant Borel probability measure $\mu$ on $M^\alpha$ we have \emph{fiberwise Lyapunov exponents}\index{Lyapunov exponent!fiberwise}.
		$$\lambda^F_{1,\mu}, \dots, \lambda^F_{p,\mu}\colon A\to \R$$
		for the restriction of the  derivative of the $A$-action on $M^\alpha$ to the fibers of $M^\alpha$ introduced in Section \ref{ss:fibLyap}.
		\item $\beta^{i,j}\colon A\to \R$ are the \emph{roots} of $\Sl(n,\R)$ and $U^{i,j}$ are the corresponding \emph{root subgroups} introduced in Section \ref{ss:roots}.
	\end{enumerate}
	\subsection{Two key propositions}
	The proof of Theorem \ref{thm:USEGOD} is by contradiction and follows from the following two propositions.  Our first key proposition is an analogue of  Proposition \ref{prop:ZLEUSEGOD}.  
	\begin{proposition}\label{eq6677}
		Suppose that $\alpha\colon \Gamma\to \Diff^1(M)$ fails to have uniform subexponential growth of derivatives.  Then there exists a Borel probability measure $\mu'$ on $M^\alpha$ such that  
		\begin{enumcount}
			\item $\mu'$ is  $A$-invariant and ergodic;
			\item  there exists a nonzero fiberwise Lyapunov exponent $\lambda^F_{j, \mu'}\colon A\to \R$.  
		\end{enumcount}
	\end{proposition}
	The proof of Proposition \ref{eq6677} is very similar  to the proof of Proposition \ref{prop:ZLEUSEGOD} with some minor modifications and notational differences.   We include an outline of the proof in \cref{Sec:spitproof}; see also  \cite[Section 4]{1608.04995} for complete details.  
	Obtaining  uniform control on the growth of a cocycle from bounds on the top Lyapunov exponent over all invariant measures is a standard technique in dynamical systems; see for instance  \cite{MR1986306,MR1643183,Hurtado_Burnside,MR2776369}.


	The measure $\mu'$ in Proposition \ref{eq6677} projects to an ergodic, $A$-invariant measure on $\Sl(n,\R)/\Gamma$.  If $\mu'$ projected to the Haar measure on $\Sl(n,\R)/\Gamma$ then, from Theorem \ref{thm:invmsr} and the bounds on the dimension $M$, the measure  $\mu'$ would be $G$-invariant and, as explained below, the proof of Theorem \ref{thm:USEGOD} would be complete.  However, there may exist ergodic  $A$-invariant measures on 
	$\Sl(n,\R)/\Gamma$ that are not the Haar measure.\footnote{In fact, for certain lattices $\Gamma$ there exist  ergodic  $A$-invariant measures on  $\Sl(n,\R)/\Gamma$ that have positive entropy for some element of $A$ as shown by M. Rees; see \cite[Section 8]{MR1989231}.}  
	
	By carefully averaging the measure $\mu'$ along root subgroups $U^{i,j}$ and applying Ratner's measure classification theorem \cite{MR1262705} to the projected measure on  $\Sl(n,\R)/\Gamma$ we obtain the following.  
	
	\begin{proposition}\label{prop:goodmeas}
		Let $\alpha\colon \Gamma\to \Diff^{1}(M)$ be an action.  
		Suppose   there exists an ergodic, $A$-invariant measure $\mu'$ on the suspension space $M^\alpha$  with a nonzero   fiberwise Lyapunov exponent $\lambda^F_{j', \mu'}\colon A\to \R$.   Then there exists a Borel probability measure $\mu$ on $M^\alpha$ such that  
		\begin{enumcount}
			\item $\mu$ is  $A$-invariant and ergodic;
			\item  there exists a nonzero fiberwise Lyapunov exponent $\lambda^F_{j, \mu}\colon A\to \R$; %
			\item $\mu$ projects to the Haar measure on  $\Sl(n,\R)/\Gamma$.  
		\end{enumcount}
	\end{proposition}

	\begin{remark}\
		\begin{enumerate}
			\item  Propositions \ref{eq6677} and  \ref{prop:goodmeas} hold in full generality; they do not depend on the comparison between the dimension of $M$ and the rank of $\Sl(n,\R)$.  The constraint on the dimension of  $M$ is used to obtain a contradiction in the  proof of Theorem \ref{thm:USEGOD} by applying Theorem \ref{thm:invmsr} and   Zimmer's cocycle superrigidity to the fiberwise derivative cocycle.  
			\item Propositions \ref{eq6677} and  \ref{prop:goodmeas}  heavily use the fact that $\Gamma$ is cocompact in $\Sl(n,\R)$ so that the manifold $M^\alpha$ is compact.  For instance, if $M^\alpha$ is not compact then  the   proof of Proposition \ref{eq6677} (compare with proof of Proposition \ref{prop:ZLEUSEGOD}) fails as there may be escape of mass into the cusp of $G/\Gamma$.  Thus,  more subtle arguments are required to establish the analogue of \cref{thm:USEGOD}    in the case that $\Gamma$ is nonuniform.  In the case that $\Gamma = \Sl(n,\Z)$, such  arguments appear in \cite{1710.02735}.
			\item Both \cref{eq6677} and \cref{prop:goodmeas}  hold for $C^1$ actions.  The $C^{1+\beta}$ hypotheses is  later used (along with the dimension bounds) to conclude that the $A$-invariant measure $\mu$ obtained in  \cref{prop:goodmeas}  is, in fact, $G$-invariant by applying \cref{thm:invmsr}.  
		\end{enumerate}
	\end{remark}

	\subsection{Proof of Theorem \ref{thm:USEGOD}}\label{sec:pp}
	
	We deduce  Theorem \ref{thm:USEGOD} 
	from Proposition \ref{eq6677}, Proposition  \ref{prop:goodmeas}, Theorem   \ref{thm:invmsr}, and  Theorem \ref{thm:ZCSR}.   
	
	\begin{proof}[Proof of Theorem \ref{thm:USEGOD}]
		Let $\alpha\colon \Gamma\to \Diff^2(M)$ be as in Theorem \ref{thm:USEGOD}.    For the sake of contradiction, assume that $$\alpha\colon \Gamma\to \Diff^2(M)$$ fails to have uniform subexponential growth of derivatives.  
		Let $\mu'$ be the measure guaranteed by Proposition \ref{eq6677}.  We then apply Proposition \ref{prop:goodmeas} to obtain an ergodic, $A$-invariant Borel probability  measure $\mu$ on $M^\alpha$ that projects to the Haar measure on $G/\Gamma$ and has a non-zero fiberwise Lyapunov exponent.  
		In either case considered in   \cref{thm:USEGOD}, it follows from   Theorem \ref{thm:invmsr} that $\mu$ is $G$-invariant. 
		

		Recall that we write $\pi\colon M^\alpha \to \Sl(n,\R)/\Gamma$ for the natural projection and let $F$ be the fiberwise tangent bundle; that is, $F$ is  sub-vector-bundle of $TM^\alpha$ given by $F = \ker D\pi$.  
		As $F$ is $G$-invariant, we may  apply Zimmer's cocycle superrigidity theorem, Theorem \ref{thm:ZCSR},  to the fiberwise derivative cocycle $\calA(g,x) = \restrict{D_x g}{F(x)}$ of the $\mu$-preserving $\Sl(n,\R)$-action on $M^\alpha$.
		Since the fibers have dimension at most $n-1$  and since there are no non-trivial representations $\rho\colon \Sl(n,\R)\to \Sl(d,\R)$ for $d<n$, it follows from Theorem \ref{thm:ZCSR} that the fiberwise derivative cocycle $ \calA(g,x) = \restrict{D_x g}{F(x)}$  is cohomologous to a compact-valued cocycle: there is a compact group $K\subset \Sl(d,\R)$ and a measurable map $\Phi\colon M^\alpha \to \Gl(d,\R)$ such that $$\Phi(g\cdot x) \restrict{D_x g}{F(x)}\Phi(x)\inv \in K.$$
		By Poincar\'e recurrence to sets on which the norm and conorm of $\Phi$ are bounded, it follows for any $g\in G$ and $\epsilon>0$  that the set of  $x\in M^\alpha$ such that $$\liminf _{n\to \infty} \frac 1 n \log\|\restrict{D_x g^n}{F(x)}\|\ge {\epsilon }$$ has $\mu$-measure zero.  This   contradicts the existence of  nonzero fiberwise Lyapunov exponent for $\mu$. This contradiction completes the proof of Theorem \ref{thm:USEGOD}. 
	\end{proof}

	\section{Discussion of the proof of Propositions  \ref{eq6677} and  \ref{prop:goodmeas}}\label{sec:owlpate}
	We outline the main steps in the proof of Propositions  \ref{eq6677} and  \ref{prop:goodmeas}.  

	\subsection{Averaging measures on $M^\alpha$} Let $H= \{h^t: t\in \R\}$ be a 1-parameter subgroup of $\Sl(n,\R)$.  Given a measure $\mu$ on $M^\alpha$ and $T\ge 0$ we define
	$$H^T\ast \mu := \frac{1}{T}\int _0^T (h^t)_*\mu \ d t$$
	to be the  measure obtained by averaging  the translates of $\mu$ over the interval  $[0, T]$.

	Let $s\in A$.  Given any $s$-invariant measure $\mu$ on $M^\alpha$, the \emph{average top fiberwise Lyapunov exponent} of $s$ with respect to $\mu$ is \index{Lyapunov exponent!top!fiberwise}
	\begin{equation}\label{topex}\lambda^F_\av(s,\mu)  =  \inf_{n\ge 1}  \frac 1 n  \int \log \|\restrict{D_x(s^n)}{F} \| \ d \mu (x).\end{equation}
	Note that if  $\mu$ is moreover $A$-invariant and $A$-ergodic with fiberwise Lyapunov exponents  $\lambda^F_{1,\mu}, \dots, \lambda^F_{p,\mu}\colon A\to \R$ then  $$\lambda^F_\av(s,\mu) =\max _{1\le i\le p} \lambda^F_{i, \mu}(s).$$
	
	We have the following  facts which we invoke throughout our averaging procedures.  
	\begin{claim} \label{olpo} Let $s\in A$ and let $\mu$ be an $s$-invariant measure on $M^\alpha$.  \label{average}
		Let $H= \{h^t, t\in \R\}$  be a one-parameter group  contained in the centralizer of $s$  in $\Sl(n,\R)$.
		\begin{enumcount}
			\item \label{olk1}The measure $H^T\ast \mu$
			is $s$-invariant for every $T\ge 0$.
			\item \label{olk2} Any weak-$*$ limit point   of 
			$\{H^T\ast \mu\}$ as $T\to \infty$ 
			is $s$-invariant.
			\item \label{olk5} Any weak-$*$ limit point   of 
			$\{H^T\ast \mu\}$ as $T\to \infty$ 
			is $H$-invariant.
			\item \label{olk4} $\lambda^F_\av(s,H^T\ast \mu)= \lambda^F_\av(s,\mu)$ for every $T\ge 0$.  
			
			\item \label{olk3} If $\mu'$ is a weak-$*$ limit point   of 
			$\{H^T\ast \mu\}$ as $T\to \infty$ 
			then $$\lambda^F_\av(s,\mu')\ge \lambda^F_\av(s,\mu).$$
		\end{enumcount}
	\end{claim}
	
	\ref{olk1}  is clear from definition and \ref{olk2} follows since     the set of $s$-invariant measures is closed. 
	\ref{olk5} follows from (the proof of) the Krylov-Bogolyubov theorem (see \cref{booboob}). 
	\ref{olk4} is a standard computation which follows from  the compactness of $M^\alpha$ and hence boundedness of the cocycle.  
	Indeed we have 
	\begin{align*}
		\lambda^F_\av&(s,H^T\ast \mu)
		= \inf _{n\to \infty} \frac 1 n  \int \log \|\restrict{D_x(s^n)}{F} \| \ d ( H^T\ast \mu) (x)\\
		&= \inf _{n\to \infty} \frac 1 n \frac{1}{T} \int \int_{t=0}^T\log \|\restrict{D_{h^t\cdot x}(s^n)}{F} \| \ d t   \, d\mu (x)\\
		&= \inf _{n\to \infty} \frac 1 n \frac{1}{T} \int \int_{t=0}^T\log \|\restrict{D_{h^t\cdot x}(h^ts^nh^{-t})}{F} \| \ d t   \, d\mu (x)\\
		&\le  \inf _{n\to \infty} \frac 1 n \frac{1}{T} \int \int_{t=0}^T
		\log \|\restrict{D_{h^t\cdot x}(h^{-t})}{F} \|\\& \quad \quad \quad \quad \quad \quad \quad \quad \quad 
		+  \log \|\restrict{D_{ x}(s^n)}{F} \| 
		+ \log \|\restrict{D_{s\cdot x}(h^{t})}{F} \| 
		\ d t \, d \mu (x)\\
		&\le  \inf _{n\to \infty} \frac 1 n \left( \int  
		\log \|\restrict{D_{ x}(s^n)}{F} \| 
		\ d \mu (x) + 2K\right)\\
	\end{align*}
	where $$K = \sup\left\{ \log \|\restrict{D_{ x}(h^{t})}{F} \| : x\in M, t\in [-T,T]\right\}.$$
	\ref{olk3} follows from the well-known fact that the   average top Lyapunov exponent is upper-semicontinuous on the set of $s$-invariant measures (see for example \cite{MR3289050} or  \cite[Lemma 3.2(b)]{1608.04995}).  Indeed,  in the weak-$*$ topology,  for each $n$ the function $$\mu\mapsto  \frac 1 n  \int \log \|\restrict{D_x(s^n)}{F} \| \ d  \mu (x)$$ is continuous.  The pointwise infimum of a family of continuous functions is  upper-semicontinuous.

	\begin{remark}\label{rem:folner}
		Recall that a \emph{\Folner sequence} \index{F{\o}lner sequence} in a Lie group $H$ equipped with a left-Haar measure $m_H$ is a sequence $\{F_n\}$  of Borel subsets $F_n\subset H$, with $0<m_H(F_n)<\infty$, such that for every compact subset $Q\subset H$ one has
		$$\lim_{n\to\infty } \sup_{h\in Q}\dfrac{m_H\big ((h\cdot  F_n) \triangle F_n\big)} {m_H(F_n)} = 0.$$
		If $H$ admits a \Folner sequence then $H$ is said to be \emph{amenable}\index{group!amenable}.  
		When $H=\R$,  a \Folner sequence is given by $F_n = [0,n]$.   
		Examples of amenable  groups include abelian groups, nilpotent groups, solvable groups, and compact groups.  See \cite{BdHV} for more details.  
		
		Consider  $H$ to be an amenable Lie subgroup of $G= \Sl(n,\R)$.
		Given a Borel probability measure $\mu$ on $M^\alpha$ and a  \Folner sequence $\{F_n\}$ in $ H$ we define $$F_n \ast \mu:= \dfrac 1 {m_H(F_n)}\int _{F_n} h_* \mu \ d m_H(h).$$
		By a computation analogous to \eqref{eq:KB} in the proof of \cref{booboob}, any weak-$*$ limit point $\hat \mu$ of the sequence $\{F_n \ast \mu\}$ as $n\to \infty$ is an $H$-invariant measure on $M^\alpha$.  Moreover,  properties analogous to those in Claim \ref{average} hold when averaging  an $s$-invariant measure $\mu$ against a \Folner sequence $\{F_n\}$ in an amenable subgroup $H$ contained in the centralizer $C_G(s)$ of $s$.  
		See \cite[Lemma 3.2]{1608.04995} for precise formulations.  
	\end{remark}

	\subsection{Averaging measures on $\Sl(n,\R)/\Gamma$}
	When averaging probability measures on $\Sl(n,\R)/\Gamma$ along 1-parameter unipotent subgroups we obtain additional properties of the limiting measures.  The results stated in the following proposition  are   consequences of \index{theorem!Ratner measure classification and equidistribution} 
	Ratner's measure classification and equidistribution theorems for unipotent flows  \cites{MR1262705,MR1054166,MR1075042}.   See also \cite{MR2158954}.  We do not formulate Ratner's theorems here but only  the consequences we use in the remainder.  
	
	\begin{proposition}
		\label{claim:Ratner} Let $\hat \mu $ be a Borel probability measure on $\Sl(n,\R)/\Gamma$.   
		For each 1-parameter  root subgroup $U^{i,j}$
		\begin{enumcount}
			\item \label{polk1}the weak-$*$ limit 
			$$U^{i,j}\ast \hat \mu  := \lim_{T\to \infty}\{(U^{i,j})^T\ast \hat   \mu  : T\ge 0 \}$$
			exists; 
			\item \label{polk2} if $\hat \mu $ is $A$-invariant, so is $U^{i,j}\ast \hat \mu $;
			
			\item \label{polk3} if $\hat \mu $ is $A$-invariant and $A$-ergodic, the measure $U^{i,j}\ast \hat \mu $ is $A$-ergodic;
			\item \label{polk4} if $\hat \mu $ is $A$-invariant and  $U^{i,j}$-invariant then $\hat \mu $ is $U^{j,i}$-invariant.  
		\end{enumcount}
	\end{proposition}
	
	 \cref{claim:Ratner}\ref{polk1}  follows from Ratner's   measure classification and equidistribution theorems for unipotent flows.  When $U$ is higher-dimensional, we use an analogue of  \cref{claim:Ratner}\ref{polk1}  due to Shah  \cite[Corollary 1.3]{MR1291701}.
	 \cref{claim:Ratner}\ref{polk2} follows from the fact that $A$ normalizes  $U^{i,j}$ and that the limit in  \cref{claim:Ratner}\ref{polk1} exists and is hence unique.  
	 \cref{claim:Ratner}\ref{polk4} is a consequence of Theorem 9 in   \cite {MR1262705} or Proposition 2.1 in \cite{MR1054166}.

	 \cref{claim:Ratner}\ref{polk3} is a short argument that uses the $A$-invariance of $\hat \mu$ and the pointwise ergodic theorem: Since there is $s\in A$ such that $U^{i,j}$-orbits are contracted by $s$,   by the pointwise ergodic theorem, the measurable hull of the partition into $U^{i,j}$-orbits refines the ergodic decomposition for $A$. (See \cref{prop:hopf} and Theorem \ref{Hopf} in Appendix \ref{App:Pinsker}.)  Let $\eta$ be the  measurable hull of the partition into $U^{i,j}$-orbits and let $\{\hat \mu^\eta_x\}$ be a family of conditional measures of $\hat \mu$ for this partition. (Note that  from Ratner's equidistribution theorem, we have that $\hat \mu^\eta_x$ is a homogeneous measure on a closed homogeneous submanifold.)   If $\phi$ is a bounded, $A$-invariant measurable function then  for $\hat \mu$-a.e.\ $x$, $\phi$ is constant $\hat \mu^\eta_x$-almost surely; in particular, $$\phi(x) = \int \phi \ d \hat \mu^\eta_x$$ for $\hat \mu$-a.e.\ $x$.   Then $x\mapsto \int \phi \ d \hat \mu^\eta_x$ is a $\mu$-almost everywhere defined, $A$-invariant function.  In particular, $x\mapsto \int \phi \ d \hat \mu^\eta_x$ is constant $\mu$ a.s.\ by ergodicity of $\mu$.  It follows that $\phi$ is constant $\hat \mu$-a.s.\ and ergodicity follows.

	\subsection{Proof of Proposition \ref{eq6677}}\label{Sec:spitproof}
	We outline the proof of \cref{eq6677}. 
	Recall the notation  introduced in Section \ref{ssec:susp}.
	In particular,  $\pi\colon M^\alpha\to G/\Gamma$  is the canonical  projection and $F = \ker (D\pi)$ is  the fiberwise tangent bundle of $M^\alpha$.
	We write the derivative of translation by $g$ in $M^\alpha$ as $Dg$ and the restriction to the fiber of $F$ through $x\in M^\alpha$ by $\restrict{D_xg} {F(x)}.$  Equip $M^\alpha$ with any Riemannian metric and write  $$\|\restrict{Dg} {F}\|=\sup_{x\in M^{\alpha}}\|\restrict{D_xg}  {F(x)}\|.$$
	
	Let $K=\So(n)$.  We equip $G$ with a right-invariant, left-$K$-invariant Riemannian metric and   induced distance function $d(\cdot, \cdot)$.   We note that relative to such a metric, all $A$-orbits are geodesically embedded in $G$. 
	We have the following elementary claim  which allows us to transfer exponential growth properties between  the $\Gamma$-action on $M$  and the  $G$-action on the  fibers of  $M^\alpha$.  
	\begin{claim}\label{claim:gammavG}
		If $\Gamma\subset \Sl(n,\R)$ is cocompact and if $M$ is compact, then any action $$\alpha\colon \Gamma\to \diff^1(M)$$  has uniform subexponential growth of derivatives if and only if for every $\epsilon>0$ there is a $C$ such that for all $g\in \Sl(n,\R)$, $$\|\restrict{Dg} {F}\|\le Ce^{\epsilon d(e, g)}.$$
	\end{claim}
	With the above claim, we outline the main steps in the proof of \cref{eq6677}.  
	
	\begin{proof}[Proof of \cref{eq6677}]
		We assume $\alpha\colon \Gamma\to \diff^1(M)$ fails to have uniform subexponential growth of derivatives.  Then, by \cref{claim:gammavG}, there exist $\epsilon>0$, integers $m_n\in \N$ with ${m_n}\to \infty$,  elements   $g_{m_n}\in G$ with $d(g_{m_n}, e)= {m_n}$, points $ x_{m_n}\in M_\alpha,$ and unit vectors $v_{m_n}\in T_{x_{m_n}}M_\alpha$  such that 
		$$\|D_{x_{m_n}}g_{m_n}(v_{m_n})\|\ge e^{\epsilon {m_n}}.$$

		Let $UF$ denote the unit sphere bundle in $F$ and, given $g\in G$, let $UDg$ denote the induced action on $UF$: given $x\in M^\alpha$ and $v\in UF(x)$  write $$UD_xg(v) = \frac{D_x g(v)}{\|D_x g(v)\|}$$ and 
		$$UDg(x,v) = \left(g\cdot x, UD_xg(v)\right).$$
		
		By the singular value decomposition of matrices, the group  $G= \Sl(n,\R)$ can be written as $G= KAK$ where $K= \SO(n)$. (For general simple Lie groups $G$ we use the  Cartan decomposition).   We can thus write each $g_{m_n}\in G$ as $$g_{m_n} = k_n a_n k'_n$$ where $k_n,k'_n\in K$ and $a_n \in A$.  Write 
		$$x'_n = k'_n\cdot x_{m_n}, \quad x''_{n} = a_nk'_n\cdot x_{m_n} , $$ $$  v'_n = UD_{x_{m_n}} k'_n (v_{m_n}), \quad v''_n = UD_{x_{m_n}} (a_nk'_n) (v_{m_n}).$$
		Then 
		\begin{align*} \|&D_{x_{m_n}}g_{m_n}(v_{m_n})\| 
			=  \|D_{x_n''}k_n(v_n'') \| \cdot \| D_{x_n'}a_n(v_n') \|  \cdot \|D_{x_{m_n}}k'_n(v_{m_n})\| 
		\end{align*}
		and so $$\epsilon \le \lim _{n\to \infty } \frac 1 {{m_n}} \log \|D_{x_{m_n}}g_{m_n}(v_{m_n})\|=
		\lim _{n\to \infty } \frac 1 {{m_n}}  \log \| D_{x'_n}a_n(v'_n)\|$$
		as $\|\restrict{D_x k}{F}\|$ is uniformly bounded over all $k\in K$ and $x\in M^\alpha$.  
		
		Note that  $$ |{m_n}-d(a_n,e)|= |d(g_{m_n},e)-d(a_n,e)| \le  d(k_{n},e)+ d(k'_n,e) $$
		is uniformly bounded in $n$.    Thus ${m_n}\inv d(a_n,e)\to 1$.  As $A\simeq \R^{n-1}$, for each $n$ there is a unique $\td a_n$ with   $a_n = (\td a_n) ^{m_n}$;  moreover,  as $A$ is geodesically embedded in $G$, we have   $d(\td a_n ,e) \to 1$.
		
		For each $n$, let $\nu_n$ be the empirical measure on $UF$ given by
		$$\nu_n = \frac{1}{{m_n}}\sum _{j=0}^{{m_n}-1}(\td a_n )^j _*\delta_{(x_n', v_n')}.$$
		Taking a  subsequence  $\{n_j\}$, we may assume that $\nu_{n_j}$ converges to some $\nu_\infty$ and that $\td a_{n_j} $ converges to some $ s\in A$.  Note that  $d(s, e) =1$.     Let $\bar \mu$ denote the image of $\nu_\infty$ under the natural projection  $UF\to M^\alpha$.  Adapting the proofs of \cref{booboob} and \cref{prop:ZLEUSEGOD} one can show that 
		\begin{enumerate}
			\item $\nu_\infty$ is $UD{s}$-invariant whence $\bar\mu$ is $s$-invariant;
			\item $\lambda^F_\av(s,\bar \mu) \ge \epsilon >0$.
		\end{enumerate} 
		
		Take a \Folner sequence $\{F_n\}$ in $A$ and let $\td \mu$ be any weak-$*$ limit point of {$\{F_n\ast \bar \mu\}$} as $n\to \infty$.  Then, from  analogues of the properties in \cref{olpo} for averaging over \Folner sequences, we have that 
		\begin{enumerate}
			\item $\td \mu$ is $A$-invariant;
			\item $\lambda^F_\av(s,\td\mu) \ge \epsilon >0$.
		\end{enumerate}
		We  take $\mu'$ to be an $A$-ergodic component of $\td \mu$ with  $\lambda^F_\av(s,\mu')\ge \epsilon  >0$.
	\end{proof}

	\subsection{Proof   of Proposition \ref{prop:goodmeas} for $\Sl(3,\R)$}\label{ss:sl3}
	To simplify ideas, we outline the proof  of Proposition \ref{prop:goodmeas} assuming that $\Gamma$ is a cocompact lattice in $\Sl(3,\R)$.  We perform two averaging procedures on the measure $\mu'$ from the hypotheses of  Proposition \ref{prop:goodmeas}  to obtain the measure $\mu$  in the conclusion of Proposition \ref{prop:goodmeas}.

	\begin{proof}[Proof   of Proposition \ref{prop:goodmeas} for $\Gamma\subset \Sl(3,\R)$]
		Take $\mu_0= \mu'$  to be the ergodic, $A$-invariant probability measure in the hypotheses of  Proposition \ref{prop:goodmeas}  with nonzero fiberwise exponent $$\text{$\lambda^F_{j, \mu_0}\colon A\to \R$, \quad $\lambda^F_{j, \mu_0}\ \neq 0$.}$$  
		
		\fakeSS{First averaging}
		Consider the elements $$\text{$s=\diag(\tfrac 1 {4},2, 2)$ \quad and \quad $\bar s = \diag(2, 2, \tfrac 1 {4})$}$$ of $A\subset \Sl(3,\R)$.  Note that $s$ and $\bar s$ are linearly independent and hence form a basis for $A\simeq \R^{2}$.  As the linear functional $ \lambda^F_{j, \mu_0} $ is nonzero,  either $$\text{$\lambda^F_{j, \mu_0} (s) \neq 0$ \quad or \quad  $\lambda^F_{j, \mu_0}(\bar s) \neq 0$. }$$  
		Without loss of generality we may assume that 
		$$\lambda^F_{j, \mu_0}( s) \neq 0.$$  
		Take $s_0$ to be either $s$ or $s\inv$ so that  $\lambda^F_{j, \mu_0}( s_0) > 0$.

		Consider the 1-parameter subgroup $$U^{2,3} = \left \{ \left(\begin{array}{ccc}1 & 0 & 0 \\0 & 1 & t \\0 & 0 & 1\end{array}\right): t\in \R\right\}.$$
		Note  that $U^{2,3}$ commutes with $s_0$.  
		Let $\mu_1$ be any weak-$*$ limit point of $\{(U^{2,3})^T\ast \mu \}$ as $T\to \infty$.  From Claim \ref{average}, $\mu_1$ is $s_0$-invariant and $\lambda^F_\av(s_0,\mu_1)\ge \lambda^F_\av(s_0,\mu_0).$
		
		We now average $\mu_1$ over a \Folner sequence in $A$:  identifying $A$ with $ \R^2$ let  $A^T= [0,T]\times [0,T]$ define a \Folner sequence $\{A^T\}$ in $A$.  Then 
		$$  A^T\ast \mu_1:= \frac{1}{T^2} \int_0^T \int _0^T (t_1, t_2)_*\mu_1 \ d (t_1, t_2).$$
		Let $\mu_2$ be any weak-$*$ limit point of $\{A^T\ast \mu_1\}$ as $T\to \infty$.  Then, from facts analogous to those in Claim \ref{average}, $\mu_2$ is $A$-invariant and 
		$$
		\lambda^F_\av(s_0,\mu_2)
		\ge \lambda^F_\av(s_0,\mu_1) >0.$$
		Note that $\mu_2$ might no longer be $U^{2,3}$-invariant.  
		
		We investigate properties of the projection of each measure $\mu_0, \mu_1, $ and $\mu_2$  to $\Sl(3,\R)/\Gamma$.  For each $j$,  we denote by $\hat \mu_j = \pi_* (\mu_j)$   the image of $\mu_j$ under the projection $\pi\colon M^\alpha \to \Sl(3,\R)/\Gamma$.  
		
		Observe that  $\hat \mu_1= U^{2,3}\ast \hat \mu_0 $ is   $U^{2,3}$-invariant.  Since $\hat \mu_0$ was $A$-invariant, from  \cref{claim:Ratner}\ref{polk2} we have that $\hat \mu_1$ is  $A$-invariant and it follows that   $\hat \mu_1= \hat \mu_2$ so $\hat \mu_2$ is $U^{2,3}$-invariant and $A$-invariant. 
		From  \cref{claim:Ratner}\ref{polk4}, 
		$\hat \mu_2$ is invariant under the subgroup 
		$$\left \{ \left(\begin{array}{ccc}\ast & 0 & 0 \\0 & \ast & \ast\\0 & \ast & \ast\end{array}\right)\right\}\subset \Sl(3,\R)$$
		generated by  $A$, $U^{2,3}$ and $U^{3,2}$ in $\Sl(3,\R)$.    
		Moreover, since  $\hat \mu_0$ was $A$-ergodic, from  \cref{claim:Ratner}\ref{polk3} the measure $\hat \mu_1 = \hat \mu_2$ is $A$-ergodic.  
		
		Returning to $M^\alpha$, as $ \lambda^F_\av(s_0,\mu_2)>0$ and as $\hat \mu_2$ is $A$-ergodic, we may replace $\mu_2$ with an $A$-ergodic component $\mu_2'$ of $\mu_2$ such that  
		\begin{enumerate}
			\item  $\lambda^F_\av(s_0,\mu_2')>0$, and 
			\item  the projection of  $\mu_2'$ to $\Sl(3,\R)/\Gamma$ is   $\hat \mu_2$.
		\end{enumerate}
		Let   $\lambda^F_{1,\mu_2'}, \dots, \lambda^F_{p',\mu_2'}\colon A\to \R$ denote the fiberwise Lyapunov exponents for the $A$-invariant, $A$-ergodic measure $\mu_2'$.  
		Then     $0<  \lambda^F_{j', \mu_2'} (s_0)= \lambda^F_\av(s_0,\mu_2')  $ for some $1\le j'\le p'$  whence some fiberwise Lyapunov exponent  $ \lambda^F_{j', \mu_2'}\colon A\to \R $ is a nonzero linear functional.

		\fakeSS{Second averaging}
		Consider now  the elements $s = (2,2, \frac 1 4)$ and $\bar s = (2, \frac 1 4, 2)$ in $A$.
		Again, either  $$\text{$\lambda^F_{j', \mu_2'} (s)   \neq 0$  \quad or \quad $\lambda^F_{j', \mu_2'}(\bar s)  \neq 0$.}$$
		
		\fakeSS{Case 1: $\lambda^F_{j', \mu_2'} (s) \neq 0$}
		Take $s_1= s$ or $s_1= s\inv $  so that $\lambda^F_{j', \mu_2'} (s_1) > 0$.
		Consider the one-parameter group $U^{1,2}$ which commutes with $s_1$.  As above, any weak-$*$ limit point $\mu_3$ of 
		$\{(U^{1,2})^T\ast \mu_2' \}$ as $T\to \infty$ is $s_1$-invariant, with 
		$$\lambda^F_\av(s_1, \mu_3)\ge \lambda^F_\av(s_1, \mu_2')>0.$$
		Let $\mu_4$ be any weak-$*$ limit point of $\{A^T\ast \mu_3 \}$ as $T\to \infty$ (where $A^T\ast \mu_3$ is as in the first averaging).  Then $\mu_4$ is $A$-invariant and 
		$$ \lambda^F_\av(s_1,\mu_4)   \ge \lambda^F_\av(s_1,\mu_3) >0.$$
		
		We claim  that the projection $\hat \mu_4$  of $\mu_4$ to $\Sl(3,\R)/\Gamma$ is the Haar measure.  
		Since the groups $U^{1,2}$ and $U^{3,2}$ commute and since $\hat \mu_2$ was $U^{3,2}$-invariant, it follows that $\hat \mu_3= U^{1,2}\ast \hat \mu_2$ is $U^{3,2}$-invariant.  Also, since $\hat \mu_2$ was $A$-invariant,   \cref{claim:Ratner}\ref{polk2} shows that $\hat \mu_3$ is $A$-invariant.  Thus   $\hat \mu_3 = \hat \mu_4$ and    $\hat \mu_4$ is also invariant under the actions of $A$, $U^{1,2}$, and $U^{3,2}.$
		By   \cref{claim:Ratner}\ref{polk4} it follows that 
		$\hat \mu_4$ is invariant under the groups $U^{2,1}$  and $U^{2,3}$; in particular $\hat \mu_4$ is invariant under the following subgroups of $\Sl(3,\R)$: 
		$$\left \{ \left(\begin{array}{ccc}\ast & 0 & 0 \\0 & \ast & \ast\\0 & \ast & \ast\end{array}\right)\right\}, 
		\quad \quad 
		\left \{ \left(\begin{array}{ccc}\ast & \ast & 0 \\ \ast & \ast & 0\\0 & 0 & \ast\end{array}\right)\right\}.$$
		These two groups generate all of $\Sl(3,\R)$, and hence $\hat \mu_4$ is the Haar measure.

		\fakeSS{Case 2: $\lambda^F_{j', \mu_2'} (\bar s) \neq 0$}
		Take $s_1= \bar s$ or $s_1= \bar s\inv $  so that $\lambda^F_{j', \mu_2'} (s_1) > 0$.
		Consider the one-parameter group $U^{1,3}$ which commutes with $s_1$.  As above, any weak-$*$ limit point $\mu_3$ of 
		$\{(U^{1,3})^T\ast \mu_2' \}$ as $T\to \infty$ is $s_1$-invariant, with 
		$$\lambda^F_\av(s_1, \mu_3)\ge \lambda^F_\av(s_1, \mu_2')>0.$$
		Let $\mu_4$ be any weak-$*$ limit point of $\{A^T\ast \mu_3 \}$ as $T\to \infty$.  Then $\mu_4$ is $A$-invariant and 
		$$ \lambda^F_\av(s_1,\mu_4) \ge \lambda^F_\av(s_1,\mu_3) >0.$$
		
		Again, we claim that $\hat \mu_4=U^{1,3}\ast \hat \mu_2 $ is the Haar measure.  
		Since the groups $U^{1,3}$ and $U^{2,3}$ commute, it follows that $\hat \mu_3$ is $U^{2,3}$-invariant.  
		Also, since $\hat \mu_2$ was $A$-invariant,    \cref{claim:Ratner}\ref{polk2} shows that $\hat \mu_3$ is $A$-invariant.  Thus   $\hat \mu_3 = \hat \mu_4$ and    $\hat \mu_4$ is also invariant under the actions of $A$, 
		$U^{1,3}$ and $U^{2,3}$.  
		By     \cref{claim:Ratner}\ref{polk4} it follows that  $\hat \mu_4$ is invariant under the   following subgroups of $\Sl(3,\R)$:  
		$$ \left \{ \left(\begin{array}{ccc}\ast & 0 & 0 \\0 & \ast & \ast\\0 & \ast & \ast\end{array}\right)\right\},
		\quad \quad 
		\left \{ \left(\begin{array}{ccc}\ast & 0 & \ast \\ 0 & \ast & 0\\\ast & 0 & \ast\end{array}\right)\right\}.$$
		Again, these two groups generate all of $\Sl(3,\R)$, and  hence $\hat \mu_4$ is the Haar measure.  
		
		\fakeSS{Completion of proof}
		In either Case 1 or Case 2, since the Haar measure $\hat \mu_4$ is $A$-ergodic, we may take an $A$-ergodic component $\mu_4'$ of $\mu_4$ projecting to the Haar measure with 
		$$\lambda^F_\av(s_1,\mu_4')  >0.$$
		If    $\lambda^F_{1,\mu_4'}, \dots, \lambda^F_{p'',\mu_4'}\colon A\to \R$ denote the fiberwise Lyapunov exponents for the $A$-invariant, $A$-ergodic measure $\mu_4'$ then     $0<  \lambda^F_{j'', \mu_4'} (s_1)= \lambda^F_\av(s_1,\mu_4')  $ for some $1\le j''\le p''$  whence some fiberwise Lyapunov exponent  $ \lambda^F_{j'', \mu_4'}\colon A\to \R $ is a nonzero linear functional.
		
		This completes the proof of Proposition \ref{prop:goodmeas}.
	\end{proof}
	\subsection{Modifications of the proof of Proposition \ref{prop:goodmeas}  in $ \Sl(n,\R)$}
	
	When $\Gamma$ is a cocompact lattice in $\Sl(n,\R)$ we replace the first averaging step with a   more complicated averaging.  
	
	
	\fakeSS{First averaging}
	We again take $\mu_0= \mu'$  to be the $A$-invariant measure in \cref{prop:goodmeas} with nonzero fiberwise exponent $$\text{$\lambda^F_{j, \mu_0}\colon A\to \R$, \quad $\lambda^F_{j, \mu_0}\ \neq 0$.}$$ 
	Without loss of generality (by conjugating by a permutation matrix) we may assume that for   the element $$ s=\diag(\tfrac 1 {2^{n-1}},2,\dots, 2) $$ of $A\subset \Sl(n,\R)$, we have  $$\lambda^F_{j, \mu_0} (s) \neq 0.$$
	Take $s_0$ to be either $s, $ or $ s\inv$   so that $\lambda^F_{j, \mu_0}( s_0) > 0$.

	Consider the unipotent subgroup $U\subset \Sl(n,\R)$ of matrices of the form $$ U = \left\{ \left(\begin{array}{ccccc}1 & 0 & 0&\cdots &0 \\0 & 1 & \ast &\cdots &\ast
	\\  \vdots &  & \ddots & &\vdots 
	\\0 & 0 & \cdots &1 & \ast 
	\\0 & 0 & \cdots &0 & 1 \\
	\end{array}\right)\right\}.$$
	Note  that $U$ commutes with $s_0$.
	
	Let $\{F_n\}$ be a \Folner sequence in $U$ and let $\mu_1$ be any weak-$*$ limit point of $\{F_n\ast \mu_0 \}$ as $n\to \infty$ where $$ F_n\ast \mu_0= \frac{1}{m_U(F_n)} \int_{F_n} u_* \mu_0 \ d u.$$  From facts analogous to those in Claim \ref{average}, we have that $\mu_1$ is $s_0$-invariant and $\lambda^F_\av(s_0,\mu_1)\ge \lambda^F_\av(s_0,\mu_0)>0.$
	Moreover, as $U$ is higher-dimensional, we   use  \cite[Corollary 1.3]{MR1291701} rather than     \cref{claim:Ratner}\ref{polk1}  to conclude (at least for certain \Folner sequences $\{F_n\}$ in $U$ with nice geometry)  that the projection $\hat \mu_1$ of $\mu_1$ to $G/\Gamma$ is the limit $$\hat \mu_1 = \lim F_n\ast \hat \mu_0$$ and is  $A$-invariant, ergodic, and $U$-invariant.

	We again average $\mu_1$ over a \Folner sequence of the form  $$A^T= [0,T]\times \dots \times  [0,T]$$ in $A$ (identified with $\R^{n-1}$) 
	and let $\mu_2$ be any weak-$*$ limit point of $\{A^T\ast \mu_1 \}$ as $T\to \infty$.  Then $\mu_2$ is $A$-invariant and 
	$$
	\lambda^F_\av(s_0,\mu_2)
	\ge \lambda^F_\av(s_0,\mu_1) >0.$$
	Again, we have equality of the projected measures $\hat \mu_1= \hat \mu_2$ so $\hat \mu_2$ is $U$-invariant and $A$-invariant. 
	From    \cref{claim:Ratner}\ref{polk4}, 
	$\hat \mu_2$ is also  invariant under the subgroup 
	$$ H =\left\{ \left(\begin{array}{ccccc}1 & 0 & 0&\cdots &0 \\0 & \ast & \ast &  &\ast
	\\0 & \ast &\ast  & &\ast
	\\  \vdots &  & & \ddots  &\vdots 
	\\0 & \ast & \ast& \cdots  & \ast \\
	\end{array}\right)\right\}.$$
	As $\hat \mu_2$ is $A$-ergodic, we may replace $\mu_2$ with an $A$-ergodic component $\mu_2'$ of $\mu_2$ such that  
	\begin{enumerate}
		\item  $\lambda^F_\av(s_0,\mu_2')>0$, and 
		\item  the projection of  $\mu_2'$ to $\Sl(n,\R)/\Gamma$ is   $\hat \mu_2$.
	\end{enumerate}
	Then, if $\lambda^F_{1,\mu_2'}, \dots, \lambda^F_{p',\mu_2'}\colon A\to \R$ denote the fiberwise Lyapunov exponents for  $\mu_2'$, we have     $0<  \lambda^F_{j', \mu_2'} (s_0)= \lambda^F_\av(s_0,\mu_2')  $ for some $1\le j'\le p'$.

	\fakeSS{Second averaging}
	Consider now  the roots $\beta^{1,2}$ and $\beta^{1,n}$ of $G$.  Since $\beta^{1,2}$ and  $\beta^{1,n}$ are not proportional, at most one of 
	$\beta^{1,2}$ and  $\beta^{1,n}$  is proportional to $\lambda^F_{j', \mu_2'}$.  In particular, we may find either $s$ or $\bar s$ in $A$ such that 
	\begin{enumerate}
		\item  $\beta^{1,2}(s) =0$ but $ \lambda^F_{j', \mu_2'} (s)   \neq 0$; or
		\item  $\beta^{1,n}(\bar s) =0$ but $\lambda^F_{j', \mu_2'} (\bar s)   \neq 0$.
	\end{enumerate}

	The two cases in the second averaging step of Section \ref{ss:sl3} are then identical to the above, where we either average over the 1-parameter group $U^{1,2}$ in the case $\lambda^F_{j', \mu_2'} (s)   \neq 0$ or $U^{1,n}$ in the case $\lambda^F_{j, \mu_2'}(\bar s)  \neq 0$.  The structure theory of $\Sl(n,\R)$ will then imply that the measure obtained after the second averaging projects to the Haar measure.

	\starsection{Zimmer's conjecture for actions by  lattices in other Lie groups}
	\label{sec:othergps}
	Consider a connected, simple Lie group $G$ with finite center.  Let $\Gamma\subset G$ be a cocompact lattice.  The proof of Theorem \ref{slnr} discussed above, particularly the use of Theorem \ref{thm:invmsr}  in Section \ref{sec:pp} can be adapted almost verbatim to show the following.  See also \cite{Cantat} where Theorem \ref{thm:serge} is stated and given  a mostly self-contained proof.   
	\begin{theorem}\label{thm:serge}
		Let $G$ be a connected, simple Lie group $G$ with finite center and rank at least $2$.  Let $\Gamma\subset G$ be a cocompact lattice.  Let $M$ be a compact manifold.   
		\begin{enumerate}
			\item If $\dim(M)< \rank(G)$
			then any homomorphism $\Gamma\rightarrow \Diff^{2}(M)$ has finite image.  
			\item In addition, if $\vol$ is a volume form on $M$
			and if $\dim(M)\le \rank(G)$ then any homomorphism $\Gamma \rightarrow \Diff^{2}_\vol(M)$ has finite image.
		\end{enumerate}
	\end{theorem}
	As mentioned in \cref{fullspit}, Theorem \ref{thm:serge}  holds for $C^1$ actions; see Theorem \ref{slnrC1}.

	Theorem \ref{thm:serge} fails to give the optimal dimension bounds for the analogue of Conjecture \ref{conj:slnr} given in Conjecture \ref{conjecture:zimmergne} for actions by lattices in Lie groups other than $\Sl(n,\R)$.  See Table \ref{tab:stupid} for various conjectured critical dimensions arising in Zimmer's conjecture for other Lie groups.  

	To state the most general (as of 2018) result towards solving Conjecture \ref{conjecture:zimmergne},  to any simple Lie group  $G$, we associate a non-negative integer $r(G)$.  See \cite[Section 2.2]{1608.04995} or  \cref{foot:r}, page \pageref{foot:r}, for equivalent definitions of $r(G)$ and Table \ref{tab:stupid} for values of $r(G)$ in various examples of $G$.   For actions of lattices in a general Lie group $G$, the main result of \cite{1608.04995}, as well as the announced extension, gives finiteness of the action up to the critical dimension $r(G)$.
	\begin{theorem}[{\cite{1608.04995}  cocompact case; \cite{BFHWM} nonuniform  case}] \label{slnr:popop}
		Let $\Gamma\subset G$ be a  lattice in a higher-rank simple Lie group $G$ with finite center. Let $M$ be a compact manifold. \begin{enumerate}
			\item If $\dim(M)< r(G)$ 
			then any homomorphism $\Gamma\rightarrow \Diff^{1+\beta}(M)$ has finite image.  
			\item In addition, if $\vol$ is a volume form on $M$
			and if $\dim(M)=r(G)$ then any homomorphism $\Gamma \rightarrow \Diff^{1+\beta}_\vol(M)$ has finite image.
		\end{enumerate}
	\end{theorem}
	
	When $G$ is exceptional or not a split real form,  our number $r(G)$ is lower than the conjectured critical dimension in Conjecture \ref{conjecture:zimmergne}\ref{gnea} and \ref{gneb}.  However,  for lattices in all Lie groups that are {non-exceptional, split real forms} \cref{slnr:popop}  confirms Conjecture \ref{conjecture:zimmergne}\ref{gnea} and \ref{gneb}.   For instance, for actions by lattices in symplectic groups we have the following.
	\begin{theorem}[{\cite[Theorem 1.3]{1608.04995}  cocompact case; \cite{BFHWM}  nonuniform case}]
		\label{thm:spnr}
		For $n\ge 2$, if $M$ is a compact manifold with  $\dim(M)< 2n-1$ and if $\Gamma \subset \Sp(2n,\R)$ is a   lattice then any  homomorphism $\alpha\colon \Gamma \rightarrow \Diff^2(M)$ has finite image. In addition, if $\dim(M)=2n-1$  then any  homomorphism $\alpha\colon \Gamma \rightarrow \Diff_\vol^2(M)$ has finite image.
	\end{theorem}

	Similarly, for actions by lattices in split orthogonal groups we have the following.
	\begin{theorem}[{\cite[Theorem 1.4]{1608.04995}  cocompact case; \cite{BFHWM} nonuniform case}] 
		\label{thm:orthogonalgroups}
		Let $M$ be a compact   manifold. 
		\begin{enumerate}
			\item For $n\ge 4$, if $\Gamma\subset\So(n,n)$ is a   lattice and if   $\dim(M)< 2n-2$ then any  homomorphism $\alpha\colon \Gamma \rightarrow \Diff^2(M)$ has finite image.  If $\dim(M) = 2n-2$ 
			then any  homomorphism $\alpha\colon \Gamma \rightarrow \Diff^2_\vol(M)$ has finite image. 
			\item For $n\ge 3$, if  $\Gamma\subset\So(n,n+1)$ is a   lattice and  if $\dim(M)< 2n-1$ then any  homomorphism $\alpha\colon \Gamma \rightarrow \Diff^2(M)$ has finite image.  If $\dim(M) = 2n -1$  
			then any  homomorphism $\alpha\colon \Gamma \rightarrow \Diff^2_\vol(M)$ has finite image.  
		\end{enumerate}
	\end{theorem}

	For actions by  lattices $\Gamma$ in simple Lie groups that are not split real forms  such as $G=\Sl(n,\mathbb{C})$, $\So(n,m)$ for $m\ge n+2$, or $\mathrm{SU}(n,m)$,    Theorem \ref{slnr:popop} above (the  main result of \cite{1608.04995} for cocompact case, \cite{BFHWM} in general) gives finiteness of all actions on manifolds whose dimension is below a certain critical  dimension.  However, this critical dimension may be below the dimension conjectured by the analogue of Conjecture \ref{conjecture:zimmergne} for these groups.   See Table \ref{tab:stupid}.

	\part{A selection of other  measure rigidity results}
	\section{Nonuniformly hyperbolic $\Z^k$-actions}\label{sec:NUHZd}
	Instead of considering $\Z^2$-actions  by automorphisms of $\T^3$ as in Theorem \ref{thm:KS}, we might  consider  $\Z^2$-actions on the torus  $\T^3$ generated by two commuting   diffeomorphisms $f,g\colon \T^3\to \T^3 $.

	Recall that for any homeomorphism  $f\colon  \T^3\to \T^3$ there exists a unique  $M\in \Gl(3, \Z)$ so that   any  lift $\td f\colon \R^3\to \R^3$  of $f$  is of the form $$\td f(x) = Mx + \psi(x)$$ where $\psi\colon \R^3\to \R^3$ is $\Z^3$-periodic. The linear map $M$ can also be seen as the induced action of $f$ on first homology of $\T^3$.  We call $M$ the \emph{linear data} of $f$.  
	By a theorem of Franks \cite{MR0271990}, if  $M$ has no eigenvalues of modulus 1 then there is a  continuous, surjective $h\colon \T^d\to \T^d$, homotopic to the identity,  such that \begin{equation}\label{eq12}  h \circ f=L_M\circ h\end{equation}
	where $L_M\colon \T^3\to \T^3$ is the induced automorphism of the torus; such a map $h$ is called a \emph{semiconjugacy}.\index{conjugacy!semi-}

	If $f,g\colon \T^3\to \T^3 $ are commuting  homeomorphisms with linear data $A$ and $B$, respectively, one can verify  that $A$ and $B$ commute.  Indeed if $\td f(x)= Ax + \psi(x)$ and $\td g(x) =  Bx + \phi(x)$ are lifts of $f$ and $g$, respectively, then 
	$$\td f\circ \td g(x) = AB x+ A \phi(x) + \psi(x)$$ and 
	$$\td g\circ \td f(x) = BA x+ B \psi(x) + \phi(x)$$ are both lifts of $f\circ g = g\circ f$ whence $AB= BA$.   
	
	If $A$ has no eigenvalues of modulus 1, we may take a map $h\colon \T^3\to \T^3$ with $$h\circ f = L_A \circ h$$ as in \eqref{eq12}.  Following \cite[Lemma 1]{MR2643892} (correcting \cite[Lemma 1.2]{MR2261075}) the map $h$ conjugates the $\Z^2$-action generated by $f$ and $g$ to an affine action on $\T^3$ whose linear part is generated by $L_A$ and $L_B$.  
	That is, if $\alpha\colon \Z^2\to \diff(\T^3)$ is the non-linear action $$\alpha(n_1, n_2) = f^{n_1} g^{n_2}$$ then there is an affine action $\alpha_0\colon  \Z^2\to \diff(\T^3)$  of the form $$\alpha_0(n_1, n_2) (x)= L_A^{n_1} L_B^{n_2}(x) + v_{(n_1, n_2)}$$ for some $v_{(n_1, n_2)} \in \T^3$ such that for all $(n_1, n_2)\in \Z^2$
	\begin{equation}\label{eq:semicon} h\circ \alpha(n_1, n_2) = \alpha_0(n_1, n_2)\circ h. 
	\end{equation}
	We note that the translation term $ (n_1,n_2)\mapsto v_{(n_1, n_2)}$ is a cocycle: $$v_{(n_1, n_2)+ (m_1, m_2)} = L_A^{n_1} L_B^{n_2} v_ {(m_1, m_2)}  +v_{(n_1, n_2)}.$$
	Moreover, the action $\alpha_0$ has a fixed point if and only if $v_{(n_1, n_2)}$ is  a coboundary: $$v_{(n_1, n_2)} =L_A^{n_1} L_B^{n_2} \eta -\eta $$ for some $\eta\in \T^3$.  
	The presence of the translation term $v_{(n_1, n_2)}$  is due to the non-uniqueness of the map $h$ satisfying \eqref{eq12}.  However, all  maps $h$  satisfying \eqref{eq12} differ by a translation by an element of the finite set of  fixed points   for $L_A$.  Thus, the  translation terms $v_{(n_1, n_2)}$ take  only finitely many possible values.  See discussion in \cite{MR2643892} for more details.  We note that it is possible to construct genuinely affine Anosov actions $\alpha_0$ without fixed points as in \cref{rem:affineAnosov}.  See for example \cite[Theorem 2]{MR1236179}.  In particular, it may be that the action $\alpha$ is not semiconjugate to any action  by automorphisms.  However, restricting to a subgroup $\Sigma\subset \Z^2$ of finite index, one has  that  $\restrict{\alpha_0}{\Sigma}\colon \Sigma \to \Diff(\T^3)$ is an action by automorphisms: $$\alpha_0(n_1, n_2) (x)= L_A^{n_1} L_B^{n_2}(x)$$
	for all $(n_1, n_2)\in \Sigma$.

	
	If  $f$ is Anosov then its linear data  $A$ is known to have no eigenvalues of modulus 1 and the map $h$ in \eqref{eq12} is a homeomorphism.  Suppose further that  $f$ and $g$  generate a ``genuine'' $\Z^2$-action so that the group of matrices generated by their linear data $A$ and $B$ is not virtually cyclic.  This implies that the linear action generated by $L_A$ and $L_B$  satisfies Theorem \ref{thm:KS}.  
	Restricted to a finite-index subgroup, the map $h$ conjugates the action $\alpha$ to a linear action of the  type of action considered in  Example \ref{ex:key}.  Since invertible maps preserve entropy, Theorem \ref{thm:KS} classifies all positive entropy measures  that are jointly $f$- and $g$-invariant.   We remark also that under the above assumptions,  from \cite{MR2318497}, we know  in this setting that the conjugating map $h$ satisfying \eqref{eq:semicon} is smooth.  
	Note that the assumption that the group of matrices generated by  $A$ and $B$ is not virtually cyclic is essential; for instance, if $g$ is a power of $f$ we expect no rigidity of jointly invariant measures or smoothness of the conjugacy $h$.

	If neither  $f$ nor $g$ is  Anosov, the map $h$ in \eqref{eq:semicon} may  be non-invertible.   In particular, there may exist  ergodic, $\alpha$-invariant measures $\mu$ on $\T^3$ with $h_\mu(f)>0$ such that the push-forward measure     $h_*(\mu)$ has zero entropy for $\alpha_0(\vecn)$ for every   $\vecn\in \Z^2.$
	Note however that if  $$h_{h_*(\mu)}(\alpha_0(\vecn) )> 0$$ for some $\vecn\in \Z^2$ then $h_*(\mu)$  is necessarily  Haar by Theorem \ref{thm:KS}.

	When the map $h$ in \eqref{eq:semicon} is non-invertible, analysis of measures invariant under the affine action $\alpha_0$  gives less information about   measures jointly invariant under $f$ and $g$.  However, the method of proof of Theorem \ref{thm:KS} can be adapted to study measures jointly invariant under $f$ and $g$; in particular, one can show the  following theorem which is a simplified version of the main results of \cites{MR2261075, MR2285730}.
	\begin{theorem}[{\cites{MR2261075, MR2285730}}] \label{jkjkjlkjo}
		Suppose $f,g\colon \T^3\to \T^3 $ are commuting $C^{1+\beta}$ diffeomorphisms.  Suppose the linear data of $f$ and $g$ are, respectively,  the matrices  $A$ and $B$    in Example \ref{ex:key}.  
		Then any ergodic probability measure $\mu$ that is invariant under both $f$ and $g$ and  such that  $h_*(\mu)$ is Haar 
		is absolutely continuous with respect to the Riemannian volume on $\T^3$.  Such a measure always exists and is, moreover,  unique.  
	\end{theorem}

	For actions on more general manifolds,  there may be  no   a priori semiconjugacy between the nonlinear action and an affine Anosov  action.  However, under certain dynamical hypotheses on the action, the structure of the algebraic toral action can be reconstructed.  Consider a $\Z^2$-action $\alpha$ on a 3-manifold $M$ generated by two commuting diffeomorphisms $f,g\colon M\to M$.  Given an ergodic, $\Z^2$-invariant probability measure $\mu$, one can define  {Lyapunov exponent functionals} for the $\Z^2$-action as in Theorem \ref{thm:oscHR}.  These extend to linear functionals on $\R^2$.  Note that there are at most 3 (the dimension of $M$) Lyapunov exponent functionals.  Under some genericity assumptions on the Lyapunov exponent functionals, an analogue of \cref{thm:KS} and \cref{jkjkjlkjo} was obtained in \cite{MR2811602}.
	\begin{theorem}[\cite{MR2811602}]\label{KKRH}
		Let $\alpha $ be a $\Z^2$-action by $C^{1+\beta}$-diffeomorphisms of a 3-manifold and let $\mu$ be an ergodic, $\alpha$-invariant measure.  Assume there are 3 nonzero,  Lyapunov exponent functionals $\lambda_\mu^1, \lambda_\mu^2, \lambda_\mu^3$ and that no pair of exponents is proportional.  
		
		If some element $\alpha(n_1, n_2)$ has positive entropy with respect to $\mu$, then $\mu$ is absolutely continuous with respect to the Riemannian volume on $M$.
	\end{theorem}
	In \cite{MR3503686}, it is shown  in the setting of Theorem \ref{KKRH} that one can   reconstruct an action by (infra-)toral automorphism and a measurable semiconjugacy $h$ between the non-linear action $\alpha$ (restricted to a finite-index subgroup of $\Z^2$) and the algebraic action.  Moreover, the semiconjugacy $h$ takes (an ergodic component of) $\mu$ to the Lebesgue measure on the (infra-)torus, is differentiable along stable manifolds, and is differentiable (in the Whitney sense) off   sets  of arbitrarily small measure.  
	This, in particular, implies that the exponents $\lambda^i_\mu(n)$ are logarithms of algebraic numbers for every $n\in \Z^2$.  
	
	\section{Invariant measures for Cartan flows}\label{sec:WCF}
	In  Section \ref{sec:diag}, we introduced an important example of a higher-rank, continuous-time algebraic  Anosov action, namely, the    diagonal action (or Cartan flow)\index{Cartan flow} on a higher-rank semisimple homogeneous space.  We review its properties, referring  back to Section \ref{sec:slstructure} for details. \begin{example}\label{ex:2}
		Let $G= \Sl(3, \R)$  and let $\Gamma = \Sl(3,\Z)$ or any lattice in $G$.  
		Let $X$ denote the coset space $X= G/\Gamma$.  This is an $8$ dimensional  manifold (which is noncompact for $\Gamma=\Sl(3,\Z)$.)  
		$G$ acts on $X$ by  left translation.

		The group  $A\subset G$ of diagonal matrices with positive entries  is isomorphic to $\R^2$. 
		The action  $\alpha \colon A\times X\to X$ of $A$ on $X$  is given by $\alpha (s)(x) = s  x.$
		There are $6$ roots $\beta^{i,j}\colon A\to \R$ given by $\beta^{i,j}(\diag(e^{t_1}, e^{t_1}, e^{t_3}) )= t_i - t_j$ each with an associated root subgroup $U^{i,j}\subset G$.  For $x\in X$,  $W^{i,j}(x)$ is the orbit of $x$ under the 1-parameter group $U^{i,j}$:
		$$W^{i,j}(x) = \{ U^{i,j}\cdot x: t\in \R\}.$$
		For $s\in A$, the action  $\alpha(s)$ dilates distances in  $W^{i,j}(x)$ by exactly $e^{\beta^{i,j}(s)}$.  
	\end{example}

	One might ask whether an analogue of Theorem \ref{thm:KS} holds in Example \ref{ex:2}.  That is, if $\mu$ is an ergodic, $A$-invariant probability measure on $X$ such that there is some $s\in A$ with $h_\mu(\alpha(s))>0$, is $\mu$ necessarily the Haar measure  on  $X$ or on   a homogeneous submanifold of $X$?

	The answer is no.  The extension of the proof of Theorem \ref{thm:KS} breaks down in this setting as the trick in Lemma \ref{lem:Pipart}  fails. 
	Indeed,   for every root  $\beta^{i,j}$,  we have that $\beta^{i,j}$ and $\beta^{j,i}= - \beta^{i,j}$  are negatively proportional.  
	Moreover, explicit examples of diagonally invariant measures with positive entropy (for some element of the diagonal) on spaces of the form $\Sl(3,\R)/\Gamma$ for certain (cocompact) lattices $\Gamma$ were constructed by Mary Rees  in an unpublished manuscript.  See \cite{MR1989231} for detailed constructions of such measures.  
		
{A related problem is the classification of orbit closures for the $A$-action on $X=\Sl(n,\R)/\Gamma$.  Rees's construction yields  $A$-orbit closures in $\Sl(n,\R)/\Gamma$ that are non-homogeneous and of intermediate Hausdorff dimensions for certain lattices $\Gamma$.   It was shown in \cite{MR2630049} that there exist non-homogeneous $A'$-orbit closures in $\SL(n,\R)/\Sl(n,\Z)$ (for $n\ge 6$) for certain higher-rank subgroups $A'$ of the full diagonal group $A$.}


	Returning to the classification of invariant measures,  in the case that $\Gamma=\Sl(3,\Z)$, the \index{conjecture!Margulis} {\it Margulis conjecture} asserts that all ergodic $A$-invariant measures  $\mu$ on $X$ should be algebraic.  See \cite[Conjecture 1.1]{MR2247967} and discussion in \cite[\S 1.2]{MR1754775}.  For measures  with positive entropy, this  conjecture was solved   in \cite{MR2247967} (see \cref{thm:EKL} below).  We outline the main results used in \cite{MR2247967}, namely the \textit{high} and \textit{low entropy} methods.  
		
	To discuss the high and low entropy methods, first note that there are some key differences in the structure of the foliations in this setting versus the setting of Example \ref{ex:key}.    First note that any two transverse foliations $W^i,W^j$  of the torus $\T^3$ by lines are jointly integrable; that is there is a foliation of $\T^3$ by planes $W^{i,j}$ with $W^i(x) \subset W^{i,j}(x)$ and $W^j(x) \subset W^{i,j}(x)$ for all $x$.  This follows as $\T^3$ has an abelian group structure.  In $X= \Sl(3,\R)/\Gamma$, Lyapunov foliations do not jointly integrate as the corresponding subgroups may not commute.  For instance, the subgroups $U^{1,2}$ and $U^{2,3}$ do not commute and thus the foliations $W^{1,2}(x)$ and $W^{2,3}(x)$ do not jointly integrate.  This is the key idea behind the \emph{high entropy method}. 
	Moreover,   translations along Lyapunov directions $E^i$ are isometries in the torus $\T^3$.  
	In $X= G/\Gamma$, translation by an element  of a 1-parameter subgroup $U^{i,j}$ is  {not}  isometric; there is some polynomial shearing.  This is  a key step in the proof of Ratner's measure classification theorem for unipotent flows (see \cite{MR1262705}) and is also a key idea in the \emph{low entropy method}.  
	
	We state the  versions of the high entropy and low entropy methods for Example \ref{ex:2}.  Given a measure $\mu$ on $X$, for $i\neq j$  let $\mu^{i,j}_x$ denote the locally finite leaf-wise measures obtained by conditioning $\mu$ along   $W^{i,j}$-manifolds.  		
	\begin{theorem}[High entropy method \cite{MR1989231}] \label{thm:HE}
		Let $\mu$ be an ergodic, $A$-invariant measure on $X=\Sl(3,\R)/\Gamma$.  
		Let $i, j,$ and $k$ be distinct elements of   $\{1,2,3\}$.  
		If $\mu^{i,j}_x$ and  $\mu^{j,k}_x$ are nonatomic for a positive measure set of $x$ then $\mu$ is $U^{i,k}$-invariant.
	\end{theorem}
	Note that the subgroups $U^{i,j}$ and $U^{j,k}$ do not commute; precisely, we have $[U^{i,j}, U^{j,k}]= U^{i,k}$.  Theorem \ref{thm:HE} states that if both Lyapunov exponents $\beta^{i,j}$ and $\beta^{j,k}$ contribute entropy (so that $\mu^{i,j}_x$ and  $\mu^{j,k}_x$ are nonatomic) then the measure $\mu$ is invariant under their bracket   $U^{i,k}= [U^{i,j}, U^{j,k}]$.  The noncommutativity of $U^{i,j}$ and $U^{j,k}$  is essential   in the proof of the  theorem. 
	
	From Theorem \ref{thm:HE} one can derive the following corollary.  
	\begin{theorem} [{\cite[Theorem 4.1]{MR1989231}}] \label{thm:HE;}
		Let $\Gamma\subset \Sl(n,\R)$ be a lattice and 
		let $\mu$ be an ergodic, $A$-invariant measure on $X=\Sl(n,\R)/\Gamma$ such that for every nontrivial $s\in A$, $$h_\mu(\alpha(s))>0.$$  Then $\mu$ is the Haar measure on $X$. 
	\end{theorem}

	We note that the statements of    Theorems \ref{thm:HE} and \ref{thm:HE;} 
	are specific for the group $\Sl(n,\R)$.  More general high entropy methods appear in \cite{MR2191228}.

	The low entropy method is a bit more difficult to state.   We state a version for $X= \Sl(3,\R)/\Gamma$.   For each $i\neq j$, let $A'_{i,j}$ denote the kernel of $\beta^{i,j}$ 
	and let $C(A'_{i,j})$ denote the centralizer of $A'_{i,j}$ in $G= \Sl(3,\R).$   
	\begin{theorem}[{Low entropy method, \cite[Theorem 2.3.]{MR2247967}}]\label{thm:LE}
		Let $\mu$ be an ergodic, $A$-invariant measure on $X= \Sl(3,\R)/\Gamma$.  
		If $\mu^{i,j}_x$ and $\mu^{j,i}_x$ are nonatomic for some $i,j$ and  $\mu^{i',j'}_x$ is atomic for all other pairs $i',j'$ then either 
		\begin{enumcount}
			\item $\mu$ is $U^{i,j}$-invariant, or 
			\item \label{LE2}  there is  $x_0\in X$ and $s\in A'_{i,j}$ with $\alpha(s)( x_0) = x_0$ such that  $\mu$ is supported on the orbit  $C(A'_{i,j}) x_0$.  
		\end{enumcount}
	\end{theorem}
	Conclusion \ref{LE2} is specific for the case that $X$ is of the form $X= \Sl(n,\R)/\Gamma$ for $n = 3$.  For $n> 3$, the appropriate version of  \ref{LE2} is slightly  more complicated.

	To show that \cref{thm:HE} and \cref{thm:LE} cover all cases it is shown \cite[Corollary  3.4]{MR2247967} for any  ergodic, $A$-invariant measure $\mu$ and  every pair $i,j$, the measure $\mu^{i,j}_x$ is nonatomic if and only if $\mu^{j,i}_x$ is nonatomic.  Thus every $A$-invariant measure $\mu$ with positive entropy is considered in either \cref{thm:HE} and \cref{thm:LE}.

	To apply the low entropy method, one typically does additional work to  rule out conclusion \ref{LE2} in \cref{thm:LE}. Note that conclusion \ref{LE2} of \cref{thm:LE} occurs in    Rees's examples so it can not be ruled out for all lattices $\Gamma$.     However, for certain lattices, it can be shown that \ref{LE2} of Theorem \ref{thm:LE} does not happen.  In particular, this is verified for $\Gamma =\SL(n,\Z)$ in \cite{MR2247967}.  
	The high and low entropy method combine to give the following.
	\begin{theorem}[{\cite[Theorem 1.3, Corollary 1.4]{MR2247967}}]\label{thm:EKL}
		Let $\mu$ be an ergodic, $A$-invariant measure on $X= \Sl(3, \R)/\Sl(3, \Z)$.  Assume $\mu$ has positive entropy for some nontrivial element of $A$.  Then $\mu$ is the Haar measure on $X$. 
	\end{theorem}
	Note the conclusion that $\mu$ is the Haar measure above follows as we assume $n=3$ which is prime.  For the  general result on $\Sl(n,\R)/\Sl(n,\Z)$, the conclusion is that $\mu$ is algebraic.

		The study of invariant measures and orbit closures for various subgroups $H\subset \Sl(n,\R)$  acting on $X=\Sl(n,\R)/\Sl(n,\Z)$ is related to several important problems  in number theory.  See for instance the proof of Margulis's proof  \cites{MR882782,MR993328} of  the Oppenheim conjecture which reduces to the study of $H= \SO(2,1)$-orbit closures in $\Sl(3,\R)/\Sl(3,\Z)$.  See \cite[Section 1.2]{MR2158954} for further discussion.  One motivation for studying $A$-invariant measures on $\Sl(n,\R)/\Sl(n,\Z)$ is its relationship to Littlewood's conjecture.  An important consequence of \cref{thm:EKL} is that the set of values for which Littlewood's conjecture fails has Hausdorff dimension zero.  See \cite{MR2247967} as well as  \cite[\S 1.2]{MR1754775} and \cite{MR2358379} for details.


\part*{References}
			

		\bigskip
		
		\noindent \textsc{Aaron Brown}\\
		University of Chicago, Chicago, IL 60637, USA \\
		\texttt{awb@uchicago.edu}


\clearpage

\part*{Appendices}
	
\appendix

We discuss some classical notions and results that may help reading the main text, without browsing through the literature on the subject.

\section{Furstenberg's Theorem (by Dominique Malicet)}\label{App:furstenberg}

Here we give a self-contained proof of Furstenberg's Theorem (Theorem \ref{thm:furst}), mainly following the original proof in \cite{furstenberg}.

\subsection{Notations and statement}
Let $S^1$ be the $1$-dimensional torus $\mathbb{R}/\mathbb{Z}$. For $\alpha$ in $S^1$ we denote by $T_\alpha:S^1\rightarrow S^1$ the translation operator defined by $T_\alpha(x)=x+\alpha \mod 1$.
For $n$ in $\mathbb{N}$ we denote by $M_n:S^1\rightarrow  S^1$ the multiplication operator defined by $M_n(x)=nx \mod 1$.
\begin{theorem}[Furstenberg]\label{Furstenberg}\index{theorem!Furstenberg}
	Let $a$ and $b$ be two positive integers which are not powers of the same integer, and let $F$ be a closed subset of $ S^1$ invariant by $M_a$ and $M_b$: $M_a(F)=M_b(F)=F$. Then either $F$ is finite or $F= S^1$.
\end{theorem}
\begin{remark}\label{r:Furstenberg}~
	\begin{enumerate}
		\item In the case where $F$ is an invariant finite set, it is actually a set of rational numbers. (Indeed if $|F|=\ell$ and $x$ is a point of $F$ then there exists $n\le \ell$ with $a^nx=x$ modulo $1$.)
		\item The conclusion does not hold if the closed set is invariant by only one transformation $M_a$. For example the triadic Cantor set is invariant by $M_3$.
		\item We can reformulate the theorem as follows: if $a$, $b$ are integers which are not powers of the same integer, then for any irrational number $x$ the set $\{a^mb^n x,(m,n)\in\mathbb{N}^2\}$ is dense modulo $1$.
	\end{enumerate}
\end{remark}

\subsection{Proof of the theorem}
We follow the proof of Furstenberg (except that we try to avoid the unnecessary use of the existence of minimal invariant closed subsets). For the whole proof, we fix integers $a$ and $b$ which are not powers of the same integer. It is equivalent to say that $\log a$ and $\log b$ are independent over $\mathbb{Q}$. Let $F$ be a closed subset of $ S^1$ invariant by $M_a$ and $M_b$. If $F$ is infinite, it means that it has some accumulation point, and we want to deduce that actually $F=S^1$. We divide the proof into two parts:

\begin{enumerate}
	\item The first part treats the particular case where the accumulation point of $F$ is a rational number. ``Spreading'' points of $F$ close to this rational number by using $M_a$ and $M_b$, we manage to prove that $F= S^1$, mainly by combinatorial techniques.
	
	\item The second part treats the general case where the accumulation point can be irrational. The idea here is to use translations $T_\alpha$ commuting with $M_a$ and $M_b$, and to prove that there is ``some $T_\alpha$-invariance'' in $F$. The first treated case will help at some key points. The following fact can be checked by a simple computation:
	
	\begin{lemma}\label{commute}
		A translation $T_\alpha$ commutes with $M_a$ and $M_b$ if and only if $(a-1)\alpha=(b-1)\alpha=0 \mod 1$, or equivalently if $\alpha$ is a rational number (modulo $1$) whose denominator divides $a-1$ and $b-1$.
	\end{lemma}
	This condition on $\alpha$ is too restrictive to be useful (there is only a finite number of solutions, and even no solution at all if $a-1$ and $b-1$ are coprime!). That is why we will actually use translations commuting with some large powers of $M_a$ and $M_b$.
\end{enumerate}

\subsubsection{The particular case}
In this part we prove the following weak version of the theorem:
\begin{proposition}\label{particulier}
	If $F$ is closed, invariant by $M_a$ and $M_b$ and has some rational number $\frac{p}{q}$ as an accumulation point, then $F= S^1$.
\end{proposition}
The proof relies on the following combinatorial lemma, which is actually the only step where we use that we have two transformations $M_a$ and $M_b$ instead of one.
\begin{lemma}\label{enumerate}
	Let us enumerate the set $S=\{a^mb^n,(m,n)\in\mathbb{N}^2\}\subset \mathbb{N}$ by an increasing sequence of integers $(s_k)_{k\in\mathbb{N}}$. Then $\lim_{k\to +\infty}\frac{s_{k+1}}{s_k}=1$
\end{lemma}
\begin{proof}
	Let $\varepsilon$ be any positive number. The additive group generated by $\log a$ and $\log b$ is dense in $\mathbb{R}$ (since $\log a$ and $\log b$ are independent over $\mathbb{Q}$) hence one can find a finite set $A\subset \mathbb{Z}^2$ such that $\{m\log a+n\log b,(m,n)\in A\}$ is $\varepsilon$-dense in $[0,1]$. Then, if $p_0$ is large enough, we have that the set $\{(m+p)\log a+(n+p)\log b,(m,n)\in A,p\geq p_0\}$ is a subset 
	of $\{\log s,s\in S\}$, and it is $\varepsilon$-dense in $[M,+\infty)$ where $M=p_0\log a+p_0\log b$. Thus if $\log s_k \geq M$, then $\log s_{k+1}\leq \log s_k+\varepsilon$. We conclude that $\lim_{k\to +\infty}\log s_{k+1}-\log s_k=0$ and hence that $\lim_{k\to +\infty}\frac{s_{k+1}}{s_k}=1$.
\end{proof}
\begin{proof}[Proof of Proposition \ref{particulier}]
	
	Take the set $S$ of Lemma~\ref{enumerate} and enumerate it by the increasing sequence $(s_k)$.
	We denote by $x\mapsto \bar{x}$ the canonical projection of $\mathbb{R}$ onto~$S^1$.
	
	Let us treat first the case where the accumulation point of $F$ is $0$ (modulo $1$). Then, up to replacing $F$ by $-F$ we assume that for any $\varepsilon>0$, there exists $x_\varepsilon$ in $(0,\varepsilon)$ such that $\bar{x}_\varepsilon$ belongs to $F$. Given $\varepsilon>0$ and  $x\in (\varepsilon,1)$,  let $k_\varepsilon$ be such that $s_{k_\varepsilon}x_\varepsilon\leq x< s_{k_\varepsilon+1}x_\varepsilon$. We have
	\begin{align*}
	d(\bar{x},F)\leq|x-s_{k_\varepsilon} x_\varepsilon|\leq s_{k_\varepsilon+1}x_\varepsilon-s_{k_\varepsilon}x_\varepsilon&=\left(\frac{s_{k_\varepsilon+1}}{s_{k_\varepsilon}}-1\right)s_{k_\varepsilon}x_\varepsilon\\&\leq \left(\frac{s_{k_\varepsilon+1}}{s_{k_\varepsilon}}-1\right)x.
	\end{align*}
	Letting $\varepsilon$ going to $0$ (so that $x$ can be arbitrary in $(0,1)$), we have that $k_\varepsilon\to +\infty$ hence the last term tends to $0$ by the lemma, and we conclude that $\bar{x}$ belongs to $F$. Thus $F=S^1$.
	
	In the general case where the accumulation point of $F$ is a rational number $\frac{p}{q}$, then the point $p=0 \mod 1$ is an accumulation point of $M_q(F)$, and since $M_q$ commutes with $M_a$ and $M_b$, the set $M_q(F)$ is also invariant by $M_a$ and $M_b$, and we deduce by the first case that $M_q(F)= S^1$. As a consequence, we also have that $M_q^{-1}(M_q(F))= S^1$, that is:
	$$F\cup T_{\frac{1}{q}}(F)\cup\cdots\cup T_{\frac{q-1}{q}}(F)= S^1.$$
	Since a finite union of closed sets with empty interiors has empty interior, we conclude that $F$ contains some non trivial interval $I$. But for sufficiently large $n$, $M_a^n(I)= S^1$, hence $F= S^1$ by $M_a$-invariance of~$F$.	
\end{proof}

\subsubsection{The general case}
We establish some lemmas relating $F$ with the dynamics of the translations $T_\alpha$.
\begin{lemma}\label{key}
	Let $F$ be a closed infinite set which is invariant by $M_a$ and $M_b$, and let $T_\alpha$ be any translation. Then $T_\alpha(F)\cap F\not=\emptyset$.
\end{lemma}
\begin{proof}
	Note that 
	$$T_\alpha(F)\cap F\not=\emptyset\Leftrightarrow \alpha\in F-F,$$
	where $F-F=\{x-y,(x,y)\in F\times F\}$. The set $F-F$ is closed and invariant by $M_a$ and $M_b$. Moreover, if $F$ is infinite, then $F$ has some accumulation point $x$ and hence $0=x-x$ is an accumulation point of $F-F$. By Proposition \ref{particulier}, $F-F= S^1$ and hence $T_\alpha(F)\cap F\not=\emptyset$.
\end{proof}
\begin{lemma}\label{invariant}
	Let $F$ be a closed infinite set which is invariant by $M_a$ and $M_b$, and let $T_\alpha$ be a translation commuting with $M_a$ and $M_b$. Then there exists a nonempty closed set $\tilde{F}\subset F$ invariant by $M_a$, $M_b$ and $T_\alpha$.
\end{lemma}
\begin{proof}
	Since $F$ is infinite, the set $F'$ of the accumulation points of $F$ is non empty. Let us define by induction $F_0=F'$ and $F_{n+1}=F_n\cap T_\alpha(F_n)=\bigcap_{k=0}^{n+1} T^k_\alpha(F')$, and let $\tilde{F}=\bigcap_{k=0}^\infty T^k_\alpha(F')$ be the intersection of all the $F_n$'s. The sequence $(F_n)_{n\in\mathbb{N}}$ is a nested sequence of closed sets, all of them invariant by $M_a$ and $M_b$ (because $T_\alpha$ commutes with $M_a$ and $M_b$).
	The intersection $\tilde{F}$ is obviously a closed subset of $F$ invariant by $M_a$, $M_b$. It verifies 
	$T_\alpha(\tilde F)\supset \tilde F$ hence, as $T_\alpha$ is a translation, $\tilde F$ is also invariant by $T_\alpha$. What remains to prove is that all the $F_n$ are non empty, in order to conclude by compactness that $\tilde{F}$ is non empty.
	
	Let us assume by contradiction that $F_n=\emptyset$ for some $n>0$. Choosing $n$ minimal we can assume that $F_{n-1}\not=\emptyset$, and we have $T_\alpha(F_{n-1})\cap F_{n-1}=F_n=\emptyset$. By Lemma \ref{key}, $F_{n-1}$ is a finite set, and in particular it contains only rational numbers (see the first observation in Remark~\ref{r:Furstenberg}). Since $F_{n-1}\not=\emptyset$ and $F_{n-1}\subset F_0=F'$, this means that we can find a rational number in $F'$, so that by Proposition \ref{particulier}, $F'=F= S^1$ and hence $F_n= S^1$, which gives a contradiction and concludes the proof.
\end{proof}
\begin{remark}
	We will use the previous lemma with rational translations $T_\alpha$, and in this case one easily checks that the set $\tilde{F}$ defined in the proof is actually the finite intersection $\tilde{F}=F'\cap T_\alpha(F')\cap\cdots T_\alpha^{k-1}(F')$
	where $k$ is the denominator of $\alpha$ when written in reduced terms.
\end{remark}
We are now ready to prove Theorem \ref{Furstenberg}:
\begin{proof}
	Let $F$ be a closed set invariant by $M_a$ and $M_b$ that we assume infinite. Let $k$ be a large integer coprime with $a$ and $b$, and let $n=\varphi(k)$ be the cardinal of $\left(\mathbb{Z}/k\mathbb{Z}\right)^\times$ so that $a^n=b^n=1 \mod k$. Then, $F$ is invariant by $M_{a^n}=M_a^n$ and $M_{b^n}=M_b^n$, and the translation $T_{\frac{1}{k}}$ commutes with $M_{a^n}$ and $M_{b^n}$ by Lemma \ref{commute}. Applying Lemma \ref{invariant} with $M_{a^n}$ and $M_{b^n}$ instead of $M_a$ and $M_b$, we find $\tilde{F}\subset F$ non empty, invariant by $T_{\frac{1}{k}}$. In particular $\tilde{F}$ is $\frac{1}{k}$-dense, and hence so is $F$. Since $k$ can be chosen arbitrarily large, $F= S^1$.
\end{proof}

\bigskip

\noindent \textsc{Dominique Malicet}\\
LAMA,  Universit\'e  Paris-Est  Marne-la-Vall\'ee, CNRS UMR 8050,\\
5  bd.~Descartes,  77454  Champs  sur Marne, France.\\
\texttt{mdominique@crans.org}

\newpage

		\addtocontents{toc}{\vspace{\normalbaselineskip}}

\section{Measurable partitions, disintegration and conditional measures (by Bruno Santiago and Michele Triestino)}\label{App:rokhlin}
\newcommand{\cC}{\mathcal{C}}
\newcommand{\cP}{\mathcal{P}}
\newcommand{\cM}{\mathcal{M}}
\newcommand{\cB}{\mathcal{B}}
\newcommand{\cQ}{\mathcal{Q}}
\newcommand{\cU}{\mathcal{U}}

%
%
%
Here we discuss more extensively the notion of conditional measures (Section \ref{sec:condi}) and review Rokhlin Disintegration Theorem. We also give some applications, mainly in relation to unstable partitions.		
For a more detailed reference, the reader may consult \cites{viana-oliveira,C1208.4550} (or \cite{rokhlin} as historical reference).\footnote{B.S. thanks Aaron Brown for helpful conversations and the organizing committee of the conference ``Workshop for young researchers: groups acting on manifolds'' for the opportunity of participating in this wonderful meeting.}  


\subsection{Introduction}

Consider a measure space $(X,\cB,\mu)$. Suppose that we \textit{partition} $X$ in an arbitrary way. Is it possible to recover the measure $\mu$ from its \textit{restriction} to the elements of the partition?
We shall address this question, defining the ``restriction'' via a classical theorem of Rokhlin, giving some affirmative answer and applying this idea to obtain interesting results.

Let us start with a simple (positive) example. 

\begin{example}[Figure \ref{toro}]
	Consider the $2$-torus $\mathbb{T}^2={S}^1\times{S}^1$, endowed with the Lebesgue measure $m$. The torus is partitioned into the vertical sets $\{y\}\times{S}^1$.
	\begin{figure}[ht!]
		\centering
		\includegraphics{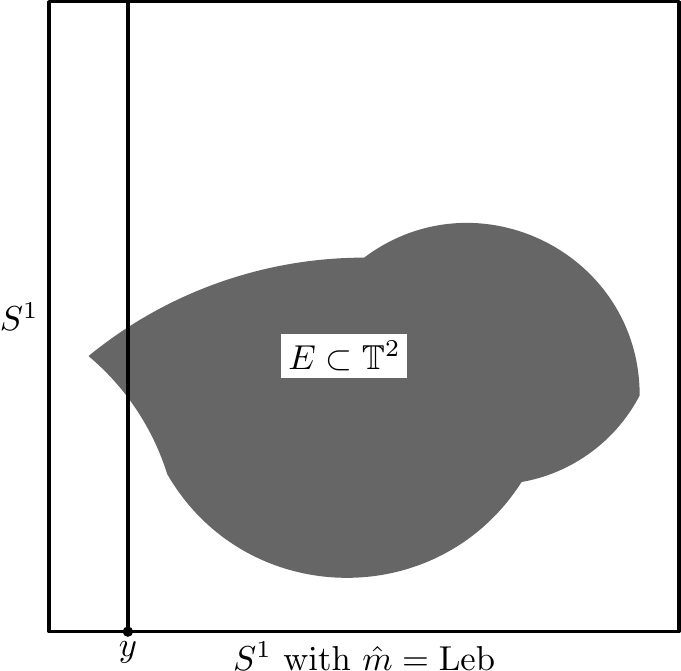}
		\caption{}
		\label{toro}
	\end{figure}
	Denote by $m_y$ the Lebesgue measure over the circle $\{y\}\times{S}^1$, and $\hat{m}$ the Lebesgue measure over ${S}^1$. 		
	If $E\subset\mathbb{T}^2$ is a measurable set we know from Fubini theorem that 
	
	\begin{equation}
		\label{e.fubini}
		m(E)=\int_{{S}^1} m_y(E)d\hat{m}(y).
	\end{equation}  
\end{example}  		


As we shall see later, for some simple, dynamically defined, partitions no disintegration  like this exists. In the next sections we shall define formally the notion of a \textit{disintegration} and try to explore a little bit this concept. 

\subsection{Disintegration and Conditional Probability Measures}\label{B:secDisintegration}

Let $Q\subset {S}^1$ be a measurable subset with $\hat{m}(Q)=0$. Notice that if in (\ref{e.fubini}) we choose to calculate the integral over ${S}^1\setminus Q$ only, equality 
is not affected. Thus, it is natural to consider partitions only modulo null sets (sets of measure zero). 

More formally, let $(X,\cB,\mu)$ be a probability space. Let $\cP$ be a partition of $X$ into measurable sets. Let $\pi:X\to\cP$ be the natural projection:
\begin{center}
	\textit{$\pi(x)$ is the unique element of $\cP$ such that $x\in\pi(x)$.}
\end{center}	
We can turn $\cP$ into a measure space $(\cP,\hat{\cB},\hat{\mu})$, by saying
$$Q\in\hat{\cB}\iff\pi^{-1}(Q)\in\cB,$$
and
$$\hat{\mu}(Q)=\mu(\pi^{-1}(Q)).$$
This clearly makes the projection $\pi$ measurable.
\begin{definition}[See Definition \ref{def:condmeas}]\label{B:disintegration}
	A \textit{disintegration}\index{measure(s)!disintegration} of $\mu$ with respect to $\cP$ is a family of probability measures $\{\mu_P;P\in\cP\}\subset\cM_1(X)$ such that for every $E\in\cB$ one has
	\begin{enumerate}
		\item $\mu_P(P)=1$ for $\hat{\mu}$-almost every $P\in\cP$. \label{Bitem:1}
		\item the assignment $P\in\cP\mapsto\mu_P(E)\in\R$ is $\hat{\cB}$-measurable.
		\label{Bitem:2}
		\item $\mu(E)=\int_{\mathcal P}\mu_P(E)d\hat{\mu}(P)$. \label{Bitem:3}
	\end{enumerate}
	Each measure $\mu_P$ is called a \textit{conditional probability measure}.\index{measure(s)!conditional}
\end{definition}	

Observe that property \eqref{Bitem:3} in \cref{B:disintegration} can be reformulated in functional terms: for every $\mu$-integrable Borel function $f:X\to \R$ one has the Fubini property
\[
\int_X f d\mu = \int_{\mathcal P}\int_X fd\mu_{P}d\hat{\mu}(P).
\]

As the lemma below states, in reasonable cases disintegrations are essentially unique.

\begin{lemma}
	\label{l.contable}
	Assume that the $\sigma$-algebra $\cB$ is countably generated. If $\{\mu_P;P\in\cP\}$ and $\{\mu_P^*;P\in\cP\}$ are disintegrations then $\mu_P=\mu_P^*$ for $\hat{\mu}$ almost every $P\in\cP$.
\end{lemma}
\begin{proof}[Sketch of proof]
	Let $\mathcal G\subset \cB$ be a countable generating family, and $\mathcal A\subset \mathcal G$ the subalgebra generated by $\mathcal G$ (which is still countable). Using the properties of disintegrations, one proves that for any $E\in \mathcal A$, the subset $\cP_E:=\{P\in \cP\mid \mu_P(E)\neq \mu_P^*(E)\}$ has $\hat{\mu}$-measure zero. Hence the countable union $\cQ=\bigcup_{E\in\calG}\cP_E$ has $\hat{\mu}$-measure zero. So if $P\notin\cQ$, the measures $\mu_P$ and $\mu_P^*$ coincide on the generating subalgebra $\mathcal A \subset \cB$ and thus (by the monotone class theorem) on $\cB$.
\end{proof}

From this lemma, we deduce the following dynamical property.

\begin{lemma}
	\label{Bprop:disintegration_invariant}
	Let $(X,d)$ be a separable metric space, endowed with the Borel $\sigma$-algebra $\cB$ and a probability measure $\mu$. Let $(f,X,\cB,\mu)$ be a measure preserving system. Assume that there exists a partition $\cP$ of $X$ into measurable invariant subsets, such that $\mu$ admits a disintegration with respect to $\cP$. Then for $\hat{\mu}$ almost every $P\in \cP$ the conditional measure $\mu_P$ is $f$-invariant.
\end{lemma}

\begin{proof}
	We want to prove that the family $\{f_*\mu_P;P\in\cP\}$ is also a disintegration, so that Lemma~\ref{l.contable} gives $f_*\mu_P=\mu_P$ for almost every $P\in\cP$.
	
	As $P\in\cP$ is $f$-invariant, one has $f_*\mu_P(P)=\mu_P(f^{-1}(P))=\mu_P(P)=1$, so \eqref{Bitem:1} in \cref{B:disintegration} is verified. Fix $E\in\cB$.
	Clearly the assignment $P\in\cP\mapsto f_*\mu_P(E)$ is $\hat{\cB}$-measurable, thus \eqref{Bitem:2} is verified. For \eqref{Bitem:3}, invariance of $\mu$ gives $\mu(E)=\mu(f^{-1}(E))$ and thus
	\begin{align*}
		\mu(E)=\mu(f^{-1}(E))&=\int_{\mathcal P}\mu_P(f^{-1}(E))d\hat{\mu}(P)\\
		&=\int_{\mathcal P}f_*\mu_P(E)d\hat{\mu}(P).\qedhere
	\end{align*}
\end{proof}

\begin{example}
	\label{ex.finito}
	Let $\cP=\{P_1,\ldots,P_n\}$ be a finite partition of $(X,\cB,\mu)$. Assume that no element of this partition has zero measure. Define probability measures $\mu_i$ supported on $P_i$ by the expression
	$$\mu_i(E)=\frac{\mu(E\cap P_i)}{\mu(P_i)},\:\textrm{for every}\:E\in\cB,i=1,\ldots,n.$$
	This defines the conditional probability measures.
	Indeed, we have $\hat{\mu}(\{P_i\})=\mu(P_i)$ and 
	$$\mu(E)=\sum_{i=1}^n\mu(P_i)\frac{\mu(E\cap P_i)}{\mu(P_i)}=\sum_{i=1}^n\hat{\mu}(\{P_i\})\mu_i(E).$$ 
	In the same way, we can show that every countable partition admits a disintegration.
\end{example}

There are very natural examples of partitions for which no disintegration exists at all. 

\begin{example}[cf.~\cref{ex:partition_orbits}]
	\label{ex.circle}
	Let $\theta\in\R\setminus\mathbb{Q}$ be an irrational number, $m$ be the normalised Lebesgue measure on the unit circle ${S}^1$, equipped with the Borel $\sigma$-algebra. We consider $R_{\theta}:{S}^1\to{S}^1$, defined as the circle rotation by an angle $2\pi\theta$.
	Let $\mathcal O=\left\{\{R_{\theta}^n(x)\}_{n\in\Z};x\in{S}^1\right\}$ be the partition into $R_\theta$-orbits, with induced measure $\hat m$. We claim that this partition admits no disintegration. Indeed, assume that there exists $\{\mu_P;P\in\mathcal O\}$, a disintegration of the Lebesgue measure with respect to this partition. 
	By \cref{Bprop:disintegration_invariant}, $\hat{\mu}$ almost every measure $\mu_P$ is $R_\theta$-invariant. Moreover $\mu_P(P)=1$, but this is a contradiction because no invariant probability measure can give positive mass to a countable infinite	 set (orbits are countable).
	
	More generally, given any ergodic system $(f,X,\cB,\mu)$, where $(X,d)$ is a separable metric space and $\cB$ the Borel $\sigma$-algebra, one has that the partition into orbits $\mathcal O$ admits no disintegration.
\end{example}

\subsection{Measurable Partitions}\label{B:secMeasurablePart}

In this section we shall define a class of partitions for which we  can always find a disintegration. Recall\index{partition!partial order} from Section~\ref{sec:parord} that a partition $\cP$ is \textit{finer} than a partition $\cQ$, which we denote by $\cQ\prec\cP$, if there exists a full measure subset $X_0\subset X$ such that for $\mu$-almost every $x\in X$ one has,
$$\cP(x)\cap X\subset \cQ(x)\cap X,$$
where $\cP(x)$ (resp.~$\cQ(x)$) denotes the atom of the partition $\cP$ (resp.~$\cQ$) containing $x$.
Given two partitions $\cP$ and $\cQ$, we denote by $\cP\vee \cQ$ the smallest partition that refines both $\cP$ and $\cQ$ (this is the \textit{join} introduced in Section~\ref{ss:condentropy}).\index{partition!join}
\begin{definition}[cf.~Section~\ref{ss:measpart}]\label{B:meas_partition}
	Let $(X,\cB,\mu)$ be a probability space.
	A partition $\cP$ is \textit{measurable}\index{partition!measurable} if there exists $X_0\subset X$ with $\mu(X_0)=1$ and a nested sequence of countable partitions $\cP_1\prec\cP_2\prec\cdots\prec\cP_n\prec\cdots$ of $X_0$ such that $\cP|_{X_0}=\bigvee_{n=1}^{\infty}\cP_n$. In other words, for every $P\in\cP$ there exists a sequence $P_n$, with $P_n\in\cP_n$ such that $P\cap X_0=\bigcap_{n=1}^{\infty}P_n$.	 
\end{definition}

Thus a measurable partition can be described as the joining of a nested sequence of countable partitions. Recall from Example \ref{ex.finito} that countable partitions always admit a disintegration. 

From this fact and from a suitable martingale argument, one can prove the following fundamental theorem.

\begin{theorem}[Rokhlin Disintegration Theorem]\index{theorem!Rokhlin disintegration}
	\label{t.rokhlin}
	Let $(X,d)$ be a complete and separable metric space, endowed with the Borel $\sigma$-algebra $\mathcal B$. Let $\mu$ be any probability measure on $(X,\mathcal B)$ and $\cP$ be a measurable partition. Then, there exists $\{\mu_P;p\in\cP\}$, a disintegration of $\mu$.
\end{theorem} 

Let us see some examples of partitions which are, and which are not, measurable.

\begin{example}[Figure~\ref{circle}]
	In the two torus $\mathbb{T}^2={S}^1\times{S}^1$, consider for each pair $i,n$, with $n$ a positive integer and $i\in\{1,2,3,\ldots,2^n\}$, the interval $J(i,n)=[\frac{i-1}{2^n},\frac{i}{2^n}]$. 
	\begin{figure}
		\centering
		\includegraphics{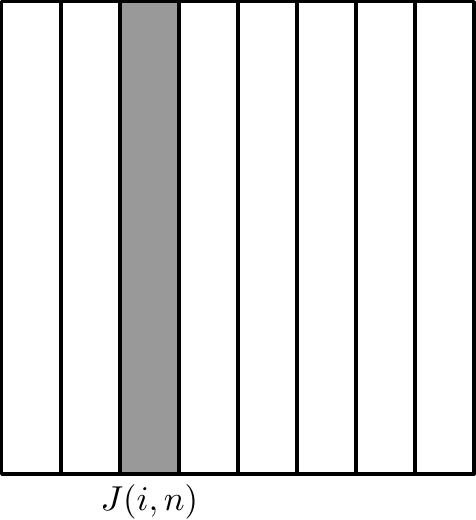}
		\caption{A measurable partition of $\mathbb{T}^2$.}
		\label{circle}
	\end{figure}
	
	Then, the partition $\cP_n=\{{S}^1\times J(i,n)\}$ is a measurable partition.
\end{example}	

\begin{example}[cf.~Example \ref{ex:truck}]
	\label{anosov}
	Denote by $m$ the Lebesgue measure on the two-dimensional torus $\mathbb T^2$. Let $f_A:\mathbb{T}^2\to\mathbb{T}^2$ be the Anosov diffeomorphism induced by the integer matrix $$A=\begin{bmatrix}
	2&1\\
	1&1 
	\end{bmatrix}.$$
	Let $\cP=\{W^u(x);x\in\mathbb{T}^2\}$ be the partition into unstable manifolds. We claim that $\cP$ is not measurable. Indeed, if $\cP$ were measurable, as it is the partition into orbits of an irrational flow, $\cP=\bigvee_{n=1}^{\infty}\cP_n$ would imply that for each $n$ there exists $P_n\in\cP_n$, with $m(P_n)=1$. Thus, the set $P=\bigcap_{n=1}^{\infty}P_n$ belongs to the partition $\cP$,
	and $m(P)=1$, which is absurd.  (This is the continuous-time version of Example~\ref{ex.circle}.) 
\end{example}	

\subsection{Ergodic Decomposition of Invariant Measures}\label{B:secErgDecomposition}

We proceed to give an important application of the disintegration theorem, namely the decomposition of invariant measures into ergodic measures.

Let $(X,\cB,\mu)$ be a probability space and $f:X\to X$ be a measurable map such that $f_*\mu=\mu$. We say that the measure preserving system $(f,X,\cB,\mu)$ is ergodic if every measurable $f$-invariant set has either zero or full $\mu$-measure. 

The goal of this section is to prove the following military principle: divide the space to conquer the ergodic decomposition. 

\begin{theorem}[Ergodic Decomposition; see \cref{def:ergdecom}] 
	\label{t.decomposicaoergodica}\index{ergodic decomposition}
	Let $(X,d)$ be a complete and separable metric space, endowed with the Borel $\sigma$-algebra $\cB$ and a probability measure $\mu$. Let $(f,X,\cB,\mu)$ be a measure preserving system. Then there exists a measurable partition $(\calE,\hat{\cB},\hat{\mu})$, with $f$-invariant atoms, whose disintegration $\{\mu_P;P\in\calE\}$ satisfies that $\hat{\mu}$-almost every $\mu_P$ is $f$-invariant and ergodic.
\end{theorem}

Furthermore, one can prove by measure theoretical arguments that the ergodic decomposition $\calE$ given by the theorem is essentially unique, in the sense that any other ergodic decomposition $(\calE',\hat\cB',\hat{\mu}')$ is measurably isomorphic to $(\calE,\hat{\cB},\hat{\mu})$ (the isomorphism is even Borel in restriction to conull subsets, see for example \cite{Schmidt}).

\smallskip

The idea for \cref{t.decomposicaoergodica} is that an ergodic system is dynamically indecomposable, since its orbits spread uniformly over the configuration space, and thus it is possible to split $X$ into the indecomposable components of the dynamics (see Example~\ref{ex.circle}).   Let us see this more closely by recalling a fundamental result in ergodic theory.

\subsubsection{The Birkhoff's ergodic theorem}

Consider the following statistical question: given a point $p\in X$ and a certain positive measure set $A\subset X$, how often does the forward $f$-orbit of $x$ visit $A$?

From a more formal point of view this means to study the behavior of the sequence 
$$\frac{1}{n}\sum_{j=0}^{n-1}\chi_A(f^j(x)).$$
So, it is natural to ask: does this sequence converges? If so, to what limit? 

%

\begin{theorem}
	\label{t.ergodic}\index{theorem!Birkhoff's ergodic}
	Let $(f,X,\cB,\mu)$ be a measure preserving system, where $\mu$ is a probability measure. Then for every measurable set $A\subset X$ the limit   
	$$\frac{1}{n}\sum_{j=0}^{n-1}\chi_A(f^j(x))$$
	exists for $\mu$-almost every $x\in X$.
\end{theorem}   

It is not hard to show (though we will not do this here) that the ergodic theorem implies the following.

\begin{corollary}
	\label{c.ergodic}
	A measure preserving system $(f,X,\cB,\mu)$ is ergodic if and only if
	$$\lim_{n\to\infty}\frac{1}{n}\sum_{j=0}^{n-1}\chi_A(f^j(x))=\mu(A),$$
	for $\mu$-almost every $x\in X$, and every measurable set $A$.
\end{corollary}

\subsubsection{Proof of Theorem~\ref{t.decomposicaoergodica}}

As we said before, we need to divide the space to conquer the ergodic decomposition. So, our first task is to choose a suitable partition of $X$. Let $\cU\subset \cB$ be a countable basis for the topology of $X$, and $\cA\subset\cB$ the algebra generated by $\cU$. Notice that $\cA$ is countable and generates $\cB$. 

Then the ergodic theorem implies that for each $A\in\cA$ there exists $X_A\subset X$ with $\mu(X_A)=1$ and such that for every $x\in X_A$ the limit
$$\tau(A,x)=\lim_{n\to\infty}\frac{1}{n}\sum_{j=0}^{n-1}\chi_A(f^j(x))$$
exists. Take $X_0=\bigcap_{A\in\cA}X_A$. Then $\mu(X_0)=1$.

We define the following equivalence relation on $X_0$: $x\sim y$ if, and only if, $\tau(A,x)=\tau(A,y)$, for every $A\in\cA$. 

\begin{lemma}
	\label{l.claim}
	The partition $\calE=\{[x];x\in X_0\}$ of $X_0$ into $\sim$-equivalence classes is measurable.	
\end{lemma}

We shall first finish the proof of Theorem~\ref{t.decomposicaoergodica} assuming Lemma~\ref{l.claim}.

\begin{proof}[Proof of \cref{t.decomposicaoergodica}]
	Let $\calE$ be the measurable partition from \cref{l.claim} and $\{\mu_P;P\in\calE\}$ be the associated disintegration. 
	Observe that $\tau(A,x)=\tau(A,f(x))$ for every $x\in X$, $A\in\cA$, thus every atom of the partition $\calE$ is $f$-invariant. By \cref{Bprop:disintegration_invariant}, almost every $\mu_P$ is $f$-invariant.
	It remains to prove that almost every $\mu_P$ is ergodic. Fix $P\in\calE$ and consider
	$$\cC=\{E\in\cB\mid \tau(E,x)\:\textrm{is defined and constant for every}\:x\in X_0\cap P\}.$$
	Notice that  $\cA\subset\cC$, by definition of $\calE$. Moreover, if $E_2\subset E_1$ are elements of $\cC$ then 
	$$\tau(E_1\setminus E_2,x)=\tau(E_1,x)-\tau(E_2,x)$$
	is well-defined and constant over $X_0\cap P$. If $\{E_i\}$ are pairwise disjoint sets then
	$$\tau\left (\bigcup_{i=1}^\infty E_i,x\right )=\sum_{i=1}^{\infty}\tau(E_i,x)$$
	is well-defined and constant over $X_0\cap P$. We conclude that $\cC$ is a monotone class (it is stable under increasing unions and decreasing intersections). By the monotone class theorem we conclude that $\cC=\cB$. By Corollary~\ref{c.ergodic} we deduce that $\mu_P$ is ergodic.	
\end{proof} 

\begin{proof}[Proof of Lemma~\ref{l.claim}]
	Let $\cA=\{A_k\}$ be an enumeration of $\cA$ and $\{q_k\}=\mathbb{Q}$ be an enumeration of the rational numbers. Fix $n\in\mathbb{N}$. We define a partition $\cP_n$ in the following way: we mark the points $q_1,\ldots,q_n$ on the line and consider the partition of $\R$ into intervals induced by these points. 
	\begin{figure}[ht!]
		\centering
		\includegraphics{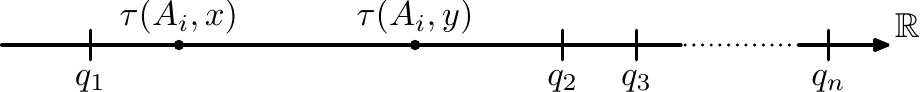}
		\caption{The partition $\cP_n$.}
		\label{lp}
	\end{figure}	
	We declare $x\sim_n y$ if and only if $\tau(A_i,x)$ and $\tau(A_i,y)$ belong to the same interval of this partition for every $i=1,\ldots,n$ (see figure \ref{lp}).  Clearly, $\tau(A_i,x)=\tau(A_i,y)$ for every $i$ if and only if $x,y\in\bigcap_{n=1}^{\infty}P_n$, with $P_n\in\cP_n$, and thus
	$\calE=\bigvee_{n=1}^{\infty}\cP_n$.
\end{proof}

%
%
%
%
%

\subsection{Non-measurable unstable partitions}
\index{measure(s)!conditional}
\cref{ex:unstable_partition_measurable} says that in general unstable partitions are not measurable. However, as we have seen in Section~\ref{sec:condi}, it is possible to define conditional measures of a measure $\mu$ conditioned on the unstable (or stable) partition. The drawback is that the conditional measures are not likely to be probability measures, and moreover they are defined only up to a multiplicative constant.  The construction is classical and goes back at least to \cites{LedrappierYoung1,LedrappierYoung2}. Since then, it appears as a standard tool in many important measure rigidity results, as we have seen for Theorem \ref{thm:KS} \cites{KatokSpatzier,KatokSpatzierC}, but one also finds it in \cites{BQ,MaT,EKL,EM,BRH}, just to cite a few. 

The idea described in Section \ref{sec:condi} consists into considering an exhaustion of the unstable partition $W^u$ by a sequence of subordinate measurable partitions~$\xi^k$ of $\mu$ conditioned on~$\xi_k$. This gives conditional  measures $\{\mu^{(k)}_P\mid P\in \xi^k\}$ that one chooses to renormalize so that the unit ball at $x$ (with respect to the intrinsic metric of $W^u(x)$), has measure $1$ as soon as this ball is contained in $\xi^k(x)$.
Then one is able to take a limit of these conditional measures.

Although the construction is somehow delicate, rigorous treatments appear rarely in the literature. A very abstract approach can be found in \cite[Section 4.1]{BQ} (in French), carefully explained in \cite{dufloux}. As conditional measures take a central place in these notes, we recall the construction as it appears in these cited works.

For this, we start with a \textit{second countable locally compact group $R$} (in practice, this will be a closed subgroup of $\R^d$, see Sections \ref{sec:proofKS} and \ref{sec:koko}) with an \textit{action on a standard Borel space $(Z,\mathcal Z)$ with discrete stabilizers} (i.e.~the stabilizer of any point $z\in Z$ is a discrete subgroup of $R$). We also fix a \textit{probability measure $\mu$} on $(Z,\mathcal Z)$.

We denote by $\mathcal M(R)$ the space of (positive) Radon measures on $R$ and by $\mathbb P\cM(R)$ the space of \textit{projective measures}, that is, of classes of Radon measures with respect to the equivalence relation that declares two measures $\sigma_1$ and $\sigma_2$ equivalent if and only if they are (positively) proportional (one writes $\sigma_1\,\propto\,\sigma_2$).

There is a natural map $\pi:\cM(R)\to\mathbb P\cM(R)$ taking a given measure to its class. Given an exhaustion $R=\bigcup_{n\in \N}X_n$ of $R$ defined by an increasing union of compact subsets $X_n\subset R$,
we define a \textit{section} $\mathbb P\cM(R)\to \cM(R)$ by assigning to a given projective measure $[\sigma]$ the unique measure $\sigma$ determined by the condition
\[
\sigma(X_n)=1,\quad\text{where $n$ is the least $k\in \N$ such that }[\sigma](X_k)>0.
\]
In practice, when $R\subset\R^d$, one may choose $X_n=R\cap[-n,n]^d$.

\begin{definition}
	A Borel subset $\Sigma\subset Z$ is a \textit{discrete section} for the action of $R$ if for any $z\in Z$, the set $\{r\in R\mid r\cdot z\in \Sigma\}$ is a closed and discrete subset of $R$. (This is called a \textit{lacunary section} in \cite{dufloux}.)
\end{definition}

A result by Kechris \cite{Kechris} states that any such action admits a \textit{complete} discrete section, that is, with the additional property that $R\cdot\Sigma=Z$.
Now, given a complete discrete section $\Sigma\subset Z$, we consider the surjective map
\[
\dfcn{a}{R\times \Sigma}{Z}{(r,z)}{r\cdot z}
\]
with countable fibers, so that the measure defined by
\[
(a^*\mu)(E)=\int_Z\#\left (E\cap a^{-1}(z)\right )d\mu(z),\quad\text{for }E\subset R\times \Sigma\text{ Borel subset},
\]
is a $\sigma$-finite measure on $R\times \Sigma$. (Observe that this is a \textit{pull-back} of a measure, so it comes from an \textit{exceptional} construction.)

Let $\pi_\Sigma:R\times \Sigma\to \Sigma$ denote the projection onto the second factor, and $\mu_\Sigma$ the push-forward by $\pi_\Sigma$ of some \textit{finite} measure equivalent to $a^* \mu$ (recall that two measures are equivalent if they share the same zero and full measure sets). By construction, $\mu_\Sigma$ is a finite measure on $\Sigma$.

The horizontal partition $\cP=\{R\times \{z\};z\in \Sigma\}$ is a measurable partition of the Borel space $R\times \Sigma$ thus (a refinement of) Rokhlin disintegration theorem (Theorem~\ref{t.rokhlin}) ensures the existence of a disintegration of the measure $a^*\mu$: for $\mu_\Sigma$-almost every $z\in \Sigma$, there exits a \textit{conditional measure} $\mu_{\Sigma,z}\in \cM(R)$ (not necessarily finite, but $\sigma$-finite) such that
\[
a^*\mu(A\times B)=\int_B\mu_{\Sigma,z}(A)\,d\mu_\Sigma(z),\quad \text{for }A\subset R,B\subset \Sigma\text{ Borel sets}.
\]

Given $r\in R$, we denote by $\rho_r$ the right multiplication by $r$ on elements of $R$. The following lemma \cite[Lemma 4.1]{BQ} tells that the \textit{class} of the conditional measure does not change as we move along one orbit:

\begin{lemma}
	For $\mu_\Sigma$-almost every $z\in\Sigma$, for any $r\in R$ such that $r\cdot z\in \Sigma$, one has
	\[
	\mu_{\Sigma,z}\,\propto\, (\rho_r)_*\mu_{\Sigma,r\cdot z}.
	\]
\end{lemma}

\begin{proof}
	The hypothesis of \textit{discrete stabilizers} implies that the set $\{(r,z)\in R\times\Sigma\mid r\cdot z\in \Sigma\}$ is a countable union of graphs of partially defined, Borel injective functions $r_i:\Sigma_i\to R$. Therefore it is enough to check the equality for graphs of injective functions only.
	%
\end{proof}

Finally, we have \cite[Prop.~4.2]{BQ}:

\begin{proposition}[Definition of conditional measures]
	There exists (an essentially unique) Borel map $\sigma:Z\to \mathbb P\cM(R)$, a Borel set $E\subset Z$ of full $\mu$-measure such that:
	\begin{enumerate}
		\item for every discrete section $\Sigma$, for $\mu_\Sigma$-almost every $z_0\in\Sigma$ and for any $r\in R$ such that $r\cdot z_0\in E$ one has
		\[
		[\mu_{\Sigma,z_0}]=(\rho_r)_*\sigma(r\cdot z_0);
		\]
		\item for any $r\in R,z\in E$ such that $r\cdot z\in E$, one has
		\[
		\sigma(z)=(\rho_r)_*\sigma(r\cdot z).
		\]
	\end{enumerate}
\end{proposition}

\begin{proof}
	Fix a complete discrete section $\Sigma_0$. We define, for $z=rz_0\in R\cdot \Sigma_0=Z$,
	\[
	\sigma(z):=\left [(\rho_r^{-1})_*\mu_{\Sigma_0,z_0}\right ]\in\mathbb P\cM(R).
	\]
	The previous lemma guarantees that $\sigma(z)$ is well defined.
\end{proof}

\bigskip

\noindent \textsc{Bruno Santiago}\\
Instituto de Matem\'atica e Estat\'istica, Universidade Federal Fluminense\\
Rua Professor Marcos Waldemar de Freitas Reis, s/n, Bloco H - Campus do Gragoat\'a \\
S\~ao Domingos - Niter\'oi - RJ - Brazil CEP 24.210-201\\
\texttt{brunosantiago@id.uff.br}

\smallskip

\noindent{\scshape Michele Triestino}\\
IMB, Universit\'e de Bourgogne Franche-Comt\'e, CNRS UMR 5584\\
9 av.~Alain Savary, 21000 Dijon, France\\
\texttt{michele.triestino@u-bourgogne.fr}

\newpage
		\addtocontents{toc}{\vspace{\normalbaselineskip}}

\section{The Pinsker Partition and the Hopf Argument (by Davi Obata)}\label{App:Pinsker}

\subsection{The Pinsker, stable and unstable partitions}

Here we treat more extensively some of the notions appearing in  Section \ref{sec:abstractergodictheory}. We define the Pinsker partition and describe  its relation with the unstable and stable partitions, given by Theorem B of \cite{ly}. We shortly recall basic ingredients from measure theory; for a better discussion on these points see \cite{ck}. 

Let $(X,\mathcal{B},\mu)$ be a measure space. Given a partition $\xi$ of $X$ define $\mathcal{B}(\xi)$ to be the $\sigma$-algebra generated by the measurable sets $C \in\mathcal{B}$ that are union of elements of $\xi$. Given two partitions $\xi$ and $\eta$, we say that $\xi$ refines $\eta$ if every element of $\eta$ can be obtained by union of elements of $\xi$, we denote it by  $\eta \prec \xi$, we also say that $\eta$ coarsens $\xi$. We say that two partitions $\eta$ and $\xi$ are equal mod zero if they are the same on a set of full $\mu$-measure, and we denote it by $\eta = \xi$ (cf.~Section~\ref{sec:parord}).	

Let $f:X \to X$ be a measurable function that preserves $\mu$; recall that given any finite measurable partition $\xi$ we can define the metric entropy with respect to this partition, denoted by $h_{\mu}(f,\xi)$ (see Section \ref{sec:ME}).

\begin{definition}[cf.~\cref{sss:pinsker}]
	
	The \textbf{Pinsker partition}\index{partition!Pinsker} $\pi(f)$ is defined as the finest measurable partition such that if $\eta$ is any finite partition with $\eta \prec \pi(f)$ then $h_{\mu}(f,\eta)=0$. We can also define the \textbf{Pinsker $\sigma $-algebra} as the biggest sub-$\sigma$-algebra of $\mathcal{B}$, which we will denote by $\mathcal{P}$ such that every $A\in \mathcal{P}$ satisfies $h_{\mu}(f, \{A, A^c\}) = 0$. 
\end{definition}

Given a partition $\xi$ we define its \textbf{measurable hull}\index{partition!measurable hull}  $\Xi (\xi)$ as the finest measurable partition which coarsens $\xi$. In other words $\Xi(\xi)\prec \xi$ and if $\eta$ is a measurable partition such that $\eta \prec \xi$ then $\eta \prec \Xi (\xi)$ (cf.~\cref{ss:meashull}).

\begin{example}[cf.~\cref{ex:partition_orbits}]\label{Bex:partition_orbits}
	Let $(X,d)$ be a complete and separable metric space, $\cB$ is Borel $\sigma$-algebra, $(f,X,\cB,\mu)$ a measure preserving system and $\calE$ the ergodic decomposition (Theorem~\ref{t.decomposicaoergodica}).
	Then the measurable hull $\Xi (\mathcal{O})$ of the partition by $f$-orbits $\mathcal{O}$ is the ergodic decomposition $\calE$. 
	To see this, notice  that atoms of the ergodic decomposition $\calE$ are $f$-invariant and therefore $\calE$ coarsens $\mathcal O$. Thus it remains to prove that $\calE$ is the finest measurable partition enjoying this property. For this, we remark that, replacing $X$ with an ergodic component $P\in\calE$ and $\mu$ with $\mu_P$, we can assume that the measure $\mu$ is ergodic.
	Now, ergodic measures admit no nontrivial disintegration with respect to a measurable partition into invariant subsets (we use \cref{Bprop:disintegration_invariant}), so any finer measurable partition $\xi$, $\calE\prec \xi\prec\mathcal O$, must coincide with $\calE$.
\end{example}

From now on let us suppose that $X=M$ is a manifold, $\mathcal{B}$ is the Borel $\sigma$-algebra and $f:M \to M$ is a $C^{1+\beta}$-diffeomorphism. By Oseledec's theorem (Theorem \ref{thm:oseled}) we know that for $\mu$-almost every point $x$ the Lyapunov exponents are defined. By Pesin's theory (see Section \ref{unstmanifold11}) we know that for $\mu$-almost every point $x\in M$ there are stable/unstable manifolds $W^s(x),W^u(x)$ tangent to the directions of the Oseledec's splitting related to the negative/positive exponents: they are defined by
\begin{align*}
	W^s(x)&=\left \{ y\in M\,:\, \limsup_{n\to \infty}\frac{1}{n} \log d\left (f^n(x),f^n(y)\right )<0\right \},\\
	W^u(x)&=\left \{ y\in M\,:\,  \limsup_{n\to \infty}\frac{1}{n} \log d\left (f^{-n}(x),f^{-n}(y)\right )<0\right \}.
\end{align*}
In the case all the exponents are zero, those manifolds are just the points. Thus we obtain two partitions, the \textit{stable} and \textit{unstable partitions} which we denote by $W^s$ and $W^u$, respectively. (Their measurable hulls are denoted by $\Xi^s$ and $\Xi^u$, respectively, in Section \ref{sec:abstractergodictheory}.)

In \cite{ly} Ledrappier and Young proved the following theorem (cf.~Proposition \ref{prop:pinsker}).

\begin{theorem}[Theorem B of \cite{ly}]\index{theorem!Ledrappier--Young}
	Let $f:M \to M$ be a $C^{1+\beta}$-diffeormorphism preserving a probability measure $\mu$. Then we have equality of partitions
	\[
	\Xi (W^s) = \pi(f) = \Xi (W^u).
	\]
	
\end{theorem}

Leddrapier and Young actually state the theorem in terms of $\sigma$-algebras, the result is the same, up to replace the partitions by the $\sigma$-algebras they generate.

\subsection{The Hopf Argument}

The Hopf argument was introduced by Hopf \cite{h} to prove that the geodesic flow on compact surfaces with constant negative curvature is ergodic with respect to the Liouville measure. Later Anosov proved, in \cite{a}, that every $C^2$-volume preserving, Anosov diffeomorphism is ergodic. The Hopf argument can be divided in two parts, the first part is that every ergodic component is saturated, up to a set of measure zero, by stable or unstable manifolds, the second part is that under an additional hypothesis called accessibility, one can exploit the Anosov property to show that the system is ergodic. Since then the ideas from Hopf argument have been the main tool to prove ergodicity for partial hyperbolic systems, see \cite{aw} for a survey on conservative partially hyperbolic dynamics. 

In our scenario we can state the Hopf argument in the following way (cf.~Proposition \ref{prop:hopf}).

\begin{theorem}[The Hopf Argument]\label{Hopf}\index{Hopf argument}
	Let $f:M \to M$ be a $C^{1+\beta}$-diffeormorphism preserving a probability measure $\mu$.
	The ergodic decomposition is refined by the measurable hull of the stable partition, in other words 
	\[
	\Xi (\mathcal{O}) \prec \Xi (W^s).
	\]
\end{theorem}

The same result holds if we change the measurable hull of the stable partition by the measurable hull of the unstable partition.

\begin{proof}
	Let $\nu$ be an invariant ergodic measure in the ergodic decomposition of $\mu$. By Birkhoff's ergodic theorem (see Theorem~\ref{t.ergodic}) for every continuous function $\varphi \in C^0(M)$ there is a measurable set $\Lambda_{\varphi}$ of full $\nu$-measure such that if $x\in \Lambda_{\varphi}$ then the limit
	\begin{equation}
		\label{eq.birkaverage}
		\displaystyle \varphi^+(x) := \lim_{n\to +\infty} \frac{1}{n} \sum_{j=0}^{n-1} \varphi(f^j(x))
	\end{equation}
	exists and it is equal to $\varphi^+(x) = \int_X \varphi\, d\nu$.
	Let $\{\varphi_k\}_{k\in \mathbb{N}}$ be a sequence which is dense in $C^0(X)$ and consider the set $\Lambda_{\nu} = \displaystyle \bigcap_{k\in \mathbb{N}} \Lambda_{\varphi_k}$;
	this set has full $\nu$-measure and has the property that if $x\in \Lambda_{\nu}$ and $\varphi \in C^0(X)$, the equality $\varphi^+(x) = \int_X\varphi\, d\nu$ holds. In other words, we have that for any $x\in \Lambda_\nu$ 
	\[
	\displaystyle \delta(x,n):= \frac{1}{n} \sum_{j=0}^{n-1}\delta_{f^j(x)} \xrightarrow{\: n\to + \infty \:} \nu, \textrm{ in the $weak^*$-topology.}
	\]
	Moreover, since the set $\Lambda_{\nu}$ has full $\nu$-measure, we can describe it as the set of all points whose forward Birkhoff's average converges to $\nu$. 
	
	We claim that if $x\in \Lambda_{\nu}$ then $W^s(x) \subset \Lambda_{\nu}$. Indeed, if $y\in W^s(x)$, then 
	\[d(f^n(x) , f^n(y)) \xrightarrow{ n \to +\infty} 0.\]
	We know that $\delta(x,n) $ converges to $\nu$; to prove that $\delta(y,n)$ also converges to $\nu$ we have to prove that for every $\varphi\in C^0(X)$ one has 
	\[
	\int_X \varphi\, d\delta(y,n) \xrightarrow{ n \to +\infty} \int_X \varphi\, d\nu.
	\]  
	By continuity, one has $| \varphi(f^n(x)) - \varphi(f^n(y))| \to 0$ as $n$ goes to infinity, thus 
	\begin{align*}
	\displaystyle \lim_{n\to +\infty} \int_X \varphi\, d\delta(y,n) &= \lim_{n\to + \infty} \frac{1}{n} \sum_{j=0}^{n-1} \varphi ( f^j(y))\\
	&= \lim_{n\to + \infty} \frac{1}{n} \sum_{j=0}^{n-1} \varphi (f^j(x)) = \int_X \varphi\, d\nu .\qedhere
	\end{align*}
\end{proof}

Since our system $f$ is invertible the same statement is also true for the measurable hull of the unstable partition. We remark that analogous results hold for flows.

\begin{example}
	We now give an example of application of this result. Consider the group $G = \mathrm{PSL}(2,\mathbb{R})$ and the subgroups
	\[
	A = 
	\begin{pmatrix}
	* & 0\\
	0 & *
	\end{pmatrix}
	, \textrm{ }
	L  =
	\begin{pmatrix}
	1& 0 \\
	* & 1
	\end{pmatrix}
	\textrm{ and }
	U= 
	\begin{pmatrix}
	1 & * \\
	0 & 1
	\end{pmatrix}.
	\]
	
	It is easy to see that the subgroups $A, L$ and $U$ generate $G$. Suppose that $G$ acts on the left on a compact manifold $M$, by $C^{1+\beta}$ diffeomorphisms. Assume that it preserves a probability measure $\mu$. We say that the measure $\mu$ is $G$-ergodic if every measurable set $B$ that is $G$-invariant has zero or full $\mu$-measure. 
	
	Of course if the measure $\mu$ is ergodic for any of the subgroups then it is ergodic for $G$. We will prove now that if the measure $\mu$ is ergodic for $G$ than it is ergodic for $D$. 
	We can parametrize the subgroups by
	\[
	A_t = 
	\begin{pmatrix}
	e^{\frac{t}{2}}& 0\\
	0 & e^{-\frac{t}{2}}
	\end{pmatrix}
	, \textrm{ }
	L _r =
	\begin{pmatrix}
	1& 0 \\
	r & 1
	\end{pmatrix}
	\textrm{ and }
	U_s= 
	\begin{pmatrix}
	1 & s \\
	0 & 1
	\end{pmatrix}.
	\]
	
	Observe that $A_t$ generates a flow and $\mu$ is an invariant measure for this flow. First we obtain that
	\begin{equation}
		\label{eq.commut}
		U_s \circ A_t = A_t \circ ( A_{-t} \circ U_s \circ A_t) =  A_t \circ 
		\begin{pmatrix}
			1& s e^{-t}\\
			0&1
		\end{pmatrix}
		= A_t \circ U_{se^{-t}}.
	\end{equation}
	Observe that from (\ref{eq.commut}), we obtain
	\[
	U_{se^t} \circ A_t = A_t \circ U_s.
	\]
	Now for any $x\in M$ and $s\in \mathbb{R}$, we have
	\[
	d(A_t(x), A_t(U_s(x)))  =  d(A_t(x), U_{se^{t}} (A_t(x))) \xrightarrow[t \to -\infty]{} 0
	\]
	and this convergence is exponentially fast. This implies that $U_s(x) \in W^u(x)$ for every $s \in \mathbb{R}$, where $W^u(x)$ is the unstable manifold of $x$ with respect to the flow $A_t$. Similarly, we can check that $ L_{re^{-t}} \circ A_t = A_t \circ L_{r}$, thus $L_r(x) \in W^s(x)$. 
	
	Let $\varphi \in C^0(M)$ be a continuous function and let $\varphi^+(\cdot)\in L^1(M,\mu)$ be the forward Birkhoff average with respect to the flow $A_t$. By Birkhoff's theorem this function is $A$-invariant and is defined on a set of full $\mu$-measure. By Theorem \ref{Hopf}, for $\mu$-almost every point $x\in M$ and for any $y\in W^s(x)$ it holds that $\varphi^+(x) = \varphi^+(y)$. By the previous calculation we know that for any $r\in \mathbb{R}$, we have $L_r(x) \in W^s(x)$. Thus, the function $\varphi^+$ is also $L$-invariant. Similarly, we conclude that $\varphi^+$ is $U$-invariant. Notice that $A$, $L$ and $U$ generate $G$, hence $\varphi^+ $ is $G$-invariant and since $\mu$ is $G$-ergodic, the only $G$-invariant functions in $L^1(M,\mu)$ are the constants. We conclude that $\varphi^+$ is constant $\mu$-almost everywhere. This is true for any continuous function, therefore $\mu$ is $A$-ergodic. 
	
	This is a simple case of a larger class of results called the ``Mautner phenomenon'' (first appearing in \cite{mautner}, see also \cite{moore}).
	
\end{example}

\bigskip

\noindent \textsc{Davi Obata}\\
LMO, Universit\'e Paris Sud, CNRS UMR 8628,\\
D\'epartement de Math\'ematiques B\^atiment 425\\
Facult\'e des Sciences d'Orsay Universit\'e Paris-Sud\\
F-91405 Orsay Cedex\\
\texttt{davi.obata@math.u-psud.fr}
%


\newpage
		
\addtocontents{toc}{\vspace{\normalbaselineskip}}
\section{Metric entropy and Lyapunov exponents after Ledrappier--Young (by S\'ebastien Alvarez and Mario Rold\'an)}\label{App:LY}
\newcommand{\NN}{\mathcal N}
\newcommand{\mass}{{\rm mass}}
\newcommand{\Leb}{{\rm Leb}}
\newcommand{\CC}{\mathcal C}

Lyapunov exponents give a geometric way to measure the complexity of a map, and metric entropy gives a probabilistic way to do so. We are interested here in comparing these two notions. In these directions, the two basic results are (see Theorem \ref{entropyfacts}):

\begin{quote}\index{Margulis--Ruelle inequality}
	\textbf{Margulis--Ruelle's inequality} proven in \cite{Ru}: For every $C^1$-mapping $f$ (not necessarily invertible) of a compact Riemannian manifold $M$ preserving a probability measure $\mu$, the metric entropy is bounded above by the sum of positive Lyapunov exponents, 
	$$h_{\mu}(f)\leq\int_M\sum_{\lambda^j>0} m^j(x)\lambda^j(x)d\mu(x).$$
\end{quote}

\begin{quote}	\index{Pesin!entropy formula}
	\textbf{Pesin's Formula} (also known as \textit{Entropy Formula}) proven in \cite{Pe} (see also \cites{LS,Ma}): For every $C^{1+\beta}$ diffeomorphism $f$ of a compact Riemannian manifold $M$ preserving a probability measure $\mu$ \textit{equivalent to the Riemannian volume}, we have
	$$h_{\mu}(f)=\int_M\sum_{\lambda^j>0} m^j(x)\lambda^j(x)d\mu(x).$$
\end{quote}
where, as usual, $\lambda^1(x)>\lambda^2(x)\cdots >\lambda^{r(x)}(x)$ denote all distinct Lyapunov exponents of $f$ at $x$, $m^j(x)$ is the multiplicity of $\lambda^j(x)$, and $h_{\mu}(f)$ denotes the metric entropy. Note that when $\mu$ is ergodic, the quantities $\lambda^j(x)$, $r(x)$ and $m^j(x)$ do not depend on the choice of a $\mu$-typical $x\in M$ and the integral sum in both formulas may be omitted.

The aim of \textit{Ledrappier--Young's theory} is to further study relations of these types. In \cite{LYI}, F. Ledrappier and L.-S. Young characterize those measures which satisfy Pesin's entropy formula. In \cite{LYII} they prove a formula which is valid for every invariant measure.

We state below the first result, i.e.~the principal result of \cite{LYI}. The statement of the general formula shall be postponed until Section \ref{S:LY2}. Before stating the result, let us recall that an ergodic probability measure $\mu$ invariant by a $C^{1+\beta}$ diffeomorphism of a compact manifold $M$ is said to be an \textit{SRB measure} if it has absolutely continuous conditional measures on unstable manifolds.\index{measure(s)!SRB}

\begin{theorem}[Ledrappier--Young I]\index{theorem!Ledrappier--Young}
	\label{t:LYI}
	Let $f$ be a $C^{1+\beta}$ diffeomorphism of a compact Riemannian manifold $M$ and $\mu$ be an ergodic, $f$-invariant probability measure. Then $\mu$ is SRB if and only if
	$$h_{\mu}(f)=\sum_{\lambda^j>0}m^j\lambda^j.$$
\end{theorem}

The ``only if'' direction is a generalization of Pesin's formula: it was proven in the conservative setting by Ya.~Pesin. R. Ma\~n\'e gave a proof in \cite{Ma} without using the theory of stable manifold. In this generality (for SRB measures rather than for smooth measures), the proof of the ``if'' part is due to F. Ledrappier and J.-M. Strelcyn in \cite{LS}.

The proof of the ``if'' part is the most difficult one and was first achieved by F. Ledrappier in \cite{L} under the hypothesis that $\mu$ is \textit{hyperbolic} (i.e.~has no zero Lyapunov exponent) (see Section \ref{ss:lyapexp}). Later, together with L.-S. Young in  \cite{LYI}, they were able to treat the difficulties that emerge when one allows the presence of zero Lyapunov exponents. 

Let us say a word about the regularity of the dynamics in Theorem \ref{t:LYI}. F. Ledrappier and L.-S. Young proved
that Theorem \textit{for $C^2$ diffeomorphisms}. As we will see later on there is a crucial step in which the $C^2$ hypothesis 
rather than the $C^{1+\beta}$ hypothesis on the dynamics was used in \cite{LYI}: obtaining the lipschitzness of the unstable 
holonomies inside center-unstable sets. A. Brown recently showed in \cite{B}, that this lipschitzness actually holds 
for $C^{1+\beta}$ dynamics. Finally let us note that the regularity can't be lowered. It comes from \cites{BCS,Pugh} 
that a $C^1$ diffeomorphism can have many hyperbolic ergodic probability measures (i.e.~without zero Lyapunov exponents)
such that the points on their supports have no unstable nor stable manifolds \cites{BCS,Pugh}.

\begin{example}[Baker vs.~horseshoe] Before entering in the details let us borrow to L.-S. Young (see \cite[Example 4.1.3.]{Yo2}) an enlightening illustration of the results above.
	
	Consider first the well known \textit{baker transformation}. This is a piecewise affine map of the unit square $C$ of $\R^2$ defined as follows. We set $T(x,y)=(2x,y/2)$ if $x<1/2$  and $T(x,y)=(2x-1,y/2+1/2)$ if $x\geq 1/2$. (See Figure \ref{f:baker}.) It is possible to prove that the Lebesgue measure $\mu$ is $T$-invariant and ergodic and that $T$ is then isomorphic to the $(\frac{1}{2},\frac{1}{2})$-Bernoulli shift. It is clear that
	$$h_{\mu}(T)=\log 2=\lambda^+,$$
	where $\lambda^+$ is the largest Lyapunov  exponent.
	
	\begin{figure}
		\[
		\includegraphics{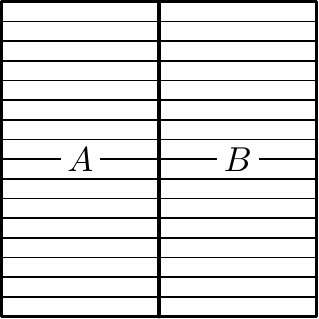}\quad\quad
		\includegraphics{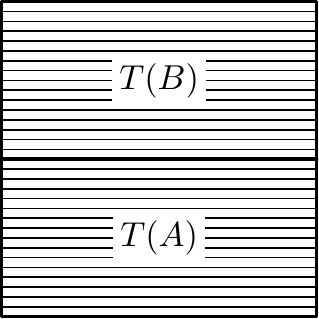}
		\]
		\caption{The baker transformation}\label{f:baker}
	\end{figure}
	
	The second transformation is a piecewise affine map defined on a subset of $C$ that may be extended to Smale's \textit{horseshoe}. (See Figure \ref{f:horseshoe}.) It has a hyperbolic invariant set $\Lambda$ of measure zero (this is a product of two Cantor set) and we endow it with its measure of maximal entropy (which is singular since it is supported on $\Lambda$). It is also isomorphic to the $(\frac{1}{2},\frac{1}{2})$-Bernoulli shift. Now we have
	$$h_{\mu}(T)=\log 2<\lambda^+.$$
	
	\begin{figure}
		\[
		\includegraphics{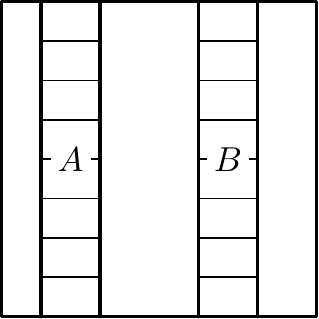}\quad\quad
		\includegraphics{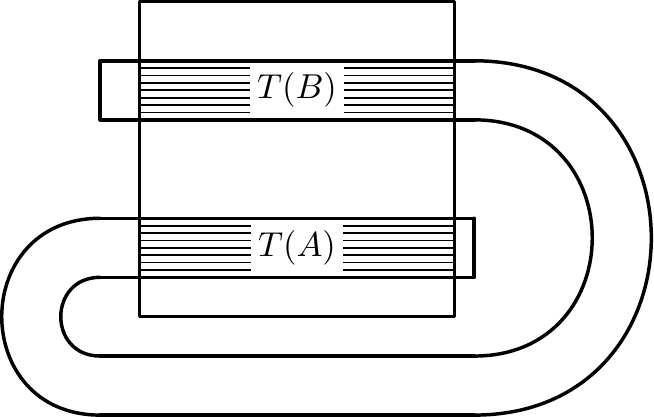}
		\]
		\caption{Smale's horseshoe}\label{f:horseshoe}
	\end{figure}
	
	The first transformation satisfies the entropy formula, while the second one does not.  One may interpret this fact as follows (see \cite{Yo2}). In a conservative system, all the expansion goes back to the system and contributes to the entropy, whence the entropy formula. In a dissipative system, some of the entropy is wasted along the way (in the second case, it is wasted to ``bend'' the horseshoe) and the entropy is concentrated in some region of smaller dimension (we will see how in the second part of Ledrappier--Young's theory).
\end{example}

\subsection{Entropy along the unstable direction}

We will focus on the second half of Ledrappier--Young's theorem (Theorem \ref{t:LYI}) i.e.~we want to prove
$$\mu\text{ satisfies the entropy formula}\quad\Longrightarrow \quad\mu\text{ is  SRB}.$$

The principal idea of this theorem is that the entropy is created by the expansion along unstable manifolds. We show below how to formalize this idea and we sketch the proof given in \cite{LYI}.

\subsubsection{Partitions subordinate to the unstable foliation}

\hyphenation{diffeo-mor-phism}
\index{partition!subordinate}
Let $f$ be a $C^{1+\beta}$ diffeomorphism of a compact Riemannian manifold $M$ and $x\in M$ be a \textit{regular point}, meaning that Lyapunov exponents $\lambda^1(x),\lambda^2(x),\ldots,\lambda^{r(x)}(x)$ and Oseledec's splitting $T_xM=E^1(x)\oplus E^2(x)\oplus\ldots\oplus E^{r(x)}(x)$ exist at $x$. By Oseledec's Theorem (Theorem~\ref{thm:oseled}) the set of such points is full for every $f$-invariant probability measure. The \textit{unstable manifold} at $x$ is defined by
$$W^u(x)=\left\{y\in M;\,\overline{\lim_{n\to\infty}}\frac{1}{n}\log\,d(f^{-n}(x),f^{-n}(y))<0\right\}.$$
By Pesin's (un)stable manifold theorem (see \cites{FHY,PeStable,RuStable}) $W^u(x)$ is a monotone union of discs tangent to $E^u(x)$ at $x$, so it is an injectively immersed Euclidean space tangent to $E^u(x)$ at $x$ where
$$E^u(x)=\bigoplus_{\lambda^j>0}E^j(x).$$
We will refer to the partition $W^u=\{W^u(x);\,x\,\,\text{regular}\}$ as the \textit{unstable foliation}. The ambient Riemannian structure induces a Riemannian structure on $W^u(x)$. This provides $W^u(x)$ with a topology that we call the \textit{internal topology}. Note that it differs from the topology induced by $M$ on $W^u(x)$.

In general unstable leaves form a non-measurable partition of $M$ and we can't disintegrate $\mu$ in unstable leaves so we will need the following definition.

Say a measurable partition $\xi$ is $(\mu$-)\textit{subordinate to the unstable foliation}
$\W^u$ if (cf.~Section \ref{unsentropy}): 
\begin{enumerate}
	\item for $\mu$-almost every $x\in M$, $\xi(x)$ is a subset of $W^u(x)$ with diameter bounded by a constant $\eps_0$ which does not depend on $x$;
	\item for $\mu$-almost every $x\in M$, $\xi(x)$ contains an open neighbourhood of $x$ inside $W^u(x)$;
	\item $\xi$ is increasing, i.e.~$f\xi\prec\xi$;
	\item $\xi$ is generating, i.e.~$\bigvee_{n=0}^{\infty}f^{-n}\xi$ is the partition into points of $M$.
\end{enumerate}

Note that unstable manifolds $W^u(x)$ are well defined for every regular point, and that every $f$-invariant probability measure gives total mass to the set of regular points. Hence the hypothesis of ergodicity of $\mu$ \textit{is not needed} in the above definition.

The existence of measurable partitions subordinate to the unstable foliation is due to F. Ledrappier and J.-M. Strelcyn \cite{LS}.

\begin{proposition}[\cite{LS}]
	\label{p:existence_partition}\index{theorem!Ledrappier--Strelcyn}
	Let $f$ be a $C^{1+\beta}$ diffeomorphism of a compact manifold $M$ and $\W^u$ be the partition into Pesin unstable manifolds of $f$. Then there exists a measurable partition $\xi$ subordinate to $\W^u$.
\end{proposition}

Let us sketch a proof of Proposition \ref{p:existence_partition} when $\W^u$ is \textit{uniformly expanding}, building on recent work of J.~Yang \cite{Ya}. This last hypothesis means that $\W^u$ is an $f$-invariant continuous foliation, tangent to a $Df$-invariant continuous plane field $E^u$, such that there exists a uniform $\lambda<1$ such that
$$||Df^{-1}|_{E^u}||<\lambda.$$

Note that in the proof below, $f$ needs only to be $C^1$ and $\mu$ does not need to be ergodic: this proof only uses the expansion property of $f$ on $\W^u$ and Borel--Cantelli's theorem.

\subsubsection*{Finite partitions and the Borel--Cantelli property} Fix a foliated atlas $\A^u$ for $\W^u$. The proof consists in constructing first a finite partition $\Pp$ of $M$ satisfying the two conditions below

\begin{enumerate}
	\item atoms of $\Pp$ are included in foliated charts of $\A^u$;
	\item the series of $\mu$-masses of the $\lambda^j$-neighbourhoods of the boundary of any atom of $\Pp$ converges.
\end{enumerate}

This partition is constructed from a covering of $M$ by finitely many small balls $B_k$ such that $\sum_j\mu(\NN_{\lambda^j}(\partial B_k))<\infty$ (in what follows $\NN_r(A)$ stands for the $r$-neighbourhood of $A$). The construction of this covering does not require any hypothesis on $\mu$ and is essentially a consequence of the following Borel--Cantelli type lemma which allows us to find the radii of these balls. For the sake of completeness we include the proof, which is both elementary and elegant.

\begin{lemma}
	\label{l:BC}
	Let $\nu$ be a finite Borel measure on $\R$ supported on $[0,r_0]$ for some $r_0>0$. Then for every $\lambda\in (0,1)$, for Lebesgue-almost every $r\in [0,r_0]$,
	$$\sum_{j=0}^\infty\nu\left(\left[r-\lambda^j,r+\lambda^j\right]\right)<\infty.$$
\end{lemma}

\begin{proof}
	For a fixed $j\geq 0$ we will define the \textit{bad set} as
	$$Y_j=\left\{r\in[0,r_0];\,\nu\left(\left[r-\lambda^j,r+\lambda^j\right]\right)\geq \frac{1}{j^2}\right\}.$$
	The bad set may be covered by finitely many \textit{bad intervals} $[r_i-\lambda^j,r_i+\lambda^j]$ for $i=1,\ldots,l$ with $r_i\in Y_j$, in such a way that any point of $Y_j$ belongs to at most two bad intervals. We can bound the number $l$ of these bad intervals because
	$$\frac{l}{j^2}\leq\sum_{i=1}^l\nu\left(\left[r_i-\lambda^j,r_i+\lambda^j\right]\right)\leq 2\nu(\R) ,$$
	so $l\leq 2\nu(\R)j^2$. We deduce that $\Leb(Y_j)\leq 4\nu(\R)j^2\lambda^j$ so $\sum_j\Leb(Y_j)<\infty$. By Borel--Cantelli's theorem for $\Leb$-almost every $r\in[0,r_0]$ there exists $j_r$ such that $r\notin Y_j$ for every $j\geq j_r$, which implies that $\sum_j\nu([r-\lambda^j,r+\lambda^j])<\infty$.
\end{proof}

This lemma being established we can construct the desired finite partition $\Pp$. We chose $r_0$ smaller than the Lebesgue number of the covering of $\A^u$. Given $x\in M$ we want to define the radius $r_x$ of a ball centered at $x$ whose boundary satisfies the second condition stated above. To do so, we define a measure $\nu_x$ on $[0,r_0]$ by
$$\nu_x([a,b])=\mu\left(\left\{y\in M;\,a\leq d(x,y)\leq b\right\}\right),$$
and we apply Lemma \ref{l:BC} to $\nu_x$. If $\nu_x$ is the zero measure, we set $r_x=r_0/2$. If it is not, Lemma \ref{l:BC} gives $r_x\in(r_0/2,r_0)$ with
$$\sum_{j=1}^{\infty}\mu\left(\NN_{\lambda^j}(\partial B(x,r_x))\right)<\infty.$$
In particular $\mu$ gives zero measure to the boundary of the spheres $B(x,r_x)$. A compactness argument allows to cover $M$ with finitely many such balls $B_k$ which are included in unstable charts (by the choice of $r_0$). The desired partition is now $\Pp=\bigvee_k \{B_k,\,^cB_k\}$.

\subsubsection*{Proof of Proposition \ref{p:existence_partition}} Now consider the measurable partition $\xi_0$ whose atoms are the intersection of atoms of $\Pp$ and unstable plaques of $\A^u$. We claim that the following partition is subordinate to $\W^u$
$$\xi=\bigvee_{j=0}^{\infty}f^j\xi_0.$$

The most complicated part is to prove that $\xi(x)$ contains a neighbourhood of $x$ inside $W^u(x)$ for $\mu$-almost every $x$. Denote by $\partial\Pp$ the union of boundaries of atoms of $\Pp$. We first use that $\mu(\partial\Pp)=0$ and the $f$-invariance of $\mu$ to find a Borel set $X$  full for $\mu$ such that $f^j(x)\notin\partial\Pp$ for every $x\in X$ and $j\in\Z$. This implies that for every $x\in X$, $\xi_k(x)$ contains an open neighbourhood of $x$ for every $k\in\N$ where
$$\xi_k=\bigvee_{j=0}^{k}f^j\xi_0.$$
The fact that $\xi(x)$ contains an open neighbourhood of $x$ comes from the fact that
$$\xi(x)=\xi_{k(x)}(x)$$
for some $k(x)\in\N$. This $k(x)$ is obtained from an argument \textit{\`a la Borel--Cantelli}. Let us explain it. Using once more the $f$-invariance of $\mu$ as well as the second property characterizing $\Pp$, we see that 
$$\sum_{j=0}^\infty \mu\left[f^j\left(\NN_{\lambda^j}(\partial\Pp)\right)\right]<\infty.$$
By Borel--Cantelli's theorem there exists a Borel set $X$ of full $\mu$-measure such that for every $x\in X$ we have $d(f^{-k}(x),\partial\Pp)>\lambda^k$ for every $k$ greater than some $k(x)$.  One easily shows that when $k\geq k(x)$ we have $\xi_k(x)=\xi_{k+1}(x)$, for the contrary would imply that $f^{k+1}(\partial\Pp)\cap\xi_k(x)\neq\vide$. Using that plaques of  $\A^u$ have uniform diameters (say smaller than $1$) we find that $d(x,f^{k+1}(\partial\Pp))\leq 1$. Using the uniform expansion of $f$ along $\W^u$, one would find that $d(f^{-(k+1)}(x),\Pp))\leq\lambda^{k+1}$, which is absurd by the definition of $k(x)$.

\begin{remark}
	\label{r:mane1}
	In order to treat the general case one has to use Pesin's theory in order to get uniform expansion in sets of positive measure. This is done by defining Pesin's sets and analysing the first return maps to these Pesin's sets. \textit{Pesin's sets}\index{Pesin!set} $\Lambda$ (we use Katok--Mendoza's terminology \cite{KH}) are sets of positive measure (but which are not invariant) enjoying the following properties 
	\begin{itemize}
		\item the size of local unstable manifolds of elements of $\Lambda$ is uniformly bounded from below; more precisely, there exists $\delta>0$ such that for every $x\in\Lambda$, the preimage of $W^u_{loc}(x)$ under the exponential map at $x$ contains the graph of a $C^1$-map from the $\delta$-neighbourhood of $0$ in $E^u(x)$ to $E^c(x)\oplus E^s(x)$;
		\item the dynamics is uniformly expanding along local unstable manifolds inside $\Lambda$; more precisely there exist constants $0<\eps<\lambda/100$ and $C>0$ such that for for every $x\in\Lambda$ every $n\geq 1$ and $m\in\Z$
		$$\left|\left|D_xf^{-n}|_{E^u(f^m(x))}\right|\right|\leq C e^{(\lambda-\eps)n} e^{-|m|\eps}.$$
		where $\lambda>0$ is the smallest positive Lyapunov exponent of $f$ for $\mu$.
	\end{itemize}
	Of course, if one wants to increase the measure of Pesin sets, one looses control on the constants $C$ and $\eps$. Nevertheless the argument sketched in the case of uniformly expanded foliations can be adapted even if one only guarantees the uniformity of the expansion in positive measure sets. This analysis is essentially an argument given by Ma\~n\'e in \cite{Ma}. We won't enter here into the details and suggest the reader to consult the classical references: \cites{Ma,LS,LYI,Yo2}.
	
\end{remark}

\subsubsection{Entropy along the unstable direction} The next step of the proof of Ledrappier and Young is to define the entropy along the unstable direction. Recall (Section \ref{ss:condentropy}) that when $\eta_1$ and $\eta_2$ are measurable partitions of $M$, $H_{\mu}(\eta_1\mid \eta_2)$ denotes the \textit{conditional entropy}\index{entropy of a partition!conditional} of $\eta_1$ given $\eta_2$ and that when $\eta$ is an increasing partition we have
$$h_{\mu}(f,\eta)=H_{\mu}(\eta\mid f\eta).$$

The next proposition allows us to define the entropy along the unstable direction (cf.~\cref{def:entsub}).

\begin{proposition}
	\label{p:indep_partition}
	Let $f:M\to M$ be a $C^{1+\beta}$ diffeomorphism of a compact manifold.
	Let $\xi_1$ and $\xi_2$ be two measurable partitions subordinate to the unstable foliation $\W^u$ of $f$. Then
	$$h_{\mu}(f,\xi_1)=h_{\mu}(f,\xi_2).$$
\end{proposition}

\begin{proof}
	Let us detail the argument proving this proposition. Let $\xi_1$ and $\xi_2$ be two measurable partitions subordinate to $\W^u$. It is enough to prove that $h_\mu(f,\xi_1)=h_\mu(f,\xi_1\vee\xi_2)$. The great idea of the proof is to note that since $f\xi_i\prec\xi_i$ we have for every $n\geq 0$
	$$f^n\xi_1\vee f^n\xi_2\prec\xi_1\vee f^n\xi_2\prec\xi_1\vee\xi_2.$$
	By $f$-invariance of $\mu$, the entropies of $f$ conditional to the first and last partitions coincide so we deduce that
	\begin{equation}
		\label{eq:mean_entropy1}
		h_\mu(f,\xi_1\vee\xi_2)=h_\mu(f,\xi_1\vee f^n\xi_2)=H_\mu\left(\left.\xi_1\vee f^n\xi_2\right|f\xi_1\vee f^{n+1}\xi_2\right),
	\end{equation}
	the last equality coming from the definition of conditonal entropy. Before carrying on with the proof, observe that $f$ expands the unstable foliation, so eventually the atoms of $\xi_1$ should be included in atoms of $f^n\xi_2$, for $n$ large enough. This gives a good hint that the entropy conditional to $\xi_1\vee f^n\xi_2$ should tend to the entropy conditional to $\xi_1$. Let us give a formal explanation of this intuition.
	
	We will use a formula of conditional  entropy proved in Rokhlin's classical paper \cite[\S 5.9]{Rok}. For measurable partitions $\A,\B$ and $\CC$ we have that
	\begin{equation}
		\label{eq:mean_entropy_gen}
		H_\mu(\A\vee\B|\CC)=H_\mu(\A|\CC)+H_\mu(\B|\A\vee\CC).
	\end{equation}
	Applying \eqref{eq:mean_entropy_gen} with $\A=\xi_1$, $\B=f^n\xi_2$ and $\CC=f\xi_1\vee f^{n+1}\xi_2$ and having in mind that $\xi_1\vee f\xi_1=\xi_1$ we find
	\begin{align*}
		H_\mu\left(\left.\xi_1\vee f^n\xi_2\right|f\xi_1\vee f^{n+1}\xi_2\right)
		&=H_\mu\left(\xi_1\left|f\xi_1\vee f^{n+1}\xi_2\right.\right)\\
		&+H_\mu\left(\left.f^n\xi_2\right|\xi_1\vee f\xi_1\vee f^{n+1}\xi_2\right)\\
		&= H_\mu\left(\xi_1\left|f\xi_1\vee f^{n+1}\xi_2\right.\right)+H_\mu\left(\xi_2\left|f^{-n}\xi_1\vee f\xi_2\right.\right).
	\end{align*}
	
	Let us recapitulate. We just prove that the following equality holds for every $n\geq 0$
	\begin{equation}
		\label{eq:mean_entropy3}
		h_\mu(f,\xi_1\vee\xi_2)=H_\mu\left(\xi_1\left|f\xi_1\vee f^{n+1}\xi_2\right.\right)+H_\mu\left(\xi_2\left|f^{-n}\xi_1\vee f\xi_2\right.\right).
	\end{equation}
	Observe that $f^{-n}\xi_1$ generates (by Item 4.~of the definition) so the second term tends to $0$ as $n\to\infty$. We must now prove that the first term converges to $H_\mu(\xi_1|f\xi_1)=h_\mu(f,\xi_1)$. We clearly have $H_\mu(\xi_1|f\xi_1\vee f^{n+1}\xi_2)\leq H_\mu(\xi_1|f\xi_1)$ (see \cite[\S 5.10]{Rok}).
	
	We will now use that $f$ expands the unstable manifold so that for most points $x\in M$, the atom $f\xi_1(x)$ is contained in an atom of $f^{n+1}\xi_2$. Consider the Borel set $D_n$ of such $x$. On the one hand $f\xi_1=f\xi_1\vee f^{n+1}\xi_2$ in restriction to $D_n$. On the other hand, since $f^{-1}$ contracts unstable manifolds we have $\mu(D_n)\to 1$ as $n\to\infty$. This yields $\lim_{n\to\infty} H_\mu(\xi_1|f\xi_1\vee f^{n+1}\xi_2)\geq H_\mu(\xi_1|f\xi_1)$, thus concluding the proof.
\end{proof}

This allows us to give sense to the following definition (see Section~\ref{unsentropy}).

\begin{definition}[Entropy along the unstable direction]\index{entropy of a transformation!unstable} The $\mu$-entropy of $f$ along the unstable direction is the value
	$$h_\mu^u(f)=h_{\mu}(f,\xi),$$
	where $\xi$ is any measurable partition subordinate to the unstable foliation $\W^u$.
\end{definition}

\subsubsection{Local entropy}\label{ssc:local_entropy} It is often convenient to work with a local version of entropy which is due to M. Brin and A. Katok see \cite{BK}. The construction is quite general, but for the sake of clarity we will state their results in our context.

Let us first define \textit{dynamical balls}. Given $x\in M$, $n\in\N$ and $r>0$ we define 
$$B_n(x,r)=\left\{y\in M;\,d(f^i(x),f^i(y))< r,\forall\,i=0,\ldots, n-1\right\}.$$
Let $\mu$ be an ergodic $f$-invariant measure. Given $x\in M$, set
$$\underline{h}_\mu(f,x)=\lim_{r\to 0}\limi_{n\to\infty}-\frac{1}{n}\log\,\mu(B_n(x,r)),$$
and
$$\overline{h}_\mu(f,x)=\lim_{r\to 0}\lims_{n\to\infty}-\frac{1}{n}\log\,\mu(B_n(x,r)).$$

\begin{theorem}[Brin--Katok]
	\label{t:brinkatok}\index{theorem!Brin--Katok}
	Let $f:M\to M$ be a $C^{1+\beta}$ diffeomorphism of a compact manifold and
	$\mu$ an ergodic $f$-invariant measure. Then for $\mu$-almost every $x\in M$
	$$h_\mu(f)=\underline{h}_\mu(f,x)=\overline{h}_\mu(f,x)=\lim_{r\to 0}\lim_{n\to\infty}-\frac{1}{n}\log\,\mu(B_n(x,r)).$$
\end{theorem}

Following this classical work, F. Ledrappier and L.-S. Young adopted a pointwise approach for defining the entropy along the unstable direction, which works well in the ergodic case. 

We will let $d^u$ denote the Riemannian distance on unstable manifolds induced by the ambient Riemannian structure. Given $x\in M$, $n\in\N$ and $r>0$ we define 
\begin{equation}
	\label{eq:dyn_ball_unstable}
	B^u_n(x,r)=\left\{y\in W^u(x);\,d^u(f^i(x),f^i(y))< r,\forall\,i=0,\ldots, n-1\right\}.
\end{equation}

We will now consider a partition $\xi$ subordinate to $\W^u$ and a system $(\mu_x^u)_{x\in M}$ of conditional measures of $\mu$ associated with $\xi$, uniquely defined up to a $\mu$-negligible set by Rokhlin's theorem (see Appendix~\ref{App:rokhlin}). We will define
$$\underline{h}^u_\mu(f,x,\xi)=\lim_{r\to 0}\limi_{n\to\infty}-\frac{1}{n}\log\,\mu_x^u(B^u_n(x,r)),$$
and
$$\overline{h}_\mu(f,x,\xi)=\lim_{r\to 0}\lims_{n\to\infty}-\frac{1}{n}\log\,\mu_x^u(B^u_n(x,r)).$$
In the second part of their work, F. Ledrappier and L.-S. Young prove the following theorem (see \cite[Proposition 7.2.1.~and Corollary 7.2.2.]{LYII}).

\begin{theorem}
	\label{t:LedYoun_local}
	Let $f:M\to M$ be a $C^{1+\beta}$ diffeomorphism of a compact manifold  and let $\mu$ be an ergodic $f$-invariant measure. Then for $\mu$-a.e.\ $x\in M$
	$$\underline{h}^u_\mu(f,x,\xi)=\overline{h}^u_\mu(f,x,\xi)=\lim_{r\to 0}\lim_{n\to\infty}-\frac{1}{n}\log\,\mu_x^u(B^u_n(x,r)),$$
	and the common value is
	$$H_\mu(\xi\mid f\xi)=h^u_{\mu}(f).$$
\end{theorem}

\subsection{Measures satisfying the entropy formula}

\subsubsection{All the expansion occurs in the unstable direction}

The principal accomplishment of Ledrappier--Young's first paper \cite{LYI} is the proof of the following key result which says that all the expansion of $f$ occurs in the unstable direction (cf.~Section \ref{sec:led}).

\begin{theorem}\index{theorem!Ledrappier--Young}
	\label{t:tec_LYI}
	Let $f:M\to M$ be a $C^{1+\beta}$ diffeomorphism of a compact manifold and $\mu$ be an ergodic $f$-invariant  probability measure. Then
	$$h_{\mu}(f)=h_\mu^u(f).$$
\end{theorem}

In \cite{L}, Ledrappier had already proved a similar statement for measures without zero Lyapunov exponents, and, as we will see later on, knew how to deduce the conclusion of Theorem \ref{t:LYI} from the equality
$$h_\mu^u(f)=\sum_{\lambda^j>0}m^j\lambda^j.$$

\subsubsection{The uniformly hyperbolic case} The case of uniformly hyperbolic dynamics is certainly an oversimplification of the general context. Nevertheless, we may find useful to understand the skeleton of the Ledrappier--Young's delicate argument and the difficulties therein.

Let us assume here that $f$ is an Anosov diffeomorphism. This means that there is a $Df$-invariant splitting $TM=E^s\oplus E^
u$ where $E^s$ and $E^u$ are respectively uniformly contracted and expanded by $Df$.

Using the local product structure, every sufficiently small ball is contained in a foliated chart for $\W^u$ of the form $D^s\times D^u$ where $D^s$ and $D^u$ are small stable and unstable discs respectively. Moreover in such a chart the Riemannian distance is uniformly equivalent to the \textit{Lyapunov distance} which we may define as the $L^1$-distance of the product $D^s\times D^u$.

Given $x\in M$ and sufficiently small $r>0$, the dynamical ball $B_n(x,r)$ (for the Lyapunov distance) is inside $B^s_n(x,r)\times B^u_n(x,r)$. Since $\W^s$ is uniformly contracted, the dynamical ball $B^s_n(x,r)$ does not depend on $n\in\N$. Using the transverse continuity of the restriction of the induced Riemannian structure on leaves of $\W^u$ we see that there exist $r_1<r_2$ converging to $0$ with $r$ such that
\begin{equation}\label{eq:decdynball}
	\bigcup_{y\in B^s_n(x,r)} B_n^u(y,r_1)\dans B_n(x,r)\dans \bigcup_{y\in B^s_n(x,r)} B_n^u(y,r_2).
\end{equation}
Let $\mu$ be an ergodic $f$-invariant probability measure and $\mu_y^u$, a system  of conditional measures of $\mu$ along unstable plaques of a chart $D^s\times D^u$ containing $B_n(x,r)$. Using \eqref{eq:decdynball} and the fact that $B_n^s(x,r)$ does not depend on $n$ we find a constant $C(r)>0$ such that
$$C(r)\,\essinf_y\, \mu^u_y(B_n^u(y,r_1))\leq\mu(B_n(x,r))\leq C(r)\,\esssup_y\, \mu^u_y(B_n^u(y,r_2)).$$
Using the pointwise versions of measure entropy one deduces
$$h_\mu(f)=h^u_\mu(f).$$

\subsubsection{Some words about the general case} The general case is much more delicate and actually the authors of \cite{LYI} don't follow such a naive pointwise approach. The most immediate difficulty is that the hyperbolicity is not uniform and we must work inside Lyapunov charts, which leads to important technicalities. But the true difficulty of the paper is to understand and analyse the role of zero Lyapunov exponents.

Let us explain some of the difficulties. For a regular point $x\in M$ one may consider a \textit{Lyapunov chart} at $x$. 
We won't enter into the details of the definition here, let us just say that this is an open set $U_x$ which is foliated by local unstable manifolds $W^u_{loc}(y)$ (which are well defined for $\mu$-a.e.~$y$). 
The size of these unstable manifolds depends on $y$ and is not uniform a priori. 

The situation is similar to what we saw in the uniformly hyperbolic setting. 
One has a system of coordinates $T\times D^u$ around the point $x$, $D^u$ being a small unstable disc and $T$ being a small transversal to the unstable foliation. 
This time $T$ is not uniformly contracted and one may think of $T$ as a center stable set. And we must analyse how $f$ acts on such sets. The most important difficulty here is the following.

\begin{center}\textit{There is no canonical choice of a transverse distance on $T$}.
\end{center}
More precisely we want to show that the separation of unstable plaques is less than the expansion along unstable manifolds. We know that this is the case for the expansion along a central transversal. But this apparent weaker expansion could be a lure and be caused by the effect of a separation \textit{inside unstable plaques}. And the information that unstable holonomies are H\"older continuous, is not enough a priori to rule out the possibility that the actual separation of unstable plaques (the dynamics in the quotient by unstable plaques) is stronger than the expansion along unstable manifolds. The authors treat this difficulty by proving that

\begin{center}\textit{unstable holonomies inside center-unstable manifolds are Lipschitz}
\end{center}

We don't enter here into the technical details of the statement and refer to \cite[\S 4.2]{LYI}). Before we carry on with Ledrappier--Young's theory let us mention that a similar difficulty appears in Hirsch--Pugh--Shub's theory of \textit{normally hyperbolic laminations}. In \cite[\S 7]{HPSbook} the authors consider a diffeomorphism $f$ with an invariant normally hyperbolic lamination and study the neighbouring diffeomorphisms. It is not quite true that close to $f$, a diffeomorphism $g$ has an invariant lamination which is (leaf)-conjugate to that of $f$. When one applies the graph transform to such a $g$, the leaves of the lamination can merge: the phenomenon of ``sliding along plaques'' could lead to a \textit{branched} invariant lamination (see \cite{BI} for the definition and \cite{CP} for more information). This phenomenon can be avoided by requiring a technical condition, called \textit{plaque expansivity}, under which M. Hirsch, C. Pugh and M. Shub prove that the branched lamination is a true lamination. This plaque expansiveness is satisfied for Lipschitz foliations (see \cite[Theorem 7.2.]{HPSbook}). We don't know examples of partially hyperbolic diffeomorphisms with a foliation tangent to the central bundle that does not satisfy plaque expansiveness, and we don't know how to prove plaque expansiveness for all such foliations.

The argument provided by F. Ledrappier and L.-S. Young works for $C^2$ diffeomorphisms, and this is the only argument of the paper that needs this regularity assumption. A. Brown showed in \cite{B} how this crucial step of Ledrappier--Young's proof can be carried on for $C^{1+\beta}$ diffeomorphisms. The proof then consists in the precise analysis of how $f$ expands the transverse distance and of dynamical balls. One may think of this part as a sophistication of the proof given in the uniformly hyperbolic setting.

\subsubsection{Idea of proof of Theorem \ref{t:LYI}}\label{D:Idea_Proof} Now that we know that $h_\mu(f)=h_{\mu}^u(f)$ we may follow an argument due to Ledrappier \cite{L} and deduce that
$$h_{\mu}(f)=\sum_{\lambda^j>0} m^j\lambda^j\quad\Longrightarrow\quad\mu\text{ is SRB}.$$

\subsubsection*{Unstable jacobian} Let us define the \textit{unstable jacobian}\index{unstable jacobian} of $f$ at $x$ as the quantity 
$$J^u(x)=\left|\Jac(D_xf|_{E^u(x)})\right|.$$
Using Oseledec's and Birkhoff's theorems we get that for $\mu$-a.e.\ $x\in M$
$$\int_MJ^ud\mu=\lim_{n\to\infty}\frac{1}{n}\sum_{i=0}^n\log J^u(f^i(x))=\sum_{\lambda^j>0} m^j\lambda^j.$$

\subsubsection*{Dynamical prescription of densities} It is a well known fact that when an $f$-invariant measure $\mu$ is absolutely continuous along an expanding foliation, the densities along unstable manifolds are \textit{dynamically prescribed}. Denote by $m^u_x$ the Riemannian volume of $W^u(x)$ and assume that
$$d\mu^u_x=\rho\, dm^u_x,$$
where $\rho$ is a positive and measurable function and $\mu^u_x$ is a system of conditional measures of $\mu$ with respect
to a measurable partition $\xi$ associated with $\W^u$ that we will fix from now on. Then for $\mu_x^u$-a.e.\ $y,z\in W^u(x)$

\begin{equation}
	\label{dyn_densities}
	\frac{\rho(z)}{\rho(y)}=\prod_{i=1}^{\infty}\frac{J^u(f^{-i}(y))}{J^u(f^{-i}(z))}.
\end{equation}
In order to derive the equation above we prove the following lemma.

\begin{lemma}
	For $\mu$-almost every $x\in M$ 
	the map given by
	\begin{equation}
		\label{eq:leg}
		g(y)=\frac{\rho(y)}{\rho(f^{-1}(y))}J^u(f^{-1}(y))
	\end{equation}
	is constant on $\xi(x)$. 
\end{lemma}

\begin{proof}
	Note that $\xi\prec f^{-1}\xi$ so an atom of $\xi$ is a countable union of atoms of $f^{-1}(\xi)$. We will consider the measurable function $h$, constant on atoms of $f^{-1}(\xi)$, by setting
	$$h(y)=\mu^u_y(f^{-1}\xi(y)).$$
	The idea now is to write the family of conditional measures with respect to $f^{-1}\xi$ in two ways. Firstly, using that $\xi\prec f^{-1}\xi$ we have that for every Borel set $K\dans M$
	\begin{equation} \label{eq1}
		\begin{split}
			\mu_{f^{-1}\xi(y)}(K\cap f^{-1}\xi(y)) & = \frac{\mu^u_y(K\cap f^{-1}\xi(y))}{\mu^u_y(f^{-1}\xi(y))} \\
			& = \frac{1}{h(y)}\int_{K\cap f^{-1}\xi(y)}\rho(z) dm^u_y(z).
		\end{split}
	\end{equation}
	Secondly, by $f$-invariance of $\mu$ we have
	\begin{equation}\label{eq2}
		\mu_{f^{-1}\xi(y)}(K\cap f^{-1}\xi(y))=\mu^u_{f(y)}\left(f(K)\cap \xi(f(y))\right).
	\end{equation}
	We get from \eqref{eq1}, \eqref{eq2} and from a change of variable formula that
	\begin{align*}
		&\frac{1}{h(y)}\int_{K\cap f^{-1}\xi(y)}\rho(z) dm^u_y(z)= \int_{f(K)\cap \xi(f(y))}\rho(z)dm^u_{f(y)}(z)\\
		= & \int_{f(K\cap f^{-1}\xi(y))}\rho(z)dm^u_{f(y)}(z)
		= \int_{K\cap f^{-1}\xi(y)}\rho(f(z))J^u(z)dm^u_y(z).
	\end{align*}
	Consequently for $\mu$-almost every $y\in M$ and $m^u_y$-almost every $z\in f^{-1}(y)$ we have
	$$\frac{1}{h(y)}\rho(z)=\rho(f(z))J^u(z).$$
	This proves that the function
	$$g\circ f=\frac{1}{h}$$
	is constant on atoms of $f^{-1}\xi$. Finally we have that $g$ is constant on atoms of $\xi$.
\end{proof}

\subsubsection*{Reconstructing $\mu$} Ledrappier's argument is an inductive one. Denote by $\B(\xi)$ the $\sigma$-algebra whose elements are union of atoms of $\xi$. We can see it as a $\sigma$-algebra ``transverse'' to $\xi$. By Rokhlin's theorem (Theorem~\ref{t.rokhlin}), a measure $\nu$ on $M$ is determined by its trace on $\B(\xi)$ and by a system of conditional measures with respect to $\xi$.

We will consider a positive function $\rho$ satisfying \eqref{dyn_densities}. It is proven in \cite[Theorem 3.1. Item viii]{L} that $\log(\rho)$ is H\"older continuous in every atom $\xi(x)$. Hence on every atom $\xi(x)$, the function $\rho$ is uniformly bounded away from $0$ and $\infty$, and in particular it is $m^u_x$-integrable. In order to prove the H\"older continuity of $\log(\rho)$, one uses that on local unstable manifolds, $z\mapsto E^u(z)$ is Lipschitz, that $z\mapsto D_zf$ is H\"older continuous, that $f^{-1}$ contracts unstable manifolds so we can apply the usual distortion controls. We will furthermore normalize $\rho$ so that for $\mu$-almost every $x\in M$
\begin{equation}
	\label{eq:normalize}
	\int_{\xi(x)}\rho(x) dm^u_x =1,
\end{equation}
Define the probability measure $\nu$ on $M$ satisfying both conditions
\begin{enumerate}
	\item $\mu$ and $\nu$ coincide on $\B(\xi)$;
	\item the disintegration  $(\nu_x)$ of $\nu$ with respect to $\xi$ satisfies
	$$d\nu_x=\rho\, dm^u_x,$$
	where $\rho$ satisfies \eqref{dyn_densities} and \eqref{eq:normalize}.
\end{enumerate}

We will prove by induction that $\mu$ and $\nu$ coincide on the $\sigma$-algebra $\B({f^{-n}}\xi)$ for every $n\in\N$. Since $\xi$ is increasing and generating this implies that $\mu=\nu$. All this follows  from the next lemma.

\begin{lemma}
	\label{l:induction}
	Assume that $\int J^ud\mu=h_\mu(f,\xi)$. Then $\mu$ and $\nu$ coincide on $\B({f^{-1}\xi}).$
\end{lemma}

\begin{proof}
	Each element $\xi(x)$ contains countably many atoms of $f^{-1}\xi$ and $\mu$ and $\nu$ coincide on $\B(\xi)$. Hence in order to show that $\mu$ and $\nu$ coincide on $\B({f^{-1}\xi})$ it is enough to prove that $\nu_y(f^{-1}\xi(y))=\mu_y(f^{-1}\xi(y))$ almost everywhere. We consider the derivative
	$$q(y)=\frac{\nu_y\left(f^{-1}\xi(y)\right)}{\mu_y\left(f^{-1}\xi(y)\right)}$$
	which is well-defined almost everywhere and positive. Furthermore we have $\int q\,d\mu=1$. Using Jensen's inequality and the concavity of the logarithm, one has:
	$$\int\log (q)\, d\mu\leq\log\left(\int q\,d\mu\right)=0.$$
	Moreover this inequality is an equality if and only if $\log(q)=0$ $\mu$-almost everywhere, which as mentioned before, would imply that $\mu=\nu$ on $\B({f^{-1}\xi})$.	
	On the one hand, we have by definition
	$$-\int\log \mu_y\left(f^{-1}\xi(y)\right)d\mu(y)=H_\mu(f^{-1}\xi\mid\xi)=h_{\mu}(f,\xi).$$
	On the other hand it is possible, thanks to the definition of $\rho$, to compute explicitly
	$$-\int\log\nu_y\left(f^{-1}\xi(y)\right)d\mu(y)=\int\log J^u\,d\mu.$$
	By hypothesis these two quantities are equal.
\end{proof}

\subsection{Ledrappier--Young II}\label{S:LY2}

The second part of Ledrappier--Young's work focuses on finding a general entropy formula for measures which are not SRB. The formula they find is similar, the role of the multiplicities being replaced by some quantities $\gamma^j$ representing roughly the dimension of the measure $\mu$ in the $E^j$-direction. They prove the following theorem.

\begin{theorem}[Ledrappier--Young II]\index{theorem!Ledrappier--Young}
	\label{t:LYII}
	Let $f$ be a $C^{1+\beta}$ diffeomorphism of a compact Riemannian manifold $M$ and $\mu$ be an ergodic, $f$-invariant probability measure. Let $\delta_j$ denote the dimension of $\mu$ along the $j$-th dimensional unstable manifolds and
	$$\gamma^j=\delta^j-\delta^{j-1}$$
	(with $\delta_0=0$). Then 
	$$h_{\mu}(f)=\sum_{\lambda^j>0}\gamma^j\lambda^j.$$
\end{theorem}

As for \cref{t:LYI}, F. Ledrappier and L.-S. Young proved \cref{t:LYII} only for $C^2$ diffeomorphisms, but again, after the works of A. Brown \cite{B} and of L. Barreira, Ya.~Pesin and J.~Schmeling \cite[Appendix]{BPS} establishing the required Lipschitz regularity of holonomies of intermediate foliations, it holds for $C^{1+\beta}$ diffeomorphisms.

Before entering into the proof of Theorem \ref{t:LYII} we need to introduce some of the objects appearing in the statement.

\subsubsection{Nested foliations and Hausdorff dimension}

\subsubsection*{Unstable foliations} For  a regular point $x\in M$ let
$$\lambda_1(x)>\lambda_2(x)>\cdots>\lambda_{k^u(x)}(x)>0$$
be the positive Lyapunov exponents of $x$. They are associated with a splitting of the tangent space
$$T_xM=E^1(x)\oplus E^2(x)\oplus\cdots\oplus E^{k^u(x)}(x).$$
We assume that $\mu$ is ergodic so $k^u(x)$, $m^j(x)=\dim E^j(x)$ and $\lambda^j(x)$ do not depend on the $\mu$-typical $x$. The intermediate spaces $E^2(x),\ldots,E^{k^u}(x)$ need not to be integrable. But by Pesin's theory there exist $C^2$-immersed manifolds at $\mu$ a.e.~$x$ denoted by $W^1(x),W^2(x),\ldots,W^{k^u}(x)$, tangent at $x$ to the spaces
\begin{eqnarray*}
	& &E^1(x) \\
	& &E^1(x)\oplus E^2(x)\\
	& &\vdots\\
	& &E^1(x)\oplus E^2(x)\oplus\cdots\oplus E^{k^u}(x).
\end{eqnarray*}
These manifolds are dynamically determined as follows
$$W^j(x)=\left\{y\in M;\,\lims_{n\to\infty}\frac{1}{n}\log d(f^{-n}(x),f^{-n}(y))\leq-\lambda^j\right\}.$$
They form (a.e.)~\textit{nested foliations} $\W^1,\ldots,\W^{k^u}$, where $\W^{k^u}$ is the unstable foliation $\W^u$ we have already worked with. The foliation $\W^j$ will be referred to as the $j$-th unstable foliation.

\subsubsection*{Pointwise dimension of measures} Let $X$ be a metric space and $m$ be a probability measure on $X$. Recall the following classical fact. A proof may be found in \cite{Yo1}.

\begin{definition}
	Say the dimension of $m$, denoted by $\dim m$, is well defined and equal to $\alpha$ if for $m$-a.e.~point $x$ the following limit is well defined 
	$$\alpha=\lim_{\eps\to 0}\frac{\log\,m(B(x,\eps))}{\log\eps}.$$
	In that case the dimension of $m$ coincides with its Hausdorff dimension, i.e.
	$$\dim m=\HD(m)=\inf_{m(Y)=1}\HD(Y).$$
\end{definition}

We can adopt this viewpoint and study the dimension of an ergodic measure along unstable manifolds. Using Ledrappier--Strelcyn's argument (see Proposition \ref{p:existence_partition}) one deduces for every $j$ the existence of a measurable partition $\xi^j$ subordinate to the $j$-th unstable manifold $\W^j$. We can moreover ask
$$\xi^1\succ \xi^2\succ \ldots \succ\xi^{k^u}.$$

For $j\in\{1,\ldots,k^u\}$ we denote by $d^j$ the distance in $j$-th unstable leaves for the induced Riemannian structure.
The corresponding balls are denoted by $B^j(x,r)$ and the dynamical balls (by analogy with \eqref{eq:dyn_ball_unstable} in Section~\ref{ssc:local_entropy}) are denoted by $B^j_n(x,r)$, $r>0$ and $n\in\N$.

We can consider for every $j$ a system $(\mu^j_x)_{x\in M}$ of conditional measures of $\mu$ associated with $\xi^j$. We define for $\mu$-a.e.~$x\in M$
$$\underline{\delta}^j(x,\xi^j)=\limi_{\eps\to 0}\frac{\log \mu(B^j(x,\eps))}{\log\eps},$$
$$\overline{\delta}^j(x,\xi^j)=\lims_{\eps\to 0}\frac{\log \mu(B^j(x,\eps))}{\log\eps}.$$
F. Ledrappier  and L.-S. Young prove in \cite[Proposition 7.3.1]{LYII} the following.

\begin{proposition}
	\label{p:dim_well_def}
	The numbers $\underline{\delta}^j=\underline{\delta}^j(x,\xi^j)$ and $\overline{\delta}^j=\overline{\delta}^j(x,\xi^j)$ don't depend on $\xi^j$ nor on the $\mu$-typical $x$. Moreover
	$$\underline{\delta}^j=\overline{\delta}^j.$$
	If $\delta^j$ is the common value then for $\mu$-a.e.~$x\in M$ and every system of conditional measures $(\mu_x^j)_{x\in M}$ associated with a measurable partition $\xi^j$ subordinate to $\W^j$ we have
	$$\delta^j=\lim_{\eps\to 0}\frac{\log \mu(B^j(x,\eps))}{\log\eps}.$$
\end{proposition}

The number $\delta^j$ is called the dimension of $\mu$ on the $j$-th unstable foliation.

\subsubsection*{Pointwise entropies} F. Ledrappier and L.-S. Young also define a  pointwise version of the entropy along the $j$-th unstable foliation. Considering a partition $\xi_j$ subordinate to $\W^j$ and a system of conditional measures $(\mu^j_x)_{x\in M}$ we define
$$\underline{h}^j_\mu(f,x,\xi^j)=\lim_{r\to 0}\limi_{n\to\infty}-\frac{1}{n}\log\,\mu_x^j(B_n^j(x,r)),$$
and
$$\overline{h}_\mu(f,x,\xi^j)=\lim_{r\to 0}\lims_{n\to\infty}-\frac{1}{n}\log\,\mu_x^j(B_n^j(x,r)).$$
They proved a result analogue to Theorem \ref{t:LedYoun_local}:
$$\underline{h}^j_\mu(f,x,\xi^j)=\overline{h}^j_\mu(f,x,\xi^j)=\lim_{r\to 0}\lim_{n\to\infty}-\frac{1}{n}\log\,\mu_x^j(B^j_n(x,r)),$$
and the common value is $$H_\mu(\xi^j\mid f\xi^j)=h^j_{\mu}(f).$$

\subsubsection{Idea of proof for Theorem \ref{t:LYII}}

\subsubsection*{Conformal case} Before explaining the idea of the proof, let us mention the simplest case.

\begin{theorem}[The conformal case]
	\label{t:confcase}
	Let $f:M\to M$ be a $C^2$ mapping and $\mu$ an ergodic $f$-invariant probability measure. Assume that $f$ has a unique Lyapunov exponent $\lambda>0$, $\mu$-a.e. Then
	$$h_\mu(f)=\lambda\dim(\mu).$$
\end{theorem}

\begin{proof} Let us give the main idea. The rigorous proof uses Pesin's theory and in particular, Ma\~n\'e's argument.
	
	Let $n\in\N$ and $\eps>0$. Locally, $f$ ``looks like'' an expansion by $e^\lambda$ and the dynamical ball $B_n(x,r)$ looks like a ball $B(x,re^{-\lambda n})$ (this affirmation is the one that needs Pesin's theory to be made rigorous). So if $\eps=re^{-\lambda n}$ we have
	$$-\frac1n \log \mu(B_n(x,r))\sim\lambda\frac{\log\mu(B(x,\eps))}{\log\eps}.$$
	The formula follows.
\end{proof}

Of course if $f$ has various Lyapunov exponents, a dynamical ball looks like an ellipsoid whose eccentricity tends to infinity and the previous argument doesn't work as easily. The precise formula is then given by Ledrappier--Young.

\subsubsection{Global strategy} The proof of Theorem \ref{t:LYII} follows three steps.

\begin{enumerate}
	\item $h^1=\lambda^1\delta^1$;
	\item $h^j-h^{j-1}=\lambda^j(\delta^j-\delta^{j-1})=\lambda^j\gamma^j$;
	\item $h^{k^u}=h^u=h_{\mu}(f)$.
\end{enumerate}

The first case is analogous to the conformal case. In restriction to the first unstable manifolds, there is only one Lyapunov exponent and $f$ looks like an expansion by $e^{\lambda^1}$. Here again one has to use Pesin's theory to make this idea rigorous.

In order to consider the second case, one has to collapse $\W^{j-1}$ inside $\W^j$ and to consider ``quotient dynamics'' on the quotient space $\W^j/\W^{j-1}$ (one rather works with quotient partitions $\xi^j/\xi^{j-1}$). Once again one of the main technical issues is that there is no canonical ``transverse distance'' on the quotient, and one has to prove and use the fact that $\W^{j-1}$-holonomies are Lipschitz inside $\W^j$. Once we manage to deal with these important technicalities, we see that the quotient dynamics induced by $f$ on $\xi^j/\xi^{j-1}$ has a unique Lyapunov exponent (this is $\lambda^j$) and that the corresponding entropy and dimension are respectively $h^j-h^{j-1}$ and $\delta^j-\delta^{j-1}$. The situation is one more time analogous to the conformal case.

The third case is treated by Theorem \ref{t:tec_LYI}. Summing those equalities, the second Ledrappier--Young's theorem follows.

\bigskip

\noindent {\scshape S\'ebastien Alvarez}\\
CMAT, Facultad de Ciencias, Universidad de la Rep\'ublica\\
Igua 4225 esq. Mataojo. Montevideo, Uruguay.\\
\texttt{salvarez@cmat.edu.uy}

\smallskip

\noindent {\scshape Mario Rold\'an}\\
Departamento de Matem\'atica, Universidade Federal de Santa Catarina\\
Campus Universit\'ario Trindade,\\
Florian\'opolis - SC - Brazil CEP 88.040-900 \\
\texttt{roldan@impa.br}


\printindex
\clearpage


\end{document}